\documentclass[10pt]{article}
\usepackage[mathscr]{eucal}
\usepackage{graphicx}
\usepackage{amsmath,amsthm,amsfonts,amscd,amssymb}
\usepackage[usenames,dvipsnames]{color}\usepackage[all]{xy} 

\newcommand\mapsfrom{\mathrel{\reflectbox{\ensuremath{\mapsto}}}}

\DeclareMathOperator\FreeProduct{{\hbox{\huge$\ast$}}}

\newcounter{enumitemp}
\newenvironment{enumeratecontinue}{
  \setcounter{enumitemp}{\value{enumi}}
  \begin{enumerate}
  \setcounter{enumi}{\value{enumitemp}}
}
{
  \end{enumerate}
}

\newcommand\displayqed[1]{
\smallskip
\centerline{\hphantom{\qed} \hfill $#1$ \hfill \qed}
}

\newcommand\pref[1]{(\ref{#1})}

\newcommand\ds\displaystyle
\theoremstyle{plain}

\newtheorem*{TheoremA}{Theorem A}
\newtheorem*{TheoremB}{Theorem B}
\newtheorem*{TheoremC}{Theorem C}
\newtheorem*{TwoOverAllThm}{Two Over All Theorem}
\newtheorem*{uniterated}{\TOAT\ (uniterated case)}
\newtheorem*{theorem*}{Theorem}
\newtheorem{theorem}{Theorem}[section]
\newtheorem{proposition}[theorem]{Proposition}
\newtheorem{lemma}[theorem]{Lemma}
\newtheorem{corollary}[theorem]{Corollary}
\newtheorem{fact}[theorem]{Fact}
\newtheorem{LemmaAndDefinition}[theorem]{Lemma and Definition}
\theoremstyle{definition}
\newtheorem{definition}[theorem]{Definition}
\newtheorem{claim}[theorem]{Claim}
\newtheorem{StepThreeFact}{Fact}

\DeclareMathOperator{\Out}{Out}
\DeclareMathOperator{\Aut}{Aut}
\DeclareMathOperator\Inn{Inn}
\DeclareMathOperator{\rank}{rank}
\DeclareMathOperator{\Length}{Length}
\DeclareMathOperator{\Stab}{Stab}
\DeclareMathOperator\diam{diam}
\DeclareMathOperator{\MCG}{\mathsf{MCG}}
\DeclareMathOperator\Isom{Isom}
\DeclareMathOperator\Axis{Axis}
\DeclareMathOperator\MaxEdges{MaxEdges}

\DeclareMathOperator\DFF{D_{\text{FF}}}
\newcommand\Act{{\mathcal A}}

\DeclareMathOperator\corank{corank}
\DeclareMathOperator\Krank{KR}
\newcommand\nat{{\text{nat}}}
\newcommand\Id{\mathrm{Id}}
\newcommand\tsp{\tau_{\text{sp}}}
\newcommand\reals{{\mathbf R}}
\newcommand\Z{{\mathbf Z}}
\newcommand{\bdy}{\partial}
\newcommand{\bdyinf}{\bdy_{\infty}}
\newcommand{\from}{\colon}
\newcommand\suchthat{\bigm|}
\newcommand\inv{{-1}}
\newcommand\union{\cup}
\newcommand\abs[1]{\left| #1 \right|}
\newcommand\intersect{\cap}
\newcommand\meet{\wedge}
\newcommand\restrict{\bigm|}
\newcommand\subgroup{\leqslant}

\renewcommand\L{\mathcal L}
\newcommand\Bline{\mathcal B}
\newcommand\C{\mathcal C}

\newcommand\V{\mathcal V}
\renewcommand\P{\mathcal P}
\newcommand\VG{\V G}
\newcommand\VS{\V S}
\newcommand\VT{\V T}
\newcommand\E{\mathcal E}
\newcommand\EG{\E G}
\newcommand\EH{\E H}
\newcommand\ET{\E T}
\newcommand\ES{\E S}
\newcommand\<\langle
\renewcommand\>\rangle
\newcommand\disjunion\sqcup
\DeclareMathOperator\interior{int}
\DeclareMathOperator\Lip{Lip}
\DeclareMathOperator\cx{cx}
\newcommand\act\curvearrowright
\newcommand\X{\mathcal{X}}
\newcommand\LymanRTTTag{Lyman:RTT}
\newcommand\LymanRTT{\cite{\LymanRTTTag}}
\newcommand\LymanCTTag{Lyman:CT}
\newcommand\LymanCT{\cite{\LymanCTTag}}
\newcommand\BHTag{BestvinaHandel:tt}
\newcommand\BH{\cite{\BHTag}}
\newcommand\BookOneTag{BFH:TitsOne}
\newcommand\BookOne{\cite{\BookOneTag}}
\newcommand\SubgroupsTag{HandelMosher:Subgroups}
\newcommand\FSHypTag{HandelMosher:FreeSplittingHyperbolic}
\newcommand\FSHyp{\cite{\FSHypTag}}
\newcommand\FSOneTag{HandelMosher:FreeSplittingHyperbolic}

\newcommand\FSTwoTag{HandelMosher:FreeSplittingLox}
\newcommand\FSTwo{\cite{\FSTwoTag}}

\newcommand\RelFSOneTag{HandelMosher:RelComplexHyp}
\newcommand\RelFSOne{\cite{\RelFSOneTag}}

\newcommand\RelFSThreeTag{HandelMosher:RelComplexHypIII}
\newcommand\RelFSThree{\cite{\RelFSThreeTag}}

\newcommand\wt\widetilde
\renewcommand\O{{\mathscr O}}

\newcommand\A{\mathscr A}
\newcommand\B{\mathscr B}
\newcommand\F{{\mathscr F}}

\newcommand\Fell[1]{\F_{\!ell} #1}
\newcommand\FellS{{\Fell  S}}
\newcommand\FellT{{\Fell \, T}}

\newcommand\CFFS{\mathcal{F\!F}}

\newcommand\FFC{{\mathcal F}}

\newcommand\FF{\mathcal{F\!F}}

\newcommand\FS{\mathcal{FS}}
\newcommand\PFS{\mathcal{PFS}}
\newcommand\FZ{\mathcal{FZ}}
\newcommand\collapsesto\succ
\newcommand\expandsto\prec
\newcommand\expands\expandsto
\newcommand\Vertices{{\mathcal V}}
\newcommand\Edges{{\mathcal E}}
\newcommand\orb[1]{{\mathcal O} #1}
\newcommand\spray{\text{spray}}
\newcommand\quadtext[1]{\quad\text{#1}\quad}

\newcommand\SetTwo{{\{2\}}}

\DeclareMathOperator\KR{KR}

\newcommand\TOAT{\emph{Two Over All Theorem}}
\newcommand\STOAT{\emph{Strong Two Over All Theorem}}
\newcommand\LPT{\emph{Lipschitz Projection Theorem}}

\newcommand\relA{\emph{rel}~$\A$}

\title{Relative Free Splitting and Free Factor Complexes II: \\
Stable Translation Lengths and the \\
Two Over All Theorem
}

\author{Michael Handel and Lee Mosher}

\makeindex
 
\begin{document}

\maketitle

\begin{abstract}
This is the second of a three part study of free splitting and free factor complexes of a group~$\Gamma$ relative to a free factor system~$\A$. Here and in Part~III, given a relative outer automorphism $\phi \in \text{Out}(\Gamma;\A)$, we study the stable translation length $\tau_\phi \ge 0$ of the simplicial isometry of the relative free splitting complex $\mathcal{FS}(\Gamma;\A)$ determined by~$\phi$, stating and proving quantitative generalizations of earlier theorems for $\text{Out}(F_n)$. The main tool developed here in Part~II is the \TOAT, which expresses a uniform exponential flaring property along arbitrary Stallings fold paths in $\mathcal{FS}(\Gamma;\A)$, a new result even for $\text{Out}(F_n)$. We apply this theorem to prove that if $\phi \in \text{Out}(\Gamma;\A)$ has a filling attracting lamination then, letting $\lambda_\phi > 1$ denote the corresponding expansion factor, $\tau_\phi$~has an upper bound of the form~$B \log(\lambda_\phi)$. We also apply it to show that the natural map from the relative outer space ${\mathscr O}(\Gamma;\A)$ to the relative free splitting complex $\mathcal{FS}(\Gamma;\A)$ is coarsely Lipschitz, with respect to the log-Lipschitz semimetric on~${\mathscr O}(\Gamma;\A)$. 
\end{abstract}

\section{Introduction to Part II}

\noindent\textbf{Stable translation lengths.} The \emph{stable translation length} of an isometry $f \from X \to X$ on a metric space $X$ is an asymptotic invariant of $f$ under iteration, a real number denoted $\tau_f \ge 0$ defined by the following limit:
$$\tau_f = \lim_{n \to \infty} \frac{d(f^n(x),x)}{n}
$$
Fekete's Lemma on subadditive functions gives existence of this limit \cite{Fekete}, and it is well-defined independent of~$x$. When $X$ is Gromov hyperbolic we have $\tau_f > 0$ if and only if $f$ is \emph{loxodromic}, meaning that for any $x$ the orbit map $\mathbb Z \to X$ defined by $n \mapsto f^n(x)$ is a quasi-isometric embedding. Also, $f$ is \emph{elliptic} if and only if orbits are bounded, hence $\tau_f=0$. Finally, $f$ is \emph{parabolic} if and only if $\tau_f=0$ and orbits are unbounded. 

In the course of proving Gromov hyperbolicity of the curve complex $\C(S)$ for $S$ a finite type, oriented surface, Masur and Minsky dynamically classified mapping classes $\phi \in \MCG(S)$ by the asymptotics of their actions on $\C(S)$: $\phi$ is either loxodromic or elliptic; and $\phi$ is loxodromic if and only if it is pseudo-Anosov \cite{MasurMinsky:complex1}. In the loxodromic case they find an upper bound, depending only on~$S$, for the ratio $\tau_\phi / \log(\lambda_\phi)$ where $\lambda_\phi > 1$ is the expansion factor of $\phi$ acting on its unstable measured foliation \cite[Section 4]{MasurMinsky:complex1}. Subsequently Bowditch found a positive lower bound for the case $\tau_\phi >0$, in fact he showed that $\tau_\phi$ is a rational number with a positive denominator bounded above by a constant depending only on~$S$ \cite[Corollary 1.5]{Bowditch:tight}. 

\medskip
\noindent\textbf{(Relative) free splitting and free factor complexes.} In studying the geometry of the outer automorphism group $\Out(F_n)$ of a rank~$n$ free group $F_n$ by analogy with $\MCG(S)$, the hyperbolic role of the curve complex $\C(S)$ is played by several different actors: 
the free factor complex $\FFC(F_n)$ \cite{BestvinaFeighn:FFCHyp}; the free splitting complex $\FS(F_n)$ \FSHyp; and the cyclic splitting complex $\FZ(F_n)$ \cite{Mann:Hyperbolicity}. Each individual outer automorphism $\phi \in \Out(F_n)$ acts on each of these complexes either loxodromically or elliptically, although the dichotomies vary amongst the complexes: $\phi$ acts loxodromically on $\FFC(F_n)$ if and only if $\phi$ is fully irreducible \cite{BestvinaFeighn:FFCHyp}; $\phi$~acts loxodromically on $\FS(F_n)$ if and only if $\phi$ has a filling attracting lamination \cite{HandelMosher:FreeSplittingLox}; and $\phi$~acts loxodromically on $\FZ(F_n)$ if and only if it has a $\Z$-filling attracting lamination \cite{GuptaWigglesworth:Centralizers}. 

In Part~I \RelFSOne, we relativized large scale geometric results about $\FS(F_n)$ and $\FFC(F_n)$ by proving hyperbolicity of the relative free splitting complex $\FS(\Gamma;\A)$ and of the complex of relative free factor systems
$\CFFS(\Gamma;\A)$,\footnote{For the natural quasi-isometric embedding of $\FFC(\Gamma;\A)$ as a subcomplex of $\CFFS(\Gamma;\A)$, see Proposition 6.3 of Part~I \RelFSOne\ .} 
given a group $\Gamma$ and a free factor system $\A$ of~$\Gamma$. 

Our main goal here and in Part III \RelFSThree\ is to establish a dynamical classification of elements of the natural action of $\Out(\Gamma;\A)$ on the relative free splitting complex $\FS(\Gamma;\A)$, generalizing the results for $\Out(F_n)$ but with mostly new proofs, and with strong new quantitative conclusions even for $\Out(F_n)$, thus finding new analogues of results from $\MCG(S)$.

\paragraph{Theorems A, B and C.}
These three theorems are each concerned with a group $\Gamma$ and a free factor system~$\A$ of $\Gamma$, and with dynamics of elements of $\Out(\Gamma;\A)$ acting on the relative free splitting complex $\FS(\Gamma;\A)$, including quantitative bounds on stable translation lengths in terms of constants that depend only on simple numerical invariants $\Gamma$ and $\A$. For the theory of filling laminations and their associated expansion factors see Section~\ref{SectionAnAxisInFS} for brief initial explanations, in particular Proposition~\ref{PropAxisInFS} and Definition~\ref{DefFillExpFact}; see also Section~\ref{SectionAxisConstruction} for more details on that theory, based on work of Lyman \cite{\LymanRTTTag,\LymanCTTag}. 

\begin{TheoremA}
There exist $A,B>0$ such that for any $\phi \in \Out(\Gamma;\A)$, if $\phi$ has a filling lamination with associated expansion factor $\lambda_\phi$ then \, $A \le \tau_\phi \le B \log(\lambda_\phi)$.
\end{TheoremA}


\begin{TheoremB}
There exists $\Omega \ge 1$ such that for any $\phi \in \Out(\Gamma;\A)$ the following are equivalent:
\begin{enumerate}
\item\label{ItemFillingLamExists}
$\phi$ has a filling attracting lamination;
\item\label{ItemActLox}
$\phi$ acts loxodromically on $\FS(\Gamma;\A)$;
\item\label{ItemActUnbdd}
$\phi$ acts with unbounded orbits on $\FS(\Gamma;\A)$;
\item\label{ItemActLargeOrbit}
Every orbit of the action of $\phi$ on $\FS(\Gamma;\A)$ has diameter $\ge \Omega$.
\end{enumerate}
It follows that no element of $\Out(\Gamma;\A)$ acts parabolically on $\FS(\Gamma;\A)$.
\end{TheoremB}

\begin{TheoremC}
For any $\phi \in \Out(\Gamma;\A)$, $\phi$ fixes a point of $\FS(\Gamma;\A)$ if and only if the set of attracting laminations of $\phi$ does not fill $\Gamma$ rel~$\A$.
\end{TheoremC}

Here in Part~II we prove the upper bound in Theorem~A, and we prove Theorem C. The proof of the lower bound in Theorem~A, and the entire proof of Theorem B, are found in Part III \RelFSThree. See also Part~III Section~5 for some partial results regarding analogues of Theorems~A and~B for the complex of relative free factor systems $\CFFS(\Gamma;\A)$.

In each of the above theorems one is given a group $\Gamma$ and a free factor system~$\A$ of~$\Gamma$. All constants will depend on $\Gamma$ and~$\A$, such as the constants $A$, $B$, $\Omega$ in Theorems A and B, and the constant $\Delta$ in the \TOAT\ below. More precisely, when we write a constant in the form $C=C(\Gamma;\A)$ we mean that $C$ depends only on two non-negative integer valued combinatorial invariants of the pair $(\Gamma,\A)$: the number of ``components'' of $\A$ denoted~$\abs{\A}$; and $\corank(\A) = \corank(\Gamma;\A) \ge 0$ which is the rank of a free ``cofactor'' of the free factor system $\A$ in the group $\Gamma$ (see Section~\ref{SectionROut}). Note that constants of this form are completely independent of the isomorphism types of the actual free factors whose conjugacy classes comprise the free factor system~$\A$. In the special case $\A=\emptyset$, the group $\Gamma \approx F_n$ is free of some finite rank $n \ge 2$, and so the constants depend only on $\corank(F_n;\emptyset)=n$.

As mentioned earlier, for the special case of $\Out(F_n)$, Theorem~C and the nonquantitative parts of Theorem~B (i.e.~the equivalencies \pref{ItemFillingLamExists}$\iff$\pref{ItemActLox}$\iff$\pref{ItemActUnbdd}) were first proved in our earlier work \FSTwo, using heavy relative train track machinery for the implication \pref{ItemActUnbdd}$\implies$\pref{ItemFillingLamExists}. Our intention in Parts~II and~III is to introduce new proofs based on new geometric tools augmented with lighter train track machinery, to develop these tools and proofs for the general case of $\Out(\Gamma;\A)$, and to give stronger, quantitative conclusions, namely: Theorem~A; and the additional equivalent statement~\pref{ItemActLargeOrbit} in Theorem~B. On the other hand our proof of Theorem~C in Section~\ref{SectionThmCProof} will be a straightforward generalization of the proof for $\Out(F_n)$ given in \FSTwo. 

%

\smallskip
We turn now to an outline of the contents of Part II.

\medskip\noindent
\textbf{Section~\ref{SectionTOAT}: Stating the \TOAT.} The main tool here in Part II is the \TOAT, which expresses a uniform exponential growth property along Stallings fold paths in $\FS(\Gamma;\A)$; this theorem is new even for $\FS(F_n)$. 

As shown in \RelFSOne, every Stallings fold path in $\FS(\Gamma;\A)$ can be naturally reparameterized to form a $K,C$-quasigeodesic path with constants $K \ge 1$, $C \ge 0$ independent of the path. Furthermore, any two vertices of $\FS(\Gamma;\A)$ can be uniformly perturbed to become endpoints $S,T$ of some Stallings fold path $S = S_0 \mapsto S_1 \mapsto \cdots \mapsto S_K = T$. To explain the notation a bit, each $S_i$ is a free splitting of $\Gamma$ rel~$\A$, each map $f_i \from S_{i-1} \to S_i$ is a ($\Gamma$-equivariant) fold map; and each composed map $f^i_j = f_j \circ\cdots\circ f_{i+1} \from S_i \to S_j$ is \emph{foldable}, a property guaranteeing that for each natural edge $E \subset S_i$ (Section~\ref{SectionRFSC}), each of its iterated images $f^i_j(E) \subset S_j$ is a path \emph{without backtracking} in the tree~$S_j$. 

The \TOAT\ is concerned with how edges grow along Stallings fold paths: Given natural edges $E \subset S_i$ and $E' \subset S_j$, how many times does the image path $f^i_j(E) \subset S_j$ cross natural edges of $S_j$ that are in the $\Gamma$-orbit of $E'$? The theorem gives a certain ``maxi-min'' exponential flaring inequality of this type:

\begin{TwoOverAllThm} (Uniterated Form) There exists an integer constant $\Delta=\Delta(\Gamma;\A) > 0$ such that for any free splittings $S,T$ of $\Gamma$ rel~$\A$, any foldable map $f \from S \to T$, 
if $d_\FS(S,T) \ge \Delta$ then there exist two natural edges $E_1, E_2 \subset S$ in different $\Gamma$-orbits, such that for each natural edge $E' \subset T$, each of the paths $f(E_1),f(E_2) \subset T$ crosses at least 
$2$ edges in the $\Gamma$-orbit of $E'$. 
\end{TwoOverAllThm}

This ``uniterated form'' is the $n=1$ case of the full iterated form of the \TOAT, stated in Section~\ref{SectionTOAStatement}, and having the stronger conclusion that the paths~$f(E_1),f(E_2)$ cross~$2^{n-1}$ edges in the $\Gamma$-orbit of $E'$, under the stronger assumption that $d_\FS(S,T) \ge n\Delta$. 

For a discussion the proof see the heading below on Section~\ref{SectionTwoOverAllProof}. In particular Section~\ref{SectionOneNatOverAll} contains the heart of the proof of exponential growth.

\subparagraph{Motivation for the \TOAT.} Our primary interest is the study of the large scale geometry of $\Out(F_n)$. Concepts developed by Masur and Minsky \cite{MasurMinsky:complex1,MasurMinsky:complex2} for studying the large scale geometry of mapping class groups of surfaces $\MCG(S)$, and generalizations of those concepts to the class of hierarchically hyperbolic groups by Behrstock, Hagen and Sisto \cite{BHS:Hierarchy,BHS:HierarchyII}, have proved surprisingly effective at studying $\Out(F_n)$, \emph{despite} the fact that $\Out(F_n)$ is not a hierarchically hyperbolic group; see for example \cite{HandelMosher:distortion,BestvinaFeighn:subfactor}.

With such motivations, and with the hopes of perhaps discovering new ``hierarchical'' concepts outside of the realm of hierarchically hyperbolic groups, we have engaged in a protracted investigation of the large scale geometry of the \emph{automorphism} group $\Aut(A)$ of a finite rank free group $A \approx F_k$, and of the subgroups $\Aut(A) < \Out(F_n)$ associated to a proper free factor $A < F_n$; see for example \cite{HandelMosher:distortion}, where these subgroups play a role in establishing an exponential lower bound to the Dehn function of $\Out(F_n)$. We suspect that these subgroups might play a deeper role in the large scale geometry of $\Out(F_n)$, analogous to the role played by subgroups $\MCG(F) < \MCG(S)$ when $F \subset S$ is a connected, essential subsurface.

We discovered the \emph{Two Over All Theorem} while studying the action of $\Aut(F_k)$ on the pointed free splitting complex $\PFS(F_k)$ on which it acts; see Section 6 of the \emph{Overview} \cite{HandelMosher:RelHypComplexIntro}, but here is a very brief account. The space $\PFS(F_k)$ may be thought of as the complex of free splittings equipped with a base point, up to base point preserving conjugation of $F_k$-actions; the complex $\PFS(F_k)$ stands in relation to $\FS(F_k)$ the \emph{autr\'e espace} of pointed marked graphs stands in relation to the outer space of unpointed marked graphs (see \cite[\S 1 Section 3]{Vogtmann:OuterSpaceSurvey}). The \TOAT, and its strengthened version in \RelFSThree, may be thought of as a description of flaring behavior, in the sense of \cite{BestvinaFeighn:combination}, within the space $\PFS(F_k)$. We intend in later works to explore applications of the \TOAT\ and the \STOAT\ to the large scale geometry of the \emph{autr\'e espace} of $F_n$, and of $\PFS(F_n)$, and of their $(\Gamma;\A)$ generalizations. One such application is upcoming work of Lyman and the second author \cite{LymanMosher:exponentialDehn} providing exponential lower bounds for the Dehn functions of many groups of the type $\Out(\Gamma;\A)$, including new proofs for the known case of $\Out(F_n)$ when $n \ge 3$ \cite{BridsonVogtmann:GeometryOfAut,HandelMosher:distortion}. 

Meanwhile, beyond these still unrealized motivations, recently discovered applications described below give us an opportunity to explain the \TOAT\ in a focussed setting. 

\medskip\noindent
\textbf{Section~\ref{SectionAppLip}: Application to the \LPT.} Masur and Minsky showed that the natural equivariant ``systole map'' $\mathcal T(S) \mapsto \C(S)$ defined on the Teichm\"uller space $\mathcal T(S)$ is coarsely Lipschitz with respect to the Teichm\"uller metric \cite{MasurMinsky:complex1}; this map assigns, to each hyperbolic structure on $S$, a simple closed geodesic of minimal length. Bestvina and Feighn proved the same coarse Lipschitz conclusion for the ``systole map'' $X(F_n) \mapsto \FFC(F_n)$ defined on Culler-Vogtmann outer space \cite{CullerVogtmann:moduli} equipped with the log-Lipschitz semimetric \cite{Bestvina:BersLike,FrancavigliaMartino:MetricOuterSpace}; this map associates, to each metric marked graph, the conjugacy class of a rank~$1$ free factor represented by a shortest closed circuit in that graph. The systole map $X(F_n) \mapsto \FFC(F_n)$ factors naturally as $X(F_n) \mapsto \FS(F_n) \mapsto \FFC(F_n)$ where the second factor is coarsely Lipschitz \cite{KapovichRafi:HypImpliesHyp}, which raises the question of whether the first factor $X(F_n) \mapsto \FS(F_n)$ is also coarsely Lipschitz. We prove this is so by applying the \TOAT, in the more general setting of the relative deformation space $\O(\Gamma;\A)$:

\begin{theorem*}[Lipschitz Projection Theorem]
The natural equivariant map $\O(\Gamma;\A) \to \FS(\Gamma;\A)$ is coarsely Lipschitz with respect to the log Lipschitz metric on $\O(\Gamma;\A)$ and the simplicial metric on $\FS(\Gamma;\A)$. 
\end{theorem*}

\paragraph{Section~\ref{SectionAppUpperBoundInA}: Application to the upper bound of Theorem~A.} Besides the \TOAT, the proof of the upper bound of Theorem~A depends on the theory of relative train track maps, attracting laminations, and expansion factors, developed for $\Out(\Gamma;\A)$ by Lyman \cite{\LymanRTTTag,\LymanCTTag}, generalizing foundational results for $\Out(F_n)$ by Bestvina, Feighn and Handel \cite{\BookOneTag}. A capsule summary of what we need for this purpose is Proposition~\ref{PropAxisInFS}, stated at the end of Section~\ref{SectionAnAxisInFS}, regarding the existence of ``fold axes'' in $\FS(\Gamma;\A)$ for elements of $\Out(\Gamma;\A)$. The proof of the upper bound of Theorem~A is found in Section~\ref{SectionProofUpperBound}, by combining Proposition~\ref{PropAxisInFS} with the \TOAT. Proposition~\ref{PropAxisInFS} itself is proved in Section~\ref{SectionAxisConstruction}.

See also Section~\ref{SectionAxisConstruction} for the proof of Theorem~C.


\medskip\noindent
\textbf{Section~\ref{SectionTwoOverAllProof}: Proving the \TOAT.} In the opening pages of Section~\ref{SectionTwoOverAllProof} we sketch an early version of our proof of the uniterated version of the \TOAT\ for the special case of $\FS(F_n)$, expressed in terms of marked graphs representing vertices of $\FS(F_n)$, highlighting key ideas and a three step outline. After Section~\ref{SectionUniteratedReduction} which proves that the uniterated form of the \TOAT\ implies the iterated form, and Section~\ref{SectionDistanceBounds} which lays out distance bounds needed for later, in Section~\ref{SectionTwoOverAllOutline} we lay out the three step outline for proving the uniterated form of the \TOAT\ in the general case of $\FS(\Gamma;\A)$. Those three steps are subsequently carried out in Sections~\ref{SectionOneNatOverAll}, \ref{SectionOneNatOverAllNat} and~\ref{SectionTwoOverAllStepThree}. 

\paragraph{Acknowledgements:} We are grateful to Robbie Lyman for conversations and suggestions regarding several aspects of this work.


\setcounter{tocdepth}{3}
\tableofcontents


\section{Stating the \TOAT}
\label{SectionTOAT}

Sections~\ref{SectionROut}--\ref{SectionFolds} are mostly a review of basic theory from \cite[Sections 2 and 3]{\RelFSOneTag}, somewhat reorganized for our present needs. Sprinkled through this material are a few new facts that we will need including: 
\begin{description}
\item[Twisted equivariance principle (in Section~\ref{SectionROut}):] Applying uniqueness of twisted equivariant maps;
\item[Lemma~\ref{LemmaTwEqUnique}~\pref{ItemEqIsoUni}:] Proving uniqueness of a twisted equivariant simplicial isomorphism of free \hbox{splittings;} 
\item[From Lemma/Definition~\ref{LemmaExtension} to Lemma~\ref{LemmaElementary}:] Facts and notations regarding restriction of a free factor system to another free factor, deriving from the Kurosh subgroup theorem; 
\item[Definition~\ref{DefSewingNeedle} and Lemma~\ref{LemmaSewingNeedleFold}:] The role of ``sewing needle folds'' in Stallings fold paths.
\end{description}
The full statement of the \TOAT\ itself is found in Section~\ref{SectionTOAStatement}.

\subsection{Free factor systems and relative outer automorphism groups.}
\label{SectionROut}

\paragraph{Free factor systems.} 

Consider a nontrivial group $\Gamma$. The conjugacy class of a subgroup $A \subgroup \Gamma$ is denoted $[A]$. A \emph{free factor system} in $\Gamma$ is a set of the form $\A = \{[A_1],\ldots,[A_I]\}$ for which there exists a free factorization $\Gamma = A_1 * \cdots * A_I * B$ such that $A_1,\ldots,A_I$ are nontrivial subgroups and $B$ is a (possibly trivial) finite rank free subgroup called a \emph{cofactor} of $\A$; any such free factorization is called a \emph{realization} of~$\A$. An individual element $[A_i] \in \A$ is called a \emph{component} of $\A$. The \emph{external cofactor} of $\A$, denoted $\Gamma / \A$, is the quotient of $\Gamma$ modulo the subgroup normally generated by $A_1 \union\cdots\union A_I$.  The quotient homomorphism $\Gamma \mapsto \Gamma/\A$ restricts to an isomorphism $B \mapsto \Gamma / \A$, for any realization of $\A$ as above. In general, the cofactors of~$\A$ are simply the sections of the quotient homomorphism $\Gamma \mapsto \Gamma/\A$. The \emph{corank} of~$\A$ is well-defined by the formula $\corank(\A) = \rank(B) = \rank(\Gamma/\A)$. Like any set of subgroups forming terms of a free factorization of a group, the subgroups $A_1,\ldots,A_I,B$ are \emph{mutually malnormal} meaning that each one of them is malnormal and the intersection of any conjugates of any two of them is trivial (see e.g.~\cite[Lemma 2.1]{\RelFSOneTag}).

The partial order on free factor systems of $\Gamma$, denoted $\A_1 \sqsubset \A_2$ and variously referred to as \emph{nesting} or \emph{extension}, is defined by requiring that for each $A_1 \subgroup \Gamma$ such that $[A_1] \in \A_1$ there exists $A_2 \subgroup \Gamma$ such that $[A_2] \in \A_2$ and $A_1 \subgroup A_2$. The unique maximum of this partial order is the \emph{full} free factor system $\{[\Gamma]\}$; a \emph{nonfull} free factor system is one which is not full.\footnote{Elsewhere, instead of the ``nonfull'' and ``full'' terminology, we have used ``proper'' and ``improper''. But the latter terminology being overused, we have settled on this new terminology, which has the advantage of agreeing with ``filling lamination'' terminology.} 


\paragraph{Non-elementary groups.} Consider a nontrivial group~$\Gamma$. It is easy to show that $\Gamma$ possesses a non-full free factor system if and only if there exists a free splitting of $\Gamma$ (see Section~\ref{SectionFreeSplittings} below for details); as a special case, $\A=\emptyset$ is a free factor system of~$\Gamma$ if and only if $\Gamma$ is free of finite rank. For any free splitting $T$ of $\Gamma$ the following dichotomy is also easy to show: either $T$ has infinitely many ends, or $T$ has two ends. Furthermore the latter case holds if and only if $\Gamma$ is infinite cyclic or infinite dihedral, in which case the action of $\Gamma$ on $T$ is conjugate to the standard action on the real line of either $\Z$ or $D_\infty \approx \Z / 2 \Z * \Z / 2 \Z$ (in the language of \cite{CullerMorgan:Rtrees} the action of $\Gamma$ on $T$ is either \emph{a shift} or \emph{dihedral}). 

We say that $\Gamma$ is \emph{elementary} if it is either infinite cyclic or infinite dihedral; otherwise $\Gamma$ is \emph{nonelementary}.

Various statements and proofs would have exceptional cases if we allowed $\Gamma$ to be elementary. Also, various statements would be vacuous if $\Gamma$ had no free splittings at all, equivalently if $\Gamma$ has no non-full free factor systems. To avoid these situations we make the following assumption:
\begin{description}
\item[Global hypothesis:] $\Gamma$ is a nonelementary group having a non-full free factor system.
\end{description}
The elementary case does show up in certain contexts, for example when a free splitting action is restricted to the action of an elementary free factor; see for example Lemma~\ref{LemmaElementary}; see also item~\pref{ItemNotEFF} of Definition~\ref{DefinitionAttrLam} where attracting laminations are defined.

\medskip

\paragraph{The outer automorphism group relative to a free factor system.} 
Fixing a group $\Gamma$ satisfying the \emph{Global Hypothesis}, it follows that $\Gamma$ has trivial center and hence the standard homomorphism $\Gamma \mapsto \Inn(\Gamma)$ taking $\gamma \in \Gamma$ to the inner automorphism $i_\gamma \in \Inn(\Gamma)$ (defined by $i_\gamma(\delta) = \gamma \delta \gamma^\inv$) is an isomorphism. The group $\Inn(\Gamma)$ includes as a normal subgroup of $\Aut(\Gamma)$ with quotient group $\Out(\Gamma) = \Aut(\Gamma) / \Inn(\Gamma)$. The group $\Out(\Gamma)$ acts naturally on the set of conjugacy classes of subgroups, inducing a natural action on the set of free factor systems of $\Gamma$, preserving the nesting partial order and the meet operation; this action is well-defined by the formula 
$\phi \cdot \A = \{[\Phi(A)] \suchthat [A] \in \A\}$ for any $\Phi \in \Aut(\Gamma)$ that represents $\phi$. 

\medskip

Fixing a nonfull free factor system~$\A$ of $\Gamma$, the group $\Out(\Gamma;\A)$, called the \emph{outer automorphism group of $\Gamma$ relative to~$\A$}, is defined to be the subgroup of $\Out(\Gamma)$ that fixes~$\A$. Letting $\Aut(\Gamma;\A) \subgroup \Aut(\Gamma)$ be the pre-image of $\Out(\Gamma;\A)$ under the quotient homomorphism $\Aut(\Gamma) \mapsto \Out(\Gamma)$, it follows that $\Aut(\Gamma;\A)$ is the subgroup of all $\Phi \in \Aut(\Gamma)$ that permutes those subgroups $A \subgroup \Gamma$ for which $[A] \in \A$. Clearly $\Inn(\Gamma) < \Aut(\Gamma;\A)$ and $\Out(\Gamma;\A) = \Aut(\Gamma;\A) / \Inn(\Gamma)$. In the special case $\A = \emptyset$ where $\Gamma \approx F_n$ is free of finite rank \hbox{$n = \corank(\emptyset) \ge 2$} we obtain an isomorphism $\Out(\Gamma;\A) \approx \Out(F_n)$.

The class of groups $\Out(\Gamma;\A)$, and aspects of their large scale geometry and dynamics, are the central focus of this work.

\paragraph{Twisted equivariance and actions by $\Aut(\Gamma)$ and its subgroups.} We describe here some simple concepts and facts regarding group actions. For any group $\Gamma$ and any pair of actions $\Gamma \act X,Y$ in some category, and for any $\Phi \in \Aut(\Gamma)$, a morphism $f \from X \to Y$ is said to be \emph{$\Phi$-twisted equivariant} if $f(\gamma \cdot x) = \Phi(\gamma) \cdot x$ for all $\gamma \in \Gamma$ and $x \in X$. Note that for any $\Phi,\Psi \in \Aut(\Gamma)$ and for any composition $X \xrightarrow{f} Y \xrightarrow{g} Z$ of a $\Phi$-twisted equivariant morphism $f$ followed by a $\Psi$-twisted equivariant morphism $g$, the composed map $g \circ f$ is $\Psi \circ \Phi$-twisted equivariant. Note also that for any $\delta \in \Gamma$ with associated inner automorphism $i_\delta(\gamma) = \delta\gamma\delta^\inv$, the map $x \mapsto \delta \cdot x$ from $X$ to itself is $i_\delta$-twisted equivariant. 

Consider now any action $\Gamma \act X$ denoted $\gamma \cdot x$, and any subgroup $H \subgroup \Out(\Gamma)$ with inverse image $\wt H \subgroup \Aut(\Gamma)$, hence $\wt H$ contains every inner automorphism. An action $\wt H \act X$ denoted $\Phi \bullet x$ is said to be an \emph{extension} of the given action $\Gamma \act X$ if for each $\delta \in \Gamma$ and $x \in X$ we have $i_\delta \bullet x = \delta \cdot x$. If this is so then, applying the easily verified identity $\Phi \circ i_\gamma = i_{\Phi(\gamma)} \circ \Phi$ in the group $\Aut(\Gamma)$, it follows that for any $\Phi \in \wt H$, any $\gamma \in \Gamma$, and any $x \in X$, we have 
\begin{align*}
\Phi \bullet (\gamma \cdot x) &= \Phi \bullet (i_\gamma \bullet x) = (\Phi \circ i_\gamma) \bullet x = (i_{\Phi(\gamma)} \circ \Phi) \bullet x = i_{\Phi(\gamma)} \bullet (\Phi \bullet x) = \Phi(\gamma) \cdot (\Phi \bullet x)
\end{align*}
In brief, this shows that the action map $x \mapsto \Phi \bullet x$ is $\Phi$-twisted equivariant. The following is a kind of converse:
\begin{description}
\item[Twisted Equivariance Principle:] For any group action $\Gamma \act X$ and any $H \subgroup \Out(\Gamma)$ as above, if there exists a \emph{unique} $\Phi$-twisted equivariant map $f_\Phi \from X \mapsto X$ for each $\Phi \in \wt H$, then the formula $\Phi \bullet x = f_\Phi(x)$ defines an action $\wt H \act X$ that extends the given action $\Gamma \act X$.
\end{description} 
This follows for two reasons. First, the composed map $f_\Phi \circ f_\Psi$ is $\Phi \circ \Psi$-twisted equivariant, and so $f_\Phi \circ f_\Psi = f_{\Phi \circ \Psi}$ by uniqueness. The formula for $\Phi \bullet x$ therefore does indeed define an action $\wt H \act X$. Also, the maps $x \mapsto f_{i_\delta}(x)$ and $x \mapsto \delta \cdot x$ are both $i_\delta$-twisted equivariant and so these two maps are identical, proving that the action $\wt H \act X$ extends the given action $\Gamma \act X$.

In Lemma~\ref{LemmaTEqBdyMaps} we will see an application of the \emph{Twisted Equivariance Principle}, where the uniqueness hypothesis is verified by applying topological/dynamical properties of the given action $\Gamma \act X$.

\subsection{Relative free splitting complexes} 
\label{SectionRFSC}

\subsubsection{Graphs, paths, directions, and tight maps.} 
\label{SectionGraphsPathsEtc}
A \emph{graph} $G$ is a 1-dimensional $\Delta$-complex, with vertex and edge sets denoted $\Vertices(G)$, $\Edges(G)$ or just $\VG$, $\EG$. A \emph{subgraph} of a graph $G$ is a subcomplex $H \subset G$ of some subdivision of~$G$; its \emph{complementary subgraph} $G \setminus H$ is the closure of the set theoretic complement $G-H$. Two subgraphs \emph{overlap} if, with respect to some subdivision, they have an edge in common.

Without otherwise specifying we work in the PL category: every map between graphs is assumed to be piecewise linear or PL, meaning a map in the category of simplicial complexes with respect to some subdivision of the domain and range. Sometimes we will require a map to be simplicial, and if so this will always be made explicit. Often (but not always) we require that a map take the vertex set of its domain to the vertex set of its range; in particular this is required in the definition of topological representatives (Section~\ref{SectionRTT}). Now and then we will consider continuous maps that are not necessarily PL.  

An \emph{oriented edge} is an edge with an orientation assigned; if $e$ denotes an oriented edge then $\bar e$ denotes the same edge with opposite orientation. \emph{Paths} in a graph $G$ come in two types: either a \emph{trivial path} which takes constant value at some vertex; or a \emph{nontrivial path} which is a finite edge path without backtracking, expressed as a concatenation $\gamma = e_1 \cdots e_K$ of $K \ge 1$ oriented edges such that $\bar e_i \ne e_{i+1}$ for all $1 \le i < i+1 \le K$.

Given a graph $G$ and $v \in \VG$, the set of nontrivial paths in $G$ with initial vertex $v$ is partially ordered by inclusion, this partial ordering generates an equivalence relation, the equivalence classes are called \emph{directions of $G$ at $v$}, and the set of all these directions is denoted $D_v G$. Every $d \in D_v G$ is uniquely represented by a single oriented edge $e$ of $G$ with initial vertex $v$; given such an $e$, we generally abuse notation and use $e$ to denote the direction in $D_v G$ that it represents. A \emph{turn} at $v$ is an unordered pair of directions $\{e,e'\} \subset D_v G$; the turn is \emph{nondegenerate} if $e \ne e'$ and is \emph{degenerate} otherwise. Given a nontrivial path $\gamma = e_1 \cdots e_K$ in $G$, a vertex $v \in G$, and a nondegenerate turn $\{e,e'\}$ at~$v$, to say that \emph{$\gamma$ takes the turn $\{e,e'\}$} means that there exists $1 < i \le K$ such that $\{\bar e_{i-1},e_i\} = \{e,e'\}$. 

Consider two graphs $G,H$ and a map $f \from G \to H$. To say that $f$ is a \emph{tight} map means that $f(\Vertices(G)) \subset \Vertices(H)$ and for each edge $e \subset G$, the restriction $f \restrict e$ is a path. Assuming that $f$ is tight, it follows that for each path $\alpha \subset G$ the restriction $f \restrict \alpha$ is a concatenation of paths, and the unique path in $G$ obtained by straightening $f \restrict \alpha$ (by homotopy rel endpoints) is denoted $f_\sharp(\alpha)$. In this situation the process of \emph{subdividing $G$ so that $f$ is simplicial} refers to the following two step process:  first, $G$ is subdivided so that $\Vertices(G) = f^\inv(\Vertices(H))$, and hence each edge of $G$ now maps to an edge or vertex of $H$; second, the barycentric coordinates on each edge of $G$ are reparameterized so that $f$ becomes a simplicial map.

Given a tight map $f \from G \to H$, consider $v \in \VG$. If $f$ is nontrivial on each edge incident to $v$ then there is an induced \emph{derivative map} $D_v f \from D_v G \to D_{f(v)} H$ defined so that for each oriented edge $e$ with $e \in D_v G$, the direction $D_v f(d)$ is represented by (the first oriented edge in) $f(e)$. We shall regard $D_v f$ as undefined if there exists an edge incident to $v$ on which $f$ is trivial (although one might wish to instead think of $D_v f$ as being partially defined).

\subsubsection{Trees, and their paths, rays, and lines.}
\label{SectionPathsRaysLines}
A \emph{tree} $T$ is a contractible graph, equivalently a connected, 1-dimensional simplicial complex with no embedded circles. Nontrivial paths in $T$ correspond one-to-one with embedded, oriented arcs having endpoints at vertices, equivalently oriented subcomplexes homeomorphic to $[0,1]$. Every finite concatenation of paths may be uniquely \emph{straightened} to get a path with the same endpoints. One sometimes encounters a \emph{hare's path} \cite{Aesop:TortoiseAndHare} which is \emph{not} a path in our current sense (see Section~\ref{SectionGraphsPathsEtc}), but is instead a finite concatenation of paths, each term of that concatenation being either a trivial path or a nontrivial path, such that each edge of the tree is crossed by \emph{at most one} nontrivial term of the concatenation. To straighten a hare's path one simply removes the trivial subpaths, thus leaving the image of the path unchanged; we shall refer to this special case of straightening as \emph{tightening}. 

A \emph{ray} $\rho$ in a tree $T$ is a subcomplex homeomorphic to $[0,\infty)$; a ray in $T$ uniquely determines, and is uniquely determined by, a singly infinite edge path in $T$ without backtracking, expressed as a concatenation $e_0 \, e_1 \, e_2 \, e_3 \, \ldots.$ A \emph{line} $\ell \subset T$ is a subcomplex homeomorphic to~$\reals$; every line is the image of a doubly infinite edge path without backtracking expressed as $\ldots \, e_{-2} \, e_{-1} \, e_0 \, e_1 \, e_2 \, \ldots$; this expression is unique up to translation and reversal, and each such expression uniquely determines a line.

\smallskip

Here is a simple and useful fact about maps between trees, which is easily proved by induction on distance from a base vertex. Given a function $f \from X \to Y$ and a subset $A$, to say that \emph{$f$ is injective on $A$} means that the restricted function $f \restrict A$ is injective.

\begin{lemma}
\label{LemmaTreeMapInjective}
A map of trees $f \from S \to T$ is injective if and only if $f$ is injective on each edge of~$S$ and $f$ is locally injective at each vertex $v \in S$ (meaning that $v$ has a neighborhood on which $f$ is injective).
\qed\end{lemma}

\subsubsection{Free splittings.} 
\label{SectionFreeSplittings}
We freely use the language of Bass-Serre theory \cite{Bass:covering,Serre:trees,ScottWall} as well as various simple facts about group actions on trees \cite{CullerMorgan:Rtrees}.
 
A \emph{free splitting} of a group $\Gamma$ consists of a nontrivial action $\Gamma \act T$ on a simplicial tree which has finitely many edge and vertex orbits, has trivial edge stabilizers, and is minimal meaning that the only $\Gamma$-invariant subtree of $T$ is $T$ itself (our global hypothesis that $\Gamma$ is not elementary implies that $T$ is not a line). We use $\gamma \cdot x \in T$ to denote the result of $\gamma \in \Gamma$ acting on $x \in T$. When $\Gamma$ and its action on $T$ are understood, $T$~itself is used as shorthand for the free splitting. On the other hand the action $\Gamma \act T$ will occasionally be denoted formally as an injective ``action homomorphism'' $\Act_T \from \Gamma \to \Aut(T)$ where $\Aut(T)$ denotes the simplicial automorphism group; with this notation we have $g \cdot x = \Act_T(g)(x)$. On $T$ there is a $\Gamma$-invariant \emph{natural} graph structure in which the \emph{natural} vertices are those vertices of $T$ which either have nontrivial stabilizer or have valence~$\ge 3$. The edges of the given graph structure on $T$ are sometimes called \emph{edgelets}, the intuition being that each natural edge $E \subset T$ is a union of edgelets which can be vary tiny in comparison with~$E$.

Consider a subgraph $\tau \subset T$ with respect to some $\Gamma$-invariant subdivision of~$T$. To say that $\tau$ is \emph{invariant} means that $\Gamma \cdot \tau = \tau$. To say that $\tau$ is \emph{nondegenerate} means that none of its components is a point.

For any free splitting $\Gamma \act T$, each nontrivial $\gamma \in \Gamma$ satisfies exactly one of the following: $\gamma$ is \emph{elliptic} meaning that $\gamma$ fixes a unique vertex of $T$; or $\gamma$ is \emph{loxodromic} meaning that $\gamma$ preserves a unique line $\Axis(\gamma) \subset T$, and the restriction of $\gamma$ to $\Axis(\gamma)$ is a nontrivial translation. Fixing $\Gamma \act T$, to say that a nontrivial subgroup $A \subgroup \Gamma$ is elliptic means that each of its elements is elliptic, implying that there is a unique vertex of $T$ fixed by each element of~$A$. An elliptic subgroup of $\Gamma$ is maximal amongst elliptic subgroups if and only if it is a vertex stabilizer. Every maximal elliptic subgroup is a free factor, and the collection of conjugacy classes of maximal nontrivial elliptic subgroups forms a nonfull free factor system of $\Gamma$ denoted $\FellT$, which may also be expressed as the set of conjugacy classes of nontrivial vertex stabilizers.

Conversely, associated to any nonfull free factor system~$\A$ of $\Gamma$ there exists a free splitting $T$ of $\Gamma$ such that $\A = \FellT$. Here is one such construction from \cite[Lemma 3.1]{\RelFSOneTag}. Start with a choice of free factorization $\Gamma = A_1 * \cdots * A_I * B$ that realizes $\A = \{[A_1],\ldots,[A_I]\}$ $(I \ge 0)$, and a choice of free basis of the cofactor $B = \<\beta_1,\ldots,\beta_L\>$ ($L \ge 0$). Using Lyman's terminology \cite[Section~1]{\LymanCTTag}, these choices determine a graph of groups presentation of $\Gamma$ called a \emph{thistle with $I$ prickles and $L$ petals}: it has a vertex $V$ of valence $I+2L$ with trivial vertex group; it has vertices $W_1,\ldots,W_I$ of valence~$1$ with associated vertex groups $A_1,\ldots,A_I$, each connected to $V$ by an edge with trivial edge group; and it has $L$ oriented edges with both endpoints on $V$, with trivial edge group and labelled by stable letters $\beta_1,\ldots,\beta_L$. The associated Bass-Serre tree $T$ of this graph of groups is a free splitting of $\Gamma$ rel~$\A$ that we call a \emph{thistle} free splitting, and it has the property that $\FellT = \A$.

For any free splitting $\Gamma \act T$, the given simplicial structure on the tree $T$ is a subdivision of a $\Gamma$-invariant \emph{natural} simplicial structure, with a uniquely determined set of \emph{natural} \hbox{vertices $v$} defined as those given vertices for which either $\Stab(v)$ is nontrivial or $v$ has valence~$\ge 3$. When further $\Gamma$-invariant subdivision of $T$ is needed, we use the terminology \emph{edgelets} to refer to the subdivision edges. The \emph{natural edges} of $T$ are, of course, the closures of the complementary components of the natural vertices. For any valence~2 vertex $v \in T$ of the given simplicial structure, to say that $v$ is \emph{essential} means that $v$ is a natural vertex, equivalently $\Stab(v)$ is nontrivial, equivalently $\Stab(v)$ is cyclic of order~$2$.

\subsubsection{Maps and relations amongst free splittings.} Given a map $f \from S \to T$ between free splittings of $\Gamma$, and given $\Phi \in \Aut(\Gamma)$, to say that $f \from S \to T$ is \emph{$\Phi$-twisted equivariant} means that $f(g \cdot x) = \Phi(g) \cdot f(x)$ for all $g \in \Gamma$, $x \in S$. In the special case that $\Phi$ is the identity automorphism, the terminology is abbreviated to say just that $f$ is \emph{equivariant}, meaning $f(g \cdot x) = g \cdot f(x)$. 

\begin{lemma}
\label{LemmaTwEqUnique}
Given free splittings $S,T,U$ and $\Phi \in \Aut(\Gamma)$
\begin{enumerate}
\item\label{ItemOtherAut}
If~$f \from S \mapsto T$ is $\Phi$-twisted equivariant, and if $\Phi' =  i_\gamma \circ \Phi$ for $\gamma \in \Gamma$, then the map $f' \from S \to T$ defined by $f'(x) = \gamma \cdot f(x)$ is $\Phi'$-twisted equivariant.
\item\label{ItemTwistComp}
If $f \from S \to T$ is $\Phi$-twisted equivariant and $g \from T \to U$ is $\Psi$-twisted equivariant then the composition $g \circ f \from S \to U$ is $\Psi \Phi$-twisted equivariant.
\item\label{ItemEqIsoUni}
For each $\Phi \in \Aut(\Gamma)$, a $\Phi$-twisted equivariant simplicial isomorphism $S \mapsto T$, if it exists, is unique. 
\end{enumerate}
\end{lemma}

\begin{proof} Items~\pref{ItemOtherAut} and~\pref{ItemTwistComp} are straightforward calculations, which actually show that \pref{ItemOtherAut} and~\pref{ItemTwistComp} are true for $\Gamma$-actions in any category. 

For proving~\pref{ItemEqIsoUni}, consider two $\Phi$-twisted equivariant simplicial isomorphisms $f,f' \from S \to T$. By postcomposing each with $(f')^\inv$ and applying item~\pref{ItemTwistComp}, we reduce to proving that if $f \from S \to S$ is an equivariant simplicial isomorphism then $f = \text{Id}$. Since $S$ has a fundamental domain that is bounded in the simplicial metric, it follows that $d(f(x),x)$ is uniformly bounded as $x \in S$ varies. Arguing by contradiction and using that $f$ is a simplicial isomorphism, if $f$ is not the identity then there exists a natural vertex $v$ such that $f(v) \ne v$. This vertex $v$ must have valence~$\ge 3$, for otherwise $\Stab(v)$ is nontrivial and $\Stab(f(v))=\Stab(v)$, implying that $f(v)=v$. Amongst the components of $S-v$ there is a unique one $C_+$ containing $f(v)$ and a unique one $C_-$ containing $f^\inv(v)$, and we choose a third component $C_0$ of $S-v$ that is distinct from $C_-$ and $C_+$. It follows that $f(C_0) \subset C_+$ which is disjoint from~$C_0$. For each $x \in C_0$ the path from $x$ to $f(x)$ therefore passes through both $v$ and $f(v)$ and so
$$d(x,f(x)) = d(x,v) + d(v,f(v)) + d(f(v),f(x)) = d(v,f(v)) + 2 \, d(x,v)
$$ 
Since $d(x,v)$ can be arbitrarily large as $x \in C_0$ varies, this contradicts uniform boundedness.
\end{proof}


\textbf{Equivalence, and action by $\Out(\Gamma)$.} Consider two free splittings $S$ and~$T$ of $\Gamma$. To say that $S$ and $T$ are \emph{equivalent}, written $S \approx T$, means that there exists an equivariant homeomorphism $S \mapsto T$; this is equivalent to saying that there exists an equivariant simplicial isomorphism with respect to $\Gamma$-invariant \emph{natural} simplicial structures on $S$ and $T$. We use $[S]$ to denote the equivalence class of~$S$, but we often abuse notation by letting $S$ stand for its own equivalence class. 

The group $\Out(\Gamma)$ acts from the right on the set of equivalence classes of free splittings of~$\Gamma$: for each $\phi \in \Out(\Gamma)$ represented by $\Phi \in \Aut(\Gamma)$, and for each free splitting $S$ of $\Gamma$, the class $[S] \cdot \phi$ is represented by precomposing the action homomorphism $\Act \from \Gamma \to \Aut(S)$ with the automorphism~$\Phi$, yielding the action $\Act \circ \Phi \from \Gamma \to \Aut(S)$ defined by $((\Act \circ \Phi)(\gamma))(x) = (\Act(\Phi(\gamma)))(x)$. Equivalently, given free splittings $S,T$ and $\phi \in \Out(\Gamma)$, the equation $[S] = [T] \cdot \phi$ holds if and only if there exists $\Phi \in \Aut(\Gamma)$ representing $\phi$, and there exists a $\Phi$-twisted equivariant homeomorphism $f \from S \mapsto T$ (Lemma~\ref{LemmaTwEqUnique}~\pref{ItemOtherAut}, this equivalence also holds with the phrase ``there exists $\Phi$'' replaced by ``for all $\Phi$''). The ``right action'' equation is proved by noticing that if $\Phi$ and $f$ witness $[S]=[T] \cdot \phi$ as above, and if $\Psi$ and $g$ witness $[T]=[U] \cdot \psi$, then $\Psi \circ \Phi$ and $g \circ f$ witness $[S]=[U] \cdot \psi\phi$. 

\smallskip\textbf{Collapse.}  A map $\pi \from S \to T$ of free splittings is a \emph{collapse} if the inverse image $\pi^\inv(v)$ of each vertex $v \in T$ is connected. Let $\sigma$ be the union of $\pi^\inv(v)$ over all vertices $v \in T$ such that $\pi^\inv(v)$ does not degenerate to a single vertex of $S$. After subdividing $S$ and~$T$ so that $\pi$ is simplicial, $\sigma$ becomes a nondegenerate subgraph of $S$ called the \emph{collapse forest}, we incorporate it into the notation by writing $\pi \from S \xrightarrow{\langle\sigma\rangle} S / \sigma \approx T$; we note that $S \setminus \sigma$ is a nondegenerate subgraph and that $f$ induces a bijection between the edges of $S \setminus \sigma$ and the edges of the quotient tree~$T$. We note also that a collapse map $S \xrightarrow{\langle\sigma\rangle} T$ is a homeomorphism if and only if it is a simplicial isomorphism if and only if $\sigma = \emptyset$. More generally we have equivalence $S \approx T$ if and only if the collapse map $\pi$ is \emph{trivial}, meaning that each component of $\sigma$ contains at most one natural vertex, equivalently no natural edge of $S$ is contained in~$\sigma$. If $\pi$ is trivial then $\pi$ induces a bijections of natural vertices and of natural edges, and $\pi$ can be tightened to an equivariant homeomorphism by equivariantly tightening its restriction to each natural edge of~$S$.

A key fact about a collapse map $\pi \from S \xrightarrow{\langle \sigma \rangle} T$ is that for each path $\alpha \subset S$ its projection $\pi \circ \alpha \subset T$ is a hare's path: the trivial subpaths of $\pi \circ \alpha$ are the images under $\pi$ of the nondegenerate components of $\alpha \intersect \sigma$; and the nontrivial subpaths of $\pi \circ \alpha$ are the images of the nondegenerate components of $\alpha \intersect S \setminus \sigma$. We write $\pi(\alpha)$ for the path in $T$ with the same image as $\pi \circ\alpha$, obtained from $\pi \circ \alpha$ by tightening. Similarly, for each ray $\rho \subset S$ its image $\pi(\rho)$ is either a ray or a path in $T$ obtained by tightening $\pi \circ \rho$; whether it is a ray or a path depends on whether $\rho$ contains infinitely many edges of $S \setminus \sigma$ or only finitely many edges. Also, for each line $\ell \subset S$ its image $\pi(\ell)$ is either a line, a ray, or a path: writing $\ell$ as a union of two rays $\ell=\rho_1 \union \rho_2$, the three possibilities ``line, ray, path'' depend on \emph{how many} of the two rays $\rho_1,\rho_2$ contain infinitely many edges of $S \setminus \sigma$.

The first sentence of the following lemma can be found in \cite[Lemma 3.2 (1)]{\RelFSOneTag}; the rest is straightforward.

\begin{lemma}\label{LemmaCollapseFFNest}
For each collapse map $S \xrightarrow{\langle\sigma\rangle} T$ we have $\FellS \sqsubset \FellT$. The free factor system $\FellT = \FS{(S / \sigma)}$ is identical to the set of conjugacy classes of nontrivial subgroups of $\Gamma$ that are stabilizers of components of the subforest $\sigma \union \VS$, i.e.\ $\CFFS(\sigma \union \VS)$. Indexing the components of $\sigma$ as $\sigma = \bigcup_{i \in I} \sigma_i$, for each $i \in I$ exactly one of the following holds:
\begin{description}
\item[trivial and finite:] $\Stab(\sigma_i)$ is trivial; this holds if and only if $\sigma_i$ is a finite subtree and every vertex of $\sigma_i$ has trivial stabilizer.
\item[nontrivial and bounded:] $\Stab(\sigma_i)$ is nontrivial and $\sigma_i$ is a bounded subtree; this holds if and only if there is a vertex $V \in \sigma_i$ such that $\Stab(\sigma_i) = \Stab(V)$ and $[\Stab(\sigma_i)] \in \FellS$.
\item[nontrivial and unbounded:] $\sigma_i$ is an unbounded subtree; this holds if and only if $\Stab(\sigma_i)$ is nontrivial and $[\Stab(\sigma_i)] \in \FellT - \FellS$.  \qed
\end{description}
\end{lemma}

Given two free splittings $S,T$ of $\Gamma$, if a collapse map $S \mapsto T$ exists then we say that \emph{$S$ collapses to $T$} denoted $S \collapsesto T$, or equivalently that \emph{$T$ expands to $S$} denoted $T \expands S$. The relation $S \collapsesto T$ induces a partial ordering on the set of equivalence classes of free splittings, formally denoted $[S] \collapsesto [T]$. The action of $\Out(F_n)$ on equivalence classes of free splittings preserves the ``collapse'' relation $[S] \collapsesto [T]$.


\subsubsection{Relative free factor systems and relative free splittings.} Consider a group $\Gamma$ and a nonfull free factor system $\A$ of $\Gamma$. A \emph{free factor system of $\Gamma$ rel~$\A$} is simply a free factor system $\B$ of $\Gamma$ such that $\A \sqsubset \B$. The action of $\Out(\Gamma)$ on the set of all free factor systems of $\Gamma$ restricts to an action of the subgroup $\Out(\Gamma;\A)$ on the set of free factor systems rel~$\A$, and on this subset (the restriction of) the nesting partial order is preserved by the action of $\Out(\Gamma;\A)$. Note that that $\A$ itself is the unique minimum of this restricted partial order; see Part~I \cite[Proposition 2.12]{HandelMosher:RelComplexHyp}. We sometimes emphasize this uniqueness by referring to $\A$ the \emph{(relative) Grushko free factor system}, motivated the consequence of the Kurosh Subgroup Theorem saying that in any finitely generated group there exists a unique minimal free factor system with respect to nesting; see Part~I \cite[Proposition 2.13]{HandelMosher:RelComplexHyp} where this minimum is referred to as the (absolute) Grushko free factor system.

The \emph{meet} of free factor systems of $\Gamma$ is initially defined as a binary operation $\A_1 \meet \A_2$; see 
\cite[Corollaries 2.7 and 2.8]{HandelMosher:RelComplexHyp}.
For any fixed free factor system~$\A$, when the binary meet operation on \emph{all} free factor systems of $\Gamma$ is restricted to the set of free factor systems rel~$\A$, that restriction can then be extended to a multivariate operation $\meet \F_i$ defined for any indexed family $\{\F_i\}$ of free factor systems of $\Gamma$ rel~$\A$, and producing another free factor system of~$\Gamma$ rel~$\A$ by the following formula (see \cite[Corollary 2.15]{HandelMosher:RelComplexHyp}):
\begin{align*}
\meet \F_i &= \{\, [\cap F_i] \, \suchthat \, [F_i] \in \F_i \,\text{for all $i$, and} \, \cap F_i \, \text{is not trivial}\} \\
&=  \text{the unique maximum with respect to $\sqsubset$ of all free factor systems $\B$ rel $\A$} \\
  &\qquad\text{such that $\B \sqsubset \F_i$ for all $i$}
\end{align*}

Given a free splitting $T$ of $\Gamma$, to say that $T$ is a \emph{free splitting rel~$\A$} means that $\A \sqsubset \FellT$; equivalently, $\FellT$ is a free factor system rel~$\A$; equivalently, for each subgroup $A \subgroup \Gamma$ such that $[A] \in \A$, the action of $A$ on $T$ is elliptic. If in addition $\A = \FellT$ then we say that $T$ is a \emph{Grushko free splitting} of $\Gamma$ rel~$\A$; for a specific example see the earlier construction of a thistle free splitting of $\Gamma$ relative to~$\A$.


\begin{definition}[Free factor support of an invariant subforest.]
\label{DefinitionFFS}
Consider a free splitting $T$. Associated to each $\Gamma$-invariant subforest $\tau \subset T$ there is a free factor system denoted $\F[\tau]$, which we sometimes call the \emph{free factor support} of $\tau$, defined to be the set of conjugacy classes of nontrivial stabilizers of components of the subforest $\VT \union \tau$; it follows that $\FellT \sqsubset \F[\tau]$. By Lemma~\ref{LemmaCollapseFFNest} this definition of $\F[\tau]$ is equivalent to the formula 
$$\F[\tau] =  \begin{cases}
\Fell{(T/\tau)} & \quad\text{if $\tau \ne T$} \\
\{[\Gamma]\} &\quad\text{if $\tau=T$}
\end{cases}
$$
\end{definition}

\begin{definition}[Visibility of a free factor system in a free splitting] 
Consider a free splitting $T$ of $\Gamma$. Given a subgroup $F \subgroup \Gamma$, to say that $F$ is \emph{visible} in $T$ means that there is a $\Gamma$-invariant subforest $\tau \subset T$ with one orbit of components such that $F$ is the stabilizer of some component of $T$; it follows that $F$ is a free factor of $\Gamma$. More generally, given a free factor system~$\B$ of $\Gamma$, to say that $\B$ is \emph{visible} in $T$ means that there exists a $\Gamma$-invariant subforest $\tau \subset T$ such that $\B = \F[\tau] = \FS(T/\tau)$. Note that if $T$ is a Grushko free splitting relative to a free factor system $\A$ and if the free factor system $\B$ is visible in $T$ then $\A \sqsubset \B$. Conversely, given a free factor system~$\A$ of $\Gamma$ such that $\A \sqsubset \B$ there exists a Grushko free splitting $T$ of $\Gamma$ rel~$\A$ such that $\B$ is visible in $T$ \cite[Lemma 3.2 (3)]{\RelFSOneTag}. 
\end{definition}

\begin{definition}[Relative free factors]
\label{DefRelativeFreeFactor}
Consider a subgroup $F \subgroup \Gamma$. To say that $F$ is a \emph{free factor rel~$\A$}, or just a \emph{relative free factor} when $\A$ is understood, means that either $F$ is trivial or its conjugacy class $[F]$ is an element of some free factor system rel~$\A$. Assuming that $F$ is a nontrivial free factor rel~$\A$, if $[F] \in \A$ then we say that $F$ and $[F]$ are \emph{atomic}, otherwise $F$ and $[F]$ are \emph{nonatomic}. Also, given a free splitting $T$ rel~$\A$, to say that $F$ is \emph{visible} in $T$ means that the $F$-minimal subtree $T_F$ is disjoint from all of its translates by elements of $\Gamma - F$; equivalently, there exists a free factor system~$\B$ rel~$\A$ such that $\B$ is visible in $T$ and $[F] \in \B$. For example if $F$ stabilizes some vertex $V \in T$ then $V=T^F$ and $F$ is visible. 
\end{definition}

The following lemma regarding free factors rel~$\A$ is based on \cite[Extension Lemma 2.11]{\RelFSOneTag}, which itself is an application of the Kurosh subgroup theorem. The notation defined in items~(\ref{ItemPartitionsA}, \ref{ItemRelFFBottomFactorization})
of this lemma will be used throughout this work.


\begin{LemmaAndDefinition}[Restricting $\A$ to a free factor of $\Gamma$ rel~$\A$]
\label{LemmaExtension} \quad\\
Given a proper, nontrivial subgroup $F \subgroup \Gamma$, the following are equivalent:
\begin{enumerate}
\item\label{ItemIsFFRelA}
$F$ is a free factor rel~$\A$.
\item\label{ItemIsVisibleSomewhere}
There exists a Grushko free splitting $T$ of $\Gamma$ rel~$\A$ in which $F$ is visible.
\item\label{ItemPartitionsA}
There exists a unique partition $\A = \A_F \union \bar\A_F$ such that $\{[F]\} \union \bar\A_F$ is a free factor system rel~$\A$ (either of $\A_F$ or $\bar\A_F$ can be empty).
\end{enumerate}
If these hold then for any cofactor $\bar B_F$ of $\{[F]\} \union \bar\A_F$ the following hold, where $I = \abs{\A_F} \ge 0$ and $J = \abs{\bar\A_F} \ge 0$:
\begin{enumeratecontinue}
\item\label{ItemRelFFBottomFactorization}
There exists a free factorization $\ds\Gamma = \underbrace{\left( \FreeProduct_{i=1}^I A_{i} \right) * B_F}_{= \, F} * \left( \FreeProduct_{j=1}^J A_{I+j} \right) * \bar B_F$
$$\A = \{ \, \underbrace{[A_1],\ldots,[A_I]}_{\A_F} \, , \,\underbrace{[A_{I+1}],\ldots,[A_{I+J}]}_{\bar\A_F} \, \}
$$
It follows that $B = B_F * \bar B_F$ is a cofactor of $\A$. Furthermore,
\begin{enumerate}
\item The subgroups $A_i$ with $1 \le i \le I$ are unique up to conjugacy within~$F$. Denote their conjugacy classes in $F$ by $[A_i]^F$.
\item We obtain a well-defined free factor system $\A \restrict F$ \emph{of the group $F$} having $B_F$ as a cofactor, called the \emph{restriction of $\A$ to~$F$}, together a natural one-to-one correspondence as follows:
$$\A \restrict F =  \{[A_1]^F,\ldots,[A_I]^F\} \leftrightarrow \{[A_1],\ldots,[A_I]\}=\A_F
$$ 
We often abuse notation by using this bijection to \emph{identify} $\A \restrict F \approx \A_F$. 
\end{enumerate}
\end{enumeratecontinue}
\end{LemmaAndDefinition}

\begin{proof} The implication \pref{ItemPartitionsA}$\implies$\pref{ItemIsFFRelA} is immediate, and \pref{ItemIsFFRelA}$\implies$\pref{ItemIsVisibleSomewhere} follows by applying \cite[Lemma 3.2 (3)]{\RelFSOneTag} to any free factor system $\B$ such that $[F] \in \B$. 

Assuming \pref{ItemIsVisibleSomewhere}, to prove~\pref{ItemPartitionsA} note that for any vertex $V \in \VT$, if $\Stab(V)$ is not the identity subgroup of $\Gamma$ then $V \in T^F$ if and only if $\Stab(V) \subgroup F$: for the ``only if'' direction, if $V \in \VT^F$ but $\Stab(V) \not\subgroup F$ then, choosing $g \in \Stab(V)-F$, it follows that $T^F \intersect g \cdot T^F$ contains $V$ and hence is not empty, which contradicts~\pref{ItemIsVisibleSomewhere}. The existence part of~\pref{ItemPartitionsA} now follows by letting $\A_F$ be the set of conjugacy classes of non-identity vertex stabilizers $\Stab(V)$ such that $V \in \VT^F$, and letting $\bar\A_F = \A - \A_F$, and then noting that $\{[F]\} \union \bar\A_F$ is the free factor system rel~$\A$ associated to the $\Gamma$-invariant forest $(\Gamma \cdot T^F) \union \VT$. For the uniqueness part of~\pref{ItemPartitionsA}, first note that \pref{ItemIsVisibleSomewhere} implies malnormality of $F$, because if $g \in \Gamma-F$ then the tree $T^{g F g^\inv}=gT^F$ is disjoint from $T^F$ hence $gFg^\inv$ intersects $F$ trivially. Combining this with the already proved \emph{existence} part of \pref{ItemPartitionsA}, it follows that for each $A \subgroup \Gamma$ such that $[A] \in \A$, exactly one of two alternatives holds: $A$ is conjugate to a subgroup of $F$; every conjugate of $A$ intersects $F$ trivially. The first alternative determines the inclusion $[A] \in \A_F$, and the second determines $[A] \in \bar\A_F$.

Assuming~\pref{ItemPartitionsA}, the existence of a free factorization of the following form is immediate:
$$\Gamma = F * \left( \FreeProduct_{j=1}^J A_{I+j} \right) * \bar B_F, \quad \bar\A_F = \{[A_{I+1}],\ldots,[A_{I+J}]\}
$$
To obtain the further free factorization of $F$, apply \cite[Extension Lemma 2.11]{\RelFSOneTag} to the nested pair of free factor systems $\A \sqsubset \{[F]\} \union \bar\A_F$. All statements after ``In addition\ldots'' are straightforward consequences.
\end{proof}

\begin{definition}[Kurosh rank, following \cite{CollinsTurner:efficient}]
\label{DefKRank}
Consider any nontrivial group $\Gamma$. For any free factor system~$\A$ of $\Gamma$ its \emph{Kurosh rank} is defined to be the integer
$$\Krank(\A) = \Krank(\Gamma;\A) = \abs{\A} + \corank(\A) 
$$
In the setting of Lemma/Definition~\ref{LemmaExtension}, for any relative free factor $F$ of $\Gamma$ rel~$\A$ with restricted free factor system $\A \restrict F$, the corresponding Kurosh rank is given by 
$$\Krank(F) = \Krank(F;\A \restrict F) = \abs{\A_F} + \corank(F;\A_F) = I + \rank(B_F)
$$
As a consequence: $\Krank(F)=0$ if and only if $F$ is the trivial subgroup; also, $\Krank(F)=1$ if and only if either $[F] \in \A$ or $[F] \not\in \A$ and $F$ is infinite cyclic.
\end{definition}
We note the following facts about Kurosh rank, in particular the bound in item~\pref{ItemKRankLength} which will be applied in the proof of the \TOAT\ in Section~\ref{SectionDistanceBounds}.

\subparagraph{Remark.} Because all three of these quantities $\abs{\A}$, $\corank(\A)$, $\Krank(\Gamma;\A)$ are non-negative, and because of the formula $\Krank(\Gamma;\A) = \abs{\A} + \corank(\A)$, it follows that any bounds on $\Krank(\Gamma;\A)$ determine bounds on $\abs{\A}$ and $\corank(\A)$ \emph{and} any bounds on $\abs{\A}$ and $\corank(\A)$ determine bounds on $\Krank(\Gamma;\A)$. We shall usually express things in terms of bounds on $\abs{\A}$ and $\corank(\A)$, because that keeps track of more information.


\begin{lemma}[Corollary of \protect{\cite[Proposition 2.14]{\RelFSOneTag}}]
\label{LemmaKrank}
Given relative free factors $F,F'$ of $\Gamma$~rel~$\A$,
\begin{enumerate}
\item\label{ItemKRankIneq}
If $F' \le F$ then 
\begin{enumerate}
\item\label{ItemKRankIneqStat} $\Krank(F') \le \Krank(F)$. 
\item\label{ItemKRankIneqEq} $\Krank(F')=\Krank(F)$ if and only if $F'=F$.
\end{enumerate}
\item\label{ItemKRankLength}
For any strictly increasing sequence $F_0 < F_1 < \cdots < F_L$ of relative free factors of $\Gamma$ rel~$\A$, its length satisfies the bound $L \le \KR(\Gamma) = \abs{\A} + \corank(\A)$.
\end{enumerate} 
\end{lemma}

\begin{proof}
Item~\pref{ItemKRankLength} is an immediate consequence of~\pref{ItemKRankIneq}. To prove~\pref{ItemKRankIneq}, consider nested free factors $F' \le F$ of $\Gamma$ rel~$\A$. Associated to $F$, consider also the free factorization of $\Gamma$ that is described in Lemma/Definition~\ref{LemmaExtension}~\pref{ItemRelFFBottomFactorization}:
$$(*) \qquad \ds\Gamma = \underbrace{\left( \FreeProduct_{i=1}^I A_{i} \right) * B_F}_F * \left( \FreeProduct_{j=1}^J \bar A_{I+j} \right) * \bar B_F
$$
There is a similar free factorization associated to $F'$ of the form
$$(*)' \qquad \ds\Gamma = \underbrace{\left( \FreeProduct_{i=1}^{I'} A'_{i} \right) * B'_{F'}}_{F'} * \left( \FreeProduct_{j=1}^{J'} \bar A'_{I'+j} \right) * \bar B'_{F'}
$$
but furthermore, by applying \cite[Extension Lemma 2.11]{\RelFSOneTag}, once the free factorization $(*)$ is chosen we may then choose the free factorization $(*)'$ so that $\{A'_i\}_{i=1}^{I'} \subset \{A_i\}_{i=1}^{I}$ and that $B'_{F'} \subgroup B_F$, hence $\Krank(F') = I' + \rank(B'_{F'}) \le I + \rank(B_{F}) = \Krank(F)$ which gives conclusion~\pref{ItemKRankIneqStat}. Furthermore equality of those Kurosh ranks implies $I'=I$ and $\rank(B'_{F'})=\rank(B_{F})$. From $I'=I$ it follows that $\{A'_i\}_{i=1}^{I'} = \{A_i\}_{i=1}^{I}$. And since $B'_{F'}$ and $B_{F}$ are nested finite rank free groups, from equality of their ranks it follows that $B'_{F'}=B'_{F'}$. Thus $F'=F$, proving conclusion~\pref{ItemKRankIneqEq}.
\end{proof}

Recalling the terminology \emph{elementary} for an infinite cyclic or infinite dihedral group, the following is an immediate consequence of Lemma/Definition~\ref{LemmaExtension}~\pref{ItemRelFFBottomFactorization}:

\begin{lemma}
\label{LemmaElementary}
For any free factor $F$ of $\Gamma$ rel~$\A$, if $[F] \not\in\A$ then $F$ is elementary if and only if one of the following holds:
\begin{enumerate}
\item $F$ is an infinite cyclic free factor of some cofactor of~$\A$;
\item $F$ is an infinite dihedral group of the form $F = A * A'$ for order 2 cyclic subgroups $A,A' \subgroup \Gamma$ such that $[A] \ne [A']$ and $[A],[A'] \in \A$. \qed
\end{enumerate}
\end{lemma}

\subsubsection{Relative free splitting complexes.} 
The \emph{free splitting complex} of $\Gamma$ relative to $\A$ is the (directed) simplicial complex denoted $\FS(\Gamma;\A)$, defined to be the geometric realization of the ``collapse'' partial ordering on the set of equivalence classes of free splittings rel~$\A$. In more detail, there is one $0$-simplex of $\FS(\Gamma;\A)$ associated to each equivalence classes of free splittings rel $\A$; we use $[S]$ to denote the $0$-simplex corresponding to a free splitting $S$ rel~$\A$. There is one directed $1$-simplex of $\FS(\Gamma;\A)$ associated to each pair of distinct $0$-simplices $[S]$, $[T]$ for which there exists a collapse map $S \collapsesto T$; we denote this $1$-simplex $[S \collapsesto T]$. In general, associated to each $n+1$-tuple of distinct $0$ simplices $[S_0]$, $[S_1]$, \ldots, $[S_n]$ for which there exists a sequence collapse maps $S_0 \collapsesto S_1 \collapsesto \cdots\collapsesto S_n$, there is an associated $n$-simplex denoted $[S_0 \collapsesto S_1 \collapsesto \cdots\collapsesto S_n]$. 

The action of $\Out(\Gamma)$ on the set of equivalence classes of free splittings preserves the collapse relation $\collapsesto$. Furthermore, that action restricts to an action of $\Out(\Gamma;\A)$ on the set of equivalence classes of free splittings rel~$\A$, also preserving the collapse relation, and thereby inducing a right action of $\Out(\Gamma;\A)$ on $\FS(\Gamma;\A)$. 

\begin{theorem}[\protect{\cite{HandelMosher:RelComplexHyp}}] 
\label{TheoremFSHyp}
The complex $\FS(\Gamma;\A)$, with its simplicial metric, is Gromov hyperbolic.
\end{theorem}

The simplicial structure on $\FS(\Gamma;\A)$ described above can also be viewed as the first barycentric subdivision of a simplicial structure which assigns one $k$-simplex $\sigma(S)$ to each equivalence class of free splittings $S$ of $\Gamma$ rel~$\A$ such that $S$ has exactly $k+1$ orbits of natural edges; \hbox{for details} in the case of $\FS(F_n)$ see \cite[Corollary 1.4]{\FSHypTag} and following passages. Letting vertices of $\sigma(S)$ correspond to edge orbits of $S$, the barycentric coordinates of each point in $\sigma(S)$ correspond to a $\Gamma$-invariant geodesic metric on~$S$ in which each edge is assigned a length equal to corresponding barycentric coordinate value, and an edge length of $0$ is interpreted by collapsing the orbit of that edge. When this is done the individual points of $\FS(\Gamma;\A)$ correspond to equivalence classes \emph{up to equivariant isometry} of free splittings $S$ of $\Gamma$ rel~$\A$ equipped with a geodesic metric that is normalized by requiring that $\Length(S)=1$; here we let $\Length(S)$ denote the sum of edge lengths over one representative edge in each edge orbit. 


\subsection{Fold paths and foldable maps}
\label{SectionFolds}

We review here basic concepts of foldable maps and Stallings fold paths in the context of $\FS(\Gamma;\A)$. We shall emphasize features of the construction that will be used later in Section~\ref{SectionDistanceBounds} for bounding distances in $\FS(\Gamma;\A)$. We shall also formalize ``fold priority'' arguments which exploits the non-uniqueness of the Stallings fold path construction; such arguments can be found in the proof of hyperbolicity of $\FS(F_n)$ in \cite{} and of $\FS(\Gamma;\A)$ in \cite{}. For application very late in the proof of the \TOAT\ we describe a special kind of fold called a \emph{sewing needle fold}, and in the Sewing Needle Lemma~\ref{LemmaSewingNeedleFold} which analyzes the occurrence of sewing needle folds in Stallings fold paths.

\subsubsection{Gates.} Given Grushko free splittings $S,T$ of $\Gamma$ rel~$\A$, an equivariant tight map $S \mapsto T$ always exists; in fact, for any $\Phi \in \Aut(\Gamma;\A)$ a $\Phi$-twisted equivariant tight map $f \from S \to T$ always exists. First, the Grushko requirement implies that $S$ and $T$ have the exact same nontrivial vertex stabilizers, and $\Phi$-twisted equivariance uniquely determines the restriction of $f$ as a bijection from vertices of $S$ with nontrivial stabilizer to vertices of $T$ with nontrivial stabilizer. Next, for every vertex orbit with trivial stabilizers, that requirement also determines $f$ uniquely on the orbit \emph{after} first choosing the image of one representative of that orbit. Finally, having defined $f$ on vertices, it then extends naturally (and tightly) over all edges in a $\Phi$-twisted equivariant manner.

Consider any tight, twisted equivariant map $f \from S \to T$ between free splittings of $\Gamma$ rel~$\A$. Given $v \in \VS$ such that $f$ is injective on each edge incident to $v$, and hence the derivative map $D_v f \from D_v S \to D_{f(v)} T$ is defined, the set $D_v S$ is partitioned into \emph{gates} of $f$ at $v$, where $e,e' \in D_v S$ are in the same gate if and only if $D_v f(e)=D_v f(e')$. A nondegenerate turn $\{e,e'\}$ at $v \in \VS$ is \emph{foldable (with respect to $f$)} if $\{e,e'\}$ is contained in some gate of $f$ at $v$; otherwise $\{e,e'\}$ is \emph{unfoldable}. Consider a path $\gamma$ and a vertex $v$ in the interior of $\gamma$, and let $e,e'$ be the oriented edges contained in $\gamma$ having initial edge $v$; the orientation on one of $e$ or $e'$ agrees with the orientation on $\gamma$, and the other disagrees. If $\{e,e'\}$ is foldable then we say that $\gamma$ \emph{takes a foldable turn at $v$}, otherwise $\gamma$ \emph{takes an unfoldable turn at $v$}.\footnote{We avoid referring to foldable turns as ``illegal'', reserving that terminology for a more restrictive property in the context of a relative train track map; see Section~\ref{SectionRTT}.}


\subsubsection{Fold maps.} Consider an equivariant map $f \from S \to T$ that is tight and is injective on each edge. To say that $f$ is a \emph{foldable map} means that $f$ has at least two gates at each vertex of~$S$, equivalently at each vertex of $S$ there exists an unfoldable turn. We collect here some results about foldable maps that are useful for constructing fold sequences. 


\begin{proposition}[\protect{\cite[Lemma 4.2]{\RelFSOneTag}}]
\label{PropFoldableProps}
For any free splittings $S,T$ of $\Gamma$ the following hold:
\begin{enumerate}
\item\label{ItemFoldableExists}
There exist maps of free splittings $S \mapsfrom S' \mapsto U \mapsto T$ such that $S' \mapsto S$ and $S' \mapsto U$ are collapse maps, $U \mapsto T$ is foldable, and the composition $S' \mapsto U \mapsto T$ is tight; if furthermore $\FellS \sqsubset \FellT$ then we may take $S=S' \mapsto S$ to be the identity. 
\item\label{ItemFoldableFactorization}
For any free splitting $W$ and any maps $S \mapsto W \mapsto T$ each taking vertices to vertices, if the composed map $S \mapsto T$ is foldable then the factor maps $S \mapsto W$ and $W \mapsto T$ are both foldable. \qed
\end{enumerate}
\end{proposition}


\begin{definition}[Folds] 
\label{DefinitionFolds}
As a special case of a foldable map $f \from S \to T$, to say that $f$ is a \emph{fold} means that there exists a vertex $v \in S$, two oriented natural edges $E \ne E'$ with initial vertex~$v$, and initial segments $e \subset E$ and $e' \subset E'$ respectively, such that $e,e'$ are identified by $f$; what this means is that there exists an orientation preserving homeomorphism $h \from e \to e'$, such that the equivalence relation on $S$ that is defined by $f(x)=f(x')$ for $x,x' \in S$ is generated by the relation $\gamma \cdot x \sim \gamma \cdot h(x)$ for all $x \in e$, $\gamma \in \Gamma$. 
\end{definition}

\subparagraph{Remark.} As shown in \cite[Section 4.1]{\RelFSOneTag}, all intersections of the closed arc $e \union e'$ with any of its nontrivial translates are contained in the set $\{w,v,w'\}$ where $w \in e$, $w' \in e'$ are the terminal endpoints. Using this fact one may classify folds using a simplified version of the fold classification scheme found in \cite[Section~2]{BestvinaFeighn:bounding}, which we shall sometimes do later on as needed.

\begin{definition}[Fold paths in the free splitting complex]
\label{DefFoldPaths}
A \emph{foldable sequence} in $\FS(\Gamma;\A)$ is a sequence of free splittings and nontrivial foldable maps
$$\cdots \xrightarrow{f_{l-1}} S_{l-1} \xrightarrow{f_l} S_l \xrightarrow{f_{l+1}} S_{l+1} \xrightarrow{f_{l+2}} \cdots
$$
where $S_l$ is defined for all $l$ on some subinterval of the integers, such that each composition $f^k_l = f_l \circ\cdots\circ f_{k+1} \from S_k \to S_l$ is a foldable map. If in addition each $f_l$ is a fold map then the sequence is called a \emph{(Stallings) fold sequence} or \emph{fold path} in $\FS(\Gamma;\A)$.\footnote{In \protect{\cite[Section 4.2]{\RelFSOneTag}} a technical distinction is drawn between ``fold sequence'' and ``fold path'' which we ignore here.} In particular a finite fold sequence defined for $K \le l \le L$ is called a \emph{fold path from $S_K$ to $S_L$}. 
\end{definition}

The distance bound in the following lemma is used repeatedly:

\begin{lemma}[\protect{\cite[Lemma 4.4]{\FSOneTag}}]
\label{LemmaFoldBound}
Given $S,T \in \FS(\Gamma;\A)$, if there exists a fold $f \from S \to T$ then $d(S,T) \le 2$. In more detail, in $\FS(\Gamma;\A)$ there is a diagram of the form $$\xymatrix{
S \ar@/^1pc/[r]^{g} & U \ar@/^1pc/[r]^h \ar@/^1pc/[l]_{\bar g}^{\<\sigma\>} \ar@/_1pc/[r]^{\hat h}_{\langle\tau\rangle} & T
}$$
in which the upper arrows give a fold factorization $f = h \circ g$, and the lower arrows $\bar g$ and $\hat h$ are collaps maps giving an ``expand--collapse'' path $S \expandsto U \collapsesto T$. It follows that if $d(S,T)=2$ then $d(S,U)=d(U,T)=1$.
\qed
\end{lemma}

A \emph{length $k$ fold} is a fold map $f \from S \to T$ such that $d(S,T)=k$; see the table following the proof of \cite[Lemma 4.4]{\FSOneTag} for a complete classification of length~$0$ folds, length~$1$ folds, and length~$2$ folds.

\subsubsection{Stallings fold theorem.} The version of Stallings fold theorem that we use is taken from \cite[Lemma 4.3]{\RelFSOneTag} and it is a straightforward generalization of the version in \FSHyp. But we shall include more information in the statement of the theorem, by making explicit the choices in the construction, expressed in the language of gates and foldable turns. 

\begin{definition}[Partial fold factorizations. Maximal fold factors]
\label{DefPartialFoldFactorization}
Consider a foldable map $f \from S \to T$ between two free splittings of $\Gamma$ rel~$\A$. A \emph{partial fold factorization} (of $f$, in $\FS(\Gamma;\A)$) is a sequence of maps of free splittings of $\Gamma$ rel~$\A$, each map taking vertices to vertices, and having the form 
$$(*) \qquad f \from S=T_0 \xrightarrow{g_1} \cdots \xrightarrow{g_J} T_J \xrightarrow{h^J} T
$$
such that each $g_j$ is a fold. It follows from Proposition~\ref{PropFoldableProps}~\pref{ItemFoldableFactorization} that the sequence is foldable, and hence $h^J$ and all of the maps $h^j = h^J \circ g_J \circ \cdots \circ g_{j+1} \from T_j \to T$ ($j=0,\ldots,J-1$) are foldable (note that $f=h^0$). The integer $J \ge 0$ is the \emph{length} of the partial fold factorization. If furthermore $h^J$ is a homeomorphism then $(*)$ is a \emph{fold factorization} of $f$.

For each length~$1$ partial fold factorization in $\FS(\Gamma;\A)$ of the form
$$f \from S \xrightarrow{g} U \xrightarrow{h} T
$$
to say that $g$ is a \emph{maximal fold factor (of $f$)} means that if $e,e' \subset S$ are oriented natural edges such $\{e,e'\}$ is a foldable turn with respect to $f$, and if $g$ folds the turn $\{e,e'\}$, and if $\eta \subset e$, $\eta' \subset e'$ are the maximal initial segments that are identified by $f$, then $\eta,\eta'$ are also identified by $g$. We also say that $g$ is \emph{the maximal fold of the turn $\{e,e'\}$} (with respect to $f$). 

Returning to the above partial fold factorization $(*)$ of length~$J$, to say that $(*)$ \emph{has maximal folds} means that for each $j=1,\ldots,J$ the map $g_j \from T_{j-1} \to T_j$ is a maximal fold factor of the foldable map $h^{j-1} \from T_{j-1} \to T$.
\end{definition}

\begin{theorem}[Stallings fold theorem; see \protect{\cite[Lemma 4.3]{\RelFSOneTag}}] 
\label{TheoremStallings}
Any foldable map $f \from S \to T$ has a fold factorization with maximal folds, and a fold factorization with length~$1$ folds. More precisely, 
\begin{enumerate}
\item\label{ItemOneMoarFold}
Given a partial fold factorization $(*)$ of length $J \ge 0$ as in Definition~\ref{DefPartialFoldFactorization},
\begin{enumerate}
\item\label{ItemMoarFolds}
The sequence $(*)$ is a fold factorization if and only if $h^J$ is locally injective, if and only if $h^J$ has no foldable turns (see Lemma~\ref{LemmaTreeMapInjective}).
\item\label{ItemArbitraryFolds}
For each foldable turn $\{e^{\hphantom\prime}_J,e'_J\}$ of $T_J$ with respect to $h^J$, the factorization $(*)$ extends to a partial fold factorization of length $J+1$, by factoring $h^J$ as
$$h^J \from T_J \xrightarrow{g_{J+1}} T_{J+1} \xrightarrow{h^{J+1}} T
$$
where $g_{J+1} \from T_J \to T_{J+1}$ is the maximal fold of the edges $e^{\hphantom\prime}_J,e'_J$ with respect to $h^J$. It follows that if $(*)$ has maximal folds then the extended partial fold factorization also has maximal folds.
\end{enumerate}
\item\label{ItemFoldSequenceBound}
There is an upper bound to the lengths of all partial fold factorizations of $f$ with maximal folds.
\item\label{ItemMaxFoldFactors}
A sequence of foldable turns as in \pref{ItemArbitraryFolds} can be chosen arbitrarily. The result will be a fold factorization of $f$ with maximal folds (by combining \pref{ItemOneMoarFold} and \pref{ItemFoldSequenceBound}).
\item\label{ItemDistOneFoldFactors}
Starting with any fold factorization of $f$ with maximal folds, each maximal fold factor $g_j \from T_{j-1} \to T_j$ of length $2$ can be refactored as two length~$1$ folds. The result will be a fold factorization of $f$ with folds of length~$\le 1$.  \qed
\end{enumerate}
\end{theorem}

\subsubsection{Sewing needle folds.} 
\label{SectionSewingNeedle}
As an addendum to Theorem~\ref{TheoremStallings} we describe the special class of ``sewing needle'' folds, and we prove the \emph{Sewing Needle Lemma}~\ref{LemmaSewingNeedleFold}. This lemma says, roughly, that when constructing a fold factorization of a foldable map using Theorem~\ref{TheoremStallings}, sewing needle folds can always be avoided until the very last stages of the factorization (see the closing remark of this section). 

The Sewing Needle Lemma and concepts of sewing needle folds will not be needed until two places somewhat later here in Part II: in the proof of Theorem~\ref{TheoremStrataLamCorr}~\pref{ItemCorollaryThreeTwoTwo}; and in the third and final step of the proof of the \TOAT\ in Section~\ref{SectionTwoOverAllStepThree}. We place our discussion of sewing needle folds here because it fits quite naturally into the setting of Theorem~\ref{TheoremStallings}.

\begin{definition}[Sewing needle folds and multifolds]
\label{DefSewingNeedle}
Consider a fold map $f \from S \to T$ between two free splittings of $\Gamma$ such that $f$ is not a homeomorphism. To say that $f$ is a \emph{sewing needle fold} means that there exists a natural edge $E$ such that for every turn $\{d,d'\}$ that is folded by~$f$, that turn lies in the subforest $\Gamma \cdot E$; we also say that $f$ is a \emph{sewing needle fold of the natural edge~$E$}. To say that $f$ is a \emph{sewing needle multifold} means that every first fold factor of $f$ is a sewing needle fold.
\end{definition}

\begin{figure}
$$\xymatrix{
& & & & v_{-1} \ar[ddr]_{E_{-1}} & & v_1 \ar[ddr]_{E_1} 
\\
& & & & & & & & & & 
\\
& & & v_{-2} \ar[uur]^{E_{-2}} & \ar[ddll]_f & v_0 \ar[uur]^{E_0} & \ar[ddr]^f & v_2 & 
\\
& & & & & & & & & \\
%
%
%
& v'_{-1} \ar@{-}[d] & & v'_1 \ar@{-}[d] & & 
& v'_{-1} \ar@{-}[ddr] & & v'_1 \ar@{-}[ddl]
\\
& z_{-1} \ar[ddr]^{E'_{-1}} & & z_1  \ar[ddr]^{E'_1} & & & & & \\
& & & & & & & z
\\ 
z_{-2} \ar[uur]_{E'_{-2}} & & z_0 \ar[uur]_{E'_0} & & z_2  & 
\\
v'_{-2} \ar@{-}[u]   & & v'_0 \ar@{-}[u] & & v'_2 \ar@{-}[u] & 
v'_{-2} \ar@{-}[uurr] &&  v'_0 \ar@{-}[uu] && v'_2\ar@{-}[uull]
} 
$$
\caption{Sewing needle folds $f \from S \to T$, with $v_n = \gamma^n \cdot v_0$, and $v'_n = f(v_n) = \gamma^n \cdot v'_0$, and $E_n = \gamma^n E_0$.
The upper portion of the diagram depicts the axis of $\gamma$ in $S$ with fundamental domain $E=E_0$ having initial vertex $v_0$ and terminal vertex $v_1 = \gamma \cdot v_0$. The ``lower left'' version of~$f$ and $T$ shows a nondegenerate sewing needle fold, in which $f(E_0)$ is expressed as the concatenated path $\overline{v'_0 \, z_0 \, z_1 \, v'_1}$, and the action of $\gamma$ of $T$ has an axis with fundamental domain $E'_0 = \overline{z_0 z_1}$. The ``lower right'' version of $f$ and $T$ shows a degenerate sewing needle fold, in which $f(E_0)$ is $\overline{v'_0 \, z \, v'_1}$, and the subgroup $\<\gamma\>$ is the stabilizer of the vertex~$z$.}
\label{FigureSewingNeedle}
\end{figure}

\noindent

We next use the definition to develop a more intuitive and practical understanding of sewing needle folds (see Figure~\ref{FigureSewingNeedle}). Suppose that $f \from S \to T$ is a sewing needle fold of the natural edge $E_0 \subset S$. Choose an orientation of $E_0$ with initial vertex $v_0$, and choose a turn $\{d_0,d'_0\}$ folded by $f$ such that $d_0$ is the initial direction of $E_0$. Let $d'_1$ denote the terminal direction of $E_0$, located at the terminal vertex denoted~$v_1$. Every direction in $\Gamma \cdot E_0$ at a natural vertex of $S$ is in the orbit of either~$d_0$ or~$d'_1$. It follows that there exists a group element $\gamma \in \Gamma$ such that $\gamma \cdot d'_0 \in \{d_0,d'_1\}$. This $\gamma$ is not trivial, because $d'_0 \not\in \{d_0,d'_1\}$. But $\gamma \cdot d'_0 \ne d_0$, because otherwise $\gamma$ would fix the direction $Df(d'_0)=Df(d_0)$ at $f(v_0) \in T$, contradicting the fact that $T$ is a free splitting. We thus have $\gamma \cdot d'_0 = d'_1$ and $\gamma \cdot v_0 = v_1$. Note that $d_0$, $d'_0$ and $\gamma$ are all \emph{uniquely determined} by the choice of $E_0$ in its orbit and the choice of orientation of $E_0$: $d_0$ is the initial direction of $E_0$; $d'_0$ is the unique direction identified with $d_0$ by the fold map $f$; and the equation $\gamma \cdot d'_0 = d'_1$ determines $\gamma$ uniquely because $S$ is a free splitting.

It follows from the above description that the edge $E_0$ is a \emph{loop edge} of $S$, meaning a natural edge whose endpoints are in the same $\Gamma$ orbit; equivalently, the image of $E_0$ in the quotient graph of groups $S/\Gamma$ is a loop. 
Also, the group element $\gamma$ acts loxodromically on the tree $S$, with an axis of the form $\bigcup_n E_n$ having fundamental domain $E_0$, where $E_n = \gamma^n \cdot E_0$; the initial and terminal vertices of $E_n$ are $v_n = \gamma^n \cdot v_0$ and $v_{n+1} = \gamma^{n+1} \cdot v_0$, and its initial and terminal directions are $d_n = \gamma^n \cdot d_0$ and $d'_{n+1} = \gamma^{n+1} \cdot d'_0$. 

There are two types of sewing needle fold, \emph{nondegenerate} and \emph{degenerate} (see again Figure~\ref{FigureSewingNeedle}). To describe them, let $\eta \subset E_0$ and $\eta' \subset \bar E_{-1}$ be the initial segments representing $d_0,d'_0$ that are identified by $f$ (Definition~\ref{DefinitionFolds}), and let $w \in \eta$, $w' \in \eta'$ denote the endpoints opposite $v$, and so $f(w)=f(w')$ is a natural vertex of $T$. Note that the segments $\eta$ and $\gamma \cdot \eta'$ represent the initial and terminal directions $d_0$ and $d'_1$ of $E_0$ respectively, and these two segments are either disjoint or they intersect at their common endpoint $w = \gamma \cdot w'$ in the interior of~$E_0$. In the case that $\eta$ and $\gamma \cdot \eta'$ are disjoint, we say that $f$ is a \emph{nondegenerate} sewing needle fold, and in this situation $\gamma$ acts loxodromically on $T$, with axis $\bigcup_n E'_n$ having fundamental domain $E'_0 = g(E_0 \setminus (\eta \union \gamma \cdot \eta'))$, and with $E'_n = \gamma^n \cdot E'_0$. But if $\eta$ and $\gamma \cdot \eta'$ are not disjoint then $\eta \intersect (\gamma \cdot \eta') = w = \gamma \cdot w'$, and we say that $f$ is a \emph{degenerate} sewing needle fold; in this case $\gamma$ acts elliptically on $T$, and furthermore the infinite cyclic group $\<\gamma\>$ is the stabilizer group of the point $f(w)=f(\gamma \cdot w') \in T$. 

The reason for the ``sewing needle'' terminology is the appearance of the quotient subgraph \break \hbox{$N=(\Gamma \cdot f(E))/\Gamma$} in the quotient graph of groups $T/\Gamma$. For any free splitting $T$ of $\Gamma$, to say that a connected subgraph $N \subset T/\Gamma$ is a \emph{sewing needle} means that there exists a vertex $V$ of valence~$1$ in $N$ such that the following hold: $N-V$ is an open subset of $T/\Gamma$; there exists a vertex $W \ne V \in N$ such that every vertex of $N$ other than $V$ or $W$ is of valence~$2$ and is labelled by the trivial group; and $W$ itself is either of valence~$3$ labelled by the trivial group or of valence~$1$ labelled by an infinite cyclic group. When $W$ satisfies the former then $N$ is a \emph{nondegenerate} sewing needle, and otherwise $N$ is a \emph{degenerate} sewing needle. The ``needle'' itself is the unique path of $N$ with endpoints $V,W$. If $N$ is nondegenerate then the ``eye'' of the sewing needle is the unique loop in $N$ with both endpoints at $W$; and if $N$ is degenerate then the ``eye'' degenerates to the vertex $W$~itself. 

To summarize, given a fold map $f \from S \to T$, to say that $f$ is a sewing needle fold of a natural edge $E$ is equivalent to each of the following statements:
\begin{description}
\item[Axis Characterization:]
For any orientation of $E$ with initial vertex $v$ and initial direction $d$, there exists a unique direction $d' \ne d$ at $v$ and $\gamma \in \Gamma$ such that the turn $\{d,d'\}$ is folded by $f$ and $\gamma \cdot d'$ is the terminal direction of $E$. Furthermore, $\gamma$ acts loxodromically on $S$, and $E$ is a fundamental domain for the action of $\gamma$ on its axis $\bigcup_n \gamma^\cdot E$.
\item[Quotient Characterization:]
In the quotient graphs of groups $S/\Gamma$ and $T/\Gamma$, the subgraph $(\Gamma \cdot E) / \Gamma \subset S / \Gamma$ is a loop, and the subgraph $N = (\Gamma \cdot f(E))/\Gamma \subset T/\Gamma$ is a sewing needle. Furthermore, $f$ is nondegenerate if and only if $N$ is nondegenerate. 
\end{description}

With this understanding of sewing needle folds in hand, we next turn to an understanding of a sewing needle multifold $f \from S \to T$. Since $f$ is not a homeomorphism, it has at least one foldable turn. Choose natural edges $E_1,\ldots,E_K$ in distinct orbits ($K \ge 1$) such that every turn folded by $f$ involves a direction in one of the subforests $\Gamma \cdot E_1,\ldots,\Gamma \cdot E_K$. For each $1 \le i \le K$, choose an orientation of~$E_i$ and a turn $\{d_i,d'_i\}$ that is folded by $f$ such that $d_i$ is the initial direction of~$E_i$. The map $f$ has a maximal first fold factor that folds the turn $\{d_i,d'_i\}$, and that maximal first fold factor must be a sewing needle fold, by definition of a sewing needle multifold. It follows, from item the \emph{Axis Characterization}, that both of $d_i,d'_i$ are directions in the subforest $\Gamma \cdot E_i$, and that there is a loxodromic $\gamma_i \in \Gamma$ such that $\gamma_i \cdot d'_i$ is the terminal direction of~$E_i$. Also, using uniqueness in the \emph{Axis Characterization} for each $i$, it follows that \emph{every} turn folded by $f$ is in the orbit of one of the turns $\{d_i,d'_i\}$, $1 \le i \le K$. Notice from this description that a sewing needle fold is the same thing as a sewing needle multifold with $K=1$.
%
%
%

For further understanding of $f$, we adapt the proof of \cite[Lemma 4.4]{\RelFSOneTag}. Choose maximal initial segments $e_i$ of $E_i$ and $e'_i$ of $\gamma_i^\inv \overline E_i$ with the same images under $f$. Next, choose proper initial subsegments $\eta_i \subset e_i$ and $\eta'_i$ of $e'_i$ with the same images under $f$. We obtain a factorization 
$$f \from S \xrightarrow{g} U \xrightarrow{h} T
$$
such that $g$ is the ``partial multifold'' defined by simultaneously folding $\eta_i$ and $\eta'_i$ for $1 \le i \le K$, meaning that for all $x \ne y \in S$, $g(x)=g(y)$ if and only if $f(x)=f(y)$ and there exists $\delta \in \Gamma$ such that $\delta \cdot x \in \eta_i$ and $\delta \cdot y \in \eta'_i$. In the free splitting $U$ define subforests 
$$\sigma = \Gamma \cdot \bigcup_{i=1}^K g(\eta_i) = \Gamma \cdot \bigcup_{i=1}^K g(\eta'_i) \qquad\text{and}\qquad \tau = \Gamma \cdot \bigcup_{i=1}^K \bigl( g(e_i \setminus \eta_i) \union g(e'_i \setminus \eta'_i) \bigr)
$$
By the same method as in \cite[Lemma 4.4]{\RelFSOneTag}, the above factorization of $f$ fits into a diagram of the form
$$\xymatrix{
S \ar@/^1pc/[r]^{g} & U \ar@/^1pc/[r]^h \ar@/^1pc/[l]_{\bar g}^{\<\sigma\>} \ar@/_1pc/[r]^{\hat h}_{\langle\tau\rangle} & T
}$$
where $\bar g$ and $\hat h$ are collapse maps with respective collapse forests $\sigma$ and $\tau$, thus demonstrating that $d(S,T) \le 2$.

Putting this altogether, we have proved conclusions~\pref{ItemNoSewingToDo} and~\pref{ItemMaxFactorNeedle} of the following:

\begin{lemma}[The Sewing Needle Lemma]
\label{LemmaSewingNeedleFold}
For any foldable map $f \from S \to T$ between free splittings $S,T$ of $\Gamma$ rel~$\A$ that is not a homeomorphism, one of two cases holds:
\begin{enumerate}
\item\label{ItemNoSewingToDo}
There exists a maximal first fold factor of $f$ that is not a sewing needle fold --- equivalently, there exists a turn $\{\delta,\delta'\}$ of $S$ that is foldable with respect to $f$ such that the natural edges $E,E' \subset S$ containing $\delta,\delta'$ (resp.) are in distinct $\Gamma$-orbits;
\,\, or
\item\label{ItemMaxFactorNeedle} Every maximal first fold factor of $f$ is a sewing needle fold --- equivalently, $f$ is a sewing needle multifold. It follows that $d(S,T) \le 2$.
\end{enumerate}
As a special case:
\begin{enumeratecontinue}
\item\label{ItemOneSewingFold}
If $S$ has exactly one orbit of natural edges then item~\pref{ItemMaxFactorNeedle} holds. Furthermore, $f$ is a sewing needle fold, and $S,T$ are in different $\Gamma$-orbits of vertices of $\FS(\Gamma;\A)$.
\qed 
\end{enumeratecontinue}
\end{lemma}

\begin{proof} Only~\pref{ItemOneSewingFold} remains to be proved. Since all turns of $S$ are in the orbit of a single natural edge $E$, conclusion~\pref{ItemMaxFactorNeedle} must hold, and $f$ is a sewing needle multifold with $K=1$; it follows that $f$ is a sewing needle fold. It also follows that, in the setting of the \emph{Quotient Characterization} above, we have
$$N = (\Gamma \cdot f(E))/\Gamma = f(\Gamma \cdot E)/\Gamma = f(S) / \Gamma = T/\Gamma
$$
If $N$ is nondegenerate then it follows that $T$ has two orbits of natural edges, whereas if $N$ is degenerate then it follows $\FellT$ is stictly larger than $\FellS$; in either case $S,T$ represent different $\Gamma$-orbits of vertices.
\end{proof}

\paragraph{Remark.} Using the analysis of the sewing needle multifold $f \from S \to T$ given above, one can also prove that $f$ factors as a composition of $K$ individual sewing needle folds,
$$S=T_0 \xrightarrow{f_1} \cdots \xrightarrow{f_K} T_K = T
$$
such that for each $i=1,\ldots,K$ the image $f^0_{i-1}(E_i) \subset T_{i-1}$ is a natural edge and $f_i \from T_{i-1} \to T_i$ is a sewing needle fold of that edge. By combining this with the Stallings Fold Theorem~\ref{TheoremStallings}, it follows that a general foldable map $U \to W$ factors into two terms $U \mapsto V \mapsto W$, such that the first term $U \mapsto V$ factors as a product of folds \emph{none} of which is a sewing needle fold, and the second term $V \mapsto W$ is a sewing needle multifold (and hence $d(V,W) \le 2$) which itself factors as a product of a uniformly bounded number of sewing needle folds. 

\subsection{Statement of the theorem}
\label{SectionTOAStatement}

\begin{definition}[Crossing numbers; fully crossing paths]
\label{DefCrossing}
Given two paths $\mu,\nu$ in a graph $G$, to say that \emph{$\mu$ crosses $\nu$} means that $\mu$ can be written as a concatenation of subpaths in one of two forms: $\mu = \alpha * \nu * \beta$ or $\mu = \alpha * \nu^\inv * \beta$. The number of distinct concatenation expressions of one of these forms is denoted $\<\mu,\nu\>$, and is called \emph{the number of times that $\mu$ crosses $\nu$}. Note that if $G$ is a tree then such a concatenation expression, if it exists, is unique, and hence $\<\mu,\nu\> \le 1$. More generally, given $\mu$ a single path and $N$ a collection of paths in $G$, we define $\<\mu,N\> = \sum_{\nu \in N} \<\mu,\nu\>$; this is called the \emph{number of times that $\mu$ crosses $N$}. To say that \emph{$\mu$ crosses~$N$} means simply that $\<\mu,N\> \ge 1$. In the statement of the \TOAT\ and other places to follow, we apply these concepts in the situation where $G=T$ is a free splitting, $\nu=e$ is either an edge of $T$ or a natural edge of~$T$, and $N = \Gamma \cdot e$ is the orbit of $e$. In such a situation $\langle \mu, \Gamma \cdot e\>$ is the number of times that $\mu$ crosses (edges in) the orbit of~$e$, and it is also the number of paths in the orbit of $\mu$ that cross~$e$. To say that $\mu$ \emph{fully crosses $T$} means that $\<\mu,E\> \ge 1$ for each natural edge $E \subset T$; in other words, $\mu$ crosses some translate of every natural edge of $T$.
\end{definition}

\begin{theorem*}[\TOAT\ (iterated form)]
For any group $\Gamma$ and any free factor system~$\A$ of $\Gamma$ there exists a integer constant $\Delta=\Delta(\Gamma;\A)$ such that for any free splittings $S,T$ of $\Gamma$ rel~$\A$, any foldable map $f \from S \to T$, and any integer $n \ge 1$, if the distance between $S,T$ in the free splitting complex $\FS(\Gamma;\A)$ satisfies the lower bound $d_\FS(S,T) \ge n\Delta$ then there exist two natural edges $E_1,E_2 \subset S$ in distinct $\Gamma$-orbits such that each path $f(E_i)$ contains $2^{n-1}$ nonoverlapping subpaths each of which fully crosses~$T$. More precisely, there is a decomposition 
$$f(E_i) = \nu_1 \, \mu_1 \, \cdots \, \nu_{N} \, \mu_{N} \, \nu_{N+1} \qquad N = 2^{n-1}
$$ 
such that each path $\mu_i$ fully crosses $T$; the paths $\nu_i$ are allowed to be trivial. As a consequence, $\<f(E_i),E'\> \ge 2^{n-1}$ for every natural edge $E' \subset T$. 
\end{theorem*}

\section{Application to the \LPT}
\label{SectionAppLip}

In their proof that the curve complex $\C(S)$ of a finite type, orientable surface $S$ is hyperbolic \cite{MasurMinsky:complex1}, Masur and Minsky considered the natural systole map $\mathcal T(S) \mapsto \C(S)$ defined on the Teichm\"uller space $\mathcal T(S)$, and they showed that map to be coarsely Lipschitz. For the special case that $S=S_g$ is closed and of genus $g \ge 2$, asymptotic estimates for the optimal Lipschitz constant $\kappa(g)$ of the systole map $\mathcal T(S_g) \mapsto \C(S_g)$ were obtained by Gadre, Hironaka, Kent and Leininger \cite{GHKL:Lipschitz}, who proved that $\frac{1}{K} \log(g) \le \kappa(g) \le K \log(g)$ for some universal constant~$K$. 

In the setting of a rank~$n$ free group $F_n$, with the Culler--Vogtmann outer space $X(F_n)$ \cite{CullerVogtmann:moduli} taking over the role of $\mathcal T(S)$, Bestvina and Feighn proved hyperbolicity of the free factor complex 
$\FFC(F_n)$ \cite{BestvinaFeighn:FFCHyp}, and their work shows that the natural systole map $X(F_n) \mapsto \FFC(F_n)$ is coarsely Lipschitz with respect to the log-Lipschitz semimetric on $X(F_n)$. This semimetric has its origins in the thesis of Tad White \cite{White:GeometryOfOuterSpace}, which was followed up by Bestvina's treatment \cite{Bestvina:BersLike} and an in-depth study by Francaviglia and Martino \cite{FrancavigliaMartino:MetricOuterSpace}. 

But when working in $X(F_n)$ by analogy with $\mathcal T(S)$, the role played by the curve complex $\C(S)$ forks into (at least) two roles: depending on the application, sometimes one uses the free factor complex 
$\FF(F_n)$, and other times one uses the free splitting complex $\FS(F_n)$. As shown by Kapovich and Rafi \cite{KapovichRafi:HypImpliesHyp} there is a natural coarse Lipschitz map $\FS(F_n) \mapsto \CFFS(F_n)$. This fact, in conjunction with the existence of the coarse Lipschitz systole map $X(F_n) \mapsto \CFFS(F_n)$, invites the following question: Does there exist a natural coarse Lipschitz map $X(F_n) \mapsto \FS(F_n)$? Our Lipschitz Projection Theorem below shows that this is so, even in the more general context of relative free splitting complexes.

Given any group $\Gamma$ equipped with a free factor system~$\A$, we follow \cite{GuirardelHorbez:Laminations} in using the notation $\O(\Gamma;\A)$ and the terminology \emph{outer space} for the deformation space of Grushko free splittings of $\Gamma$ with respect to $\A$; see below for a review. These spaces $\O(\Gamma;\A)$ are special cases of the deformation spaces of group actions on simplicial trees introduced by Forester \cite{Forester:Deformation} and studied in depth by Guirardel and Levitt \cite{GuirardelLevitt:outer}, and they generalize the McCullough--Miller outer spaces \cite{McCulloughMiller:symmetric} which themselves generalize the foundational Culler--Vogtmann outer spaces \cite{CullerVogtmann:moduli}.

The log-Lipschitz semimetric on $\O(\Gamma;\A)$ was studied by Francaviglia and Martino \cite{FrancavigliaMartino:TrainTracks} and by Meinert \cite{Meinert:TrainTrackMaps}; for more recent work on $\O(\Gamma;\A)$ by Guirardel and Horbez see \cite{GuirardelHorbez:Laminations}. The outer space $\O(\Gamma;\A)$ can be identified with a dense open subset of the free splitting complex $\FS(\Gamma;\A)$ (in fact the complement of a nowhere dense subcomplex), namely the equivalence classes of all Grushko free splittings of $\Gamma$ rel~$\A$; the inclusion therefore gives us a natural $\Out(\Gamma;\A)$-equivariant map $\iota \from \O(\Gamma;\A) \mapsto \FS(\Gamma;\A)$. Our next theorem shows that this inclusion is coarsely Lipschitz, with a certain amount of control over the optimal Lipschitz constant, although the control we obtain is not --- on the face of it --- as strong as that of \cite{GHKL:Lipschitz}:

\begin{theorem*}[Lipschitz Projection Theorem ($\FS$ version)]
The inclusion map $\iota \from \O(\Gamma;A) \to \FS(\Gamma;\A)$ is coarsely Lipschitz with respect to the log Lipschitz semimetric $d_\O$ on $\O(\Gamma;A)$ and the simplicial metric $d_\FS$ on $\FS(\Gamma;\A)$. More specifically, for each $S,T \in \O(\Gamma;A)$ we have
$$d_\FS(S,T) \le  \frac{\Delta}{\log 2} \, d_\O(S,T) + 2\Delta + 1
$$
using the constant $\Delta=\Delta(\Gamma;\A)$ from the \TOAT.
\end{theorem*}
\noindent
The proof of this theorem will be a reasonably quick application of the Two-Over-All Theorem, once the basic definitions have been carefully reviewed.

\medskip

In \cite[Proposition~6.5]{\RelFSOneTag} we used the Kapovich--Rafi method \cite{KapovichRafi:HypImpliesHyp} to prove that the natural systole map $\FS(\Gamma;\A) \mapsto \CFFS(\Gamma;\A)$ is coarse Lipschitz. Combining this with the above theorem we immediately obtain the following generalization of the result of Bestvina and Feighn mentioned earlier:

\begin{theorem*}[Lipschitz Projection Theorem ($\CFFS$ version)]
The natural systole map $\O(\Gamma;A) \to \CFFS(\Gamma;\A)$ is coarse Lipschitz. \qed
\end{theorem*}

\paragraph{The outer space $\O(\Gamma;\A)$ and its log-Lipschitz semimetric.} Consider a metric Grushko free splitting $S$ of $\Gamma$ rel~$\A$, meaning a Grushko free splitting equipped with a $\Gamma$-equivariant geodesic metric. Let $\abs{S}$ denote the sum of lengths of a choice of edges, one in each $\Gamma$-orbit; equivalently $\abs{S}$ is the total length of the quotient graph of groups~$S / \Gamma$. To say that $S$ is \emph{normalized} means that $\abs{S}=1$. The equivalence class of $S$ up to topological conjugacy corresponds to a certain open simplex in $\FS(\Gamma;\A)$, using the (coarse) simplicial structure on $\FS(\Gamma;\A)$ described at the end of Section~\ref{SectionFolds}. If the metric on $S$ is normalized then the metric conjugacy class of $S$ corresponds to a particular point in that open simplex, namely the point which assigns to each natural edge orbit of $S$ a barycentric coordinate equal to the length of edges in that orbit; but even if $S$ is not normalized, there is still a point in that open simplex corresponding to its normalization $\frac{1}{\abs{S}}S$. In summary, up to the equivalence relation of $\Gamma$-equivariant homothety, the outer space $\O(\Gamma;A)$ is the subset of the free splitting complex $\FS(\Gamma;\A)$ consisting of all homothetic equivalence classes of metric Grushko free splittings $S$ of $\Gamma$ rel~$\A$. To be formal we let $[S] \in \O(\Gamma;A)$ denote the homothetic equivalence class of $S$ rel~$\A$, although we very often abuse notation by letting $S$ stand for $[S]$. 

\newcommand\Smax{S_{\text{max}}}
\newcommand\Umax{U_{\text{max}}}
 
For any two metric Grushko free splittings $S,T$, there exist $\Gamma$-equivariant quasi-isometries in both directions $S \mapsto T$ and $T \mapsto S$. Also, for each $\gamma \in \Gamma$, $\gamma$ is loxodromic in $S$ if and only if it is loxodromic in $T$, if and only if $\gamma \not\in A$ for all subgroups $A \subgroup \Gamma$ such that $[A] \in \A$. For both of these statements see \cite[Theorem 1.1]{Forester:Deformation} or \cite[Theorem 3.8]{GuirardelLevitt:outer}.


\begin{definition}[Optimal maps \protect{\cite[Definition 4.11]{Meinert:TrainTrackMaps} or \cite[Definition 6.4]{FrancavigliaMartino:TrainTracks}}]
\label{DefOptimal}
Consider an equivariant map $f \from S \to T$ between two metric Grushko free splittings equipped with their natural cell structure, such that the restriction of $f$ to each edge of $S$ has constant speed. The maximum of those speeds over all edge orbits of $S$ is equal to the Lipschitz constant $\Lip(f)$, and the union of all natural edges of $S$ which achieve that maximum is called the \emph{tension forest} $\Smax(f) \subset S$. To say that $f$ is \emph{optimal} means that 
\begin{enumerate}
\item\label{ItemOptNorm}
$S$ and $T$ are normalized;
\item\label{ItemOptMin}
$\Lip(f) = \min_{f'} \Lip(f')$ where $f' \from S \to T$ varies over all equivariant Lipschitz maps;
\item\label{ItemOptGates}
At~each vertex $v \in \Smax(f)$ there exist two directions $e,e'$ in $\Smax(f)$ such that the paths $f(e), f(e') \subset T$ represent distinct directions at $f(v)$.
\end{enumerate}
\end{definition}

\subparagraph{Remarks.}
The tension forest $S_{\max}(f)$ is always nonempty, because otherwise \hbox{$\Lip(f)=0$} and $f$ is a constant map, which contradicts equivariance. Optimal maps always exist; see~Fact~\ref{OptimalFacts} below for references. 

We shall refer to item~\pref{ItemOptGates} informally by saying that \emph{$f$ has at least two gates in $\Smax(f)$}, although formally our earlier definition of gates only applies when $f$ is injective on each edge of~$S$, i.e.\ when $f$ has positive speed on every edge.
 
%
%

\begin{definition}[The log Lipschitz semimetric on $\O(\Gamma;\A$) and its witnesses.]
The \emph{log Lipschitz semimetric} on $\O(\Gamma;A)$ is the asymmetric metric (satisfying the separation axiom and the triangle inequality, but not necessarily the symmetry axiom) which is defined by
$$d_\O([S],[T]) = \log \Lip(f) \quad\text{for any optimal $f \from S \to T$, where $S,T$ are both normalized}
$$
Given a loxodromic $\gamma \in \Gamma$, to say that $\gamma$ is a \emph{witness} (relative to $S,T$) means that the following equations hold
$$\log\left( \frac{\ell_T(\gamma)}{\ell_S(\gamma)} \right) = \max_\delta \log\left( \frac{\ell_T(\delta)}{\ell_S(\delta)} \right) = d_\O(S,T)
$$
where $\delta$ varies over the loxodromic elements of $\Gamma$. 
\end{definition}

\begin{fact} \label{OptimalFacts}
For each normalized $S,T \in \O(\Gamma;A)$:
\begin{enumerate}
\item An optimal map $f \from S \to T$ exists. \hfill \textup{\cite[Theorem 4.6]{Meinert:TrainTrackMaps}, \cite[Corollary 6.8]{FrancavigliaMartino:TrainTracks}}
\item A witness $\gamma \in \Gamma$ exists. \hfill \textup{\cite[Theorem 4.14]{Meinert:TrainTrackMaps},  \cite[Theorem 6.11]{FrancavigliaMartino:TrainTracks}}
\item For every optimal map $f \from S \to T$, and for every witness $\gamma$ relative to $S,T$ we have $\Axis(\gamma) \subset \Smax(f)$.  \hfill \textup{\cite[Lemma 4.15]{Meinert:TrainTrackMaps}}
\end{enumerate}
\qed\end{fact}

With respect to the coarse simplicial structure on $\FS(\Gamma;\A)$ described at the end of Section~\ref{SectionFolds}, the deformation space $\O(\Gamma;A)$ is identified with the complement of the subcomplex of non-Grushko free splittings, and in particular $\O(\Gamma;A)$ is a union of open simplices of $\FS(\Gamma;\A)$. At the barycenter of each open simplex of $\O(\Gamma;A)$ there sits a vertex of $\FS(\Gamma;\A)$, the point of that simplex where all natural edge lengths are equal. Let $\beta \from \O(\Gamma;A) \to \FS(\Gamma;\A)$ be the $\Out(\Gamma;\A)$ equivariant map that takes an open simplex to its barycenter; with respect to the simplicial metric on $\FS(\Gamma;\A)$, this map $\beta$ has sup-distance~$\le 1$ from the inclusion map $\iota \from \O(\Gamma;\A) \to \FS(\Gamma;\A)$. It follows that one of $\beta$ or $\iota$ is coarse Lipschitz if and only if both are. 

In the statement of the \LPT\ and the proof to follow, given $S,T \in \O(\Gamma;A)$ we use $d_\FS(S,T)$ as a shorthand for $d_\FS(\beta(S),\beta(T))$ (at the expense of an additive error $\le 1$). More generally we re-use the notation $S$ as an abbreviation for the vertex $\beta(S) \in \FS(\Gamma;\A)$; it should be clear by context whether we are using the metric on $S$ and hence thinking of $S \in \O(\Gamma;A)$, or whether we are working with $S$ combinatorially and hence thinking of $\beta(S) \in \FS(\Gamma;\A)$.

\begin{proof}[Proof of the Lipschitz Projection Theorem] Given any normalized $S,T \in \Out(\Gamma;\A)$, for the proof we need a single map $S \mapsto T$ that combines features of optimal maps and foldable maps. One could try to obtain such an ``optimal foldable map'' by starting with a foldable map and then altering it using the standard methods for constructing optimal maps, or by starting with an optimal map and then altering it using the standard methods for constructing foldable maps. We shall follow the latter strategy, using for example the methods of \cite[Lemma 4.2~(1)]{\FSHypTag}.

\smallskip

Given a metric free splitting $U$ of $\Gamma$ rel~$\A$, to say that $U$ is \emph{subnormalized} means that $\abs{U} \le 1$. 

\begin{claim} \quad For any normalized $S,T \in \O(\Gamma;A)$ equipped with their natural cell structures, and for any optimal map $f \from S \to T$, there exists a subnormalized $U \in \O(\Gamma;A)$ equipped with its natural cell structure, and there exist maps 
$$\xymatrix{
S \ar[r]^\pi_{\langle\sigma\rangle} & U \ar[r]^{g} & T
}$$ 
such that the following hold
\begin{enumerate}
\item\label{ItemClaimCollapse}
$\pi$ is a collapse map. It follows that \hbox{$d_\FS(S,U) \le 1$.}
\item\label{ItemClaimConstant}
For each edge $E \subset S \setminus \sigma$,
\begin{enumerate}
\item\label{ItemEdgeRestriction}
The restriction $\pi \restrict E$ is an isometry, 
\item\label{ItemRestrictionEquation}
We have an equation of restrictions $f \restrict E = (g \, \circ \, \pi) \restrict E$. 
\item\label{ItemTensionNotcollapsed}
$\Smax(f) \subset S \setminus \sigma$.
\end{enumerate}
Furthermore, $\Umax(g) = \pi(\Smax(f))$ and that $\log \Lip(g) = \log \Lip(f) = d_\O(S,T)$.
%
\item\label{ItemClaimFoldable}
$g$ is a foldable map.
\end{enumerate}
\end{claim}
\noindent
Note that the claim does not assert optimality of~$g$; see remarks after the ``basis step'' of the proof of the claim. Note also that in item~\pref{ItemClaimConstant} of the claim, from~\pref{ItemEdgeRestriction} and~\pref{ItemRestrictionEquation} it follows that $g \restrict \pi(E)$ is an isometry, and combined with~\pref{ItemTensionNotcollapsed} the ``Furthermore'' clause follows immediately.

\medskip

Accepting this claim for the moment, we apply it as follows. There exists an integer $k \ge 0$ so that 
$$k \Delta \le d_\FS(U,T) < (k+1)\Delta
$$
If $k=0$ then $d_\FS(S,T) \le \Delta+1$ and we are done. We may therefore assume that $k \ge 1$. Choosing an optimal map $f \from S \to T$, applying the claim, and then applying the \TOAT\ to the foldable map $g$, we obtain a natural edge $E \subset U$ such that $g(E)$ crosses the orbit of each natural edge of $T$ at least $2^{k-1} \ge 1$ times. Using that $U$ is subnormalized it follows that $\Length(E) \le 1$. Since $T$ is normalized, so $\Length(T)=1$, it follows that $\Lip(g) \ge 2^{k-1}$, and so
$$d_\O(S,T) = \log(\Lip(g)) \ge (k-1) \log 2 \ge \left(\frac{d_\FS(U,T)}{\Delta}  - 1\right) \log 2 
$$
$$d_\FS(S,T) \le d_\FS(S,U) + d_\FS(U,T) \le 1 + \left(\frac{\Delta}{\log 2} \, d_\O(S,T) + \Delta\right)
$$
which proves the desired conclusion of the Lipschitz Projection Theorem.

\medskip

For proving the claim we set up an induction argument as follows:
\begin{description}
\item[Induction hypothesis:] There exists a subnormalized $U \in \O(\Gamma;A)$ equipped with its natural cell structure, and there exist maps 
$$\xymatrix{
S \ar[r]^\pi_{\langle\sigma\rangle} & U \ar[r]^{g} & T
}$$ 
such that items~\pref{ItemClaimCollapse} and~\pref{ItemClaimConstant} of the claim hold, and such that $g$ is tight and nonconstant on each edge.
\end{description}
For the basis step of the induction we simply take $\sigma$ to be the subforest of natural edges of $S$ on which $f$ has speed zero, equivalently those edges on which $f$ is constant, we use $\sigma$ to define the collapse $S \mapsto U$, and we let $g \from U \to T$ be the induced map. We note in this case that $g$ is still optimal hence the induction hypothesis holds. 

\smallskip
\textbf{Remarks:} After the basis step is carried out by collapsing the speed zero edges, if the resulting induced map $g$ is already foldable then the proof is done \emph{and} $g$ is optimal. But if $g$ is not already foldable then, as the induction proceeds, we will lose the property that $g$ is optimal for the reason that we will lose that $U$ is normalized, but we will retain that $U$ is subnormalized. Otherwise the induction hypothesis is designed so as to carry along as much optimality of $g$ as is possible. 

\smallskip

Assuming the induction hypothesis, we may also assume that $g$ is not already foldable. Since $g$ is tight, its gates are defined at every vertex of $U$, but since $g$ is not foldable there is a vertex of $U$ with only $1$ gate. Using these facts, the goal of the induction step is to find a factored map $U \xrightarrow{\<\sigma'\>} U' \xrightarrow{g'} T$, and hence a new collapse map $\pi' \from S \xrightarrow{\<\sigma\>} U \xrightarrow{\<\sigma'\>} U'$, so that the new factorization $S \xrightarrow{\pi'} U' \xrightarrow{g'} T$ still satisfies the induction hypothesis but has smaller complexity in the following sense: either $U'$ is in the same open simplex as $U$ and hence has the same number of natural edge orbits, but $g'$ has fewer gate~$1$ vertex orbits than $g$; or $U'$ is contained in a proper face of the open simplex that contains $U$ and hence $U'$ has fewer natural edge orbits. The obvious induction then completes the proof. 

To carry out the induction step, choose a vertex $v \in U$ at which $g$ has only $1$ gate. Note that $v$ has trivial stabilizer, for if $v$ is fixed by some nontrivial $\gamma \in \Gamma$ then $\gamma$ also fixes the common $g$-image in $T$ of two oriented edges with initial vertex $v$, contradicting that $T$ is a free splitting. Note also that $v \not\in \Umax(g)$ for otherwise, using that item~\pref{ItemClaimConstant} of the Claim holds, there would exist $w \in \Smax(f)$ such that $\pi(w)=v$, but $f$ is optimal hence by Definition~\ref{DefOptimal}~\pref{ItemOptGates} the vertex $w$ has at least $2$ gates with respect to $f$; it would follow that $v$ has at least $2$ gates with respect to $g$, a contradiction. 

Enumerate the oriented edges of $U$ with initial endpoint $v$ as $\eta_1,\ldots,\eta_K$; there are only finitely many, given that $v$ has trivial stabilizer. Since $v \not\in \Umax(g)$, also $\eta_k \not\subset \Umax(g)$ for each $k=1,\ldots,K$. Let $e_k \subset \eta_k$ (for $k=1,\ldots,K$) be the collection of maximal initial segments such that $g$ maps each $e_k$ to the same path $g(e_k) \subset T$ independent of $k$. Let $e_k$ inherit its orientation from $\eta_k$ hence $v$ is the initial endpoint of each $e_k$. Consider $\spray(v) = e_1 \union\cdots\union e_K$. The frontier of $\spray(v)$ in $U$ is the set of terminal endpoints $w_1,\ldots,w_K$ of $e_1,\ldots,e_K$ (respectively). 

For any nontrivial $\gamma \in \Gamma$ we verify that $\spray(v) \intersect \spray(\gamma \cdot v) \subset \{w_1,\ldots,w_K\}$. Arguing by contradiction, if this does not hold then there exist $i,j \in \{1,\ldots,K\}$ such that the intersection $\gamma(e_i) \intersect e_j$ contains interior points of both $\gamma(e_i)$ and $e_j$. From this it follows that $\gamma(\eta_i) = \eta_j$, reversing orientation (since $\gamma(v) \ne v$). If $i=j$ then $\gamma$ fixes the midpoint of $\eta_i$, contradicting that $\eta_i$ is a natural edge of~$U$. If $i \ne j$ then $\gamma(w_i)=v$ and $\gamma(v)=w_j$, the element $\gamma$ acts loxodromically on $U$, and the natural edges $\eta_i$, $\eta_j$ are consecutive fundamental domains for the action of $\gamma$ on its axis in~$U$. Since $U$ and $T$ are both Grushko free splittings of $\Gamma$ rel~$\A$, they have the same loxodromic elements, hence $\gamma$ also acts loxodromically on $T$. But $\gamma$ maps the midpoint of $\eta_i$ to the midpoint of $\eta_j$, and $g$ identifies those midpoints to the same point of $T$, hence $\gamma$ fixes that point in $T$, contradicting that $\gamma$ acts loxodromically on~$T$.

Now form the quotient $q \from U \to \hat U$ by identifying the edges $e_1,\ldots,e_K$ of $\spray(v)$ to a single oriented edge $\hat e \subset \hat U$ with initial vertex $\hat v = q(v)$ and terminal vertex $\hat w=q(w_i)$ $i=1,\ldots,K$. Of course we also equivariantly identify each of the translates $\spray(\gamma \cdot v) = \gamma \cdot \spray(v)$, to an edge $\gamma \cdot \hat e \subset \hat U$. Each of these new edges has a valence~$1$ vertex $\gamma \cdot \hat v$, so the action $\Gamma \act \hat U$ is not minimal. Let $U' \subset \hat U$ be the minimal subtree of that action, obtained from $\hat U$ by removing each vertex $\gamma \cdot \hat v$ and the interior of each edge $\gamma \cdot \hat e$. Evidently the restricted action $\Gamma \act U'$ is a Grushko free splitting action. Let $U \mapsto U'$ be the collapse map whose restriction to $U - \bigcup_\gamma \spray(\gamma \cdot v)$ is equal to the restriction of $q$, and whose restriction to each $\spray(\gamma \cdot v)$ has constant value $\gamma \cdot \hat w$. Clearly $\abs{U'} < \abs{U} \le 1$ and so $U'$ is subnormalized. The composition $\pi' \from S \mapsto U \mapsto U'$ is the desired collapse map. The natural graph structure on $U'$ with respect to the action of $\Gamma$ is evidently the graph structure that is induced from the natural graph structure on $U$ as follows: each of the new vertices $\gamma \cdot \hat w$ is a natural vertex of~$U'$; and the collapse map $U \mapsto U'$ induces a bijection from those natural vertices of $U$ not in the set $\Gamma \cdot \{v,w_1,\ldots,w_K\}$ to those natural vertices of $U'$ not in the set $\Gamma \cdot \{\hat w\}$. The map $g \from U \to T$ induces a map $\hat g \from \hat U \to T$ which restricts to the desired map $g' \from U' \to T$ which is tight and nonconstant on edges. This completes the construction of the new factorization $S \xrightarrow{\pi'} U' \xrightarrow{g'} T$, together with verification of the required item~\pref{ItemClaimCollapse} from the Claim. The remaining requirement of the induction hypothesis, namely item~\pref{ItemClaimConstant}, is easily verified starting from the fact that $\eta_1,\ldots,\eta_K$ are not in the tension forest~$\Umax(g)$.

It remains to check that the new factored map $S \xrightarrow{\pi'} U' \xrightarrow{g'} T$ has smaller complexity than the given one $S \xrightarrow{\pi} U \xrightarrow{g} T$, and there are two cases. If there exists $i \in 1,\ldots,K$ such that $\eta_i = e_i$ then the entire natural edge $\eta_i$ is collapsed by the map $S \mapsto U'$ hence $U'$ is contained in a proper face of the open simplex containing~$S$. In the other case where no such $i$ exists, each inclusion $e_i \subset \eta_i$ is proper for all $i=1,\ldots,K$. It follows that $U'$ is contained in the same open simplex of $\O(\Gamma;A)$ as $S$, and in fact the collapse map $U \mapsto U'$ tightens to a simplicial isomorphism taking $v$ to $\hat w$. Also, for each vertex $u \in U \setminus \Gamma \cdot v$, its image $u' \in U'$ has the same number of gates with respect to $g'$ as $u$ has with respect to $g$; indeed, the collapse map $U \mapsto U'$ takes the gates of $g$ at $u$ bijectively to the gates of $g'$ at $u'$.  However, $v$ has only $1$ gate with respect to $g$, whereas by maximality of the choice of $e_1,\ldots,e_K$ it follows that $\hat w$ has $\ge 2$ gates with respect to $g'$.

This completes the proof of the Lipschitz Projection Theorem.
\end{proof}

\section{Application to the upper bound on translation lengths in Theorem~A}
\label{SectionAppUpperBoundInA}

In this section we carry out one portion of the proof of Theorem~A as outlined in the introduction:
\begin{description}
\item[The Upper Bound in Theorem A:] There is a constant $B > 0$ depending only on $\corank(\A)$ and $\abs{\A}$ such that for each $\phi \in \Out(\Gamma;\A)$ the following holds: If $\phi$ has a filling attracting lamination $\Lambda \in \L(\phi)$ with expansion factor $\lambda_\phi = \lambda(\phi,\Lambda)$  (see Definition~\ref{DefFillExpFact}), then the translation length $\tau_\phi$ for the action of $\phi$ on $\FS(\Gamma;\A)$ satisfies the upper bound
$$\tau_\phi \le B \log \lambda_\phi
$$
\end{description}
The proof is carried out over the next three subsections. 

In Section~\ref{SectionAnAxisInFS} we review concepts of boundaries and of attracting laminations, and we state Proposition~\ref{PropAxisInFS} which describes a certain ``train track axis'' in $\FS(\Gamma;\A)$ for any $\phi \in \FS(\Gamma;\A)$ that possesses a filling lamination: this axis is a $\phi$-periodic, bi-infinite fold path that encodes the expansion factor $\lambda_\phi$. An important feature of Proposition~\ref{PropAxisInFS} is the assertion that the number $\lambda_\phi$ is a well-defined invariant of $\phi$, independent of the choice of train track axis; see Definition~\ref{DefFillExpFact}.

In Section~\ref{SectionProofUpperBound} we prove the upper bound of Theorem A by applying the \TOAT\ to the axis described in Proposition~\ref{PropAxisInFS}.

In Section~\ref{SectionAxisConstruction} we prove Proposition~\ref{PropAxisInFS}, constructing the desired train track axis by applying the relative train track theory for $\Out(\Gamma;\A)$ developed by Lyman \cite{\LymanRTTTag}. We actually prove a broader version of the proposition, stated in Proposition~\ref{PropAxisInFSEmbellished}, which incorporates further details about attracting laminations that will be needed in Part~III (as explained in the ``Organizational Outline'' at the end of the Introduction). 

\medskip
\emph{Remark on constructions of train track maps.} Consider the action on the relative outer space $\O(\Gamma;\A)$ of an element $\phi \in \Out(\Gamma;\A)$ which is fully irreducible \relA. In the special case of $\Out(F_n)$ fold axes of $\phi$ in $\O(F_n)$ were constructed in \cite{HandelMosher:axes} by suspending train track representatives of~$\phi$; the proof of Proposition~\ref{PropAxisInFS} found in Section~\ref{SectionAxisConstruction} can be regarded as a generalization of this suspension construction. For the general case of $\Out(\Gamma;\A)$ fold axes were constructed in \cite{FrancavigliaMartino:TrainTracks} using train track maps constructed by an optimization procedure using the log-Lipschitz semimetric; this optimization construction can also be used as the basis for proving Proposition~\ref{PropAxisInFS}. But for our present purposes, particularly for Proposition~\ref{PropAxisInFSEmbellished} as mentioned just above, we will need the suspension construction.

\medskip

\emph{Remark on relative train track theory.} The theories of relative train maps and attracting laminations presented in \cite{\LymanRTTTag} and \cite{\LymanCTTag} are set mostly downstairs in graphs of groups, whereas we work mostly upstairs in Bass-Serre trees; the translation between the two settings is straightforward, using the concepts of Bass-Serre theory.

\subsection{Describing fold axes in $\FS(\Gamma;\A)$}
\label{SectionAnAxisInFS}

Our goal in this section is to state Proposition~\ref{PropAxisInFS} which, if $\phi \in \Out(\Gamma;\A)$ has a filling lamination, produces a bi-infinite $\phi$-periodic fold line in $\FS(\Gamma;\A)$ --- a \emph{fold axis} for $\phi$ --- such that the transition matrix encoded in that fold axis is a Perron-Frobenius matrix. Furthermore, the proposition tells us that for all such fold axes of $\phi$, the corresponding Perron-Frobenius eigenvalues are all equal. This allows us to formulate Definition~\ref{DefFillExpFact}, in which the filling expansion factor $\lambda_\phi$ is well-defined as the expansion factor of any EG-aperiodic fold axis for~$\phi$.

The statement of Proposition~\ref{PropAxisInFS} comes in Section~\ref{SectionFoldAxes}, after a review of Gromov boundaries and Bowditch boundaries in Section~\ref{SectionBoundariesRelFFSs} following \cite{GuirardelLevitt:DefSpaces,GuirardelHorbez:Laminations}, and a review of attracting laminations in Section~\ref{SectionLamsAndExpFact} following \cite{\LymanCTTag}. These will be extensive reviews, incorporating some new relations and notations needed here and in the sequel \cite{HandelMosher:RelComplexHypIII}.


\subsubsection{Ends of free splittings and boundaries of groups rel free factor systems}
\label{SectionBoundariesRelFFSs}



\paragraph{Gromov boundaries and Bowditch boundaries.} In \BookOne, attracting laminations for $\Out(F_n)$ are expressed in terms of the action of $\Aut(F_n)$ on the Gromov \hbox{boundary $\bdy F_n$.} In the presence of a free factor system $\A$ on a group~$\Gamma$, the role played by $\bdy F_n$ splits into roles played by two actors: the Bowditch boundary $\bdy(\Gamma;\A)$ \cite{GuirardelLevitt:DefSpaces}; and the (relative) Gromov boundary $\bdyinf(\Gamma;\A)$ \cite{GuirardelHorbez:Laminations}, the latter of which is included as a dense subspace of the former:
$$\xymatrix{
\bdyinf(\Gamma;\A) \,\, \ar@{^{(}->}[rr] && \,\, \bdy(\Gamma;\A)
}
$$
Here and the sequel \cite{HandelMosher:RelComplexHypIII} we shall need both of these boundaries: here we use the Bowditch boundary $\bdy(\Gamma;\A)$ as the setting for Lyman's theory of attracting laminations; in the sequel we focus more on the Gromov boundary $\bdyinf(\Gamma;\A)$, in order to apply the subboundary correspondence of Dowdall and Taylor \cite{DowdallTaylor:cosurface} to generic leaves of attracting laminations. 

As explained in \cite{GuirardelHorbez:Laminations}  and \cite{GuirardelLevitt:DefSpaces} in respective cases, the ``canonical boundaries'' $\bdy(\Gamma;\A)$ and $\bdyinf(\Gamma;\A)$ are represented, respectively, by the boundaries $\bdy T$, $\bdyinf T$ that are associated to each Grushko free splitting $T$ rel~$\A$. It has become common to \emph{identify} the canonical boundaries with their representatives (see e.g.\ \cite{GuirardelHorbez:Laminations,\LymanCTTag}. But for use here and in \cite{HandelMosher:RelComplexHypIII} we set up notation for distinguishing these various boundaries, for various maps amongst them, and for their actions by the group~$\Aut(\Gamma;\A)$; in Section~\ref{SectionLamsAndExpFact} to follow we will do the same for line spaces and attracting laminations.

\smallskip

Any simplicial tree~$T$ has an \emph{end space} $\bdyinf T$ and a \emph{Bowditch boundary} $\bdy T$; the ``end space'' may be identified canonically with the ``Gromov boundary'' of $T$ with respect to the simplicial metric on $T$, and we use those terms as synonyms. The elements of $\bdyinf T$, called \emph{ends of~$T$}, are the asymptotic classes of rays in~$T$ (see Section~\ref{SectionPathsRaysLines}), two rays being asymptotic (by definition) if their intersection is a ray. The subset of infinite valence vertices is denoted $V_\infty T \subset T$. Setting $\overline T = T \union \bdyinf T$, we have $\bdy T = V_\infty T \sqcup \bdyinf T \subset \overline T$. A~\emph{half tree in~$T$} is a component $\tau$ of $T-x$ for some $x \in T$, and corresponding to $\tau$ is a \emph{half tree in~$\overline T$} denoted $\overline\tau = \tau \union \bdyinf \tau \subset \overline T$ where $\bdyinf\tau \subset \bdyinf T$ denotes those ends of $T$ that are represented by some ray contained in~$\tau$. The \emph{observer's topology} on $\overline T = T \disjunion \bdyinf T$ is generated by the sub-basis of all half trees in $\overline T$. The \emph{end topology} on $\overline T$ (equivalent to the Gromov topology when working in the metric category) is the refinement of the observer's topology obtained by augmenting the observer's sub-basis with all open subsets of~$T$ itself. In $\overline T$ the Bowditch boundary $\bdy T$ is the closure of the end space $\bdyinf T$ with respect to the observers topology. On $\bdyinf T$ the observer's topology and the end topology on $\overline T$ induce the exact same subspace topology on the end space $\bdyinf T$. An individual ray in $T$, based at $v \in T$ and representing $\xi \in \bdyinf T$, is denoted $[v,\xi) \subset T$; the end $\xi$ is called the \emph{ideal endpoint} of the ray $[v,\xi)$, and the closure of $[v,\xi)$ in either topology on $\overline T$ is well-defined as $[v,\xi] = [v,\xi) \union \{\xi\}$.



We specialize now to the case that $T$ is a Grushko free splitting of $\Gamma$ rel~$\A$. Noting that $V_\infty T = \{v \in T \suchthat \, \text{$\Stab(v)$ is an infinite subgroup of $\Gamma$}\}$, we have a bijection $v \leftrightarrow \Stab(v)$ between the set $V_\infty T$ and the set of infinite subgroups $A \subgroup \Gamma$ such that the conjugacy class $[A]$ is an element of the free factor system $\A$. The action $\Gamma \act T$ extends naturally to an action $\Gamma \act \overline T$ (by homeomorphisms with respect to either topology), which then restricts to \emph{boundary actions} $\Gamma \act \bdy T$ and~$\Gamma \act \bdyinf T$.

Consider a nontrivial $\gamma \in \Gamma$. For any Grushko free splitting $T$ of $\Gamma$ rel~$\A$, recall that either $\gamma$ is \emph{elliptic} meaning it fixes a vertex of $T$, or $\gamma$ is \emph{loxodromic} meaning it translates along some line of $T$ (see the discussion just prior to Definition~\ref{DefOptimal}). Furthermore, $\gamma$ is elliptic if and only if $\gamma \in A$ for some free factor $A \subgroup \Gamma$ such that $[A] \in \A$. It follows that ellipticity and loxodromicity of $\gamma$ rel~$\A$ are well-defined properties, independent of the choice of a Grushko free splitting rel~$\A$, depending only on $\gamma$ itself and on~$\A$. In the case that $\gamma$ acts loxodromically on $T$: its axis $\Axis(\gamma) \subset T$ has exactly two limit points in $\bdyinf T$, represented by any pair of opposite rays in $\Axis(\gamma)$. These two points are the only points of $\bdyinf T$ fixed by the action of $\gamma$ on $\bdyinf T$; exactly one of those fixed points, denoted $\bdy^T_+\gamma$ is an attractor; and the other one, denoted $\bdy^T_-\gamma$, is a repelling fixed point. As $\gamma$ varies over all loxodromic elements, the set $\bigcup_\gamma \{\bdy^T_-\gamma,\bdy^T_+\gamma\}$ is dense in $\bdy T$ and in $\bdyinf T$. Even more, if we fix one loxodromic $\gamma$ and we vary over its set of conjugates $\gamma^\delta = \delta^\inv\gamma\delta$, each of the sets $\{\bdy^T_+\gamma^\delta \suchthat \delta \in \Gamma\}$ and $\{\bdy^T_-\gamma^\delta \suchthat \delta \in \Gamma\}$ is dense in $\bdy T$ and in $\bdyinf T$. These density statements both follow by observing that for $\tau$ any half tree of~$T$, there exists $\delta \in \Gamma$ such that $\delta \cdot \Axis(\gamma) \subset \tau$. In the case that $\gamma$ acts elliptically on~$T$ we let $v^T_\gamma$ denote the unique point of $T$ fixed by $\gamma$, that point being a natural vertex of~$T$.

\smallskip

The following lemma includes a compilation of results regarding Bowditch boundaries and Gromov boundaries, but re-expressed so as to emphasize the unique nature of actions by the group $\Aut(\Gamma;\A)$ as extensions of actions of~$\Gamma$.

\begin{lemma}
\label{LemmaTEqBdyMaps} 
For any automorphism $\Phi \in \Aut(\Gamma;\A)$, and for any two Grushko free splittings $S$ and $T$ of $\Gamma$ rel~$\A$ there exists a $\Phi$-twisted equivariant map $g \from S \mapsto T$. Furthermore,~for any such $g$ the following hold:
\begin{enumerate}
\item\label{ItemTwistedExistence}
\protect{\cite[Lemma 2.2]{GuirardelHorbez:Laminations}}
$g$ has a $\Phi$-twisted equivariant continuous extension \hbox{$\overline g \from \overline S \to \overline T$} which restricts to $\Phi$-twisted equivariant homeomorphisms $\bdy g \from \bdy S \to \bdy T$ and \hbox{$\bdyinf g \from \bdyinf S \to \bdyinf T$.}
\item\label{ItemTwistedUniqueness}
The homeomorphisms $\bdy g$ and $\bdyinf g$ are well-defined independent of $g$ and depending only on~$\Phi$, and are characterized as the unique $\Phi$-twisted equivariant homeomorphisms \break
$\bdy S \mapsto \bdy T$ and $\bdyinf S \mapsto \bdyinf T
$.
\end{enumerate}
\begin{enumeratecontinue}
\item\label{ItemTwistedActions}
For every Grushko free splitting $T$, the boundary actions of $\Gamma$ on $\bdy T$ and~$\bdyinf T$ extend uniquely to actions of $\Aut(\Gamma;\A)$. Furthermore:
\begin{enumerate}
\item\label{ItemTAWD} The actions of each $\Phi \in \Aut(\Gamma;\A)$ are the unique $\Phi$-twisted equivariant homeomorphisms $\bdy T \mapsto \bdy T$ and $\bdyinf T \mapsto \bdyinf T$. 
\item\label{ItemTAnat}
For every equivariant map \hbox{$f \from S \to T$} between Grushko free splittings the induced homeomorphisms $\bdy f \from \bdy S \to \bdy T$ and $\bdyinf f \from \bdyinf S \to \bdyinf T$ are $\Aut(\Gamma;\A)$ equivariant. 
\end{enumerate}
%
%
\end{enumeratecontinue}
\end{lemma}

\begin{proof} The construction of a $\Phi$-twisted equivariant map $g \from S \to T$ follows well-known ideas. First, $g$ takes a vertex $v$ with nontrivial stabilizer $\Stab(v)$ to the unique vertex of $T$ stabilized by $\Phi(\Stab(v))$. Next, on each orbit of vertices having trivial stabilizer, $g$ is defined arbitrarily on one vertex of that orbit, and $g$ then extends uniquely over the whole orbit so as to be $\Phi$-twisted equivariance. Finally, $g$ extends tightly over each edge.

For item~\pref{ItemTwistedExistence}, the equivariant case (i.e.\ $\Phi = \Id$) is found in \cite[Lemma 2.2]{GuirardelHorbez:Laminations}. The $\Phi$-twisted equivariant case reduces to the equivariant case as follows. Let $\Gamma \xrightarrow{\A} \Isom(S)$ be the homomorphism representing the given free splitting action represented as a group of simplicial isomorphisms of~$S$, so $\gamma \cdot x = \A(\gamma)(x)$. By composition we obtain a homomorphism $\Gamma \xrightarrow{\Phi^\inv} \Gamma \xrightarrow{\A} \Isom(S)$ defining a new free splitting action $\gamma \bullet x = \A(\Phi^\inv(\gamma))(x)$. The given $\Phi$-twisted equivariant map $g \from S \to T$ is then equivariant with respect to the new action on $S$ and the given action on $T$:
$$g(\gamma \bullet x) = g(\A(\Phi^\inv(\gamma))(x)) = g(\Phi^\inv(\gamma) \cdot x) = \Phi(\Phi^\inv(\gamma)) \cdot g(x) = \gamma \cdot g(x)
$$
The conclusions of~\pref{ItemTwistedExistence} in the equivariant case may now be applied using the action $\gamma \bullet x$ on $S$ and the given action on $T$; these imply, in turn, the desired conclusions for the $\Phi$-equivariant case with respect to the given actions on~$S$ and~$T$.

For item~\pref{ItemTwistedUniqueness}, for each nontrivial elliptic $\gamma \in \Gamma$, by combining uniqueness of elliptic fixed points with $\Phi$-twisted equivariance it follows that $\bdy g(v^S_\gamma) = g(v^S_\gamma)=v^T_{\Phi(\gamma)}$. Also, for each loxodromic $\gamma \in \Gamma$, by combining uniqueness of attracting/repelling fixed points with $\Phi$-twisted equivariance it follows that $\bdyinf g(\bdy^S_+\gamma) = \bdy^T_+\Phi(\gamma)$ and $\bdyinf g(\bdy^S_-\gamma)=\bdy^T_-\Phi(\gamma)$. The restriction of $\bdy g$ to the subset of points of $\bdy S$ that are fixed by some nontrivial element of $\Gamma$ is therefore well-defined independent of~$g$. That subset is dense in~$\bdy S$, proving uniqueness of the homeomorphism $\bdy g \from \bdy S \to \bdy S$. Similarly, the subset of points of~$\bdyinf S$ fixed by some loxodromic element of $\Gamma$ is dense in $\bdyinf S$, proving uniqueness of $\bdyinf g \from \bdyinf S \to \bdyinf S$.

For item~\pref{ItemTwistedActions}, the existence and uniqueness of the extended actions, and their characterizations in item~\pref{ItemTAWD}, follow by applying the \emph{Twisted Equivariance Principle} of Section~\ref{SectionROut} in conjunction with the existence (item~\pref{ItemTwistedExistence}) and uniqueness (item~\pref{ItemTwistedUniqueness}) of the appropriate twisted equivariant homeomorphisms.
%
%
%
%
%
For item~\pref{ItemTAnat}, one chooses $\Phi$-twisted equivariant self maps $g \from S \to S$ and $h \from T \to T$ and then applies the uniqueness conclusion~\pref{ItemTwistedUniqueness} to the two maps 
$$\bdy f \circ \bdy g, \,\, \bdy h \circ \bdy f \,\, \from \,\, \bdy S \to \bdy T
$$
using that both of those maps are twisted equivariant with respect to $\text{Id} \circ \Phi = \Phi = \Phi \circ \text{Id}$. 
\end{proof}

\paragraph{The canonical boundaries $\bdy(\Gamma;\A)$ and $\bdyinf(\Gamma;\A)$.} 
%
Following Guirardel and Horbez \cite[Section 2]{GuirardelHorbez:Laminations}, we describe these boundaries and their identifications with Bowditch and Gromov boundaries (resp.) of Grushko free splittings of $\Gamma$ rel~$\A$. See Figure~\ref{FigureBoundaryDiagrams} for summaries, including: details regarding actions by $\Gamma$ and extended actions by $\Aut(\Gamma;\A)$; notations for distinguishing between canonical boundaries and their representations as boundaries of free splittings. 

\begin{figure}
$$\xymatrix{
& \bdy(\Gamma;\A) \ar@2{~>}[dl]_{I_S} \ar@2{~>}[dr]^{I_T} \\
\bdy S \ar[rr]^{\bdy f} & & \bdy T
}
\qquad\qquad\qquad
\xymatrix{
\bdy(\Gamma;\A) \ar@2{~>}[d]_{I_S} \ar[rr]^{\bdy\Phi}
& & 
\bdy(\Gamma;\A) \ar@2{~>}[d]^{I_T} \\
\bdy S \ar[rr]^{\bdy g} & & \bdy T
}
$$
$$\xymatrix{
& \bdyinf(\Gamma;\A) \ar@2{~>}[dl]_{I_S} \ar@2{~>}[dr]^{I_T} \\
\bdyinf S \ar[rr]^{\bdyinf f} & & \bdyinf T
}
\qquad\qquad
\xymatrix{
\bdyinf(\Gamma;\A) \ar@2{~>}[d]_{I_S} \ar[rr]^{\bdyinf\Phi}
& & 
\bdyinf(\Gamma;\A) \ar@2{~>}[d]^{I_T} \\
\bdyinf S \ar[rr]^{\bdyinf g} & & \bdyinf T
}
$$
\caption{Various commutative diagrams involving canonical boundaries $\bdy(\Gamma;\A)$ and $\bdyinf(\Gamma;\A)$. For each Grushko free splitting $T$ of $\Gamma$ rel~$\A$, a squiggly arrow labelled $I_T$ denotes an equivariant homeomorphism (``identification'') of a canonical boundary with the corresponding boundary of~$T$. On the left are diagrams associated to each $\Gamma$-equivariant simplicial map $f \from S \to T$ between Grushko free splittings of $\Gamma$ rel~$\A$. On the right are diagrams associated to each choice of $\Phi \in \Aut(\Gamma;\A)$ and of a $\Phi$-twisted equivariant simplicial map $g \from S \to T$ between free splittings of $\Gamma$ rel~$\A$; the maps $\bdy\Phi$ and $\bdyinf\Phi$ are $\Phi$-twisted equivariant self-homeomorphims. The two diagrams on the left fit into a single commutative diagram with upward pointing ``inclusion'' arrows (not shown) from the lower diagram to the upper diagram, and similarly for the two on the right.
}
\label{FigureBoundaryDiagrams}
\end{figure}

Here are some details regarding $\bdy(\Gamma;\A)$ (after appropriately sprinkling in the symbol~$\infty$, these details apply as well to~$\bdy(\Gamma;\A)$). Consider the category whose objects are the induced boundary actions $\Gamma \act \bdy T$, one for every Grushko free splitting $T$ of $\Gamma$ rel~$\A$, and with a unique morphism $\bdy S \mapsto \bdy T$ for every pair of Grushko free splittings $S,T$ of $\Gamma$ rel~$\A$, that morphism being the unique $\Gamma$-equivariant homeomorphism that extends every $\Gamma$-equivariant map $S \mapsto T$ (see Lemma~\ref{LemmaTEqBdyMaps}~\pref{ItemTwistedUniqueness}, applied to $\Phi = \Id$); note in particular that each object $\Gamma \act T$ has a unique automorphism, namely its identity morphism. Using this category one may canonically identify \emph{all} of the spaces $\bdy T$ to a single space $\bdy(\Gamma;\A)$ on which $\Gamma$ acts, together with a family of $\Gamma$-equivariant homeomorphisms $I_T \from \bdy(\Gamma;\A) \approx \bdy T$, and with commutative diagrams as depicted on the left side of Figure~\ref{FigureBoundaryDiagrams}. For every free splitting of $\Gamma$ rel~$\A$ the action $\Gamma \act \bdy T$ has a unique extended action $\Aut(\Gamma;\A) \act \bdy T$ (by Lemma \ref{LemmaTEqBdyMaps}~\pref{ItemTwistedActions}). By transport of structure via $I_T$ it follows that the action $\Gamma \act \bdy(\Gamma;\A)$ has a unique extended action $\Aut(\Gamma;\A) \act \bdy(\Gamma;\A)$. As a consequence of uniqueness, this extended action is well-defined independent of~$T$. The action of each $\Phi \in \Aut(\Gamma;\A)$ on $\bdy(\Gamma;\A)$ is a $\Phi$-twisted equivariant self-homeomorphism denoted $\bdy\Phi$, with commutative diagrams as depicted on the right side of Figure~\ref{FigureBoundaryDiagrams}. 

Recall from before Lemma~\ref{LemmaTEqBdyMaps} that loxodromicity of $\gamma \in \Gamma$ is well-defined independent of the choice of Grushko free splitting $T$ rel~$\A$. By using $I_T$ to transport the defining property of ``loxodromic'' from $\bdy T$ and $\bdyinf T$ over to $\bdy(\Gamma;\A)$ and $\bdyinf(\Gamma;\A)$, we obtain the following: 
\begin{description}
\item[Canonical loxodromicity:] For all $\gamma \in \Gamma$, $\gamma$ is loxodromic if and only its action on $\bdy(\Gamma;\A)$ has exactly two fixed points $\bdy_- \gamma, \bdy_+\gamma \in \bdy(\Gamma;\A)$, a repeller and an attractor respectively, and these points lie in $\bdyinf(\Gamma;\A)$. Furthermore, as $\delta \in \Gamma$ varies each of the sets $\{\bdy_- \gamma^\delta\}$ and $\{\bdy_+\gamma^\delta\}$ is dense in $\bdy(\Gamma;\A)$ and in $\bdyinf(\Gamma;\A)$.
\end{description}

%

\paragraph{Embedding canonical boundaries of relative free factors.} Consider a proper, nonatomic free factor $F$ of $\Gamma$ rel~$\A$. Its free factor system~$\A \restrict F$ is nonfull, and so the canonical boundaries of $F$ rel~$\A \restrict F$ are nonempty, and these boundaries fit into the following $F$-equivariant commutative diagrams with upward pointing ``inclusion arrows'':
$$\xymatrix{
\bdy(F,\A \restrict F) \,\,
	\ar[rrr]
&&& \,\,\bdy(\Gamma;\A) \\
\bdyinf(F,\A \restrict F) \,\,
	\ar[rrr]
	\ar@{^{(}->}[u]
&&& \,\, \bdyinf(\Gamma;\A)  
	\ar@{^{(}->}[u]
}$$
Following \cite[Section]{GuirardelHorbez:Laminations} and~\cite[Lemma 3.3]{\LymanCTTag}, these embeddings 
are described by first choosing any Grushko free splitting $T$ of $\Gamma$ rel~$\A$ and letting $T_F \subset T$ be the minimal subtree of the restricted action $F \act T$.\footnote{As discussed in the sources cited, these boundary embeddings exist for any subgroup $F$ of $\Gamma$ having ``finite Kurosh rank rel~$\A$''. We need only the case of a free factor of $\Gamma$ rel~$\A$.} The inclusion $T_F \hookrightarrow T$ extends to continuously to an $F$-equivariant embedding $\overline T_F \hookrightarrow \overline T$, and then restricts to $F$-equivariant boundary embeddings depicted in the diagram below. The horizontal arrows of the above diagram are then well-defined, independent of~$T$, as the unique dashed arrows that make the diagram below commute:
$$\xymatrix{
	& \bdy T_F \,\, 
		\ar[r] 
	& \,\, \bdy T 
		\\
\bdy(F,\A \restrict F) 
		\ar@2{~>}[ur]^{I_{T_F}}
		\ar@{-->}[rrr] 
	&&& \bdy(\Gamma;\A) 
		\ar@2{~>}[ul]_{I_T}  \\
\bdyinf(F,\A \restrict F) 
		\ar@2{~>}[dr]_{I_{T_F}}
		\ar@{^{(}->}[u]
		\ar@{-->}[rrr] 
	&&& \bdyinf(\Gamma;\A) 
		 \ar@2{~>}[dl]^{I_T}
		\ar@{^{(}->}[u] \\	
	& \bdyinf T_F  \,\,  
		\ar@{^{(}->}[uuu]
		\ar[r] 
	& \,\, \bdyinf T 
		\ar@{^{(}->}[uuu]
}$$

\renewcommand\line[3]{{\overline{#1#2}}{\,}^{#3}}

\medskip
\emph{An intrinsic construction of canonical boundaries.} Here is a construction of $\bdyinf(\Gamma;\A)$ that avoids ZFC paradoxes that are encountered when identifying the boundaries of the proper class of \emph{all} Grushko free splittings of $\Gamma$ rel~$\A$. Define a \emph{free basis of $\Gamma$ rel~$\A$} to be a subset of $\Gamma$ of the form $\Sigma = A_1 \union\cdots\union A_K \union \{b_1,\ldots,b_n\} \subset \Gamma$ where $\Gamma$ has a free factorization $\Gamma = A_1 * \cdots * A_K * B$ such that $\A=\{[A_1],\ldots,[A_K]\}$, and where $\{b_1,\ldots,b_n\}$ is a free basis of~the cofactor~$B$. The set $\Sigma$ generates $\Gamma$, its word metric is called a \emph{word metric rel~$\A$}, and the Milnor-Svarc lemma holds: the identity map on $\Gamma$ is a quasi-isometry between any two word metrics rel~$\A$; also, for any Grushko free splitting $T$ of $\Gamma$ rel~$\A$ and any base point $x \in T$, the orbit map $\gamma \mapsto \gamma \cdot x$ is a quasi-isometry $\Gamma \mapsto T$ from a word metric rel~$\A$ to the simplicial metric on~$T$. Since~$T$ is Gromov hyperbolic, it follows that $\Gamma$ is word hyperbolic with respect to any word metric rel~$\A$. These word metrics all determine the same class of quasigeodesic rays in~$\Gamma$, the same Hausdorff equivalence relation amongst such rays, and the same Gromov topology on the set of Hausdorff classes of rays, thus defining $\bdyinf(\Gamma;\A)$. The orbit map $\Gamma \mapsto T$ to a Grushko free splitting of $\Gamma$ rel~$\A$ induces a well-defined $\Gamma$-equivariant homeomorphism $I_T \from \bdyinf(\Gamma;\A) \to \bdyinf T$, and the diagrams above (for $\bdyinf$) commute for all choices of $S$ and~$T$. The group $\Aut(\Gamma;\A)$ acts on $\Gamma$ by quasi-isometries with respect to word metrics rel~$\A$, thus inducing an action $\Aut(\Gamma;\A) \act \bdyinf(\Gamma;\A)$, and the identifications $I_T$ are all $\Aut(\Gamma;\A)$ equivariant. Furthermore, if $F \subgroup \Gamma$ is a free factor rel~$\A$ (or, more generally, a subgroup of finite Kurosh rank rel~$\A$) then one can verify that the inclusion map $F \hookrightarrow \Gamma$ is a quasi-isometric embedding from word metrics on $F$ rel~$\A \restrict F$ to word metrics on $\Gamma$ rel~$\A$, using the fact that for any Grushko free splitting $T$ of $\Gamma$ rel~$\A$ the inclusion into $T$ of its $F$-minimal subtree is a quasi-isometric embedding with respect to simplicial metrics; this defines the canonical $F$-equivariant embedding $\bdy(F;\A \restrict F) \hookrightarrow \bdy(\Gamma;\A).$

It is also interesting to ponder how one could extend the above description of $\bdyinf(\Gamma;\A)$ to obtain an intrinsic definition of the observer's topology on $\bdy(\Gamma;\A)$, perhaps by incorporating a point $v_A \in \bdy(\Gamma;\A)$ for each $A \subgroup \Gamma$ with $[A] \in \A$, and using some kind of canonical ``observer's neighborhood basis'' of~$v_A$ in $\bdyinf(\Gamma;\A)$. 

\subsubsection{Line spaces and attracting laminations}
\label{SectionLamsAndExpFact}

\paragraph{The space of lines $\wt\Bline(\Gamma;\A)$ and its $\Gamma$-quotient $\Bline(\Gamma;\A)$.} For any Hausdorff space $X$, its double space $X^\SetTwo$ may be described in two more-or-less equivalent ways. One description (arising in \cite{\BookOneTag}) is the orbit space
$$X^\SetTwo = (X \times X - \Delta) \bigm/ \mathbb Z / 2 \mathbb Z
$$
where $\Delta \subset X \times X$ denotes the diagonal, $X \times X - \Delta$ has the subspace topology relative to the product topology, and $\mathbb Z / 2 \mathbb Z$ acts on $X \times X$ by transposing coordinates. Another convenient description of $X^\SetTwo$ is the set of two point subsets $\{\xi,\eta\} \subset X$ with topology generated by basis elements $B\{U,V\}$, one for each unordered pair of disjoint open subsets \hbox{$U,V \subset X$,} defined by the formula
$$B\{U,V\} = \bigl\{\{\xi,\eta\} \subset X \suchthat  \{\xi,\eta\} \intersect U \ne \emptyset \text{ and } \{\xi,\eta\} \intersect V \ne \emptyset\bigr\}
$$
We note that the natural bijection between these two descriptions of $X^\SetTwo$ is a homeomorphism. From the second description it follows also that for each subset $Y \subset X$, the double space topology on the subset $Y^\SetTwo \subset X^\SetTwo$ coincides with the subspace topology induced from the double space topology on $X^\SetTwo$.

The \emph{space of lines} of $\Gamma$ rel~$\A$ is $\wt\Bline(\Gamma;\A) = \bdy(\Gamma;\A)^\SetTwo$. Elements of $\wt\Bline(\Gamma;\A)$ are sometimes called \emph{(abstract) lines}.
For any Grushko free splitting $T$ of $\Gamma$ rel~$\A$ there is an identification $\wt\Bline(\Gamma;\A) \approx \wt\Bline T = \bdy T^\SetTwo$ that is induced by the identification $I_T \from \bdy(\Gamma;\A) \approx \bdy T$. Given an abstract line $\ell = \{\xi,\eta\} \in \wt\Bline T \approx \wt\Bline(\Gamma;\A)$, the associated path in $T$ whose closure has endpoints $\xi$ and $\eta$ in $\overline T$ is called the \emph{(concrete) realization of $\ell$ in $T$} and is denoted variously as $\ell(T)$ or $\line\xi\eta T$ or $(\xi,\eta)^T$. We say that $\ell(T)$ is \emph{bi-infinite}, \emph{singly infinite}, of \emph{finite} depending on whether $\{\xi,\eta\} \intersect \bdyinf T$ has cardinality $2$, $1$ or $0$, respectively;
while this terminology should be clear by context, in the setting of laminations we shall try to avoid confusion by using the terminology ``bi-infinite line'' or ``bi-infinite leaf''.


The natural action $\Gamma \act \bdy(\Gamma;\A) \approx \bdy T$ induces an action $\Gamma \act \wt\Bline(\Gamma;\A) \approx \wt\Bline T$ whose orbit space is denoted $\Bline(\Gamma;\A) = \Bline T$. One may therefore identify $\Bline(\Gamma;\A)$ pointwise with the set of $\Gamma$-orbits of concrete lines in~$T$. Recall from \BookOne\ that in the classical context of $\Out(F_n)$ one may also visualize and study elements of $\Bline(F_n)$ as bi-infinite edge paths without cancellation in any marked graph; see \LymanRTT\ for the similar point of view using lines in graphs of groups. Given an element $\beta \in \Bline(\Gamma;\A)$, we shall for the most part express $\beta$ with a choice of orbit representative $\ell = \{\xi,\eta\} \in \wt\Bline(\Gamma;\A)$ and/or with the realization of $\ell$ as a concrete line $L = \line\xi\eta T \in \wt\Bline T$ in some Grushko free splitting~$T$. In such a situation we occasionally use the operator $\orb$ to denote the ``orbit'' of an abstract or concrete line, leading to the notation like $\beta = \orb\{\xi,\eta\} = \orb \line\xi\eta T = \orb L$. We also refer to this situation with such language as ``\emph{$\ell$ is an orbit representative of $\beta$}'' or an ``\emph{abstract line representing $\beta$}'', or ``\emph{$L$ is a concrete line in $T$ representing~$\beta$}''.

A \emph{periodic} line is one of the form $\{\bdy_-\gamma,\bdy_+\gamma\} \in \wt\Bline(\Gamma;\A)$, for some loxodromic~$\gamma \in \Gamma$. Periodic lines are dense in $\wt\Bline(\Gamma;\A)$, equivalently they are dense in $\wt\Bline T$ for any Grushko free splitting $T$ of $\Gamma$ rel~$\A$; the latter holds because any finite path in $T$ is a subpath of some finite path whose first and last oriented edges are in the same orbit under the action of $\Gamma$ on oriented edges of~$T$. Note that if $F \subgroup \Gamma$ is an elementary free factor rel~$\A$ (see~Lemma~\ref{LemmaElementary}) then there exists a loxodromic $\gamma \in F$, and the subset $\wt\Bline(F;\A \restrict F) \subset \wt\Bline(\Gamma;\A)$ is identical to the periodic line $\{\bdy_-\gamma,\bdy_+\gamma\}$ for any choice of loxodromic $\gamma \in F$. 

\paragraph{Laminations.} A \emph{lamination}\footnote{The terminology in \cite{GuirardelHorbez:Laminations} is ``algebraic lamination''.} of $\Gamma$ rel~$\A$ is \emph{either} a closed $\Gamma$-invariant subset upstairs in $\wt\Bline(\Gamma;\A)$ \emph{or} a closed subset downstairs in $\Bline(\Gamma;\A)$. Note that the quotient map $\wt\Bline(\Gamma;\A) \mapsto \Bline(\Gamma;\A)$ induces a bijection between the ``upstairs'' and ``downstairs'' versions of laminations.  

The natural action $\Aut(\Gamma;\A) \act \bdy(\Gamma;\A)$, described following Lemma~\ref{LemmaTEqBdyMaps}, induces a natural action $\Aut(\Gamma;\A) \act \wt\Bline(\Gamma;\A)$; for each $\Phi \in \Aut(\Gamma;\A)$ we denote the resulting homeomorphism as $\Phi_\sharp \from \wt\Bline(\Gamma;\A) \to \wt\Bline(\Gamma;\A)$. The restricted action by inner automorphisms $\Inn(\Gamma;\A) \act\wt\Bline(\Gamma;\A)$ preserves each individual $\Gamma$ orbit, and so there is a natural induced action $\Out(\Gamma;\A) \act \Bline(\Gamma;\A)$; for each $\phi \in \Out(\Gamma;\A)$ the resulting homeomorphism is denoted $\phi_\sharp \from \Bline(\Gamma;\A) \to \Bline(\Gamma;\A)$.


\paragraph{The free factor support of a lamination.} On this topic we follow \cite[Corollary 4.12]{\LymanCTTag}, with some equivalent reformulations, particularly in the definition of ``carrying'' and in Lemma~\ref{LemmaFFSupportLambda}). 

For any free factor $F$ of $\Gamma$ rel~$\A$ with restricted free factor system $\A \restrict F$, the $F$-equivariant embedding $\bdy(F;\A \restrict F) \mapsto \bdy(\Gamma;\A)$ induces an $F$-equivariant embedding $\wt\Bline(F;\A \restrict F) \mapsto \wt\Bline(\Gamma;\A)$; we identify $\wt\Bline(F;\A \restrict F)$ with the image of this embedding. 

Consider a line $\ell \in \wt\Bline(\Gamma;\A)$ representing $\beta  \in \Bline(\Gamma;\A)$. Given a free factor system~$\F$ of $\Gamma$ rel~$\A$, to say that $\F$ \emph{carries} or \emph{supports} $\beta$ (or $\ell$) means that there exists a free factor $F$ of $\Gamma$ rel~$\A$ such that $[F] \in \F$ and such that $\ell \in \wt\Bline(F;\A \restrict F)$. Since periodic lines in $\wt\Bline(F;\A \restrict F)$ form a dense subset of $\wt\Bline(F;\A \restrict F)$, this definition of ``carries'' is equivalent to the one in \cite{\LymanCTTag} saying that $\beta$ is in the closure of the periodic lines in $\wt\Bline(F;\A \restrict F)$. Also, this definition of ``carries'' depends only on $\beta$, independent of the choice of orbit representative~$\ell$. Finally, only a nonatomic free factor system rel~$\A$ can carry a line; clearly $\A$ itself carries no line at all.

Given a free factor system~$\F$ of $\Gamma$ rel~$\A$, to say that $\F$ \emph{carries} or \emph{supports} $\Lambda$ means that $\F$ carries every $\beta \in \Lambda$. To say that $\Lambda$ \emph{fills $\Gamma$ rel~$\A$} means that the only free factor system rel~$\A$ that carries $\Lambda$ is the full free factor system $\{[\Gamma]\}$. More generally the \emph{free factor support} of $\Lambda$ rel~$\A$, denoted $\F\Lambda$, will be defined by the equivalent conditions in Lemma~\ref{LemmaFFSupportLambda} to follow, after reviewing some notation.

Consider any free factor system~$\F$ of $\Gamma$ rel~$\A$.\footnote{Free factor systems rel~$\A$ are defined slightly differently in \cite{Lyman:CT}: they do \emph{not} include what we call ``atomic components'', i.e.\ elements of~$\A$ itself. This difference is unimportant, because the atomic components are determined by the nonatomic ones.} 
We recall from \cite[Section 2.5]{HandelMosher:RelComplexHyp} its \emph{free factor system depth} $\DFF(\F)$, a numerical invariant also know as its \emph{complexity} \cite[Proposition 2.11]{Lyman:RTT}. Writing $\F = \{[F_1],\ldots,[F_K]\}$ for some free factorization $\Gamma = F_1 * \cdots * F_K * B$ with cofactor~$B$, this invariant is
$$\DFF(\F) = K + 2 \rank(B) - 1
$$
For any nested pair of free factor systems $\A \sqsubset \F$ of $\Gamma$ the \emph{complexity of $\F$ rel~$\A$}, denoted $\cx(\F;\A)$, is defined as follows \cite[Section 4]{Lyman:CT}: letting $\{[F_1],\ldots,[F_J]\}$ denote the nonatomic components of $\F$, the complexity is the numerical sequence $\DFF(\A \restrict F_j)$ rewritten by an index permutation so as to be in nonincreasing order:
$$ \cx(\F;\A) = \bigl(\DFF(\A \restrict F_1) \ge \cdots \ge \DFF(\A \restrict F_J)\bigr)
$$
Complexities of free factor systems rel~$\A$ are well-ordered lexicographically. In the following lemma, the first conclusion is Corollary 4.12 of \LymanCT:

\begin{lemma}
\label{LemmaFFSupportLambda}
For every subset $\Lambda \subset \Bline(\Gamma;\A)$ there exists a unique free factor system~$\F\Lambda$ of $\Gamma$ rel~$\A$ such that $\F\Lambda$ carries $\Lambda$ and such that, amongst the set of \emph{all} free factor systems of $\Gamma$ rel~$\A$ that carry $\Lambda$, any of the following equivalent conditions holds:
\begin{enumerate}
\item\label{ItemSuppLambdaCX}
$\F\Lambda$ has minimal value of the complexity $\cx(\cdot)$;
\item\label{ItemSuppLambdaSqsubset}
$\F\Lambda$ is minimal with respect to the partial order $\sqsubset$;
\item\label{ItemSuppLambdaMeet}
$\F\Lambda$ is the meet of all free factor systems that carry $\Lambda$.
\end{enumerate}
\end{lemma}

\subparagraph{Remark.} Note that $\Lambda$ fills~$\Gamma$ rel~$\A$ if and only if $\F\Lambda=\{[\Gamma]\}$.

\begin{proof} Let $\{\F_i\}_{i \in I}$ be an indexing of the free factor systems of $\Gamma$ rel~$\A$ that carry~$\Lambda$. Knowing already from \cite[Corollary 4.12]{\LymanCTTag} the existence and uniqueness of $\F\Lambda$ having minimal complexity (item~\pref{ItemSuppLambdaCX}), we choose the notation so that $0 \in I$ and $\F_0 = \F\Lambda$. 

To prove conclusion~\pref{ItemSuppLambdaSqsubset} we will apply the relation $\F_i \meet \F_j \sqsubset \F_i$ which holds for all $i,j \in I$ \cite[Corollary 2.7]{HandelMosher:RelComplexHyp}. As shown in the proof of \cite[Corollary 4.12]{\LymanCTTag}, if $\F_i \sqsubset \F_j$ then $\cx(\F_i) \le \cx(\F_j)$ with equality if and only if $i=j$. For all $j \in I$, we therefore have $\cx(\F_0 \meet \F_j) \le \cx(\F_0)$; it then follows by minimality of $\cx(\F_0)$ that $\F_0 \meet \F_j = \F_0$. Using that $\F_0 \meet \F_j \sqsubset \F_j$ it follows that $\F_0 \sqsubset \F_j$ for all $j \in I$, proving that $\F_0$ is a minimal with respect to $\sqsubset$. If there was some different $\F_k$ that is also minimal with respect to $\sqsubset$ then at least one of the nesting relations $\F_0 \meet \F_k \sqsubset \F_0$ or $\F_0 \meet \F_k \sqsubset \F_k$ would be proper, contradicting minimality of $\cx$.

For conclusion~\pref{ItemSuppLambdaMeet} recall that $\meet \F_i$ is the unique $\sqsubset$-maximal free factor system of $\Gamma$ rel~$\A$ having the property that $\meet\F_i \sqsubset \F_j$ for all $j \in I$ (see \cite[Corollary 2.15]{HandelMosher:RelComplexHyp}, and see Section~\ref{SectionROut}). Taking $j=0$ it follows that $\meet\F_i \sqsubset \F_0$, and then by maximality of $\meet\F_i$ it follows that $\meet\F_i = \F_0$, proving~\pref{ItemSuppLambdaMeet}.
\end{proof}

\subparagraph{Remark.} We note that $\F\Lambda$ is also the unique free factor system carrying $\Lambda$ with maximum value of the numerical invariant~$\DFF(\cdot)$: this follows by using $\F_i \meet \F_j \sqsubset \F_i$ together with \cite[Lemma 2.14]{HandelMosher:RelComplexHyp} to conclude that that $\DFF(\F_i \meet \F_j) \ge \DFF(\F_i)$, with equality if and only if $\F_i \meet \F_j = \F_i$.

\bigskip

For any subset $\Lambda \subset \Bline(\Gamma;\A)$ and any Grushko free splitting $T$ of $\Gamma$ rel~$\A$, we define the \emph{concrete support} of $\Lambda$ in $T$ to be the $\Gamma$-invariant subforest $\tsp(\Lambda;T) \subset T$ equal to the union of concrete lines in $T$ representing elements of $\Lambda$. By tracing through the definitions and using the identification $\wt\Bline(\Gamma;\A) \approx \wt\Bline T$ one obtains the following fact, which is a translation into the language of Bass-Serre trees of a corresponding graphs-of-groups statement found in \cite[Section 4]{\LymanCTTag} under the heading \emph{Laminations}.

\begin{lemma}
\label{LemmaFillingLamFillsT}
For any $\Lambda \subset \Bline(\Gamma;\A)$, any Grushko free splitting $T$ of $\Gamma$ rel~$\A$, and any $\Gamma$-invariant subforest $\tau \subset T$, we have \, $\F\Lambda \sqsubset \F\tau \iff \tsp(\Lambda;T) \subset \tau$.
\qed
\end{lemma}

Consider an abstract line $\ell \in \wt\Bline(\Gamma;\A)$ representing $\beta = \orb\ell = \Bline(\Gamma;\A)$, and its concrete realization $L \subset T$ in some Grushko free splitting $T$ of $\Gamma$ rel~$\A$. To say that $\beta$ (or~$\ell$, or $L$) is \emph{birecurrent} means that $L$ is bi-infinite and for any finite subpath $\alpha \subset L$ and any subray $R \subset L$ there exists $g \in \Gamma$ such that $g \cdot \alpha \subset R$; that is, $L$ contains infinitely many translates of $\alpha$ approaching each of its ends. This property depends only on~$\beta$, independent of $\ell$ and $L$, and this property is invariant under the action $\Out(\Gamma;\A) \act \Bline(\Gamma;\A)$ (see \cite[Lemma 4.1]{\LymanCTTag}).

\begin{definition}[Attracting laminations]
\label{DefinitionAttrLam}
Given $\phi \in \Out(\Gamma;\A)$ and a lamination $\Lambda \subset \Bline(\Gamma;\A)$, to say that $\Lambda$ is an \emph{attracting lamination of $\phi$} means that there exists $\ell \in \wt\B(\Gamma;\A)$ representing $\beta = \orb\ell$ such that the following hold:
\begin{enumerate}
\item $\Lambda$ is the closure of the subset $\{\beta\} \subset \Bline(\Gamma;\A)$;
\item $\beta$ is birecurrent;
\item\label{ItemNotEFF}
For any loxodromic $\gamma \in \Gamma$ that is contained in an elementary free factor of $\Gamma$ rel~$\A$ we have $\ell \ne \{\bdy_-\gamma,\bdy_+\gamma\}$ (see~Lemma~\ref{LemmaElementary});
\item There exists an \emph{attracting neighorhood} $U \subset \Bline(\Gamma;\A)$ of $\beta$ with respect to some iterate $\phi^p$ (for some $p=p(\Lambda) \ge 1$), meaning that 
\begin{enumerate}
\item $\phi_\sharp^p(U) \subset U$;
\item $\{\phi_\sharp^{pj}(U) \suchthat j \ge 0\}$ is a neighborhood basis of $\beta$ in $\Bline(\Gamma;\A)$.
\end{enumerate}
\end{enumerate}
Any such $\beta$ (or $\ell$) is called a \emph{generic leaf} of~$\Lambda$. The set of all attracting laminations of $\phi \in \Out(\Gamma;\A)$ is denoted $\L(\phi)$. 
\end{definition}

\begin{lemma}[\protect{\cite{\LymanCTTag} Lemma 4.2; and see Theorem~\ref{TheoremStrataLamCorr} below}]
For each $\phi \in \Out(\Gamma;\A)$ the set $\L(\phi)$ is finite, and it is $\phi$-invariant under the natural action of $\Out(\Gamma;\A)$ on subsets of $\Bline(\Gamma;\A)$.
\qed\end{lemma}

The following lemma describes the special form of $\F\Lambda$ in the case that $\Lambda$ is an attracting lamination:

\begin{lemma}[\protect{\cite[Lemma 4.11]{\LymanCTTag}}]
\label{LemmaAttLamSupp}
For any $\phi \in \Out(\Gamma;\A)$ and any $\Lambda \in \L(\phi)$ there exists a nonatomic free factor $F$ of $\Gamma$ rel~$\A$ with associated decomposition $\A=\A_F \union \bar\A_F$ (see Lemma~\ref{LemmaExtension}~\pref{ItemPartitionsA}) such that for any generic leaf $\beta$ of $\Lambda$ we have 

\displayqed{\F\Lambda=\F\{\beta\} = \{[F]\} \union \bar\A_F}
%
\end{lemma}

The following result generalizes \cite[Lemma 3.2.4]{\BookOneTag}, and follows the same lines of proof once enough tools from relative train track theory are established. We put off the proof until Section~\ref{SectionAxisConstruction}, just following the statement of Theorem~\ref{TheoremStrataLamCorr}, which is a summary statement of results from \cite{} that supply the requisite tools.

\begin{lemma} 
\label{LemmaSupportDistinction}
Attracting laminations are distinguished by their free factor supports: for any $\phi \in \Out(\Gamma;\A)$ and any $\Lambda,\Lambda' \in \L(\phi)$, if $\F\Lambda=\F\Lambda'$ then $\Lambda=\Lambda'$. In particular a filling attracting lamination, if it exists, is unique: there is at most one $\Lambda \in \L(\phi)$ such that $\F\Lambda=\{[\Gamma]\}$. Also, there is a natural bijection $\L(\phi) \leftrightarrow \L(\phi^\inv)$ defined by $\Lambda \leftrightarrow \Lambda'$ if and only if $\F\Lambda=\F\Lambda'$.
\end{lemma}

\subsubsection{Fold axes (statement of Proposition~\ref{PropAxisInFS})}
\label{SectionFoldAxes}
In this section we define fold axes and their transition matrices, we define EG-aperiodicity of a fold axis, and we define the expansion factor of an EG-aperiodic fold axis. We also state Proposition~\ref{PropAxisInFS} asserting existence of EG-aperiodic fold axes and well-definedness of expansion factors under appropriate hypotheses.



\begin{definition}[Fold axes in $\FS(\Gamma;\A)$]
\label{DefFoldAxes}
Given $\phi \in \Out(\Gamma;\A)$, a \emph{fold axis} of $\phi$ rel~$\A$ is a bi-infinite fold path in $\FS(\Gamma;\A)$ of the form
$$\cdots \xrightarrow{f_{-2}} T_{-2} \xrightarrow{f_{-1}} T_{-1} \xrightarrow{f_0} T_0 \xrightarrow{f_1} T_1 \xrightarrow{f_2} T_2 \xrightarrow{f_3} \cdots
$$
such that $f_i(\Vertices T_{i-1}) \subset \Vertices T_i$ for each $i \in \Z$, and such that for some integer $p \ge 1$ called the \emph{period} of the axis, and for some automorphism $\Phi \in \Aut(\Gamma;\A)$ representing $\phi$, there exists a sequence of $\Phi$-twisted equivariant simplicial isomorphisms \hbox{$h^l_{l-p} \from T_{l} \to T_{l-p}$} $(l \in \Z)$ making the following diagram commute:
$$\xymatrix{
& \cdots \ar[r]  & T_{l-1} \ar[rr]^{f_{l}} \ar[dl]^<<<<{h^{l-1}_{l-p-1}} && T_{l} \ar[rr]^{f_{l+1}}  \ar[dl]^<<<<{h^l_{l-p}} && T_{l+1} \ar[r]  \ar[dl]^<<<<{h^{l+1}_{l-p+1}}& \cdots\\
\cdots \ar[r]  & T_{l-p-1} \ar[rr]_{f_{l-p}}  && T_{l-p} \ar[rr]_{f_{l-p+1}} && T_{l-p+1}  \ar[r] & \cdots
}$$
It follows from this definition that $T_l = T_{l-p} \cdot \phi$ for all $l \in \Z$. It also follows that the free factor system $\B = \FellT_i$ is well-defined independent of~$i$ and $\phi$-invariant, and that $\A \sqsubset \B$; we shall say that this is a fold axis \emph{of $\phi$, rel~$\A$, with respect to $\B$}.

For each $l \in \Z$ the \emph{first return map} on $T_l$ is the $\Phi$-equivariant map $F_l \from T_l \to T_l$ that is defined by composing arrows in the following commutative diagram:
$$\xymatrix{
                         && T_{l+p} \ar[drr]^{h^{l+p}_l} \\
T_l \ar@{-->}[rrrr]^{F_l} \ar[urr]^{f^l_{l+p}} \ar[drr]_{h^l_{l-p}} &&&& T_l \\
                         &&T_{l-p} \ar[urr]_{f^{l-p}_l}
}$$
\end{definition}

\subparagraph{General properties of fold axes.} 
In Definition~\ref{DefFoldAxes}, the quantifier ``for some representative~$\Phi$'' may be changed to ``for any representative $\Phi$'': when $\Phi \from \Gamma \to \Gamma$ is replaced by $i_g \circ \Phi$ with $i_g \in \Inn(\Gamma)$, each map $h^l_{l-p} \from T_l \to T_{l-p}$ is changed by precomposing it with the translation of $T_l$ corresponding to~$g$. The maps $h^l_{l-p}$ therefore induce well-defined bijections, independent of $\Phi$, between the orbits of the induced edge actions $\Gamma \act \Edges(T_{l})$ and $\Gamma \act \Edges(T_{l-p})$.

For all $l,d \in \Z$ we have $T_l = T_{l-dp} \cdot \phi^d$. In fact we have a system of $\Phi^d$-twisted equivariant simplicial isomorphisms $h^l_{l-dp} \from T_l \to T_{l-dp}$, satisfying the following index contraction formula whenever $d=d_1+d_2$:
$$h^l_{l-dp} : T_l \xrightarrow{h^l_{l-d_1p}} T_{l-d_1p} \xrightarrow{h^{l-d_1p}_{l-(d_1+d_2)p}} T_{l-(d_1+d_2)p} = T_{l-dp}
$$
To define these maps, for $d \ge 1$ we may use this index contraction formula to give a definition by induction on $d$ starting from the given maps $h^l_{l-p}$; whereas for $d \le -1$ we first define $h^l_{l+p} = (h^{l+p}_l)^\inv$ and then we use the formula to define $h^l_{l-dp}$ by induction on $\abs{d}$.

If $d \ge 1$ then any fold axis of $\phi$ of period $p$ is also a fold axis of $\phi^d$ of period $dp$, with first return map $F^d_l = (F_l)^d \from T_l \to T_l$; this follows from foldability of the map $f^{l-dp}_{l} \from T_{l-dp} \to T_{l}$ together with commutativity of the following diagram:
$$\xymatrix{
                         && T_{l+dp} \ar[drr]^{h^{l+dp}_l} \\
T_l \ar[rrrr]^{F_l^d} \ar[urr]^{f^l_{l+dp}} \ar[drr]_{h^l_{l-dp}} &&&& T_l \\
                         &&T_{l-dp} \ar[urr]_{f^{l-dp}_l}
}$$
Furthermore, from this diagram it is clear that for each $d \ge 1$ the $\Phi^d$-twisted equivariant map $F^d_l \from T_l \to T_l$ is injective on each edge $E \subset T_l$ for each $d \ge 1$. Applying this to all multiples of a fixed $d \ge 1$, this says that $F^d_l$ is a train track representative of $\phi^d$.

\bigskip

We turn now to a discussion of transition matrices, progressing from the general setting of foldable maps to the specific setting of fold axes.

\subparagraph{The transition matrix of a foldable map.} Consider a foldable map $f \from S \to T$ such that $f(\Vertices S) \subset \Vertices T$ (for example $f$ could be any foldable map along any fold axis as defined using Definition~\ref{DefFoldAxes}). Let $J,I \ge 1$ be the numbers of edge orbits of $S,T$ respectively, and suppose that we have chosen enumerated edge orbit representatives $e_1,\ldots,e_J$ in $S$ and $e'_1,\ldots,e'_I$ in $T$. From these choices we obtain an $I \times J$ \emph{transition matrix} $Mf$, defined in terms of crossing numbers (Definition~\ref{DefCrossing}): 
$$M_{ij}f = \<f(e_j), \Gamma \cdot e'_i\> = \<f(e),\Gamma \cdot e'_i\>
$$ 
namely the number of times that the path $f(e_j) \subset T$ crosses edges in the orbit of the edge $e'_i \subset T$. The matrix $Mf$ is unchanged if $e_j$ and $e'_i$ are rechosen \emph{within the same orbits}.

\subparagraph{Transition matrices along fold paths.} Consider now a fold path denoted as in Definition~\ref{DefFoldPaths},
$$\cdots \xrightarrow{f_{l-1}} T_{l-1} \xrightarrow{f_l} T_l \xrightarrow{f_{l+1}} T_{l+1} \xrightarrow{f_{l+2}} \cdots
$$
and satisfying the property that $f_l(\Vertices T_{i-1}) \subset \Vertices T_l$ (for example, this fold path could be a bi-infinite fold axis as defined using Definition~\ref{DefFoldAxes}). We suppose also that edge orbit representatives are chosen for each free splitting $T_k$ along the sequence. For each $k \le l$ we therefore have a transition matrix $M f^k_l$. For each $e \in \Edges(T_k)$ its image path in $T_l$ can be written as a concatenation of edges 
$$f^k_l(e) = e_{i(1)} \ldots e_{i(A)}
$$
By foldability, for each $l \le m$ the path $f^k_m(e)$ can be written as a concatenation, without backtracking, of the form
$$f^k_m(e) = f^l_m(e_{i(1)}) \ldots f^l_m(e_{i(A)})
$$
Combining this with a simple counting argument, it follows that

\begin{lemma}
\label{LemmaTransMatMult} 
For each $k \le l \le m$ the transition matrix of the composition $f^k_m = f^l_m \circ f^k_l$ equals the product of the transition matrices of $f^l_m$ and $f^k_l$. That is, $M\!f^k_m \, = \, M\!f^l_m \,\, M\!f^k_l$.
\qed\end{lemma}

\subparagraph{Transition matrices along a fold axis.} Consider now a fold axis for some $\phi \in \Out(\Gamma;\A)$ as denoted as in Definition~\ref{DefFoldAxes}. We may choose enumerated edge orbit representatives for each free splitting $T_l$ along the fold axis, subject to the following constraint:
\begin{description}
\item[Edge Orbit Periodicity:]
Each simplicial isomorphism $h^{l}_{l-dp} \from T_l \mapsto T_{l-p}$ maps each edge orbit representative to an edge orbit representative, preserving enumeration. 
\end{description}
This criterion is met by first choosing enumerated edge orbit representatives arbitrarily for each of $T_0,\ldots,T_{p-1}$, and then for each $l$ outside that range, letting $l = qp+r$ with $q \in \Z$ and $r \in \{0,\ldots,p-1\}$, using $h^l_r \from T_l \to T_r$ pull back the enumerated edge orbit representatives from $T_r$ to $T_l$. Having chosen such representatives, the transition matrices along the fold axis are now defined. We may also use the enumerated orbit representative $\{e_i\}$ of $T_l$ to define the transition matrix $MF_l$ of the first return map $F_l \from T_l \to T_l$, whose entries  are $M_{ij} F_l = \<F_l(e_j), \Gamma \cdot e_i\>$; note that if the enumeration indices are permuted then $MF_l$ is changed by conjugation with the associated permutation matrix.

The following is a quick consequence of these definitions:

\begin{lemma}\label{LemmaTransMatSame}
Given a fold axis for $\phi \in \Out(\Gamma;\A)$ with enumerated orbit representatives satisfying \emph{Edge Orbit Periodicity}, we have the equations of transition matrixes $MF_l = M f^l_{l+p}$ for any $l \in \Z$. Furthermore, for any $d \ge 1$, we have $(MF_l)^d = M f^l_{l+dp}$.
\qed\end{lemma}

\subparagraph{EG-aperiodic fold axes and their expansion factors.} Recall that a square matrix $M$ of non-negative integers is \emph{irreducible} if for each $i,j$ there exists $k$ such that $M^k_{ij} \ge 1$. If furthermore $k$ can be chosen independent of $i,j$ then $M$ is \emph{aperiodic}, also known as \emph{Perron-Frobenius}. If $M$ is irreducible then, by the Perron-Frobenius theorem, there is a unique $\lambda \ge 1$ which is an eigenvalue for some non-negative eigenvector of~$M$. To say that $M$ is \emph{EG-aperiodic} means it is aperiodic and $\lambda > 1$.

Consider now $\phi \in \Out(\Gamma;\A)$, and a fold axis of $\phi$ in $\FS(\Gamma;\A)$ (with respect to some $\phi$-invariant free factor system~$\B$), denoting that fold axis as in Definition~\ref{DefFoldAxes}. Consider also a choice of enumerated edge orbit representatives satisfying \emph{Edge Orbit Periodicity}, hence all transition matrices discussed earlier are defined. 

To say that this axis is \emph{EG-aperiodic} means that there exists $l \ge 1$ such that the transition matrix $M F_l$ is EG-aperiodic. The existential quantifier ``\emph{there exists $l \ge 1$ such that}'' may be replaced by the universal quantifier ``\emph{for all $l \ge 1$}'' for the following reasons. Suppose that $M F_l$ is EG-aperiodic, and hence every entry of the square matrix $(M F_l)^d = Mf^{l}_{l+dp}$ is positive, for some $d \ge 1$. Besides the $\phi$-orbit $(T_{l+ip})_{i \in \Z}$, any other $\phi$-orbit along the given axis may be written uniquely as $(T_{k+ip})_{i \in \Z}$ with $l-p < k < l$, hence $l+dp < k+(d+1)p < l + (d+1)p$. The transition matrix $MF_k$ therefore has $d+1^{\text{st}}$ iterate equal to
$$(*) \qquad (MF_k)^{d+1} = \underbrace{Mf^k_l}_A \circ Mf^l_{l+dp} \circ \underbrace{Mf^{l+dp}_{k+(d+1)p}}_B = A \circ \, (MF_l)^d \circ B
$$
But every row of $A$ has a positive entry, because every edge of $T_l$ is crossed by the image of \emph{some} edge of $T_k$, under the map $f^k_l$. Also every column of $B$ has a positive entry, because the image of every edge of $T_{l+dp}$ crosses \emph{some} edge of $T_{k+(d+1)p}$, under the map $f^{l+dp}_{k+(d+1)p}$. It follows that every entry of $(MF_k)^{d+1}$ is positive. 

Assuming that the given fold axis is EG-aperiodic, it follows furthermore from equation $(*)$ that the matrices $M F_k$ and $M F_l$ have the same Perron-Frobenius eigenvalue, because the matrices $A$ and $B$ are independent of the exponent $d$ and so the exponential growth rates of entries of $(MF_k)^{d+1}$ and of $(MF_l)^d$ are identical; but these rates are equal to the Perron-Frobenius eigenvalues of $MF_k$ and $MF_l$ respectively. We shall refer to this number as the \emph{expansion factor} of the given EG-aperiodic fold axis, and we record its well-definedness here:

\begin{lemma}
\label{ItemAxisExpFactorWD}
For each $\phi \in \Out(\Gamma;\A)$ and each EG-aperiodic fold axis of $\phi$ in $\FS(\Gamma;\A)$, associated to that fold axis is a well-defined number $>1$ called the \emph{expansion factor} of that axis, equal to the expansion factor of the EG-aperiodic first return train track map $F_i \from T_i \to T_i$ for each $i \in \Z$. \qed
\end{lemma}


The following proposition which will be applied in the proof of the upper bound of Theorem~A in Section~\ref{SectionProofUpperBound} to follow.

\begin{proposition}
\label{PropAxisInFS}
For each $\phi \in \Out(\Gamma;\A)$ which has a filling attracting lamination $\Lambda \in \L(\phi)$ the following hold:
\begin{enumerate}
\item\label{ItemFoldAxisExists}
For each non-filling, $\phi$-invariant free factor system~$\B$ rel~$\A$ that is maximal with respect to these properties, there exists an EG aperiodic fold axis of $\phi$ in $\FS(\Gamma;\A)$ with respect to~$\B$.
\item\label{ItemExpFactWD}
As one varies over all $\B$ as in item~\pref{ItemFoldAxisExists}, and over all EG-aperiodic fold axes of $\phi$ in $\FS(\Gamma;\A)$ with respect to~$\B$, the expansion factor of the axis is well-defined independent of $\B$ and of the axis. 
\end{enumerate}
\end{proposition}
\noindent
The proof of this proposition can be found in Section~\ref{SectionAxisConstruction}, following methods of \cite{HandelMosher:axes}. In extreme brevity, the existence clause~\pref{ItemFoldAxisExists} is proved by starting with any appropriate train track representative of~$\phi$, and then suspending that representative using a Stallings fold factorization.

\begin{definition}\label{DefFillExpFact}
Applying Proposition~\ref{PropAxisInFS}, associated to each $\phi \in \Out(\Gamma;\A)$ that has a filling lamination there is an invariant $\lambda_\phi > 1$ that we call the \emph{(filling) expansion factor} of $\phi$, namely the unique real number given in Proposition~\ref{PropAxisInFS}~\pref{ItemExpFactWD}. 
\end{definition}


\subsubsection{Bounded cancellation} 
Here and for application in the sequel we need Lemma~\ref{LemmaLineBddCancellation} below, a version of the \emph{bounded cancellation lemma}. 

The original bounded cancellation lemma, describing the cancellation effects of applying an automorphism $\Phi \in \Aut(F_n)$ to a reduced word in $F_n$, is found in work of Cooper \cite{Cooper:automorphisms} who attributes it to W.~Thurston and M.~Grayson.  When $\Phi$ is applied letter-by-letter to a reduced finite word in the generators of $F_n$, the resulting finite word may not be reduced. But as one subsequently cancels letters to obtain a reduced word, the lemma provides a quantitative bound, depending only on $\Phi$, to the ``size of cancellation''.

A more general bounded cancellation lemma can be found in \cite[Section~3]{BFH:laminations}, applying to any (twisted) equivariant continuous map $f \from S \to T$ from any free, minimal, simplicial tree action $F_n \act S$ to any minimal $\mathbb R$-tree action $F_n \act T$ that represents a point in the compactified outer space of $F_n$. The conclusion of bounded cancellation is an expression of an efficient bound, depending only on~$f$, to the ``size of cancellation'' of $f$ applied to an arbitrary finite path in~$S$; see Lemma~\ref{LemmaLineBddCancellation}~\pref{ItemBCCPaths} below for one formulation of this expression. The proof in \cite[Section~3]{BFH:laminations} uses an argument with Stallings fold paths that works in all cases where $T$ is simplicial (an additional approximation argument is needed for the remaining cases). Lyman, in \cite[Lemma 1.8]{Lyman:CT}, applied that same Stallings fold argument to the class of equivariant maps $f \from S \to T$ between Grushko free splittings of a group $\Gamma$ relative to a free factor system~$\A$. In the special case that the map $f$ is simplicial and foldable, this class of arguments gives a very nice expression for the cancellation constant which one can observe in both \cite[Section~3]{BFH:laminations} and \cite[Lemma 1.8]{Lyman:CT}; we record this expression below. 

We note also that the equivariant case of bounded cancellation \emph{implies} the $\Phi$-twisted equivariant case: for any $\Phi \in \Aut(\Gamma;\A)$ and for any $\Phi$-twisted equivariant map $f \from S \to T$ between Grushko free splittings of $\Gamma$ rel~$\A$, simply precompose the action homomorphism $\Gamma \mapsto \Isom(S)$ by the automorphism $\Phi^\inv$ to convert $f$ into an equivariant map. The conclusion of this argument is also recorded in item~\pref{ItemBCCPaths} of Lemma~\ref{LemmaLineBddCancellation} below.

Additional versions of bounded cancellation regarding rays and lines are also given in \cite[Lemma 1.8]{Lyman:CT}, again for equivariant maps $f \from S \to T$ between Grushko free splittings relative to a given free factor system. In those situations, one can use the end space identification $\bdyinf S \approx \bdyinf T$ (see Section~\ref{SectionBoundariesRelFFSs}) to given an exact description of the appropriate ray or line in~$T$ that is obtained from a given ray or line in~$S$ by applying $f$ and cancelling; we incorporate this description into Lemma~\ref{LemmaLineBddCancellation}.

The one improvement that we offer to the above methods is to note that they work without assuming the ``Grushko'' hypothesis: indeed, they work for (twisted) equivariant maps $f \from S \to T$ between any two free splittings of a group $\Gamma$. In Lemma~\ref{LemmaLineBddCancellation} to follow we do this for bounded cancellation of paths. Starting with a ray or line in $S$, the Dowdall--Taylor subboundary correspondence \cite{DowdallTaylor:cosurface} is needed in order to describe the appropriate ray or line in~$T$. We shall put this off for the sequel: the Dowdall--Taylor correspondence is studied in \cite[Section 3.2]{\RelFSThreeTag}; and bounded cancellation for rays and lines is covered in \cite[Lemma 3.1 and Lemma/Definition~3.4]{\RelFSThreeTag} respectively.

To state the bounded cancellation lemma, a \emph{metric free splitting} is a free splitting $\Gamma \act T$ equipped with a $\Gamma$-invariant geodesic metric; in this situation $\Length(T)$ is defined as the total length of the metric quotient graph $T / \Gamma$, equal to the sum of lengths of a set of edge-orbit representatives of $T$.  


\begin{lemma}[Bounded Cancellation]
\label{LemmaLineBddCancellation}
For any group $\Gamma$, any $\Phi \in \Aut(\Gamma)$, and any $\Phi$-twisted equivariant map $f \from S \to T$ between metric free splittings of $\Gamma$ rel~$\A$, there exists a constant $C$ with the following properties:
\begin{enumerate}
\item\label{ItemBCCPaths} \textbf{Bounded Cancellation for Paths:} For any finite path $\overline{pq} \subset S$ we have $[f(p),f(q)] \subset f[p,q] \subset N_C[f(p),f(q)]$.
\end{enumerate}
Furthermore, if there exists a factorization $f \from S \xrightarrow{g} S' \xrightarrow{h} T$ such that $g$ is a collapse map that restricts to an isometry on each uncollapsed edge of $S$, and such that $h$ is a foldable map that restricts to an isometry on each edge of $S'$, then for the cancellation constant one may take 
$$C = \Length(S') - \Length(T)
$$
Assuming equality of elliptic subgroup systems $\Fell S=\Fell T$, thus providing an induced $\Phi$-twisted equivariant homeomorphism $\bdyinf f \from \bdyinf S \to \bdyinf T$ (Lemma~\ref{LemmaTEqBdyMaps}), and  using the same cancellation constant $C$ that was used for finite paths, the following also hold: 
\begin{enumeratecontinue}
\item\label{ItemBCCRays} For any ray $[p,\eta) \subset S$ we have 
$$[f(p),\bdyinf f(\eta)) \subset f[p,\eta) \subset N_C[f(p),\bdyinf f(\eta))
$$
and $\bdyinf f(\xi)$ is the unique point in $\bdyinf T$ satisfying this condition.
\item\label{ItemBCCLines} For any line $(\xi,\eta) \subset S$ we have 
$$(\bdyinf f(\xi), \bdyinf f(\eta)) \subset f(\xi,\eta) \subset N_C(\bdyinf f(\xi), \bdyinf f(\eta))
$$
and $\bdyinf f(\xi),\bdyinf f(\eta)$ is the unique pair of points in $\bdyinf T$ satisfying this condition.
\end{enumeratecontinue}
\end{lemma}

\begin{proof} As stated above, the twisted equivariant case reduces to the equivariant case. We assume henceforth that $f$ is equivariant. 

We may also assume that there exists a factorization $f \from S \xrightarrow{g} S' \xrightarrow{h} T$ such that $g$ is a collapse map and $h$ is a foldable map. To see why, from the existence of an equivariant map $S \mapsto T$ it follows that $\FS(S) \sqsubset \FS(T)$ and hence Proposition~\ref{PropFoldableProps}~\pref{ItemFoldableExists} applies, from which we obtain another equivariant map $f' \from S \mapsto T$ having a factorization as just described. Since the distances $d(f(x),f'(x))$ are uniformly bounded over $x \in S$, conclusions~\pref{ItemBCCPaths}, \pref{ItemBCCRays} and~\pref{ItemBCCLines} for the two maps $f$ and $f'$ are equivalent (albeit with different cancellation constants), and so we may replace $f$ with $f'$. 

We may also assume that $g$ restricts to an isometry on each uncollapsed edge of $S$ and that $h$ restricts to an isometry on each edge of $S'$, for the following reasons. If this does not already hold, we replace the metric on each edge of $S'$ using the pullback metric via $h$; and then we replace the metric on each uncollapsed edge of $S$ using the pullback metric via $g$. This operation does not change any distances $d(f(x),f(y))$, and so conclusions~\pref{ItemBCCPaths}, \pref{ItemBCCRays} and~\pref{ItemBCCLines} for the two metrics on $S$ are equivalent.

Item~\pref{ItemBCCPaths} with cancellation constant $C = \Length(S')-\Length(T)$ is now proved exactly as in \cite[Section~3]{BFH:laminations} and~\cite[Lemma 1.8]{Lyman:CT}. First factor $h$ into Stallings folds, thereby factoring the map $f$ as 
$$\xymatrix{
S \ar[r]_<<<<<{g} \ar@/^2pc/[rrrrr]^{f}
& S' = S_0 \ar[r]_>>>>>{h_1} & S_1 \ar[r]_<<<<{h_2} & \cdots  \ar[r]_<<<<{h_{K-1}} & S_{K-1} \ar[r]_<<<<<{h_K} & S_K = T
}$$
Next, pull metrics back from $S_K$ so that each map $h_k$ restricts to an isometry on each edge of~$S_{k-1}$. Next, prove that the collapse map $g$ has cancellation constant $0$. Next, prove that each fold map $h_k \from S_{K-1} \to S_K$ has cancellation constant equal to $\Length(S_{K-1}) - \Length(S_K)$, which is equal to the common length of a pair of segments that are folded together by $h_k$. Finally, prove additivity of cancellation constants along the factorization of $f$, giving a cancellation constant for $f$ equal to $0 + \sum_{k=1}^K (\Length(S_{k-1}) - \Length(S_{k)} = \Length(S') - \Length(T)$.

\smallskip

What's left is to prove \pref{ItemBCCRays} and~\pref{ItemBCCLines} assuming that $f$ is equivariant and \hbox{$\FS(S)=\FS(T)$}. These proofs are standard (see e.g.\ \cite[Lemma 1.8]{\LymanCTTag}), but we shall spell out details that will be re-used in the sequel in a context where things are somewhat less standard (see \cite[Lemma 3.1 and Lemma/Definition 3.4]{\RelFSThreeTag}).

For the proof of~\pref{ItemBCCRays} we use that $f \from S \to T$ is a quasi-isometry of Gromov hyperbolic spaces that extends equivariantly to a map of Gromov bordifications $\bar f \from \overline S \to \overline T$ which then restricts to the unique equivariant homeomorphism of Gromov boundaries (a.k.a.\ end spaces) $\bdy f \from \bdyinf S \to \bdyinf T$ (\cite{GuirardelLevitt:DefSpaces} and see Lemma~\ref{LemmaTEqBdyMaps}). Also, when $f$ is applied to a geodesic ray $[p,\eta) \subset S$, the resulting continuous quasigeodesic ray has finite Hausdorff distance from the geodesic ray $[f(p),\bdyinf f(\eta))$, and so we have the following conclusion:
\begin{description}
\item[$(*)$] There exists a sequence $w_0,w_1,w_2,w_3,\ldots$ in $f[p,\eta)$ that converges in $\overline T$ to $\bdyinf f(\eta)$.
\end{description}
We can of course just choose $w_i=f(v_i)$ where $p=v_0,v_1,v_2,\ldots$ is the ordered sequence of vertices along $[p,\xi)$. But for use in the sequel \RelFSThree\ we couch the rest of the proof of~\pref{ItemBCCRays} so as to apply using any sequence $(w_i)$ that witnesses~$(*)$. 

In $T$, let $q_0 = f(p)$, and for $i \ge 1$ let $q_i$ be the point on the ray $[f(p),\bdyinf f(\eta))$ that is closest to the point $w_i$, hence $q_i$ also converges in $\overline T$ to $\bdyinf f(\eta)$. It follows that $\bigcup_{i \ge 1}[q_0,q_i]=[q_0,\bdyinf f(\eta))$. Starting with $p_0=p$, for each $i \ge 1$ we may choose $p_i \in [p,\eta)$ such that $f(p_i)=q_i$: if $q_i=q_0$ we choose $p_i=p$; whereas if $q_i \ne q_0$, in $\overline T$ the two points $q_0 = f(p)$ and $\bdyinf f(\eta)$ are in separate path components of the subspace $\overline T - \{q_i\}$, and the path connected subset $f[p,\eta] \subset \overline T$ contains both of those points, hence $q_i \in f[p,\eta)$. Since $q_i$ converges in $\overline T$ to $\bdyinf f(\eta)$, it follows that the sequence $p_i$ has no bounded subsequences, and therefore $p_i$ converges in $\overline S$ to $\eta$. It follows that $\bigcup_{i \ge 1} [p_0,p_i] = [p_0,\eta)$. From the first inclusion of~\pref{ItemBCCPaths} we obtain $[q_0,q_i] \subset f[p_0,p_i]$ for all $i \ge 1$ and hence 
$$[f(p),\bdyinf f(\eta)) = [q_0,\bdyinf f(\eta)) = \bigcup_{i \ge 1} [q_0,q_i] \subset \bigcup_{i \ge 1} f[p_0,p_i] = f \biggl( \,\bigcup_{i \ge 1} [p_0,p_i] \biggr) = f[p,\eta)
$$
In the other direction, applying the second inclusion of~\pref{ItemBCCPaths} we obtain $f[p_0,p_i] \subset N_C[q_0,q_i]$ for all $i \ge 1$ and hence
\begin{align*}
f[p,\eta) = \bigcup_{i \ge 1}  f[p_0,p_i] &\subset \bigcup_{i \ge 1} N_C [f(p_0),f(p_i)] = \bigcup_{i \ge 1} N_C [q_0,q_i] = N_C  \bigcup_{i \ge 1} [q_0,q_i] = N_C[f(p),\bdyinf f(\eta))
\end{align*}
To prove uniqueness of $\bdyinf f(\xi)$ in conclusion~\pref{ItemBCCRays}, if $\zeta \in \bdyinf T$ and if \pref{ItemBCCRays} is satisfied with $[f(v),\zeta)$ in place of $[f(v),\bdyinf f(\xi))$ then in $T$ the two rays $[f(v),\zeta)$ and $[f(v),\bdyinf f(\eta))$ both have finite Hausdorff distance from the set $f[v,\eta)$, hence they have finite Hausdorff distance from each other, hence $\zeta = \bdyinf f(\eta)$. 

The proof of~\pref{ItemBCCLines} follows similar lines. To start, pick any bi-infinite sequence $(w_i)_{i \in \Z}$ in \hbox{$f(\xi,\eta)\subset T$} so that in $\overline T$ we have $w_i \to \bdyinf f(\xi)$ as $i \to -\infty$, and $w_i  \to \bdyinf f(\eta)$ as $i \to +\infty$. Next, let $q_i$ be the point on the line $(\bdyinf f(\xi),\bdyinf f(\eta)) \subset T$ which is closest to $f(p_i)$, and so $q_i \to \bdyinf f(\xi)$ as $i \to -\infty$ and $q_i \to \bdyinf f(\eta)$ as $i \to +\infty$. Next, choose $p_i \in (\xi,\eta)$ so that $f(p_i)=q_i$; this choice is possible since $q_i$ separates $\bdyinf f(\xi)$ from $\bdyinf f(\eta)$ in $\overline T$. Using the first inclusion of~\pref{ItemBCCPaths} we have $[q_{-i},q_i] \subset f[p_{-i},p_i]$ for all $i \ge 1$ and hence $(\bdyinf f(\xi),\bdyinf f(\eta)) \subset f(\xi,\eta)$. Using the second inclusion of~\pref{ItemBCCPaths} we have $f[p_{-i},p_i] \subset N_C [f(q_{-i},q_i)]$ for all $i \ge 1$ and hence $f(\xi,\eta) \subset N_C(\bdyinf f(\xi),\bdyinf f(\eta))$.
\end{proof}

\subsection{Proof of the upper bound (application of Proposition~\ref{PropAxisInFS})}
\label{SectionProofUpperBound}
Recall (from the opening paragraph of the introduction) that for any self-isometry $f \from X \to X$ of a metric space $X$ its \emph{asymptotic translation length} is defined by the limit formula
$$\tau_f = \lim_{n \to \infty} \frac{d(f^n(x),x)}{n}
$$
This limit exists, is independent of $x \in X$, and depends only on the conjugacy class of $f$ in the group $\Isom(X)$. If in addition $X$ is a Gromov hyperbolic geodesic metric space, for example if $X = \FS(\Gamma;\A)$, then the inequality $\tau_f > 0$ is equivalent to the statement that for some (any) $p \in X$ the orbit map $\Z \mapsto X$ defined by $n \mapsto f^n(p)$ is a quasi-isometric embedding; this is the very meaning of the \emph{loxodromic} property of an isometry of~$X$.

We prove \emph{The Upper Bound} using the constant $B = \Delta/\log 2$, where $\Delta$ is the constant from the ``Two Over All'' Theorem, which itself depends only on $\corank(\A)$ and $\abs{\A}$.  

Let $\phi \in \Out(\Gamma;\A)$ and $\Lambda \in \L(\phi)$ be a filling lamination. 
\emph{The Upper Bound} clearly holds if $\tau_\phi=0$.\footnote{From the lower bound of Theorem A, it follows that $\tau_\phi=0$ is impossible when a filling lamination exists.} Henceforth we may therefore assume that
$$\tau_\phi > 0
$$
Since $\Lambda$ fills, by applying Proposition~\ref{PropAxisInFS} we obtain an EG-aperiodic fold axis for $\phi$, which we express using the notation of Definition~\ref{DefFoldAxes}, such that the filling expansion factor $\lambda_\phi$ of $\phi$ is equal to the expansion factor of this axis. Let $p$ be the period of this fold axis, and let $F$ denote the first return map at $T_0$, namely the composition 
$$F : T_0 \xrightarrow{f_1} \cdots \xrightarrow{f_p} T_p = T_0 \cdot \Phi \xrightarrow{h^p_0} T_0
$$
Choosing enumerated orbit representatives according to the \emph{Edge Orbit Periodicity} criterion of Section~\ref{SectionFoldAxes}, and applying Lemma~\ref{LemmaTransMatSame}, for all $n \ge 1$ we have an equation of transition matrices $(MF)^n=M f^0_{np}$ where $M f^0_{np}$ is the transition matrix of the following portion of the fold axis:
$$f^0_{np} \from T_0 \xrightarrow{f_1} T_1 \xrightarrow{f_2} \cdots T_2 \xrightarrow{f_3} \cdots \xrightarrow{f_{pn}} T_{pn} = T_0 \cdot \phi^n
$$
and therefore $\lambda_\phi$ is the Perron-Frobenius eigenvalue of $MF$.

Knowing that $\lim_{n \to \infty} \frac{1}{n} d(T_0,T_{pn}) = \tau_\phi > 0$, we may assume that $n$ is sufficiently large so that $d(T_0,T_{pn}) = d(T_0,T_0 \cdot \phi^n) \ge \Delta$, and so there exists a unique integer $Q_n \ge 1$ such that
$$(*) \qquad Q_n \Delta \le d(T_0,T_{pn}) \le (Q_n+1) \Delta
$$
From the right hand inequality of $(*)$ it follows that
$$(**) \qquad Q_n \ge \frac{d(T_0,T_{pn})}{\Delta} - 1
$$
Applying the \emph{Two Over All Theorem} together with the left hand inequality of $(*)$ (and using that $Q_n \ge 1$) there exist natural edges $E_1,E_2 \subset T_{0}$ in two different $\Gamma$-orbits such that for each natural edge $E$ of $T_{pn}$, each of the paths $f^0_{pn}(E_1)$ and $f^0_{pn}(E_2)$ crosses at least $2^{Q_n-1}$ natural edges in the $\Gamma$-orbit of $E$. Focussing on $E_1$, it follows that in the matrix $(MF)^{n} = M f^0_{np}$, letting $N_n$ denote the norm of the column of $(MF)^{n}$ corresponding to $E_1$, we have
$$N_n \ge 2^{Q_n-1}
$$
Applying the Perron-Frobenius Theorem to $MF$ it follows that
\begin{align*}
\lambda_\phi &= \lim_{n \to \infty} \sqrt[n]{N_n} \\
\log\lambda_\phi  &= \lim_{n \to \infty} \frac{\log N_n}{n} \ge \log 2 \cdot \limsup_{n \to \infty} \frac{(Q_n-1) }{n} \\
\intertext{Applying $(**)$ we get}
\log \lambda_\phi    &\ge \log 2 \cdot \limsup_{n \to \infty} \frac{1}{n}  
\biggl(
\frac{d(T_0,T_{pn})}{\Delta} - 2\biggr)  \\
    &\ge \frac{\log 2}{\Delta} \cdot \limsup_{n \to \infty}  \frac{1}{n} \cdot d(T_0,T_{pn})  \\
    &= \frac{\log 2}{\Delta} \cdot \tau_\phi \\
\tau_\phi &\le B \, \lambda_\phi \quadtext{with} B = \frac{\Delta}{\log 2}
\end{align*}

\subsection{Constructing fold axes in $\FS(\Gamma;\A)$}
\label{SectionAxisConstruction}

%

In this section we employ relative train track theory to prove Proposition~\ref{PropAxisInFS} (which was itself already applied in Section~\ref{SectionProofUpperBound} to prove the upper bound of Theorem~A), and to prove Theorem~C. We also formulate and prove Proposition~\ref{PropAxisInFSEmbellished}, a stronger version of Proposition~\ref{PropAxisInFS} which incorporates more information about attracting laminations, and will be applied in Part~III to prove the implication \pref{ItemFillingLamExists}$\implies$\pref{ItemActLox} of Theorem~B. 

The original construction of relative train track representatives for elements of $\Out(F_n)$, given by Bestvina and Handel \BH, was first generalized to $\Out(\Gamma)$ for groups $\Gamma$ of finite Kurosh rank by Collins and Turner \cite{CollinsTurner:efficient}, and was generalized to our still more general current setting of $\Out(\Gamma;\A)$ in more recent work of Lyman \LymanRTT, recounted here in Section~\ref{SectionRTT}, Theorem~\ref{TheoremLyman}. 

The tight relationship between attracting laminations and relative train track representatives, originally described for $\Out(F_n)$ by Bestvina, Feighn and Handel \BookOne, was generalized to $\Out(\Gamma;\A)$ by Lyman in \LymanCT\ (recounted here in Section~\ref{SectionLamStratumCorr}, Theorem~\ref{TheoremStrataLamCorr}).

Given $\phi \in \Out(\Gamma;\A)$, consider a $\phi$-invariant, nonfull free factor system~$\B$ satisfying $\A \sqsubset \B$ which is \emph{maximal} subject to those constraints (with respect to the partial order $\sqsubset$). In that case $\phi$ is irreducible relative to~$\B$, and the representative $f \from T \to T$ of $\phi$ whose existence is asserted in Proposition~\ref{PropAxisInFS}~\pref{ItemFoldAxisExists} is an irreducible \emph{train track} representative of $\phi$ \emph{with respect to~$\B$}. If that was all we wanted then we could apply constructions of irreducible train track representatives found in \cite{FrancavigliaMartino:TrainTracks,Meinert:TrainTrackMaps}. But under the additional assumption that $\phi$ has a filling lamination, in order to obtain an expansion factor $\lambda_\phi$ that is independent of the choice of $\B$ and $f$ (as asserted in Proposition~\ref{PropAxisInFS}~\pref{ItemExpFactWD}), we need a further result of \cite{\LymanCTTag} regarding a well-defined expansion factor associated to every attracting lamination of $\phi$, namely Theorem~\ref{TheoremLamExpFac} in Section~\ref{SectionLamStratumCorr}.

\subsubsection{Relative train track representatives}
\label{SectionRTT}

We follow Lyman's treatment \cite{Lyman:RTT,Lyman:CT} with the following change: rather than using graphs-of-groups language which is how most of \cite{Lyman:RTT} is expressed, we use groups-acting-on-trees language. The translation between these languages is a straightforward application of Bass-Serre theory, although we will have to avoid certain clashes of terminology which arise; see the \emph{Remark on ``invariant subforests''} just below.

\subparagraph{Tight representatives and topological representatives.} Recall from Section~\ref{SectionRFSC} the concept of a tight map between graphs. Given $\phi \in \Out(\Gamma;\A)$, a \emph{tight representative} of $\phi$ is a tight map $f \from T \to T$ defined on a Grushko free splitting $T$ rel~$\A$ such that $f$ is $\Phi$-twisted equivariant for some $\Phi \in \Aut(\Gamma;\A)$ that represents~$\phi$. Using the definition of ``tight'' and the fact that $T$ is a tree, it follows that $f$ takes each vertex to a vertex, and for each edge $e \subset T$ the restriction $f \restrict e$ is either a trivial path taking constant value at some vertex, or an injective edge path. 

A \emph{topological representative} of $\phi$ is a tight representative $f \from T \to T$ which is injective on each edge. The induced direction maps $D_v f \from D_v T \to D_{f(v)} T$ are therefore defined for all $v \in \VT$. Given $v \in \VT$, a nondegenerate turn $\{e,e'\} \subset D_v T$ is \emph{illegal} if some iterate $\{D_v f^n(e), D_v f^n(e')\} \subset D_{f^n(v)} T$ is degenerate ($n \ge 1$); otherwise $\{e,e'\}$ is \emph{legal}. Note that every iterated image of a legal turn is nonfoldable. 

In the context of a Grushko free splitting $T$ of $\Gamma$ rel~$\A$, given a $\Gamma$-invariant subforest $\tau \subset T$ such that each component of $\tau$ contains at least one edge, to say that $\tau$ is \emph{collapsible} means that the collapsed tree $T/\tau$ is also a Grushko free splitting; equivalently, the stabilizer of each component of $\tau$ is either trivial or equal to the stabilizer of some vertex of $T$ contained in $\tau$. 

Consider a topological representative $f \from T \to T$ of~$\phi$. An \emph{invariant subforest} of $f \from T \to T$ is a subforest $\tau \subset T$ which is both $\Gamma$-invariant and $f$-invariant. 

\smallskip
\emph{Remark on ``invariant subforests''.} The terminologies of \emph{invariant subgraphs} and \emph{invariant subforests} refer to two distinct concepts in \BH\ and in the more general setting of \LymanRTT. But in our current setting there is no distinction between \emph{subgraph} and \emph{subforest}, and we are thereby forced to make a terminological choice. Given a Grushko free splitting $T$ of $\Gamma$ rel~$\A$ with quotient graph of groups $T/\Gamma$, to say that $f \from T \to T$ is topological representative in our current setting is equivalent to saying that the induced map $\hat f \from T/\Gamma \to T/\Gamma$ is a topological representative in the setting of \LymanRTT. Given also a $\Gamma$-invariant subforest $\tau \subset T$ with quotient $\hat\tau \subset T/\Gamma$, we have the following equivalences: $\tau \subset T$ is an invariant subforest of $f$ (in our setting) if and only if $\hat\tau \subset T/\Gamma$ is an invariant subgraph of $\hat f$ (in \LymanRTT); and $\tau$ is a collapsible invariant subforest of $f$ (in our setting) if and only if $\hat\tau$ is an invariant subforest of $\hat f$ (in \LymanRTT). 

\subparagraph{Collapsing a maximal pretrivial forest.} In a tight representative $f \from T \to T$ of $\phi$, for each edge $e$ of $T$ the path $f \restrict e$ is either injective or it is a trivial path. A standard construction shows that for any $\phi \in \Out(\Gamma;\A)$, any automorphism $\Phi \in \Aut(\Gamma;\A)$ representing~$\phi$, and any Grushko free splitting $T$ of $\Gamma$ rel~$\A$, a $\Phi$-twisted equivariant tight representative $f \from T \to T$ of~$\phi$ exists, as follows. The image of every vertex $v \in T$ with nontrivial stabilizer is uniquely determined by requiring $\Stab(f(v))=\Phi(\Stab(v))$. For every orbit of vertices with trivial stabilizer, choosing one representative $v$ and its image $f(v)$ arbitrarily, one uses $\Phi$-twisted equivariance to extend $f$ over the whole orbit of $v$, defining $f(\gamma \cdot v) = \Phi(\gamma) \cdot f(v)$. Having defined $f$ on the whole vertex set, the image of every edge is determined as the unique path connecting the images of its endpoints. 

From any tight representative $f \from T \to T$ of $\phi$ one obtains a topological representative $g \from S \to S$ by the operation of \emph{collapsing the maximal pretrivial subforest} $\sigma \subset S$, as follows. That subforest~$\sigma$ is the union of all edges $e \subset T$ such that some iterated image $f^n(e)$ is contained in the vertex set of $S$. It follows that $\sigma$ is collapsible, resulting in a Grushko free splitting $T = S / \sigma$. Letting $\hat g \from T \to T$ be the map that is induced by~$f$, for each edge $e \subset E$ the restriction $\hat g \restrict e$ is a hare's path that crosses at least one edge. By equivariantly tightening $\hat g$ on each edge, one obtains the desired topological representative $g \from T \to T$.

\subparagraph{Filtrations and strata.} 
%
Given $\phi \in \Out(\Gamma;\A)$, a topological representative $f \from T \to T$ of $\phi$, and enumerated edge orbit representatives $e_1,\ldots,e_K$ of $T$. The associated \emph{transition matrix} is the $K \times K$ matrix $M$ whose entry $M_{ij}$ counts the number of times that injective edge path $f(e_j)$ crosses translates of $e_i$. This matrix is unchanged if each $e_i$ is replaced by another edge in the same orbit; and if the enumeration indices are permuted then $M$ is conjugated by the associated permutation matrix.

An \emph{invariant filtration} for $f$ is a strictly nested sequence of invariant forests $\emptyset = T_0 \subsetneq T_1 \subsetneq \cdots \subsetneq T_R=T$, the \emph{length} of the filtration is~$R$. The \emph{strata} of this filtration are the $\Gamma$-invariant subforests $H_r = T_r \setminus T_{r-1}$ (strata need not be $f$-invariant). The associated \emph{filtration by free factor systems} is the nested sequence of $\phi$-invariant free factor systems $\A = \F{[T_0]} \sqsubset \F{[T_1]} \sqsubset \cdots \sqsubset \F{[T_R]} = \{[\Gamma]\}$ (this sequence need not be strictly nested). Choose enumerated edge orbit representatives of $T$ so that each edge orbit of $H_{r-1}$ comes before each edge orbit of $H_{r}$ ($1 \le r \le R$). With this choice, the transition matrix $M=Mf$ has a block upper triangular structure with one diagonal block $M_r$ per stratum $H_r$. 
To say that the invariant filtration is \emph{maximal} means that for each 
%
%
$r=1,\ldots,R$ one of the following occurs:
\begin{description}
\item[$H_r$ is a zero stratum,] meaning that $M_r$ is the zero matrix; or
\item[$H_r$ is an irreducible stratum,] meaning $M_r$ is an irreducible matrix. Denoting the Perron-Frobenius eigenvalue of $M_r$ by $\lambda_r \ge 1$, one of the following occurs:

\textbf{$H_r$ is an NEG stratum,} meaning that $\lambda_r = 1$; or

\textbf{$H_r$ is an EG stratum,} meaning that $\lambda_r > 1$, in which case $\lambda_r$ is called the \emph{expansion factor} of $H_r$.

Furthermore, to say that an irreducible stratum $H_r$ is \emph{aperiodic} means that $M_r$ is an aperiodic matrix.
\end{description} 
Every topological representative has a maximal filtration. The filtration is not unique, and the ordering of the EG strata is not unique, but the \emph{set} of EG strata is uniquely determined independent of the choice of maximal filtration. We will use the terminology \emph{filtered topological representative} to refer to a topological representative $f \from T \to T$ equipped with a maximal filtration as denoted above.

Given a path $\gamma$ in some filtration element $T_r$, to say that $\gamma$ is \emph{$H_r$-legal (with respect to $f$)} means that for each turn $\{e,e'\}$ taken by $\gamma$, if $e,e' \subset H_r$ then $\{e,e'\}$ is legal. 

\subparagraph{Relative train track representatives.} Given $\phi \in \Out(\Gamma;\A)$ and a filtered topological representative $f \from T \to T$ of $\phi$, to say that $f$ is a  \emph{relative train track representative} (with respect to~$\A$) means that for each EG stratum $H_r$, the following properties hold:
\begin{description}
\item[(RTT-i)] For each $e \in \EH_r$ the path $f(e)$ begins and ends with edges of $H_r$.
\item[(RTT-ii)] For each nontrivial path $\gamma$ in $T_{r-1}$ with endpoints $\bdy\gamma = \{v,w\} \subset H_r$ we have $f(v) \ne f(w)$; equivalently, $f_\sharp(\gamma)$ is nontrivial.
\item[(RTT-iii)] For any $H_r$-legal path $\gamma \subset H_r$, the image path $f(\gamma)$ is $H_r$-legal.
\end{description}

\noindent
\emph{Remark 1:} The argument of \cite[Lemma 5.8]{\BHTag}, suitably adapted, shows that if (RTT-i) and (RTT-ii) both hold then the version of (RTT-iii) given here is equivalent to the statement ``For any $H_r$-legal path $\gamma \subset T_r$, the straightened image $f_\sharp(\gamma)$ is $H_r$-legal''.

\smallskip
\noindent
\emph{Remark 2:} If $H_r$ is an EG stratum of a relative train track representative $f \from T \to T$ then for every edge $e \in H_r$ there exists an $H_r$-legal path $\gamma$ beginning and ending with edges of $H_r$ such that some translate of $e$ is contained in the interior of $\gamma$. In fact, one can arrange that the number of translates of $e$ in the interior of $\gamma$ is $\ge m$ for any given integer $m \ge 1$: choose any edge $e' \subset H_r$, and use the EG property to choose an exponent $k$ such that $\gamma = f^k(e')$ crosses at least $m+2$ translates of $e'$; all but at most the first and last such crossings are contained in the interior of $\gamma$.


\smallskip


Lyman's construction of relative train track representatives, expressed in graphs--of--groups language, is found in \cite[Theorem 1.1]{\LymanRTTTag}, and her followup work \LymanCT\ adds extra information allowing free factor supports of filtration elements to be specified appropriately. Here is an expression of these results in our current language of Bass-Serre trees: 

\begin{theorem}[\protect{\cite[Proposition 1.3]{\LymanCTTag}}]
\label{TheoremLyman}
For every free factor system $\A$ of $\Gamma$, every $\phi \in \Out(\Gamma;\A)$, and every nested sequence $\B_1 \sqsubset \cdots \sqsubset \B_K$ of $\phi$-invariant free factor systems rel~$\A$, there exists a relative train track representative $f \from T \to T$ of $\phi$ rel~$\A$ such that for each $\B_k$ there is a filtration element $T_r$ such that $\B_k = \F[T_r]$. \qed
\end{theorem}

The most common way in which we shall apply Theorem~\ref{TheoremLyman} in this work is to choose $\B$ to be a maximal, non-filling, $\phi$-invariant free factor system rel~$\A$, where ``maximality'' of $\B$ refers to the partial ordering~$\sqsubset$. Applying Theorem~\ref{TheoremLyman}, there exists a relative train track representative $f \from T \to T$ of $\phi$ with top filtration element $T_R$ having \emph{some} filtration element $T_r$ such that $\F[T_r]=\B$. But since $\B=\F[T_r] \sqsubset \F[T_{R-1}]$, and since $\F[T_{R-1}]$ is a non-filling and $\phi$-invariant free factor system rel~$\A$, by maximality of $\B$ it follows that $\B=\F[T_{R-1}]$. We record this for multiple later uses:

\begin{corollary}\label{CorollaryMaximal}
For every free factor system $\A$ of $\Gamma$, every $\phi \in \Out(\Gamma;\A)$, and every maximal, non-filling, $\phi$-invariant free factor system~$\B$ rel~$\A$, there exists a relative train track representative $f \from T \to T$ of $\phi$ with penultimate filtration element $T_{R-1}$ such that $\B= \F[T_{R-1}]$. \qed
\end{corollary}

\subsubsection{The correspondence between attracting laminations and EG strata.} 
\label{SectionLamStratumCorr}
Given $\phi \in \Out(\Gamma;\A)$, there is a close relation between the set of attracting laminations $\L(\phi)$ and the EG strata of relative train track representatives of $\phi$. For the classical case of $\Out(F_n)$ this relation is laid out in \cite[Definition 3.1.7 -- Lemma 3.1.15]{\BookOneTag}. For the general case of $\Out(\Gamma;\A)$ see \cite[Lemmas 4.3 -- 4.8]{\LymanCTTag}, summarized in Theorem~\ref{TheoremStrataLamCorr} below.

Consider $\phi \in \Out(F_n)$ and a relative train track representative $f \from T \to T$ with filtration $T_1 \subset \cdots \subset T_R = T$, strata $H_r = T_r \setminus T_{r-1}$, and transition matrices $M_r$. By definition, if $H_r$ is an EG stratum then the matrix $M_r$ is irreducible; and although $H_r$ and $M_r$ need not be aperiodic, if they are then we say in addition that that $H_r$ is an \emph{EG-aperiodic stratum}. If each EG stratum is aperiodic then we say that $f$ is \emph{EG-aperiodic}. From the general theory of irreducible matrices (see e.g.\ \cite{Senata:matrices}) it follows that for each EG stratum $H_r$ there is some least integer $p_r$ such that, after appropriately re-indexing the edges of $H_r$, the matrix $M_r$ has a $p_r \times p_r$ block structure so that the nonzero blocks of $M_r$ itself are the superdiagonal blocks and the lower left block, and so that the corresponding block diagonal structure of $(M_r)^{p_r}$ has EG-aperiodic blocks along the diagonal. Using this block diagonal structure we may partition $H_r$ into pairwise edge disjoint subgraphs called the \emph{aperiodic substrata of $H_r$}, denoted $H_r = H_{r,1} \union\cdots\union H_{r,p_r}$ such that 
$$f(H_{r,j}) \subset T_{r-1} \union H_{r,j+1} \quad\text{(for each $j=1,\ldots,p_r$ modulo $p_r$)}
$$
Letting $p$ be the least common multiple of the $p_r$'s as $H_r$ varies over the EG strata of $f$, the straightened iterate $f^p_\sharp \from T \to T$ is an EG-aperiodic relative train track representative of $\phi^p$ such that the EG strata of $f^p_\sharp$ are precisely the aperiodic substrata $H_{r,i}$ of $f$. We shall refer to $f^p_\sharp$ as the \emph{EG-aperiodic iterate} of~$f$. 

For each EG-aperiodic substratum $H_{r,j}$ of $f$, and for each $k \ge 0$, the \emph{$k$-tiles} of $H_{r,j}$ (with respect to $f^p_\sharp$) are certain paths in $H_{r,j} \union T_{r-1}$ defined inductively as follows: the $0$-tiles of $H_{r,j}$ are the edges of $H_{r,j}$; assuming that $k-1$ tiles are defined, the $k$-tiles of $H_{r,j}$ are the paths of the form $f^{p}_\sharp(\tau)$, as $\tau$ varies over the $k-1$-tiles of $H_{r,j}$. A concrete line $\ell \subset T$ is said to be \emph{exhausted by tiles of $H_{r,j}$} if every finite subpath of $\ell$ is contained in some $k$-tile of $H_{r,j}$ for some $k \ge 0$. Associated to $f$ and $H_{r,j}$ is a lamination denoted $\Lambda(f;H_{r,j})$, abbreviated $\Lambda_{r,j}$, equal to the closure of the subset of $\B(\Gamma;\A)$ represented by those lines in $T$ that are exhausted by tiles of $H_{r,j}$. In the case that $H_r$ is EG-aperiodic we abbreviate further to $\Lambda_r$.

As a very important special case, to say that $f \from T \to T$ is an \emph{EG-aperiodic train track map} means that $f$ has just one stratum $H_1 = T_1 = T$, and that stratum is EG-aperiodic.

\medskip

The following theorem is, for the most part, a summary statement of results from \LymanCT\ as listed, translated from graph-of-groups language into Bass-Serre tree language; but a few items are added 

 Compare also \cite[Fact~1.56]{HandelMosher:Subgroups} for a similar summary statement in the special case of $\Out(F_n)$, derived from \cite[Definition 3.1.7 -- Lemma 3.1.15]{\BookOneTag}. 



\begin{theorem}[\protect{\cite[Lemmas 4.3, 4.4, 4.6]{\LymanCTTag}}]
\label{TheoremStrataLamCorr}
For any $\phi \in \Out(\Gamma;\A)$, its set $\L(\phi)$ of attracting laminations is finite. More precisely, for any relative train track representative $f \from T \to T$, the following relation defines a bijection between the set of attracting laminations $\Lambda \in \L(\phi)$ and the set of EG-aperiodic substrata $H_{r,j} \subset T$ of $f$:
$$(*) \qquad \Lambda \leftrightarrow H_{r,j} \iff H_{r,j} \subset \tsp(\Lambda;T) \subset T_{r-1} \union H_{r,j} \iff \Lambda = \Lambda_{r,j} \qquad \hphantom{(*)}
$$
Furthermore, 
\begin{enumerate}
\item\label{ItemLymanEGAperiodic}
For each EG-aperiodic substratum $H_{r,j}$, 
\begin{enumerate}
\item\label{ItemNestedUnionOfTiles}
A concrete line $L \subset T$ is the realization of a generic leaf of $\Lambda_{i,j}$ if and only if~$L$ is exhausted by tiles of $H_{r,j}$. 
\item $\phi(\Lambda_{r,j}) = \Lambda_{r,j+1}$ (for $j$ modulo $p_r$). 
\item\label{ItemLamPeriod}
$\Lambda_{r,j} \in \L(\phi)$ has period $p_r$ under the action of $\phi$ on $\L(\phi)$. 
\item\label{ItemCorollaryThreeTwoTwo}
$H_{r,j}$ contains two edgelets in distinct orbits; in fact there exist two natural edges $E,E' \subset T$ in distinct orbits each containing an edgelet of $H_{r,j}$.
\end{enumerate} 
\item\label{ItemActsAsIdentity}
Letting $H_r$ vary over all EG-aperiodic substrata, the period of the action of $\phi$ on the whole set~$\L(\phi)$ is $\text{LCM}\,(p_r)$. 
\item\label{ItemLamSupports}
After passing to the EG-aperiodic iterate $f^p_\#$ (by refining the filtration of $f$ so that the EG strata of $f^p_\#$ are the same as the EG-aperiodic strata of $f$), let $r_1 < \cdots < r_{\abs{\L(\phi)}}$ be the indices of the EG-aperiodic strata of $f^p_\#$, let $\Lambda_i$ be the attracting lamination corresponding to $H_{r_i}$, and let $\F_i = \F T_{r_i}$. Then we have a sequence of nesting relations
$$\emptyset = \F_0 \sqsubset \F_1 \sqsubset \cdots \sqsubset \F_{\abs{\L(\phi)}}
$$
such that $\Lambda_i$ is carried by $\F_i$ but not by $\F_{i-1}$, hence $\F_i \ne \F_{i-1}$ (for each $1 \le i \le \abs{\L(\phi)})$. 
\item\label{ItemLamPeriodBound}
There is a uniform upper bound $\abs{\L(\phi)} \le 2 \corank(\A) + \abs{\A} - 1$, and so there is also a uniform bound on the period $\text{LCM}\,(p_r)$ of the action of $\phi$ on $\L(\phi)$.
\end{enumerate}
\textbf{As a special case:} If $f$ is an EG-aperiodic train track map then $\phi$ has a unique attracting lamination $\Lambda$, and a concrete line in $L \subset T$ is the realization of a generic leaf of $\Lambda$ if and only if $L$ is exhausted by tiles with respect to $f \from T \to T$. 
\end{theorem}

\subparagraph{Remark.} Recall from Definition~\ref{DefinitionAttrLam} that for each $\Lambda \in \L(\phi)$ there exists some integer $p(\Lambda)$ which witnesses that a generic leaf $\beta$ of $\Lambda$ has an attracting neighborhood with respect to the action of $\phi^{p(\Lambda)}$. Taking $p(\Lambda)$ to be the least such integer and applying Theorem~\ref{TheoremStrataLamCorr} with $\Lambda=\Lambda_{r,j}$, one sees that $p(\Lambda)=p_r$. 

\medskip

We can now prove Lemma~\ref{LemmaSupportDistinction}, saying that attracting laminations are distinguished by their free factor supports, using the same proof that works for the classical case of $\Out(F_n)$ (see \cite[Lemma 3.2.4]{\BookOneTag}; see also \cite[Fact 1.14 (3)]{\SubgroupsTag} and the following remarks).

\begin{proof}[Proof of Lemma~\ref{LemmaSupportDistinction}] Using the notation of conclusion~\pref{ItemLamSupports}, any two distinct attracting laminations of $\phi$ may be represented as $\Lambda_i,\Lambda_j$ with $1 \le i \le j-1 < j \le \abs{\L(\phi)}$. The lamination $\Lambda_i$ is carried by $\F_i$ and hence by $\F_{j-1}$, but the lamination $\Lambda_j$ is not carried by $\F_{j-1}$, and so the free factor supports of $\Lambda_i$ and $\Lambda_j$ are not equal. 

To obtain the bijection $\L(\phi) \leftrightarrow \L(\phi^\inv)$ one follows the inductive argument in \cite[Lemma 3.2.4]{\BookOneTag}, in this case inducting on the Kurosh rank $\KR(\Gamma;\A) = \corank(\Gamma;\A) + \abs{\A}$. Choose $\Lambda \in \L(\phi)$ with free factor support $\F\Lambda$, and let $F \subgroup \Gamma$ be a free factor realizing the unique nonatomic component $[F] \in \F\Lambda$. If $\Lambda$ is nonfilling, equivalently if $\KR(F;\A \,|\, F) < \KR(\Gamma;\A)$, then by applying induction to $\phi \restrict F \in \Out(F;\A \,|\, F)$ we obtain a unique $\Lambda' \in \L(\phi^\inv)$ with $\F\Lambda'=\F\Lambda$. We thus obtain a bijection between nonfilling elements of $\L(\phi)$ and $\L(\phi^\inv)$ such that corresponding laminations have the same free factor support. By the ``special case'' of Theorem~\ref{TheoremStrataLamCorr}, each of $\phi$ and $\phi^\inv$ has at most one filling lamination. By symmetry it suffices to assume that $\phi$ has a filling lamination and to prove that~$\phi^\inv$ alsh has a filling lamination. After passing to a power of $\phi$ we obtain a nonfilling free factor system $\F$ which is both $\phi$ and $\phi^\inv$-invariant, and such that there is free factor system strictly nested between $\F$ and $\{[\Gamma]\}$ that is either $\phi$-periodic or $\phi^\inv$-periodic. Let $f \from T \to T$ and $f' \from T' \to T'$ be relative train track representatives of $\phi$ and $\phi^\inv$ respectively, each having a filtration element that realizes $\F$. Let $H_R \subset T$, $H'_{R'} \subset T'$ be the top strata. Since $\phi$ has a filling lamination, $H_R$ is an EG-stratum. By Theorem~\ref{TheoremStrataLamCorr}~\pref{ItemCorollaryThreeTwoTwo}, $H_R$ contains edgelets of two natural edges $E_1,E_2 \subset T$ in different orbits, and it follows that $\DFF(\F) \ge 2$. If $H'_{R'}$ were not EG then it would be NEG, consisting of a union of natural edgelets whose orbits are permuted up to homeomorphism by $f'$, and we may assume that permutation is the identity (by passing to a further power); then, by removing any one of those orbits, we obtain a $\phi^\inv$-periodic free factor system $\F'$ with $\DFF(\F') \ge 1$, but no such free factor system exists. It follows that $H'_{R'}$ is an EG-stratum, and so there exists an attracting lamination $\Lambda' \in \L(\phi^\inv)$ which is not supported by the free factor system~$\F$. This lamination $\Lambda'$ must be filling, because we already have a support-preserving bijection between nonfilling laminations of $\phi$ and $\phi^\inv$, and all of them are supported by~$\F$.
\end{proof}

\begin{proof}[Proof of Theorem~\ref{TheoremStrataLamCorr}] 
Regarding conclusions~\pref{ItemLymanEGAperiodic} and~\pref{ItemActsAsIdentity}, all but item~\pref{ItemCorollaryThreeTwoTwo} follow from the citations given; we briefly put off the proof of item~\pref{ItemCorollaryThreeTwoTwo}.

Regarding conclusion~\pref{ItemLamSupports}, the sequence of nesting relations clearly follows from the equivalences~$(*)$, as does the inclusion $\tsp(\Lambda;T) \subset T_{r_i}$, hence $\Lambda$ is carried by $\F_i$. What remains is to show that $\Lambda_i$ is not carried by $\F_{i-1}$. Otherwise, letting $\ell$ be a generic leaf of $\Lambda_i$, there exists a free factor $F$ rel~$\A$ such that $[F] \in \F_{i-1}$ and such that $\ell \subset T_{F}$ (the minimal subtree of $T$ with respect to the action of $F$). Letting $\tau$ be the component of $T_{r_{i-1}}$ that is stabilized by $F$, it follows that $T_{F} \subset \tau$, hence $\ell \subset \tau \subset T_{r_{i-1}}$. But this is a contradiction, because a generic leaf of $\Lambda_i$ crosses a translate of every edge of $H_{r_i}$ and $H_{r_i} \subset T_{r_{i}} \setminus T_{r_{i-1}}$.

Turning now to item~\pref{ItemCorollaryThreeTwoTwo}, we first consider the special case that $H_{r,j}=H_r$ is the top stratum; afterwards we explain how to reduce to this case. 

In the special case we argue by contradiction: if item~\pref{ItemCorollaryThreeTwoTwo} fails then~$T$ has a natural edge $E$ such that $H_r \subset \Gamma \cdot E$. Consider the collapse map $T \mapsto U = T / (T \setminus \Gamma \cdot H_r)$. The free splitting $U$ has just one orbit of natural edges, represented by the image of~$E$. The map $f$ induces a $\Phi$-twisted equivariant train track representative $f_U \from U \to U$ of $\phi$ with respect to the free factor system $\F[T \setminus \Gamma \cdot H]$.
There is a factorization $f_U \from U \xrightarrow{g} V \xrightarrow{h} U \cdot \Phi$ such that $g$ is an equivariant foldable map, and $h$ is a $\Phi$-twisted equivariant homeomorphism, hence $V$ and $U$ represent the same orbit of vertices of $\FS(\Gamma;\A)$. But since $U$ has just one orbit of natural edges, we may apply conclusion~\pref{ItemOneSewingFold} of the Sewing Needle Lemma~\ref{LemmaSewingNeedleFold} to conclude that $U$ and $V$ represent different orbits.

To carry out the reduction, after passing to the appropriate power of $\phi$ and replacing $f$ by $f_\#$, we may assume that $H_{r,j}=H_r$ is EG-aperiodic. In the filtration element $T_r$, every line $L \subset T_r$ representing a generic leaf of $\Lambda_r$ crosses an edgelet in every natural edgelet orbit of $H_r$, hence the components of $T_r$ that contain edgelets of $H_r$ all fall into a single orbit under the action of $\Gamma$; choose $T''$ to be one such component of $T_r$. Letting $\Phi$ be the representative of $\phi$ such that $f$ is $\Phi$-twisted equivariant, after applying an appropriate element of $\Inn(\Gamma;\A)$ we may assume that $f(T'')=T''$; and then, letting $\Gamma_r \subgroup \Gamma$ be the free factor rel~$\A$ that stabilizes $T_r$, it follows that $\Phi_r$ preserves $\Gamma_r$ and that the restricted free factor system $\A_r = \A \restrict \Gamma_r$ is $\Phi_r$-invariant, hence $\Phi_r \restrict \Gamma_r$ is a representative of some $\phi_r \in \Out(\Gamma_r;\A_r)$. Letting $T' \subset T_r$ be the minimal subtree for the action of $\Gamma_r$ on $T_r$, it follows that $f' = f \restrict T' \from T' \to T'$ is a relative train track representative of $\phi_r$, that $H'_r = H_r \intersect T'$ is its top stratum, and that $H'_r$ is an EG-aperiodic stratum. Once having proved that $H'_r$ contains edgelets in two natural edges of $T'$ in distinct orbits under the action of $\Gamma_r$, it follows that $H_r$ contains edgelets in two natural edges of $T$ in distinct orbits under the action of $\Gamma$, completing the reduction. 

The bound in conclusion~\pref{ItemLamPeriodBound} comes from the upper bound to the length of any strictly nested chain of free factor systems of $\Gamma$ rel~$\A$, as proved in \cite[Lemma 2.14]{\RelFSOneTag}. 
\end{proof}

\subparagraph{Remarks on bounding $\abs{\L(\phi)}$.} The bound $\abs{\L(\phi)} \le 2 \corank(\Gamma;\A) + \abs{\A}  - 1$ in conclusion~\pref{ItemLamPeriodBound} is far from optimal. The origin of that bound is the optimal bound on ``free factor system depth'' $\DFF(\F) \le \DFF(\A) = 2 \corank(\Gamma;\A) + \abs{\A}  - 1$, where $\F$ is a free factor system of $\Gamma$ rel~$\A$ (see \cite[Section 2.5, Lemma 2.14]{\RelFSOneTag}). 

Here we describe a relative train track map for an element of $\Out(\Gamma;\A)$ which we believe achieves the optimum upper bound on $\abs{\L(\phi)}$. Our description is expressed using graph of groups language as in \cite{\LymanRTTTag}. 

Denote $C = \corank(\Gamma;\A)$ and $N = \abs{\A}$. Construct a graph $G$ with components as follows: start with $\lfloor C/2 \rfloor$ components which are rank~$2$ roses; in the case that $C$ is odd, add an additional component which is a rank~$1$ rose; and then add $N$ more single vertex components labelled by the elements of~$\A$. So far the graph has $\lfloor (C+1)/2 + N \rfloor$ components. The rank~$1$ rose component, if it exists, becomes a fixed edge stratum. Each rank~$2$ rose component becomes an EG-aperiodic stratum, by labelling its edges $a,b$ and using the formula $a \mapsto b$, $b \mapsto ab$. Now hook up the components of $G$ with a tree $T$: choose one component of $G$ and denote its vertex as $V$; the tree $T$ has an edge connecting $V$ to the vertex of every other component of $G$, with $E = \lfloor (C+1)/2  + N - 1 \rfloor$ edges in total. Partition the edges of $T$ into $\lfloor E/2 \rfloor$ subsets of 2 edges each, with one additional edge subset of $1$ edge in the case that $E$ is odd. The $1$ edge partition element, if it exists, becomes a fixed edge stratum. Each $2$ edge partition element becomes an EG-aperiodic stratum by denoting those $2$ edges as $c,d$, orienting them to have initial vertex~$V$, and applying formula of the form
$$c \, \mapsto \, c \, X \, \bar c \, d \, Y \, \bar d \, c \, \bar X \, \bar c \, Z \, c \qquad\qquad d \, \mapsto \, c \, X \, \bar c \, d
$$
In this formula, each of $X$, $Y$, $Z$ represents a nontrivial group element in the fundamental group of a component of $G$: a loop in a rose component; or an element of the group labelling an $\A$ component. 

We suspect that $\lfloor C/2 \rfloor + \lfloor E/2 \rfloor$, which is the number of EG-aperiodic strata in this construction, is the optimal bound for $\abs{\L(\phi)}$. In the case of $\Out(F_n)$ where $\A=0$, this number simplifies to $\lfloor 3n/4 \rfloor + D$ where $D=-1,0,1,-1$ depending respectively on $n \equiv 0,1,2,3$ mod~$4$. A proof of optimality would perhaps require analyzing the combinatorics of the filtrations that can occur in relative train track maps.

\bigskip

The following theorem shows that each attracting lamination of $\phi$ has a well-defined expansion factor, independent of the choice of relative train track representative. 

\begin{theorem}[\protect{\cite[Proposition 4.14]{\LymanCTTag}}]
\label{TheoremLamExpFac}
For each $\phi \in \Out(\Gamma;\A)$ and each $\Lambda \in \L(\phi)$ there exists a unique expansion factor $\lambda(\phi;\Lambda) > 1$ characterized by the following property: for each relative train track representative $f \from T \to T$ of $\phi$, letting $H_r \subset T$ be the EG stratum that contains the EG-aperiodic substratum $H_{r,j}$ associated to $\Lambda$ (see Theorem~\ref{TheoremStrataLamCorr}), $\lambda(\phi;\Lambda)$ is equal to the expansion factor of $f$ on~$H_r$.\hfill\qed
\end{theorem}

\noindent
The following corollary will be the foundation of the proof of well-definedness in Proposition~\ref{PropAxisInFS}~\pref{ItemExpFactWD} to be carried out in Section~\ref{SectionPropAxisInFSProof}.

\begin{corollary} 
\label{CorollaryTopWhenFillingExists}
If $\phi \in \Out(\Gamma;\A)$ has a filling attracting lamination $\Lambda$ then for every relative train track representative $f \from T \to T$, its top stratum $H_R$ is EG-aperiodic, and $\Lambda$ is the attracting lamination associated to $H_R$. Furthermore, $\lambda(\phi;\Lambda)$ is the Perron-Frobenius eigenvalue of the transition matrix $M_R$ of $f$ on $H_R$.
\end{corollary}

\begin{proof} Using that $\Lambda$ fills, it follows that $H_R$ is the top stratum that is crossed by leaves of $\Lambda$ in~$T$. Again using that $\Lambda$ fills, $\F \phi(\Lambda)) = \phi(\F\Lambda) = \phi\{[\Gamma]\} = [\Gamma] = \F\Lambda$, and by applying Lemma~\ref{LemmaSupportDistinction} we have $\phi(\Lambda)=\Lambda$; it follows that $\Lambda$ has period~$1$ under the action of $\phi$. Combining these with Theorem~\ref{TheoremStrataLamCorr} it follows that $H_R$ is EG-aperiodic, and that $H_R$ is the EG-aperiodic substratum corresponding to $\Lambda$. The ``Furthermore'' clause follows from Theorem~\ref{TheoremLamExpFac}.
\end{proof}

\subsubsection{Irreducible train track representatives. Penultimate collapse.}
\label{SectionIrrTTCollapse}
Consider a continuous equivariant map $F \from T \to T$ representing $\phi$. We say that $F$ is a \emph{train track} representative if either of the following two equivalent statements hold:
\begin{description}
\item[Definition 1:] $F$ takes vertices to vertices, and for every $x \in T$ there exists an embedded arc $\gamma \subset T$ such that $x$ is in the interior of $\gamma$ and such that $F^k \restrict \alpha$ is injective for all~$k \ge 1$. 
\item[Definition 2:] $F$ is a topological representative, and there exists an $F$-invariant gate structure on $T$ having at least two gates at each vertex.
\end{description}
In Definition~2, a \emph{gate structure} is a $\Gamma$-invariant family of equivalence relations on the set of directions $D_v T$ for each vertex $v \in T$; the equivalence classes are called \emph{gates} at $v$. To say a gate structure is \emph{$F$-invariant} means that for any $d,d' \in D_v T$ in the same gate, their images $D_v(d),D_v(d') \in D_{F(v)} T$ are also in the same gate. Definition~2 is common in more recent literature: see for example \cite[Section 8.3]{FrancavigliaMartino:TrainTracks}. Clearly Definition~2 implies Definition~1. For the converse, assuming Definition~1 we obtain local injectivity of $F \restrict \interior(e)$ for any edge $e \subset T$. Using that $T$ is a tree, we next obtain local injectivity of $F \restrict e$, and then global injectivity of $F \restrict e$, and it follows that $F$ is a topological representative. A gate structure satisfying the requirements of Definition~2 is then obtained by requiring, for each vertex $v \in T$, that $d,d' \in D_v T$ be in the same gate if and only if there exists $k \ge 1$ such that $DF^k(d)=DF^k(d')$ \hbox{(see e.g.\ \cite[Lemma 8.11]{FrancavigliaMartino:TrainTracks}).}

Given a train track representative $F \from T \to T$ of $\phi \in \Out(\Gamma;\A)$, we next give two definitions of what it means for $F$ to be an \emph{irreducible} train track representative: 
\begin{description}
\item[Definition 3:] The transition matrix of $F$ is irreducible.
\item[Definition 4:] $F$ is a relative train track representative with exactly one stratum.
\end{description}
Clearly Definition 3 implies Definition 4. For the converse, the one stratum of Definition 4 cannot be a zero stratum because then there would exist another, lower stratum; that stratum is therefore irreducible. Further properties of that one stratum thus become properties of $F$ itself, such as NEG, EG, and EG-aperiodic. 

If $\phi$ has an EG-aperiodic train track representative $F \from T \to T$ then, using Definition~4, one may apply Theorem~\ref{TheoremStrataLamCorr} to conclude that $\phi$ has just one attracting lamination \emph{with respect to~$\F$}, which we shall denote $\Lambda(F;\F) \subset \Bline(\Gamma;\F)$. Furthermore, the concrete lines in $T$ representing generic leaves of $\Lambda(F;\F)$ are precisely those lines that are nested unions of tiles of~$F$.

\medskip

In the rest of this section we describe a construction called \emph{penultimate collapse} which produces irreducible train track representatives from relative train track representatives, at the expense of changing the free factor system. In brief, starting with $\phi \in \Out(\Gamma;\A)$ and a relative train track representative with respect to~$\A$ of $\phi$, and letting~$\F$ be the free factor rel~$\A$ that is associated to the penultimate filtration element of the given relative train track map, by collapsing that filtration element we obtain a train track representative of $\phi$ relative to~$\F$.

\begin{definition}[Penultimate collapse]
\label{DefPenColl}
Consider $\phi \in \Out(\Gamma;\A)$, and a relative train track representative $f \from S \to S$ rel~$\A$ with associated filtraiton of length $R$, its penultimate filtration element denoted as $S_{R-1}$ and its top stratum as $H_R = S \setminus S_{R-1}$. Let $\F = \F[S_{R-1}]$, a $\phi$-invariant free factor system rel~$\A$. Consider also the collapse map $\pi \from S \xrightarrow{\langle S_{R-1} \rangle} T$, and so $T$ is a free splitting of $\Gamma$ rel~$\F$. The map $\pi$ induces a $\Gamma$-equivariant bijection denoted $e \leftrightarrow e' = \pi(e)$, between all top stratum edges $e \subset \Edges(H_R)$ and all edges $e' \in \Edges(T)$. There is a unique induced map $\hat F \from T \to T$ such that $\hat F \circ \pi = \pi \circ f \from T \to T$. Note that for each corresponding pair of edges $e \leftrightarrow e'$ the restriction $\hat F \restrict e'$ is a hare's path that crosses at least one edge of $T$, which follows because $f(e)$ contains at least one edge of $H_R$. Let $F \from T \to T$ be the tight map obtained by tightening $\hat F$. The map $F \from T \to T$ is therefore a topological representative of $T$, and we obtain the following diagram which is commutative up to tightening:
$$\xymatrix{
S \ar[rr]^{\pi}_{\<S_{R-1}\>} \ar[d]_f && T \ar[d]^F \\
S \ar[rr]^{\pi}_{\<S_{R-1}\>}  && T
}$$ 
We shall refer to this as the \emph{penultimate collapse diagram} associated to the relative train track representative $f \from S \to S$.
\end{definition}


\begin{proposition}[Penultimate collapse]
\label{PropCollapseTTMap}
Given $\phi \in \Out(\Gamma;\A)$, a relative train track representative $f \from S \to S$ of $\phi \in \Out(\Gamma;\A)$, and the penultimate collapse diagram associated to $f$ as denoted above, the following hold: 
\begin{enumerate}
\item\label{ItemTransMatSame}
The transition matrix of $f$ on $H_R$ equals the transition matrix of $F$ on $T$ (once a choice of enumerated edge orbit representatives of $H_R$ has been projected to $T$ via $\pi$)
\item\label{ItemCollapseIrreducible}
$F$ is an irreducible train track representative of $\phi$ with respect to $\F$. 
\item\label{ItemEGnessCorresponds}
$f$ is NEG/EG/EG-aperiodic on $H_R$ if and only if $F$ is NEG/EG/EG-aperiodic on $T$ (respectively), and in each of these cases we have an additional conclusion:
\begin{enumerate}
\item\label{ItemUltimateNEG}
If NEG holds then $\phi$ fixes $T$ (as a vertex of $\FS(\Gamma;\A)$).
\item\label{ItemUltimateEGFactors}
If EG holds then the expansion factor of $F$ equals the expansion factor of~$f$ on~$H_r$.
\item\label{ItemUltimateEGLeaves}
If EG-aperiodicity holds then, letting $\Lambda_f$ be the unique attracting lamination of $\phi$ rel~$\A$ that is associated to the stratum $H_R$, and letting $\Lambda_F$ be the unique attracting lamination of $\phi$ rel~$\F$ (see Theorem~\ref{TheoremStrataLamCorr} for both of these laminations), the following hold:
\begin{enumerate}
\item\label{ItemTilesBiject} 
For each $k \ge 0$ the collapse map $\pi$ induces a bijection from $k$-tiles of $f$ on $H_R$ to $k$ tiles of $F$.
\item\label{ItemGenericsBiject}
The collapse map $\pi$ induces a bijection between the set of concrete lines in $S$ representing generic leaves of $\Lambda_f$ and the set of concrete lines in $T$ representing generic leaves of $\Lambda_F$.
\end{enumerate}
\end{enumerate}
\end{enumerate}
In particular if $f$ has a filling lamination then EG-aperiodicity holds (in the context of~\pref{ItemEGnessCorresponds}), hence the conclusions of~\pref{ItemUltimateEGLeaves} all hold, and $\Lambda_f$ is the filling lamination of $\phi$.
\end{proposition}

\begin{proof}
Choose enumerated edge orbit representatives $\{e_i\}_{i \in I} \subset \Edges(H_R)$ and $\{e'_i\}_{i \in I} \in \Edges(T)$ with $e'_i = \pi(e_i)$. Note for each $\gamma \in \Gamma$ that the path $f(e_j)$ crosses $\gamma \cdot e_i$ if and only if the path $F(e'_j)$ crosses $\gamma \cdot e'_j$. Item~\pref{ItemTransMatSame} and the main clause of~\pref{ItemEGnessCorresponds} follow immediately. Letting $M$ be the transition matrix, item~\pref{ItemUltimateEGFactors} also follows immediately. 

Assuming that $H_R$ is NEG, the matrix $M$ is an irreducible matrix of non-negative integers with Perron-Frobenius eigenvalue equal to $1$, from which it follows that $M$ is a cyclic permutation matrix of some order $p \ge 1$. The map $f$ therefore permutes the edges of $H_R$, hence $F$ permutes the edges of $T$, and so $F$ is a homeomorphism. This proves item~\pref{ItemCollapseIrreducible} in the NEG case, and furthermore since $F$ is $\Phi$-twisted equivariant for some $\Phi \in \Aut(\Gamma;\A)$ representing $\phi$, it follows that $\phi$ fixes $T$, proving item~\pref{ItemUltimateNEG}.

Assuming that $H_R$ is EG, and using that $f$ is a relative train track map, for each edge $e \in \Edges(H_R)$ with image $e' = \pi(e) \in \Edges(T)$, and for each $k \ge 1$, the act of straightening the restriction $f^k \restrict e$ to form the path $f^k_\sharp(e)$ is carried out by straightening each subpath in $T_{R-1}$ without altering the sequence of $H_R$ edges crossed by $f^k \restrict e$. It follows that the restriction $\pi \circ f^k \restrict e = \hat F^k \circ \pi \restrict e = \hat F^k \restrict e'$ is a hare's path, and that tightening that restriction is carried out by simply removing maximal constant subpaths to form the path; this tightened path is therefore equal to the restriction $F^k \restrict e'$, and hence that restriction has no cancellation at all. This proves item~\pref{ItemCollapseIrreducible} in the EG case. If EG-aperiodicity holds then 
item~\pref{ItemTilesBiject} follows from the equation $\pi(f^k_\sharp(e)) = F^k(e')$; and then, by applying Theorem~\ref{TheoremStrataLamCorr}~\pref{ItemNestedUnionOfTiles} to $f$ on $H_R$ and to $F$ on $T$, item~\pref{ItemGenericsBiject} follows immediately. 
\end{proof}

\paragraph{Remarks.} There are more general versions of Definition~\ref{DefPenColl} and Proposition~\ref{PropCollapseTTMap} where one starts from a relative train track representative $f \from S \to S$ and collapses an arbitrary filtration element $S_r$ (not necessarily the penultimate one). In that case Proposition~\ref{PropCollapseTTMap} generalizes with ease, producing a relative train track representative of $\phi$ with respect to~$\F = \F S_r$, although in addition to collapsing $S_r$ \emph{one must also} collapse any zero stratum edges $e \subset S$ such that $f^k(e) \subset S_r$ for some $k$; equivalently, after collapsing just $S_r$ itself, one must follow up by collapsing the maximal pretrivial subforest. In this case the collapse map $\pi \from S \to T$ induces a bijection between the set of irreducible strata of $f$ that lie above $S_r$ and the set of all irreducible strata of $g$; also, a version of conclusion~\pref{ItemEGnessCorresponds} holds for each corresponding pair of such strata.

\subsubsection{Proof of Theorem C.} 
\label{SectionThmCProof}

Given $\phi \in \Out(\Gamma;\A)$, we must prove that $\phi$ fixes some free splitting of $\Gamma$ rel~$\A$ if and only if its set $\L(\phi)$ of attracting laminations does not fill~$\Gamma$ rel~$\A$. Our proof here does not have any particularly new ideas, it is more-or-less an updating of the proof from \FSTwo\ obtained by applying the general relative train track theory of Lyman \cite{Lyman:RTT,Lyman:CT}.

\smallskip

For the ``if'' direction, assuming that $\L(\phi)$ does not fill~$\Gamma$ rel~$\A$, the finite union $\bigcup \L(\phi) = \bigcup_{\Lambda \in \L(\phi)} \Lambda$ is a $\phi$-invariant lamination of $\Gamma$ rel~$\A$ whose free factor support $\F\left(\bigcup\L(\phi)\right)$ is nonfilling and $\phi$-invariant. Choose $\F$ to be a maximal, nonfilling, $\phi$-invariant free factor system such that $\F\left(\bigcup\L(\phi)\right) \sqsubset \F$. Let $f \from S \to S$ be a relative train track representative with penultimate filtration element $S_{R-1}$ such that $\F = \F S_{R-1}$ (Corollary~\ref{CorollaryMaximal}). For each EG-aperiodic substratum $H_{r,j} \subset S$ with corresponding attracting lamination $\Lambda_{r,j} \in \L(\phi)$ (Theorem~\ref{TheoremStrataLamCorr}), knowing that $\Lambda_{r,j} \subset \bigcup \L(\phi)$ is supported by $\F = \F S_{R-1}$, it follows that $r \le R-1$. The stratum $H_R$ is therefore NEG, and so $\phi$ fixes the free splitting obtained from $T$ by collapsing~$T_{R-1}$ (Proposition~\ref{PropCollapseTTMap}~\pref{ItemUltimateNEG}).

\smallskip

For the ``only if'' direction, assume that $\phi(S)=S$ for some free splitting $S$ rel~$\A$, and so there is a $\Phi$-twisted equivariant homeomorphism $g \from S \to S$, where $\Phi \in \Aut(\Gamma;\A)$ represents~$\phi$. Let $\F = \FS S$, a $\phi$-invariant, non-filling free factor system rel~$\A$. Choosing a maximal filtration of $S$ by $\Gamma$-invariant and $g$-invariant subforests, the map $g$ becomes a relative train track representative with respect to~$\F$ in which every stratum is NEG, and therefore $\phi$ has no attracting laminations rel~$\F$ (by Theorem~\ref{TheoremStrataLamCorr}).

Let $\F'$ be a \emph{maximal} $\phi$-invariant, nonfilling free factor system of $\Gamma$ rel~$\A$ such that \hbox{$\F \sqsubset \F'$}. Choose a relative train track representative $g' \from S' \to S'$ of $\phi$ with respect to~$\F$ with penultimate filtration element $S'_{R-1}$ such that $\F S'_{R-1}=\F'$ (Corollary~\ref{CorollaryMaximal}). By penultimate collapse we obtain an irreducible train track representative $g'' \from S'' \to S''$ with respect to~$\F'$ (Proposition~\ref{PropCollapseTTMap}). Since $\phi$ has no attracting laminations rel~$\F$, the top stratum of $S'$ is NEG, and so the unique stratum of $S''$ is NEG, hence $\phi$ has no attracting laminations rel~$\F'$ (Theorem~\ref{TheoremStrataLamCorr} again). Choose a relative train track representative $h \from U \to U$ of $\phi$ with respect to~$\A$ with top stratum $U_L$ such that $\F\,U_{L-1} = \F'$. By penultimate collapse we obtain an irreducible train track representative $h' \from U' \to U'$ with respect to~$\F'$ (Proposition~\ref{PropCollapseTTMap}). Having shown that $\phi$ has no attracting laminations rel~$\F'$, it follows that $h'$ is NEG (Theorem~\ref{TheoremStrataLamCorr}), hence the top stratum $U_L \setminus U_{L-1}$ of $h$ is NEG (Proposition~\ref{PropCollapseTTMap}). Every EG stratum of $h \from U \to U$ is therefore contained in $U_{L-1}$, hence the realization in $U$ of every attracting lamination of $\phi$ rel~$\A$ is contained in $U_{L-1}$ (Theorem~\ref{TheoremStrataLamCorr}). This proves that $\L(\phi)$ is supported by~$\F'$. \qed

\subparagraph{Remarks.} In the last sentences of the proof above there are some circumlocutions that could probably be avoided if, as suggested in the \emph{Remarks} at the end of Section~\ref{SectionIrrTTCollapse}, one had a stronger ``relative train track'' version of Proposition~\ref{PropCollapseTTMap}.

\subsubsection{Suspending train track maps to obtain fold axes in $\FS(\Gamma;\A)$}
\label{SectionAxis}

In earlier work \cite[Section~7.1]{HandelMosher:axes} we described how to suspend an irreducible train track representative of an element $\phi \in \Out(F_n)$ to get a fold axis of $\phi$ in the outer space $\X(F_n)$. Although that construction was couched in the language of metric marked graphs, it is easily translated into free splitting language by lifting to universal covers, producing an axis in $\FS(F_n)$ having the given train track representative as a first return map. We generalize that construction here in the setting of $\FS(\Gamma;\A)$, and then we apply it (with other tools) to prove Proposition~\ref{PropLamCollapsed} which will itself be applied in the sequel \cite{HandelMosher:RelComplexHypIII}.

\paragraph{Suspension axes.} Consider $\phi \in \Out(\Gamma;\A)$, a nonfull $\phi$-invariant free factor system~$\F$ rel $\A$ and an EG-irreducible train track representative $F \from T \to T$ of~$\phi$ with respect to~$\F$. We shall construct a ``suspension axis'' of $F$, meaning a fold axis whose first return map is the given map~$F$; see Proposition~\ref{PropFoldAxisConstruction} and Definition~\ref{DefSuspensionAxis} below for a summary. There is a subtlety to the construction regarding how vertex sets fit together. The vertex set on $T=T_0$ is \emph{given} as part of the data of the train track representative $F \from T \to T$. The vertex set on $T_p$ is then \emph{determined} by the requirement that $h^0_p \from T_0 \to T_p$ be a simplicial isomorphism. One must then \emph{construct} the factorization $F \from T_0 \xrightarrow{f_1} T_1 \xrightarrow{f_2} \cdots\xrightarrow{f_p} T_p \xrightarrow{h^p_0} T_0$ and the vertex sets on each $T_i$ so that each fold factor $f_i$ takes vertices to vertices. This may be impossible if fold factors are chosen carelessly; the solution is to choose maximal fold factors.

To start the construction we note that $F$ is $\Phi$-twisted equivariant for some $\Phi \in \Aut(\Gamma;\A)$ representing~$\phi$. Let $U_k = T \cdot \phi^k$ for all $k \in \Z$, and so there exists a unique $\Phi$-twisted equivariant simplicial isomorphism $H \from U_1 = T \cdot \phi \to T = U_0$ (for uniqueness see Lemma~\ref{LemmaTwEqUnique}~\pref{ItemEqIsoUni}; for existence see the heading \emph{Equivalence, and action by $\Out(\Gamma;\A)$} in Section~\ref{SectionRFSC}). By composition (see Lemma~\ref{LemmaTwEqUnique}~\pref{ItemTwistComp}) we obtain a $\Gamma$-equivariant map $f \from U_0 = T \xrightarrow{F} T \xrightarrow{H^\inv} U_1$, and we note that 
\begin{align*}
f (\Vertices U_0) &= (H^\inv \circ F)(\Vertices T) = H^\inv(F(\Vertices(T)) \\
&\subset  H^\inv(\Vertices T) = \Vertices(U_1)
\end{align*}
Applying Theorem~\ref{TheoremStallings}, we may factor $f$ as a Stallings fold sequence
$$f \from U_0 = T_0 \xrightarrow{f_1} T_1 \xrightarrow{f_2} \cdots \xrightarrow{f_p} T_p = U_1
$$
such that the map $f_i \from T_{i-1} \to T_i$ is a maximal fold factor of the map $f^{i-1}_p = f_p \circ\cdots\circ f_i \from T_{i-1} \to T_p$, for each $i=1,\ldots,p$. For index values $0 < i < p$, we define the simplicial structure on $T_i$ by induction so that $\Vertices T_{i} = f(\Vertices T_{i-1}) \union \Vertices_\nat T_{i}$; inclusion $f(\Vertices T_{i-1}) \subset \Vertices T_i$ is then evident for that range of index values. We must still prove the inclusion $f(\Vertices T_{p-1}) \subset \Vertices T_p$, but for the moment we assume that inclusion is true and complete the rest of the construction. 

For each $l \in \Z$, taking $q \in \Z$ and $0 \le r \le p-1$ so that $l=qp+r$, we define $T_l = T_r \cdot \phi^q$, and we let $h^l_{l-p} \from T_{l} \to T_{l-p}$ be the $\Phi$-twisted equivariant simplicial isomorphism (unique by Lemma~\ref{LemmaTwEqUnique}). For each $i < j \in \Z$, from the identity
$$F^j = h^p_0 \circ \cdots \circ h^{jp}_{(j-1)p} \circ f^0_{jp}
$$
together with the train track property for $F^j$ and the fact that each $h$-term in that identity is a simplicial isomorphism, it follows that $f^0_{jp} \from T_0 \to T_{jp}$ is a foldable map, which implies more generally that each $f^{ip}_{jp} \from T_{ip} \to T_{jp}$ is foldable. Foldability of $f^l_m \from T_l \to T_m$ for arbitrary $l \le m$ then follows from Proposition \ref{PropFoldableProps}. Using that $\F=\FellT=\FellT_0$ it easily follows that $\F=\FellT_i$ for all~$i$. 

It remains to prove the inclusion $f_i(\Vertices T_{i-1}) \subset \Vertices T_i$ for all $i \in \Z$. We know this inclusion already for $0 < i < p$, and so by periodicity we know it for all $i$ not equal to a multiple of $p$. Again applying periodicity, it suffices prove the inclusion for $i=p$, that is, $f^{p-1}_p(\Vertices T_{p-1}) \subset \Vertices T_p$. By definition of train track map we have $F_0(\Vertices T_0) \subset \Vertices T_0$, and so from the composition expression 
$$f^0_p \from T_0 \xrightarrow{F_0} T_0 \xrightarrow{(h^p_0)^\inv} T_p
$$ 
together with the fact that $h^p_0$ is a simplicial isomorphism we have $f^0_p(\Vertices T_0) \subset \Vertices T_p$. Taking that as the basis step we proceed by induction: assuming for $1 \le i \le p-1$ that $f^{i-1}_p(\Vertices T_{i-1}) \subset \Vertices T_p$, we must prove that $f^i_p(\Vertices T_i) \subset \Vertices T_p$. Note first that, by applying the induction hypothesis, we have
$$(*) \qquad f^i_p(f_i(\Vertices T_{i-1})) = f^i_p \circ f^{i-1}_i(\Vertices T_{i-1}) = f^{i-1}_p(\Vertices T_{i-1}) \subset \Vertices T_p \qquad\hphantom{(*)}
$$
We now break the proof into two cases depending on whether $f_i$ is a full fold or a partial fold (see \cite[Section 4.1]{\RelFSOneTag}). Choose oriented natural edges $E,E' \subset T_{i-1}$ with common initial vertex $v$, and proper initial segments $\eta \subset E$, $\eta' \subset E'$ with endpoints $w$, $w'$ opposite~$v$. 

In the first case where $f_i$ is a \emph{full fold}, meaning that $\eta=E$ or $\eta'=E'$, it follows that $\Vertices T_i = f_i(\Vertices T_{i-1})$, and $f^i_p(\Vertices T_i) = f^i_p(f_i(\Vertices T_{i-1}))$; using $(*)$, we are done in that case. 

Consider now the remaining case where $f_i$ is a \emph{partial fold}, meaning $\eta \ne E$ and $\eta' \ne E'$, and so $w,w'$ are interior points of the edges $E,E'$ respectively. We note that $f_i$ is not a degenerate sewing needle fold (see Section~\ref{SectionSewingNeedle}), for in that case $\FellT_i$ is strictly larger than $\FellT_{i-1}$. It follows that $\hat w = f(w)=f(w') \in T_i$ is a valence~$3$ vertex with trivial stabilizer, at which the three directions $d_0,d_1,d_2$ are as follows: $d_0$ is the common image of the terminal directions of $\eta$ at $w$ and of $\eta'$ at $w'$; $d_1$ is the image of the initial direction of $E \setminus \eta$ at $w$; and $d_2$ is the image of the initial direction of $E' \setminus \eta'$ at $w'$. It follows that $\Vertices T_{i} = f_{i}(\Vertices(T_{i-1})) \union \Gamma \cdot \hat w$. Using~$(*)$, all that remains to be proved is that $W = f^{i}_p(\hat w)$ is a natural vertex of $T_p$, in fact $W$ has valence~$\ge 3$. To see why, the turn $\{d_0,d_1\}$ in $T_i$ is the image of the unique turn at the non-natural vertex $w$ of $T_{i-1}$, and so since $f^{i-1}_p \from T_i \to T_p$ is a foldable map it follows that the turn $\{d_0,d_1\}$ is not foldable with respect to $f^i_p$. Using the non-natural vertex $w'$ in a similar manner, the turn $\{d_0,d_2\}$ is not foldable with respect to $f^i_p$. Also, since $f_i$ is a maximal first fold factor of the map $f^{i-1}_p$, the turn $\{d_1,d_2\}$ is not foldable with respect to $f^i_p$. The three directions $d_0,d_1,d_2$ at $\hat w$ therefore map to three distinct directions at~$W$.

\smallskip

This completes the construction, which we summarize in the following proposition and definition:


\begin{proposition} 
\label{PropFoldAxisConstruction}
Given $\phi \in \Out(\Gamma;\A)$, a nonfull $\phi$-invariant free factor system~$\F$ rel $\A$ and an EG-irreducible train track representative $F \from T \to T$ of~$\phi$ with respect to~$\F$, the construction above yields a fold axis of~$\phi$ with respect to $\F$ having first return map $F$. \qed
\end{proposition}

\begin{definition}
\label{DefSuspensionAxis}
A fold axis constructed as in Proposition~\ref{PropFoldAxisConstruction} is said to be a \emph{suspension axis} of $F \from T \to T$.
\end{definition}

\paragraph{A fold axis dichotomy.} Proposition~\ref{PropLamCollapsed} to follow states a dynamical dichotomy for an element of $\phi \in \Out(\Gamma;\A)$ and a choice of maximal, $\phi$-invariant, nonfull free factor system~$\F$ rel~$\A$. It may be regarded as a partial result towards the proof of the implication \pref{ItemActLargeOrbit}$\implies$\pref{ItemFillingLamExists} of Theorem~B, and it will be applied in the sequel \cite{HandelMosher:RelComplexHypIII} where the full proof of that implication is given. 

While the statement of the proposition does not involve the ``penultimate collapse'' construction, nonetheless that construction is the underpinning of the proof: the proof will combine Propositions~\ref{PropCollapseTTMap} and~\ref{PropFoldAxisConstruction} with results of~\cite[Section 4.4]{\RelFSOneTag} that are concerned with the complexity of invariant subgraphs of free splittings.  

\begin{proposition}
\label{PropLamCollapsed}
For any $\phi \in \Out(\Gamma;\A)$ and any maximal, $\phi$-invariant, non-filling free factor system $\F$ of $\Gamma$ rel~$\A$, letting $m \ge 0$ be the number of attracting laminations of $\phi$ rel~$\A$ that are not supported by $\F$, exactly one of the following holds:
\begin{enumerate}
\item\label{ItemBoundedByFour}
$m \ne 1$ in which case there exists a Grushko free splitting $T$ of $\Gamma$ rel~$\F$ such that its $\phi$-orbit \break $\{T \cdot \phi^k\}_{k \in \Z} \subset \FS(\Gamma;\A)$ has diameter~$\le 4$ with respect to the simplicial metric on~$\FS(\Gamma;\A)$.
\item\label{ItemLamCollapsed}
$m=1$, i.e.\ there is a unique attracting lamination $\Lambda$ of $\phi$ rel~$\A$ not supported by~$\F$. 
In this case there exists a fold axis for $\phi$ with respect to $\F$, notated as in Definition~\ref{DefFoldAxes}, such that the following hold:
\begin{enumerate}
\item\label{ItemLamCollEGAperiodic}
The first return map $F_0 \from T_0 \to T_0$ is an EG-aperiodic train track representative of $\phi$~rel~$\F$. Let $\Lambda_\F$ be the unique attracting lamination of $\phi$ rel~$\F$ (obtained by applying the ``special case'' at the end of the statement of Theorem~\ref{TheoremStrataLamCorr}).
\item\label{ItemLamCollLeaves}
There exists a collapse map $\pi \from S \to T_0$ defined on some Grushko free splitting $S$ rel~$\A$, such that $\pi$ induces a bijection between generic leaves of $\Lambda$ realized in $S$ and generic leaves of $\Lambda_\F$ realized in $T_0$.
\end{enumerate}
\end{enumerate} 
\end{proposition}

\begin{proof}
Choose a relative train track representative $f \from S \to S$ of $\phi$ rel $\A$ with top stratum $H_R \subset S$ and with some filtration element $S_i$ ($i < R$) such that $\F = \F[S_i]$. By maximality of $\F$ it follows that $\F = \F[S_{R-1}]$. Applying Proposition~\ref{PropCollapseTTMap} we obtain:
\begin{itemize}
\item[$(*)$] A penultimate collapse diagram as in Definition~\ref{DefPenColl}, namely a collapse map $\pi \from S \to T$ and irreducible train track representative $F \from T \to T$ of $\phi$ rel~$\F$ such that $\pi \circ f$ and $F \circ \pi \from S \to T$ are the same up to tightening. 
\end{itemize}
We break the proof into cases depending on the dynamical behavior of $F$: first, whether $F$ is NEG or EG; and then, in the EG case, whether or not $F$ is EG-aperiodic. We note that by Proposition~\ref{PropCollapseTTMap}~\pref{ItemEGnessCorresponds}, each respective case hypothesis on $F$ applies as well to the dynamical behavior of $f$ on its top stratum $H_R$, allowing us to apply the appropriate sub-conclusions of Proposition~\ref{PropCollapseTTMap}~\pref{ItemEGnessCorresponds} in the various cases.

\medskip
\textbf{Case 1: $F$ is NEG.} It follows from Theorem~\ref{TheoremStrataLamCorr} that $m=0$. The conclusions of \pref{ItemBoundedByFour} therefore hold: applying Proposition~\ref{PropCollapseTTMap}~\pref{ItemUltimateNEG} we have $T \cdot \phi = T$, and so the orbit $\{T \cdot \phi^k\}$ has diameter~$0$.

\medskip
\textbf{Case 2: $F$ is EG.} The EG-irreducible stratum $T$ subdivides into an edge disjoint union of EG-aperiodic substrata $T = \tau_1 \union\cdots\union \tau_m$ where $m \ge 1$, and we may choose the enumeration so that $F(\tau_j)=\tau_{j+1}$ for each $j$ modulo~$m$, and so $F^m(\tau_j)=\tau_j$. Applying Theorem~\ref{TheoremStrataLamCorr}, $m$~is~also the number of attracting laminations of $\phi$ rel~$\A$ that are not supported by $\F$, one such lamination for each the EG-aperiodic substratum $\tau_1,\ldots,\tau_m$.

Applying Proposition~\ref{PropFoldAxisConstruction} and using the notation of Definition~\ref{DefFoldAxes}, we obtain a bi-infinite fold path which is a fold axis for $\phi$ of period $p$ with $T=T_0$ and with first return map \hbox{$F = F_0 \from T_0 \to T_0$.} We depict here a portion of this fold axis and its associated commutative diagram, together with some of its $\Phi$-twisted equivariant maps $h^i_{i-p}$ and first return maps $F_l$ (where $\Phi \in \Aut(\Gamma;\A)$ is a chosen representative of $\phi \in \Out(\Gamma;\A)$):
$$\xymatrix{
\protect{\hphantom{T_{-2p}}}&\cdots \ar[r] &T_{-p} \ar[r]^{f_{-p+1}} \ar[d]_{F_{-p}} 
& \cdots \ar[r]^{f_0}  
&T_0 \ar[r]^{f_1} \ar[dll]_{h^0_{-p}} \ar[d]_{F_0} & \cdots \ar[r]^{f_p} 
& T_p \ar[dll]_{h^p_0} \ar[r]^{f_{p+1}} \ar[d]_{F_p} & \cdots &
\\
&\cdots \ar[r] &T_{-p} \ar[r]_{f_{-p+1}} & \cdots \ar[r]_{f_0} 
&T_0 \ar[r]_{f_1} & \cdots \ar[r]^{f_p} 
& T_p \ar[r]_{f_{p+1}} & \cdots
}$$
As usual we have foldable maps denoted $f^i_j \from T_i \to T_j$. Also, for each $d \ge 1$ the same bi-infinite fold path is a fold axis for $\phi^d$ of period $dp$, with $\Phi^d$-twisted equivariant maps $h^i_{i-dp} \from T_i \to T_{i-dp}$ and with first return maps $F^d = f^{-dp}_0 \circ h^0_{-pd} \from T_0 \to T_0$. 

\smallskip
We now break into two subcases.

\paragraph{Case 2a: $F_0$ is EG-aperiodic, equivalently $m=1$.} The existence and uniqueness of $\Lambda_\F$ comes from  Proposition~\ref{PropCollapseTTMap}~\pref{ItemUltimateEGLeaves}. By applying Proposition~\ref{PropCollapseTTMap}~\pref{ItemGenericsBiject}, the collapse map $\pi \from S \to T=T_0$ witnesses that conclusion \pref{ItemLamCollLeaves} holds.

\paragraph{Case 2b: $F_0$ is not EG-aperiodic, equivalently $m \ge 2$.} In this case we must shall verify the conclusions of item~\pref{ItemBoundedByFour} of the proposition, showing that the set $\{T \cdot \phi^k\}$ has diameter~$\le 4$ by applying \cite[Lemma 4.13 (3a)]{\RelFSOneTag}. The hard work is to verify the hypotheses of that lemma, based on the decomposition $T=\tau_1 \union \cdots\union \tau_m$ of $T$ into edge-disjoint, nondegenerate subgraphs (in the sense of Section~\ref{SectionFreeSplittings}).

We recall from \cite[Section 4.3]{\RelFSOneTag} the following ``pullback'' operation on nondegenerate subgraphs, considered there only for equivariant maps, but extending in a straightforward way to twisted equivariant maps: given a twisted equivariant map $g \from U \to U'$ between two free splittings of $\Gamma$ rel~$\A$ such that $g$ is injective on each edge of $U$ (e.g.\ $g$ is a foldable map), and given a nondegenerate subgraph $\upsilon' \subset U'$, the pullback $\upsilon = g^*(\upsilon') \subset U$ is the nondegenerate subgraph obtained from $g^\inv(\upsilon')$ by removing all isolated points. For example, using that $F(\tau_{j-1})=\tau_{j}$ for all $j$ modulo $m$, and that the $\tau_j$'s are edge disjoint and have no isolated vertices, it follows that $\tau_{j-1} = F^*(\tau_j)$, and more generally that $(F^d)^*(\tau_j) = \tau_{j-d}$ (using indices modulo~$m$).

We claim that for each $i \in \Z$ there is an edge-disjoint decomposition into invariant subgraphs $T_i = \tau_{i,1} \union\cdots\union\tau_{i,m}$ (with $\tau_{0,j}=\tau_j \subset T_0=T$), such that the following hold:
\begin{description}
\item[Fold Invariance:] $\tau_{i,j} = (f_{i+1})^* (\tau_{i+1,j})$ ($i \in \Z$, $j=1,\ldots,m$); 
\item[$h$-Invariance:] $h^0_{dp}(\tau_{0,j}) = \tau_{dp,j-d}$ ($d \in \Z$, $j=1,\ldots,m$ mod $m$).
\end{description}
Here are details of proof for $i \le 0$ and $d \le 0$, which is the only case that we actually apply. For $i \le 0$ we use ``Fold Invariance'' as the inductive definition of $\tau_{i,j}$, starting from the base case $\tau_{0,j}=\tau_j$. It follows, by induction, that this defines an edge disjoint decomposition of $T_i$. 
For $d \le 0$ we prove ``$h$-Invariance'' as follows. We know that $(F^{-d})^*(\tau_{0,j-d}) = \tau_{0,j}$ and that $(f^{dp}_0)^*(\tau_{0,j-d}) = \tau_{dp,j-d}$. Together with the commutativity equation $F^{-d} = f^{dp}_0 \circ h^0_{dp}$ and the fact that $h^0_{dp}$ is a twisted equivariant simplicial isomorphism, the $h$-invariance equation follows for $d \le 0$. For $i \ge 0$, in brief, one may inductively define $\tau_{i,j} = f_i(\tau_{i-1,j})$ for $i \ge 1$, but one must still show that $T_i = \tau_{i,1} \union\cdots\union \tau_{i,m}$ is an edge-disjoint decomposition; we leave further details to the reader, including the proof of $h$-invariance for $d > 0$.

\medskip

To complete Case 2(a) we must prove for all $i<j$ that $d(T_i,T_j) \le 4$. We may assume that $j \le 0$, by first choosing $d \in \Z$ so that $i-dp < j-dp \le 0$ and then using that $\phi$ acts as an isometry of $\FS(\Gamma;\A)$ to conclude that $d(T_i,T_j) = d(T_i \cdot \phi^{-d},T_j \cdot \phi^{-d}) = d(T_{i-dp},T_{j-dp})$. Now choose $d \ge 1$ so that $-dmp \le i$, and let $q=mp$. We shall prove that the following fold subsequence (which contains $T_i$ and $T_j$) has diameter $\le 4$:
$$T_{-dq} \xrightarrow{f_{-dq+1}} T_{-dq+1} \xrightarrow{f_{-dq+2}} \cdots \xrightarrow{f_{-1}} T_{-1} \xrightarrow{f_0} T_0
$$
For each $k$ such that $-dq \le k \le 0$ consider the following $\Gamma$-invariant subgraph 
$$\beta_k = \tau_{k,1} \subset T_k
$$
This subgraph $\beta_k$ is proper, as a consequence of the case hypothesis $m \ge 2$. In \cite[Section 4.4.1]{\RelFSOneTag}, a natural number valued complexity $C(\beta)$ is defined for all pairs $\beta \subset T$ such that $\beta$ is a proper invariant subgraph of a free splitting $T$ of $\Gamma$ rel~$\A$. For any other such pair $\beta' \subset T'$, the following properties are proved in \cite[Section 4.4.1 and 4.4.2]{\RelFSOneTag}: 
\begin{description}
\item[\protect{\cite[Lemma 4.10(b)]{\RelFSOneTag}}] If there is a twisted equivariant simplicial isomorphism \hbox{$h \from T \to T'$} such that $h(\beta)=\beta'$ then $C(\beta)=C(\beta')$.
\item[\protect{\cite[Lemma 4.12]{\RelFSOneTag}}] If there is a fold map $f \from T \to T'$ such that $\beta = f^*(\beta')$ then \hbox{$C(\beta) \ge C(\beta')$.}
\end{description}
Applying the first property in conjunction with ``$h$-Invariance'' we conclude that $C(\beta_{-kq})$ is a constant for $0 \le k \le d$; in particular $C(\beta_{-dq}) = C(\beta_0)$. Applying the second property in conjunction with ``Fold Invariance'' it follows that the complexity values $C(\beta_k)$ are nonincreasing along the interval $-dq \le k \le 0$, but the first and last terms are equal and so the complexity values are constant on that interval. This is exactly the hypothess of \cite[Lemma 4.13 (3a)]{\RelFSOneTag}, and so the conclusion of that lemma applies: the entire vertex set $\{T_k \mid -dq \le k \le 0\}$ along the fold subsequence has diameter $\le 4$. 
\end{proof}

\subsubsection{Proof of Proposition~\ref{PropAxisInFS}}
\label{SectionPropAxisInFSProof}
We refer to reader to the original statement of the proposition, or to Proposition~\ref{PropAxisInFSEmbellished} below which is a fuller statement incorporating information about the filling attracting laminations.

Let $\phi \in \Out(\Gamma;\A)$ have a filling lamination $\Lambda \in \L(\phi)$. Fix $\F$ to be any maximal, non-filling, $\phi$-invariant free factor system of $\Gamma$ rel~$\A$. Fix $f \from S \to S$ to be a relative train track representative of $\phi$ with penultimate filtration element $S_{R-1}$ such that $\F=\F[S_{R-1}]$ (Corollary~\ref{CorollaryMaximal}). Since $\Lambda$ fills, the stratum $H_R$ is EG-aperiodic and $\Lambda$ is the attracting lamination associated to $H_R$ (Corollary~\ref{CorollaryTopWhenFillingExists}). Applying penultimate collapse to $f$ (Definition~\ref{DefPenColl} and Proposition~\ref{PropCollapseTTMap}), we obtain an EG-aperiodic train track representative $F \from T \to T$ of $\phi$ with respect to~$\F$ and a collapse map $\pi \from S \xrightarrow{\langle S_{R-1} \rangle} T$.  Applying Proposition~\ref{PropFoldAxisConstruction} and Definition~\ref{DefSuspensionAxis}, using $F$ we obtain an EG-aperiodic suspension axis for $F$ with respect to~$\F$, which completes the proof of Proposition~\ref{PropAxisInFS}~\pref{ItemFoldAxisExists}. 

To prove Proposition~\ref{PropAxisInFS}~\pref{ItemExpFactWD}, we start by using that $\Lambda \in \L(\phi)$ fills $\Gamma$ rel~$\A$ to conclude that $\lambda(\phi;\Lambda)>1$ is the expansion factor of $f$ on $H_R$ (Corollary~\ref{CorollaryTopWhenFillingExists}). Next, $\lambda(\phi;\Lambda)$ is the expansion factor of $F$ on $T$ (Proposition~\ref{PropCollapseTTMap}~\pref{ItemUltimateEGFactors}). Regarding $\phi$ as an element of the subgroup $\Out(\Gamma;\F) \subgroup \Out(\Gamma;\A)$, and using that $\phi$ possesses an EG-aperiodic train track representative with respect to~$\F$, namely $F \from T \to T$, by applying Theorem~\ref{TheoremLamExpFac} \emph{relative to~$\F$} we conclude that the expansion factor of any EG-aperiodic train track representative of $\phi$ \emph{relative to~$\F$} is well-defined, independent of the choice of such representative, and hence equals the expansion factor of $F$ which equals $\lambda(\phi;\Lambda)$. Finally, for any EG-aperiodic fold axis of $\phi$ with respect to~$\F$, any first return map of that axis is an EG-aperiodic train track representative of $\phi$ rel~$\F$, hence its expansion factor equals $\lambda(\phi;\Lambda)$.

This completes the proof of Proposition~\ref{PropAxisInFS}. 

\smallskip

It also completes the proof of the following expanded proposition, \emph{except} for a new conclusion \pref{ItemLamUnderCollapse} regarding generic leaves of $\Lambda$; that new conclusion will be applied as part of the proof of Theorem~B in Part~III of this work, near the end of \cite[Section 3.5]{\RelFSThreeTag}.


\begin{proposition}\label{PropAxisInFSEmbellished}
For each $\phi \in \Out(\Gamma;\A)$ which has a filling lamination $\Lambda$ the following hold:
\begin{enumerate}
\item\label{ItemFoldAxisExistsV2}
For each maximal, non-filling, $\phi$-invariant free factor system~$\F$ rel~$\A$ there exists an EG aperiodic fold axis of $\phi$ in $\FS(\Gamma;\A)$ with respect to~$\F$.
\item\label{ItemFixBAndAxis}
As one varies over all $\F$ as in item~\pref{ItemFoldAxisExistsV2}, and over all EG-aperiodic fold axes of $\phi$ in $\FS(\Gamma;\A)$ with respect to~$\F$, the following hold:
\begin{enumerate}
\item\label{ItemExpFactWDV2} The expansion factor of the axis is well-defined independent of $\F$ and of the axis;
\item\label{ItemLamUnderCollapse} The attracting lamination of the axis is also well-defined independent of $\F$ and of the axis, in the following sense. For any free splitting $T_i$ in the given axis with corresponding first return map $F_i \from T_i \to T_i$ (using notation from Definition~\ref{DefFoldAxes}), and for any Grushko free splitting $U$ of $\Gamma$ rel~$\A$ and any collapse map $\rho \from U \to T_i$, the map $\rho$ induces a bijection $L \leftrightarrow L' = \rho(L)$ between concrete lines $L \subset U$ representing generic leaves of $\Lambda$ and concrete lines $L' \subset T_i$ representing generic leaves of the attracting lamination $\Lambda_{F_i}$.
\end{enumerate}
\end{enumerate}
\end{proposition}

\begin{proof}[Proof of Conclusion~\protect{\pref{ItemLamUnderCollapse}}] Following up on the proof and notation of conclusion~\pref{ItemFoldAxisExistsV2}, and fixing $\F$ and a fold axis as in~\pref{ItemFixBAndAxis} we well as $T_i$ and $U$ as in~\pref{ItemLamUnderCollapse}, we set up a bounded cancellation argument based by the following diagram:
$$\xymatrix{
U \ar[d]_{\rho} \ar[r]^g & S \ar[d]^{\pi} \\
T_i & T \ar[l]_h
}$$
In this diagram $\pi \from S \to T$ is the collapse map used in the proof of~\pref{ItemFoldAxisExistsV2}. From that proof, consider also the EG-aperiodic train track representative $F \from T \to T$ with respect to~$\F$. The maps $g \from U \to S$ and $h \from T_i \to T$ are any choices of $\Gamma$-equivariant maps, which exist since $U$ and $S$ are both Grushko free splittings rel~$\A$, and $T_i$ and $T$ are both Grushko free splittings rel~$\F$. The diagram above commutes up to bounded distance, because the maps $\rho$ and $h \circ \pi \circ g$ are both $\Gamma$-equivariant. 

Since each of $F \from T \to T$ and $F_i \from T_i \to T_i$ is an EG-aperiodic train track representative of $\phi$ with respect to~$\F$, it follows from Theorem~\ref{TheoremStrataLamCorr} that these two representatives produce the same attracting lamination of $\phi$ rel~$\F$, with the same generic leaves
$$\Lambda(F;\F) = \Lambda(F_i;\F) \subset \Bline(\Gamma;\F)
$$
Using this equation of laminations it follows, together with the bounded cancellation lemma for lines (Lemma~\ref{LemmaLineBddCancellation}), that we have a bijection denoted $L_T \leftrightarrow L_{T_i}$ between concrete lines $L_T \subset T$ representing generic leaves of $\Lambda(F;\F)$ and concrete lines in $L_{T_i} \subset T_i$ representing generic leaves of $\Lambda(F_i;\F)$, and this bijection is characterized by the property
$$L_{T_i} \subset h(L_T) \subset N_{C(h)}(L_{T_i})
$$
where $C(h) \ge 0$ is a cancellation constant for $h$. From Proposition~\ref{PropCollapseTTMap} we have a bijection denoted $L_S \leftrightarrow L_T = \pi(L_S)$ between concrete lines $L_S \subset S$ representing generic leaves of $\Lambda$ and concrete lines $L_T \subset T$ representing generic leaves of $\Lambda(F;\F)$. Applying Lemma~\ref{LemmaLineBddCancellation} again we have a bijection $L_U \leftrightarrow L_S$ between concrete lines $L_U \subset U$ representing generic leaves of the filling attracting lamination $\Lambda$ and concrete lines $L_S \subset S$ representing generic leaves of $\Lambda$, and this bijection satisfies the property
$$L_S \subset g(L_U) \subset N_{C(g)} L_S
$$
where $C(g) \ge 0$ is a cancellation constant for $g$. Putting these altogether, and using that equivariant maps between free splittings are Lipschitz, it follows that the composed bijection $L_U \leftrightarrow L_S \leftrightarrow L_T \leftrightarrow L_{T_i}$ has the property that
$$L_{T_i} \subset h_i \circ \pi \circ g(L_U) \subset N_{C'}(L_{T_i})
$$
for some constant $C' \ge 0$. These latest inclusions tells us that the line $L_{T_i}$ has bounded Hausdorff distance from the image $h_i \circ \pi \circ g(L_U)$. But the maps $h_i \circ \pi \circ g$ and $\rho$ differ by a constant, and hence the image $h_i \circ \pi \circ g(L_U)$ and the  path $\rho(L_U)$ have bounded Hausdorff distance. It follows that the line $L_{T_i}$ and the path $\pi(L_U)$ have bounded Hausdorff distance, and therefore $\pi(L_U)$ is a line equal to $L_{T_i}$, completing the proof of~\pref{ItemLamUnderCollapse}.
\end{proof}

\subparagraph{Remark.} The proofs of Propositions~\ref{PropLamCollapsed} and~\ref{PropAxisInFSEmbellished} raise the following question. Consider a nested pair of free factor systems $\A \sqsubset \F$ of $\Gamma$, and consider an outer automorphism $\phi \in \Out(\Gamma)$ which fixes both $\A$ and $\F$ and hence $\phi \in \Out(\Gamma;\A) \intersect \Out(\Gamma;\F)$. Is there a conceptual way in which one may identify the attracting laminations of $\phi$ rel~$\F$ with those attracting laminations of $\phi$ rel~$\A$ that are not supported by~$\F$\,? We take this topic up in the sequel \cite{HandelMosher:RelComplexHypIII}, applying tools of Dowdall and Taylor from \cite{DowdallTaylor:cosurface}.

\section{Proving the \TOAT}
\label{SectionTwoOverAllProof}

The \TOAT, in its full ``iterated'' form stated in Section~\ref{SectionTOAStatement}, gives a uniform exponential growth property satisfied by any foldable map $f \from S \to T$ between any two free splittings $S,T$ of $\Gamma$ rel~$\A$, with an integer constant $\Delta=\Delta(\Gamma;\A)$ independent of $S$, $T$ and $f$: for any integer $n \ge 1$, if the free splitting complex distance $d_\FS(S,T)$ is at least $n\Delta$ then in $S$ one can find two natural edges $E_1,E_2 \subset S$ in distinct orbits such that each of the images $f(E_1),f(E_2)$ contains $2^{n-1}$ nonoverlapping subpaths each of which crosses a translate of every natural edge of $T$.

The first two subsections carry out some preliminary work. Section~\ref{SectionUniteratedReduction} uses a simple exponential growth argument to reduce the general case of the \TOAT\ to the special ``uniterated'' case where $n=1$.   Section~\ref{SectionDistanceBounds} collects various distance bounds in $\FS(\Gamma;\A)$ needed in later sections. Section~\ref{SectionDistanceBounds} also contains Definition~\ref{DefPriority} regarding the ``fold priority construction'', which formalizes a method for exploiting the non-uniqueness of a Stallings fold path in order to find such paths that are particularly useful; this concept of ``fold priority'', which is already implicit in the proofs of \cite[Lemma 5.2 (3)]{\FSOneTag} and in \cite[Lemma 4.13 (3a)]{\RelFSOneTag}, will play a central role throughout the proof of the \TOAT.

Section~\ref{SectionTwoOverAllOutline} breaks the proof of the uniterated case into a three step outline, those steps being carried out in the three subsections~\ref{SectionOneNatOverAll}---\ref{SectionTwoOverAllStepThree}. 

\paragraph{A sketch of the special case $\FS(F_n)$.} Before launching into formal details of the proof of the \TOAT\ starting with Section~\ref{SectionUniteratedReduction}, here we sketch out some key ideas in a special case, using the same three step outline that is laid out for the general case in Section~\ref{SectionTwoOverAllOutline}. The case we consider is that $f \from S \to T$ is a foldable map between two Grushko free splittings of~$F_n$, and so $f$ descends to a foldable map $F \from G \to H$ between the quotient graphs-of-groups $G = S / F_n$, $H = T/F_n$, which we may regard as a homotopy equivalence between two marked graphs representing two vertices of $\FS(F_n)$.

In this sketch we shall work entirely downstairs in the realm of homotopy equivalences of marked graphs which preserve the marking. We assume that the reader is familiar with that realm; this is indeed the exact setting where we first perused the issues of the \TOAT. We focus solely on the conclusion of the ``uniterated'' case of the \TOAT, showing that if the distance in $\FS(F_n)$ between $G$ and $H$ is sufficiently large then there exist two natural edges of $G$ such that, when mapped over by $F$, each of the two resulting paths in $H$ crosses every natural edge of~$H$.

\smallskip
\textbf{Step 1: One natural edge over all edges (Section~\ref{SectionOneNatOverAll}).} We show that if the distance in $\FS(F_n)$ between $G$ and $H$ is sufficiently large then there exists a natural edge $E \subset G$ such that the map $F$ restricts to a path $F \from E \to H$ that crosses every edge of~$H$. While this conclusion is stronger than the equation $F(E)=H$, it is only slightly stronger. Ignoring this difference for purposes of this sketch, we assume for every natural edge $E \subset G$ that $F(E) \ne H$, and we find an upper bound to the distance between $G$ and $H$ in $\FS(F_n)$.

For the first idea (and see ``Folding to increase rank'' in Section~\ref{SectionOneNatOverAll}), consider the following \emph{much} stronger assumption:
\begin{itemize}
\item Every natural edge $E \subset G$ has distinct endpoints and the restriction $F \from E \to H$ is injective (hence the subgraph $F(E) \subset H$ has rank~$0$ and so the restriction is not surjective).
\end{itemize}
Enumerating the natural edges $E_1,\ldots,E_K \subset G$, we have $K \le 3n-3$ (for a general group $\Gamma$ and free factor system~$\A$, the general bound is $K \le \MaxEdges(\Gamma;\A)$). 

Under the much stronger assumption, and using that that $E_1 \union \ldots \union E_K=G$, it follows that every edge $e \subset H$ has at most one preimage edge in each $E_k$ and hence at most $K$ preimage edges in $G$. One then applies a result of Bestvina and Feighn \cite[Lemma 4.1]{BestvinaFeighn:subfactor} to obtain an upper bound on $d_\FS(G,H)$ that is a linear function of $K$ (the general bound in $\FS(\Gamma;\A)$ is found in Proposition~\ref{PropsEdgeInverseDiamBound}); in the present situation where $K \le 3n-3$ we thus obtain a bound on $d_\FS(G,H)$ depending only on~$n$.

And while the assumption above is indeed much too strong, it motivates the second main idea (and see the ``Jumping Bound'' in item~\pref{ItemFoldEstimate} of Section~\ref{SectionOneNatOverAll}): Consider how each natural edge $E_k$ evolves under a Stallings fold factorization of $F \from G \to H$. The edge $E_k$ starts out as a subgraph of rank~$0$ (or rank~$1$ if its endpoints are equal). As one moves along the factorization the ranks of the images of $E_k$ form a nondecreasing sequence bounded above by $n-1$; that rank sequence can therefore jump at most $n-1$ times. And then, since $K \le 3n-3$, any fold factorization $F \from G = G_0 \mapsto G_1 \mapsto \cdots \mapsto G_L = H$ breaks naturally into a concatenation of a uniformly bounded number of fold subsequences --- that bound being $(n-1) (n-3) = n^2-4n+3$ --- such that along each such subsequence, and for each $k=1,\ldots,K$, the ranks of the images of $E_k$ do not jump. 

The third idea is this (and see the heading ``Folding to injectivity'' in Section~\ref{SectionOneNatOverAll}). The problem is now reduced to bounding the distance between the marked graphs $G_i$, $G_j$ at the beginning and end of each ``nonjumping'' subsequence. By focussing on the foldable map $G_i \mapsto G_j$, one is now in a situation somewhat like the ``much stronger assumption'' above: one write $G_i = \beta_1 \union \cdots \union \beta_K$ where $\beta_k$ is the image of $E_k$, and such that the image of each $\beta_k$ in $G_j$ is a subgraph of the same rank as $\beta_k$. But there is a problem, in that the restricted map $\beta_k \to G_j$ need no longer be injective. This problem is resolved by carefully prioritizing the choice of a fold factorization of the map $G_i \mapsto G_j$: first fold $\beta_1$ until no more folding is possible; next fold $\beta_2$ until no more folding is possible; and so on, through the whole list $\beta_1,\beta_2,\ldots,\beta_K$. Once again we get a distance bound: for as long as one is prioritizing folds in (the image of) $\beta_k$, one moves at most distance~$2$ in the free splitting complex, because no edges in the complement of $\beta_k$ are ever folded (see Lemma~\ref{LemmaInjOverEdgeBound}).

\smallskip
\textbf{Step 2: One natural edge over all natural edges (Section~\ref{SectionOneNatOverAll}).} In this step, we strengthen Step 1 by showing that if $d_\FS(G,H)$ has a certain lower bound then there exists a natural edge $E \subset G$ such that its image $F(E) \subset H$ crosses every \emph{natural} edge of~$H$. By applying Step 1 we may assume that $f(E)$ crosses every edge (of the given subdivision) of $H$. Using that $f(E)$ has \emph{only two} endpoints, one proves that the conclusions of Step 2 fail for \emph{at most one} natural edge $E' \subset H$. One then further observes a very particular pattern that governs how $f(E)$ can overlap itself in the interior of~$E'$. First, the path $f(E)$ starts at an initial point $p$ in the interior of $E'$ and crosses one of the two subpaths into which $p$ subdivides $E'$. Next, the path continues by crossing full natural edges of $H$ distinct from $E'$. Finally, the path ends at a terminal point $q$ in the interior of $E'$, having just crossed one of the two subpaths into which $q$ subdivides $E'$. One then figures out how to slightly increment the distance bound from Step 1 in order to avoid this pattern, thus obtaining a distance bound that works for Step 2.

\subparagraph{Step 3: Another natural edge over all natural edges (Section~\ref{SectionTwoOverAllStepThree}).} Knowing that there is a natural edge $E^* \subset G$ such that $f(E^*)$ crosses every natural edge of $H$, and assuming that this property characterizes $E^*$ uniquely, one must deduce a bound on $d_\FS(G,H)$. 

For this problem we were inspired by \cite[Corollary 3.2.2]{\BookOneTag} which is concerned with a self-homotopy equivalence $F : G \to G$. The essence of that corollary if $E \subset G$ is a natural edge, if $F(E)$ crosses $E$, and if no other natural edge of $G$ has an $F$ image that crosses $E$, then $F(E)$ crosses $E$ exactly once. The proof uses a Stallings fold argument.

The trick that gets our proof going is to assume that $d_\FS(G,H)$ is \emph{at least as large as} the bound in Step 2, and then to use that assumption to obtain a foldable factorization of the form
$$F : G = G_0 \xrightarrow{F_1} G_1 \xrightarrow{F_2} G_2 = H
$$
such that $d_\FS(G_1,G_2)$ is \emph{exactly equal to} the bound in Step 2. Thus there exists a natural edge $E^\sharp \subset G_1$ such that $F_2(E^\sharp)$ crosses \emph{every} natural edge of $G_2=H$. But then, from the uniqueness of $E^*$ regarding its image under $F(E^*)$, we deduce a sharper uniqueness property regarding its image $F_1(E^*)$, saying that $E^*$ is the \emph{unique} natural edge of $G=G_0$ such that $F_1(E^*)$ crosses \emph{one specific} natural edge of the marked graph $G_1$, namely $E^\sharp$. And now --- despite that $F_1$ is not a \emph{self}-homotopy equivalence --- the situation is similar enough to \cite[Corollary 3.2.2]{\BookOneTag} that, inspired by its proof, we found a Stallings fold argument with a similar conclusion that $F_1(E^*)$ crosses $E^\sharp$ exactly once. From this we can deduce a bound on $d_\FS(G_0,G_1)$ and thus, adding to it the bound already found for $d_\FS(G_1,G_2)$, we obtain a bound on $d_\FS(G,H)$.

\subsection{Reducing to the uniterated case}
\label{SectionUniteratedReduction}

The general \TOAT\ can be thought of as an iterated version of the following special case in which we set the parameter $n$ to equal $1$. And in fact the general case follows quickly from this special case, as we shall show.

\begin{uniterated}
For any group $\Gamma$ and any free factor system~$\A$ of $\Gamma$ there exists an integer constant $\Delta=\Delta(\Gamma;\A) > 0$ such that for any free splittings $S,T$ of $\Gamma$ rel~$\A$, and any foldable map $f \from S \to T$, if $d_\FS(S,T) \ge \Delta$ then there exist two natural edges $E_1,E_2 \subset S$ in different $\Gamma$-orbits such that for each natural edge $E' \subset T$, each of the paths $f(E_1)$, $f(E_2)$ crosses some edge in the orbit $\Gamma \cdot E'$.
\end{uniterated}

\begin{proof}[Proof of the iterated version \TOAT, assuming the uniterated version] \quad\break
Using the same integer constant $\Delta \ge 1$ as in the uniterated version, we proceed by induction on $n \ge 1$. For the base case, given a fold map $f \from S \to T$ between two free splittings $S,T$ of $\Gamma$ rel~$\A$ such that $d_\FS(S,T) \ge \Delta$, the uniterated version gives us two natural edges $E_0,E_1 \subset S$ in distinct orbits such that each path $f(E_i)$ fully crosses $T$; this proves the conclusion of the iterated version for the case $n=1$.

Assuming by induction that $n \ge 2$ and that the iterated version holds for $n-1$, consider $f \from S \to T$ such that $d_\FS(S,T) \ge n\Delta$. Apply the Stallings Fold Theorem \ref{TheoremStallings} to obtain a fold factorization
$$f \from S=T_0 \mapsto\cdots\mapsto T_K=T
$$
such that that $d(T_{k-1},T_k) \le 1$ for all $k$.

Let $k \in \{1,\ldots,K-1\}$ be the first index such that $d_\FS(T_0,T_k) \ge \Delta$. Since $d(T_0,T_{k-1}) < \Delta$ and $d(T_{k-1},T_k) \le 1$ it follows that $d_\FS(T_0,T_k)=\Delta$, and by the triangle inequality we then have $d_\FS(T_k,T_K) \ge (n-1)\Delta$. 
%
%
%
%
By induction there exist two natural edges $E'_1,E'_2 \subset T_k$ in distinct $\Gamma$-orbits such that each of the two paths $f^k_K(E'_1), f^k_K(E'_2) \subset T_K$ has $2^{n-2}$ nonoverlapping subpaths, each of which fully crosses $T_K$. 

Since $d(T_0,T_k) \ge \Delta$ we can apply the uniterated version of the \TOAT, to obtain two natural edges $E_1,E_2 \subset T_0=S$ in distinct $\Gamma$-orbits such that for $i=1,2$ the path $f^0_k(E_i)$ fully crosses $T_k$; in particular, $f^0_k(E_i)$ crosses translates of $E'_1$ and $E'_2$, which are of course nonoverlapping subpaths of $f^0_k(E_i)$. In the path $f(E_i) = f^K_k \circ f^k_0(E_i)$ we therefore obtain $2$ nonoverlapping subpaths, namely translates of $f^k_K(E'_1)$ and $f^k_K(E'_2)$, each of which has $2^{n-2}$ nonoverlapping fully crossing subpaths, and so in $f(E_i)$ we obtain $2 \times 2^{n-2}=2^{n-1}$ nonoverlapping fully crossing subpaths.
%
%
\end{proof}


\subsection{Some distance bounds}
\label{SectionDistanceBounds}
In this section we collect some distance bounds that will be used in proving the uniterated case of the \TOAT, starting with the following basic bounds:
\begin{description}
\item[Free factor chain bound (\protect{Lemma~\ref{LemmaKrank}~\pref{ItemKRankLength}}):] The length $L$ of any strictly nested chain of free factors $F_0 < F_1 < \cdots < F_L$ of $\Gamma$ rel~$\A$ is bounded by
$$L \le \Krank(\Gamma) = \abs{\A} + \corank(\A)
$$
\item[Natural edge orbit bound (\protect{\cite[Proposition 3.4]{HandelMosher:RelComplexHyp}}):] For any free splitting $T$ of $\Gamma$ rel~$\A$, the number of natural edge orbits of $T$ is at most
$$\MaxEdges(\Gamma;\A) = 2 \abs{\A} + 3 \corank(\A) - 3
$$
\end{description}

\paragraph{Bounding diameters of fold paths.} The freedom of choice in the Stallings fold theorem (see Theorem~\ref{TheoremStallings}~\pref{ItemArbitraryFolds}) makes clear that fold factorizations are not unique. While in some applications this nonuniqueness is an undesirable feature to be avoided (see \cite[Section~4.1]{\RelFSOneTag}), here we take it as a feature to be exploited, giving us a flexibility one can use to obtain upper bounds on distances in $\FS(\Gamma;\A)$, as was done for instance in \cite{BestvinaFeighn:subfactor,\RelFSOneTag,\FSTwoTag}. In this section we review and formalize several such techniques. 

To fix some terminology and notation, when referring to diameter in $\FS(\Gamma;\A)$ we will regard ``foldable sequences'' in general and ``foldable paths'' in particular as subsets of the $0$-skeleton of $\FS(\Gamma;\A)$. So, for example,
$$\diam(T_0 \mapsto T_1 \mapsto\cdots\mapsto T_N) = \diam\{T_0,T_1,\ldots,T_n\}
$$

The following diameter bound is a straightforward generalization of the similar bound found in \cite[Lemma 5.5]{\FSHypTag}; see the proof of that lemma up to the top of \cite[Page 1630]{\FSHypTag}). Given a function $f \from X \to Y$ and a subset $B \subset Y$, to say that \emph{$f$ injective over $B$} means that $f$ is injective on $f^\inv(B)$, i.e.\ for any $p,q \in X$, if $f(p)=f(q) \in B$ then $p=q$.

\begin{lemma}
\label{LemmaInjOverEdgeBound}
For any foldable map $f \from S \to T$ which is a simplicial map, and for any edge $e' \in \ET$, if $f$ is injective over the interior $e'$ then, letting $e \in \ES$ be the unique edge of $S$ such that $f(e)=e'$, the map $f$ induces an equivariant homeomorphism $S / \Gamma \cdot e \approx T / \Gamma \cdot e'$. It follows that $d_\FS(S,T) \le 2$. \qed
\end{lemma}

\smallskip
The following lemma formalizes a method used for example in \cite[Lemma 4.2]{\FSHypTag} to verify distance bounds. See Lemma~\ref{PropPriorityBound} for the bounds we attain by using this method in conjunction with repeated application of~Lemma~\ref{LemmaInjOverEdgeBound}.


\begin{definition}[Prioritizing folds]
\label{DefPriority}
Consider a foldable map $f \from S \to T$, a nondegenerate subgraph $Z \subset S$, and a partial fold factorization
$$f \from S=T_0 \mapsto\cdots\mapsto T_J \mapsto T
$$
We denote $Z_j = f^0_j(Z)$ ($j=0,\ldots,J$). To say that this partial factorization \emph{prioritizes folding $Z$} means the following:
\begin{enumerate}
\item\label{ItemPriorityFolds}
For each $1 \le j \le J$ there exists a foldable turn $\{d,d'\}$ in $Z_{j-1} \subset T_{j-1}$ with respect to the map $T_{j-1} \mapsto T$, and initial segments $\eta,\eta'$ of the oriented natural edges of $T_{j-1}$ that represent $d,d'$ (resp.), such that $\eta,\eta' \subset Z_{j-1}$ and such that $f_j \from T_{j-1} \to T_{j}$ folds the segments $\eta,\eta'$. 
\item\label{ItemPriorityDone}
Every turn in $Z_J$ is nonfoldable with respect to the map $T_J \mapsto T$. 
\end{enumerate}
These defining properties \pref{ItemPriorityFolds}, \pref{ItemPriorityDone} of ``prioritization'' imply the following additional properties:
\begin{enumeratecontinue}
\item\label{ItemInjOnComps} The map $T_J \mapsto T$ restricts to an embedding of each component of $Z_J$.
\item For each $1 \le j \le J$, the map $f_j \from T_{j-1} \to T_j$ satisfies the following:
\begin{enumerate}
\item\label{ItemPriorityInjOverEdge} for each edge $e \subset T_j \setminus Z_j$, the map $f_j$ is injective over the interior of $e$.
\item\label{ItemCompBij}
The map $f_j$ induces a $\Gamma$-equivariant bijection from the set of components of $Z_{j-1}$ to the set of components of $Z_j$.
\end{enumerate}
\item\label{ItemStayApart}
For each $1 \le j \le J$, the following hold in $T_j$:
\begin{enumerate}
\item The two nondegenerate subgraphs $Z_j=f^0_j(Z)$ and $f^0_j(T_0 \setminus Z)$ have no edgelets in common.
\item\label{ItemNoneOutIDWithAnyIn}
No edgelet of $T_0 \setminus Z$ is identified with any other edgelet of $T_0$ by $f^0_j$.
\end{enumerate}
\end{enumeratecontinue}
Property~\pref{ItemInjOnComps} follows from Lemma~\ref{LemmaTreeMapInjective}. Property~\pref{ItemPriorityInjOverEdge} is an evident consequence of the definition of folds combined with Property~\pref{ItemPriorityFolds}. To prove item~\pref{ItemCompBij}, using that $f_j \from Z_{j-1} \to Z_j$ is a quotient map, it suffices to note that for all $x,x' \in Z_{j-1}$, if $f_j(x)=f_j(x')$ then $x,x' \in g \cdot (\eta\union\eta')$ for some $g \in \Gamma$, but $\eta\union\eta'$ is connected and so $x,x'$ are contained in the same component of $Z_{j-1}$. Property~\pref{ItemStayApart} follows by induction on $j$, using property~\pref{ItemPriorityFolds}.

This completes Definition~\ref{DefPriority} and the verification of its extra properties~\pref{ItemPriorityInjOverEdge} and~\pref{ItemCompBij}.
\end{definition}


\begin{proposition}
\label{PropPriorityExists}
For any foldable map $f \from S \to T$ and any nondegenerate subgraph $Z \subset S$, there exists a partial fold factorization of $f$ that prioritizes $Z$.
\end{proposition}

\begin{proof} Apply Theorem~\ref{TheoremStallings} to construct a fold factorization of $f$ by maximal folds, always choosing a foldable turn $d,d'$ in $Z_j$ represented by oriented natural edges $E,E'$, and always choosing to fold initial segments $\eta \subset E$, $\eta' \subset E'$ that are maximal subject to the requirement that $\eta,\eta' \subset Z_j$ and that $\eta,\eta'$ have the same image in $T$. Continue until the first moment $j=J$ that no foldable turn exists in~$Z_J$. 
\end{proof}


\begin{proposition}
\label{PropPriorityBound}
Consider a foldable map $f \from S \to T$ and a partial fold factorization 
$$f \from S=T_0 \mapsto\cdots\mapsto T_J \mapsto T
$$
If there exists a proper, nondegenerate subgraph $Z \subset S$ that is prioritized by this factorization then
$$\diam(T_0 \mapsto\cdots\mapsto T_J) \le 2
$$
\end{proposition}

\begin{proof}  Subdivide each $T_j$ at the inverse image of the vertex set of $T_J$. Adopting the notation of Definition~\ref{DefPriority}, applying properness of $Z=Z_0 \subset T_0$ as a basis step, and applying Definition~\ref{DefPriority}~\pref{ItemPriorityFolds} by induction, it follows that $Z_j \subset T_j$ is proper for $j=0,\ldots,J$. Choosing an arbitrary edge $e_J \subset T_J \setminus Z_J$, and applying Definition~\ref{DefPriority}~\pref{ItemPriorityInjOverEdge} by backwards induction, we obtain a sequence of edges $e_j \subset T_j \setminus Z_j$ for $0 \le j \le J$, such that if $j \ge 1$ then $e_{j-1}$ is the unique edge of $T_{j-1}$ mapping to $e_j$. Repeated application of Lemma~\ref{LemmaInjOverEdgeBound} produces a sequence of equivariant homeomorphisms
$$U \equiv T_0 / \Gamma \cdot e_0 \approx T_1 / \Gamma \cdot e_1 \approx\cdots\approx T_J / \Gamma \cdot e_J
$$
It follow that $d(T_j,U) \le 1$ for each $j$, hence $d(T_i,T_j) \le 2$ for all $i,j = 0,\ldots,J$.
\end{proof}

\begin{proposition}[c.f. \protect{\cite[Lemma 4.1]{BestvinaFeighn:subfactor}}]
\label{PropsEdgeInverseDiamBound}
Consider a foldable map $f \from S \to T$, and subdivide $S$ at $f^\inv(\VT)$ so that $f$ is a simplicial map. For any $e \in \Edges(T)$, letting $f^*(e) \subset S$ denote the subgraph of all edges of $S$ mapped to~$e$, and letting $\abs{f^*(e)}$ denote the number of such edges, the diameter in $\FS(\Gamma;\A)$ of any fold factorization of $f$ is bounded by $4 \abs{f^*(e)}$.
\end{proposition}

\begin{proof} Let $S_0 = S$ and let $g_0=f \from S_0 \to T$, so in this notation we have $f^*(e)=g_0^*(e)$. Also let $Z_0 = S_0 \setminus f^*(e)$. Define a partial fold factorization
$$S_0 \xrightarrow{f_1} S_1 \xrightarrow{g_1} T
$$
as follows: if there exist two oriented edges of $g_0^*(e)$ that form a foldable turn with respect to $g_0$, let $f_1$ fold that turn; otherwise let $f_1 \from S_0 \to S_1=S_0$ be the identity map (here we are abusing the concept of ``partial fold factorization'' by allowing identity maps). Thus we have $d(S_0,S_1) \le 2$. Now let $Z_1 = S_1 \setminus \Gamma \cdot \left((g_1)^*(e)\right)$. Extend the partial fold factorization by prioritizing folding $Z_1$ to obtain
$$S_{0} \mapsto S_{1} \mapsto\cdots\mapsto S_{i(1)} \xrightarrow{g_{i(1)}} T
$$
If $g_{i(1)}$ is a homeomorphism then the fold path is complete. Otherwise, let $Z_{i(1)} \subset S_{i(1)}$ be the image of $Z_1$, and note that $Z_{i(1)} = S_{i(1)} \setminus  \Gamma \cdot (g_{i(1)})^*(e)$. By Definition~\ref{DefPriority}~\pref{ItemInjOnComps} the map $g_{i(1)}$ restricts to an injection on $Z_{i(1)}$. For any pair of oriented edges in $S_{i(1)}$ that form a foldable turn with respect to $g_{i(1)}$, at least one of those edges is \emph{not} in $Z_{i(1)}$ and so is in $(g_{i(1)})^*(e)$; but then both of those edges map to $e$ and so both are in $(g_{i(1)})^\inv(e)$. Choose such a pair and fold them, to obtain the next fold map $f_{i(1)+1}$ and extending the partial fold factorization by one more term:
$$S_{0} \mapsto S_{1} \mapsto\cdots\mapsto S_{i(1)} \mapsto S_{i(1)+1} \xrightarrow{g_{i(1)+1}} T
$$
If $g_{i(1)+1}$ is a homeomorphism then the fold path is complete. Otherwise, let $Z_{i(1)+1} = S_{i(1)+1} \setminus \Gamma \cdot (g_{i(1)+1})^*(e)$ and extend the partial fold factorization by prioritizing $Z_{i(1)+1}$ obtaining
$$S_{0} \mapsto S_{1} \mapsto\cdots\mapsto S_{i(1)} \mapsto S_{i(1)+1} \mapsto \cdots \mapsto S_{i(2)} \xrightarrow{g_{i(2)}} T
$$
Note that the number of edges in the inverse image of $e$ along this fold path decreases strictly from $S_0$ to $S_{i(2)}$: that number does not increase from $S_0$ to $S_1$; it is constant from $S_1$ to $S_{i(1)}$; it decreases by $1$ from $S_{i(1)}$ to $S_{i(1)+1}$; and it is constant from $S_{i(1)+1}$ to $S_{i(2)}$. In summary:
$$\abs{g_{i(2)}^*(e)} < \abs{g_{i(1)}^*(e)} \le \abs{g_0^*(e)}
$$
We may continue extending the partial fold factorization by induction, alternating between prioritization of folds in the complement of the inverse image of $e$, and doing a single fold of two edges in the inverse image of $e$, until reaching a homeomorphism to $T$ and thus ending with a fold factorization of $f \from S \to T$ having the following form:
\begin{align*}
S_{0} \mapsto S_{1} &\mapsto\cdots\mapsto S_{i(1)} \mapsto S_{i(1)+1} \mapsto \cdots \mapsto S_{i(2)} \mapsto S_{i(2)+1} \mapsto\cdots\cdots \\ 
&\cdots\cdots \mapsto S_{i(N-1)} \mapsto S_{i(N-1)+1} \mapsto\cdots\mapsto S_{i(N)} \xrightarrow{g_{i(N)}} T
\end{align*}
where $g_{i(N)}$ is a homeomorphism (it is possible that any of the ``prioritization'' fold paths $S_{i(n-1)+1} \mapsto\cdots\mapsto S_{i(n)}$ can collapse to a homeomorphism, but this is inconsequential). We have
$$1 = \abs{g_{i(N)}^*(e)} < \abs{g_{i(N-1)}^*(e)} < \cdots < \abs{g_{i(2)}^*(e)} < \abs{g_{i(1)}^*(e)} \le \abs{g_0^*(e)} = \abs{f^*(e)}
$$
and so $N \le \abs{f^*(e)}$. Denoting $i(0)=0$, we have diameter bounds
$$\diam\{S_{i(n-1)} \mapsto\cdots\mapsto S_{i(n)}\} \le 2, \quad \diam\{S_{i(n)} \mapsto S_{i(n+1)}\} \le 2
$$
and so the diameter of the whole fold sequence is bounded above by $4N \le 4 \abs{f^*(E)}$.
\end{proof}

\subsection{Outline of the uniterated case in 3 steps}
\label{SectionTwoOverAllOutline}

In this section we formalize and generalize the outline used earlier for $\FS(F_n)$, obtaining a three step outline of the proof of the uniterated case of the \TOAT\ for $\FS(\Gamma;\A)$, to be carried out in the three subsections to follow. Fix free splittings $S,T$ of $\Gamma$ rel~$\A$ and a foldable map $f \from S \to T$. Assuming a lower distance bound, the conclusion we must prove is that there are two different natural edges of $S$ each of whose images in $T$ crosses some natural edge in every orbit of natural edges of $T$. The outline of the proof starts with a weak version of this conclusion in Step 1, which is successively strengthened to get the full conclusion in Step 3. In each Step $N$, there will be an integer constant $\Delta_N = \Delta_N(\Gamma;\A)$.

\begin{description}
\item[Step 1: One natural edge over all edges.] There is an integer $\Delta_1 = \Delta_1(\Gamma;\A) \ge 1$ such that if $d(S,T) \ge \Delta_1$ then there is a natural edge $E \subset S$ such that $f(E)$ crosses some edge in the orbit of every edge of $T$.
\item[Step 2: One natural edge over all natural edges.] There is an integer $\Delta_2 = \Delta_2(\Gamma;\A)$ such that $\Delta_2 \ge \Delta_1$ and such that if $d(S,T) \ge \Delta_2$ then there is a natural edge $E \subset S$ such that $f(E)$ crosses some natural edge in the orbit of every natural edges of $T$.
\item[Step 3: Another natural edge over all natural edges.] There is an integer $\Delta_3 = \Delta_3(\Gamma;\A) \ge \Delta_2$ 
such that if $d(S,T) \ge \Delta_3$ then there are natural edges $E_1,E_2 \subset S$ in different orbits such that for each natural edge $E \subset T$, each of the paths $f(E_1)$, $f(E_2)$ crosses some natural edge in the orbit of $E$.
\end{description}
The final constant needed to prove the \TOAT\ is then $\Delta=\Delta_3$.

\subsection{Step 1: One natural edge over all edges}
\label{SectionOneNatOverAll}

We fix $f \from S \to T$ to be a foldable map of free splittings of $\Gamma$ rel~$\A$ such that for each natural edge $E \subset S$, its image $f(E)$ does \emph{not} crosses every edge orbit in $T$. Using that assumption we shall derive an upper bound to the distance $d(S,T)$ in $\FS(\Gamma;\A)$; by adding $1$ to that upper bound we obtain an expression for our desired constant $\Delta_1$, completing the proof of Step~1. This expression will also involve other integer constants with notational format $\Delta_{1.(k)} = \Delta_{1.(k)}(\Gamma;\A)$ where ``$k$'' enumerates various steps of the proof, for example the \emph{Jumping Bound} $\Delta_{1.\pref{ItemFoldEstimate}}$ of Step~\pref{ItemFoldEstimate}.

\begin{definition}[The covering forest of a path]
\label{DefinitionCoveringForest}
For each free splitting $U$ of $\Gamma$ rel~$\A$ and each path $\alpha \subset U$, the \emph{covering forest} of $\alpha$ in $U$ is the following $\Gamma$-invariant subforest of $U$:
$$\beta(\alpha)=\beta(\alpha;U) = \bigcup_{g \in \Gamma} g \cdot \alpha
$$
Note that the group $\Gamma$ acts transitively on the set of components of $\beta(\alpha;U)$, because $\Gamma$ acts transitively on the set of translates $\{g \cdot \alpha \mid g \in \Gamma\}$ of $\alpha$. Note also that $\alpha$ crosses the $\Gamma$-orbit of every edge of $U$ if and only if $\beta(\alpha;U)=U$. Associated to the covering forest $\beta(\alpha)$ is the free factor system $\F[\beta(\alpha)]$ (Definition~\ref{DefinitionFFS}).
\end{definition}

For each natural edge $E \subset S$ consider its covering forest $\beta_E = \beta(E;S)$. The subforests $\beta_E$ are finite in number, one per natural edge orbit of~$S$, and we collect them into a finite set denoted $B_0$: letting $\{E_i\}_{i = 1}^{I}$ be a bijectively indexed set of representatives of the natural edge orbits of~$S$, and denoting $\beta_{0,i} = \beta_{E_i}$, we have
\begin{align*}
B_0 &= \{\beta_E \suchthat E \in \E_\nat(S)\} \\
        &= \{\beta_{0,i}\}_{i = 1}^{I} \\
\abs{B_0} &= I   \\
                  &\le \MaxEdges(\Gamma;\A) \, = \, 2\abs{\A} + 3 \corank(\A) - 3
\end{align*}

We record the following properties:
\begin{enumerate}
\item The set of subforests $B_0$ covers the tree $S$.
\item\label{ItemCompActTrans}
For each $i \in I$ the group $\Gamma$ acts transitively on the set of components of~$\beta_{0,i}$.
\end{enumerate} 
In the course of the proof we will inductively construct partial fold factorizations of the map $f \from S \to T$ of the form
$$\xymatrix{
(*)_J \qquad S=S_0 \ar[r]_-{f_1}  \ar@/^2pc/[rrr]^{f^0_J} & S_1 \ar[r]_{f_2} & \cdots \ar[r]_{f_J} & S_J \ar[r]_{h_J} & T 
}$$
When the fold factorization $(*)_J$ has been specified, for each $j=0,\ldots,J$ and each $i \in I$ we define the $\Gamma$-invariant forest $\beta_{j,i} \subset S_j$, where $\beta_{0,i}$ is as defined above and $\beta_{j,i} = f^0_j(\beta_{0,i}) = f_j(\beta_{j-1,i})$. We collect these subforests into a finite set $B_j = \{\beta_{j,i}\}_{i = 1}^{I}$ of $\Gamma$-invariant subforests of~$S_j$. These subforests satisfy the following properties:
\begin{enumeratecontinue}
\item
For each $j=0,\ldots,J$ and each $i =1,\ldots, I$ we have 
\begin{enumerate}
\item\label{ItemEdgeImageOrbit}
$\beta_{j,i} = \beta(f^0_j(E_i);S_j)$ is the covering forest of the path $f^0_j(E_i)$ in $S_j$.
\item\label{ItemBetaDoesntCover}
$\beta_{j,i} \ne S_j$
\end{enumerate}
\end{enumeratecontinue}
Item~\pref{ItemBetaDoesntCover} holds for $j=J$ because otherwise, by item~\pref{ItemEdgeImageOrbit}, we would have 
$$\beta(f(E_i);T) = \beta\bigl(h_J(f^0_J(E_i));T\bigr) = h_J\bigl(\beta(f^0_J(E_i));S_J\bigr) = h_J(\beta_{J,i}) = h_J(S_J) = T
$$
which contradicts the assumption that $f(E_i)$ does not cross every edge orbit of $T$. Item~\pref{ItemBetaDoesntCover} then holds more generally for all $j = 0,\ldots,J$ because $\beta_{j,i}=S_j$ implies $\beta_{J,i}=S_J=T$. 

We remark that while the indexing of $B_0$ is bijective, in general for $1 \le j \le J$ the indexing $B_j = \{\beta_{j,i}\}_{i \in I}$ need not be bijective: it is possible, for example, that although two natural edges $E_i,E_{i'} \subset S$ are in different orbits, their image paths $f_j(E_i)$, $f_j(E_{i'})$ cross the exact same edge orbits in $S_j$ in which case we obtain identical covering forests $\beta_{j,i}=\beta_{j,i'}$ in $S_j$. Nonetheless the equation $f^0_j(\beta_{0,i})=\beta_{j,i}$, together with $\Gamma$-equivariance of $f^0_j$, implies the same upper bound for the cardinality of $B_j$:
\begin{enumeratecontinue}
\item\label{ItemBjBound}
$\abs{B_j} \le \abs{B_0} = I \le \MaxEdges(\Gamma;\A)$
\end{enumeratecontinue} 
Further properties of the set of subforests $B_j$ include:
\begin{enumeratecontinue}
\item\label{ItemForestsCover}
The set of subforests $B_j$ covers the tree $S_j$.
\item\label{ItemBetaSurjection} The map $f_j \from S_{j-1} \to S_j$ induces a surjection $B_{j-1} \to B_j$ given by $f_j(\beta_{j-1,i}) = \beta_{j,i}$ for each $i \in I$.
\item\label{ItemTransOnB}
For each $j=0,\ldots,J$ and each $i = 1,\ldots, I$ the group $\Gamma$ acts transitively on the set of components of $\beta_{j,i}$. 
\end{enumeratecontinue}
Let $b_{j,i} \subset \beta_{j,i}$ be the component containing $f^0_j(E_i)$, so $\Stab(\beta_{j,i})$ is a free factor rel~$\A$, possibly trivial or atomic. We denote 
\begin{align*}
A_{j,i} &= \Stab(b_{j,i})
\end{align*}
Applying item~\pref{ItemBetaSurjection}, for each $i$, letting $j$ vary we obtain a nested sequence of free factors rel~$\A$:
$$A_{0,i} \le A_{1,i} \le \cdots \le A_{J,i}
$$
Applying Lemma~\ref{LemmaKrank}~\pref{ItemKRankIneqStat} we obtain for each $i=1,\ldots,I$ a nondecreasing sequence of relative Kurosh ranks denoted $\KR_{j,i} = \KR(A_{j,i})$:
\begin{enumeratecontinue}
\item\label{ItemRelKuroshSequence}
$0 \, \le \, \KR_{0,i} \le \KR_{1,i} \le \cdots\le \KR_{J,i} \, \le \, \abs{\A} + \corank(\A)$.
\end{enumeratecontinue} 
Furthermore,
\begin{enumeratecontinue}
\item\label{ItemNoJump}
For each $i = 1,\ldots, I$ and each $j=1,\ldots,J$ the following are equivalent: 
\begin{enumerate}
\item\label{ItemNoJumpRank}
$\KR_{j-1,i}=\KR_{j,i}$ 
\item\label{ItemNoJumpSubgroup}
 $A_{j-1,i}=A_{j,i}$
\item\label{ItemNoJumpComponents}
The restricted surjection $f_j \from \beta_{j-1,i} \to \beta_{j,i}$ induces a bijection of component sets.
\end{enumerate}
\end{enumeratecontinue}
The equivalence \pref{ItemNoJumpRank}$\iff$\pref{ItemNoJumpSubgroup} is an application of Lemma~\ref{LemmaKrank}~\pref{ItemKRankIneqEq}, and the implication \pref{ItemNoJumpSubgroup}$\impliedby$\pref{ItemNoJumpComponents}, follows from $\Gamma$-equivariance. The converse implication \pref{ItemNoJumpSubgroup}$\implies$\pref{ItemNoJumpComponents} follows by observing that for given components $b \ne b' \subset \beta_{j-1,i}$, applying \pref{ItemTransOnB} to choose $g \in \Gamma$ such that $g \cdot b = b'$, if $f_j(b)=f_j(b')$ then $g \in \Stab(f(b)) - \Stab(f)$. 

To say that \emph{$\KR_{j,i}$ jumps} means that a strict inequality holds, namely $\KR_{j-1,i} < \KR_{j,i}$. As an immediate consequence of~\pref{ItemNoJump}, jumping of $K_{j,i}$ can be detected topologically, for each $1 \le i \le I$ and $1 \le j \le J$:
\begin{enumeratecontinue}
\item\label{ItemDetectJumping}
$\KR_{j,i}$ jumps if and only if the restricted surjection $f \from \beta_{j-1,i} \to \beta_{j,i}$ does \emph{not} induce a bijection of component sets.
\end{enumeratecontinue} 
\textbf{Remark:} As a consequence of~\pref{ItemDetectJumping}, what happens when $\KR_{j,i}$ jumps is that in a pair of edges $e_1,e_2 \subset S_{j-1,i}$ that are folded by $f_j$, there exist points $x_1 \in e_1$, $x_2 \in e_2$ in distinct components of $\beta_{j-1,i}$ that are identified to a single point by $f_j$. Typically we may choose $x_1,x_2$ so that one of them is an endpoint of the arc $e_1 \union e_2$ and that one is the unique point on $e_1 \union e_2$ that lies in the corresponding component of $\beta_{j-1,i}$.

\smallskip

Knowing from~\pref{ItemBjBound} that $\abs{B_j} \le \MaxEdges(\Gamma;\A)$, and knowing from \pref{ItemRelKuroshSequence} that for each $i =1,\ldots, I$ there are at most $\abs{\A} + \corank(\A)$ values of $j = 1,\ldots,J$ for which $\KR_{j,i}$ jumps, we obtain:
\begin{enumeratecontinue}
\item\label{ItemFoldEstimate} \textbf{Jumping bound:} In any partial fold factorization $(*)_J$ of $f \from S \to T$, the number ordered pairs $(j,i) \in \{1,\ldots,J\} \times \{1,\ldots,I\}$ such that $\KR_{j,i}$ jumps is bounded above by 
\begin{align*}
\Delta_{1.\pref{ItemFoldEstimate}} &= \MaxEdges(\Gamma;\A) \cdot (\abs{\A} + \corank(\A)) = (2\abs{\A} + 3 \corank(\A) - 3) \cdot (\abs{\A} + \corank(\A)) \\
&= 2 \abs{\A}^2 - 3 \abs{\A} + 3 \corank(\A)^2 - 3 \corank(\A) 
  + 5 \abs{\A} \cdot \corank(\A)
\end{align*}
\end{enumeratecontinue}

\medskip

We now set up the induction step for construction of partial fold factorizations of~$f$, starting from the factorization $(*)_0$ which is just $S=S_0 \xrightarrow{f=h_0} T$. Suppose that the partial fold factorization $(*)_J$ of $f$ has already been constructed. If $h_J$ is not already homeomorphism, and so there is at least one fold still to do, we shall describe two types of fold processes to extend the partial fold factorization $(*)_J$ of $f$ by inserting a partial fold factorization $h_J$, thereby increasing the length $J$. We shall prove that each of these fold processes has a uniformly bounded finite diameter in the free splitting complex, and we shall find a uniform finite bound to the overall number of processes that need to be applied in order to finally reach~$T$. The~outcome of each fold process will be denoted as follows, using $M$ for \hbox{``the new $J$'',} i.e.~for the length of the extended partial fold factorization:
 $$\xymatrix{
(*)_M \qquad S=S_0 \ar[r]^-{f_1} \ar@/_2pc/[rrr]^{f^0_J}  \ar@/^3pc/[rrrrr]^{f^0_M}  & S_1 \ar[r]^{f_2} &\cdots \ar[r]^{f_J} 
 & S_J \ar[r]^{f_{J+1}} \ar@/_2pc/[rrr]^{h_J} &\cdots \ar[r]^{f_M} & S_M \ar[r]^{h_M} & T
}
$$
The two fold processes we shall describe are carried out under complementary hypotheses, so one of the two processes will always be applicable, as we now explain. For each $i \in I$, it follows from~\pref{ItemTransOnB} that $h_J$ is injective on \emph{some} component of $\beta_{J,i}$ if and only if $h_J$ is injective on \emph{every} component of $\beta_{J,i}$, and if this holds then we say that \emph{$h_J$ is injective on each component of $\beta_{J,i}$.} One carries out Process \#1 if for all $i \in I$ the map $h_J$ is injective on each component of $\beta_{J,i}$, and Process \#2 if there exists $i \in I$ such that $h_J$ is not injective on each component of $\beta_{J,i}$. 

\begin{description}
\item[Process \#1: Fold to increase rank.] Assuming that for each $i \in I$ the map $h_J$ is injective on each component of $\beta_{J,i}$, fold arbitrarily until the first moment $M > J$ that some rank $\KR_{M,i}$ jumps or the map $h_M$ is a homeomorphism. To be precise, extend $(*)_J$ to $(*)_M$ so that $M$ is the minimum value such that $\KR_{j,i}$ does not jump for $J < j < M$ and all $i \in I$ and so that one of two possible outcomes occurs:
\begin{description}
\item[Folding to increase rank is successful:] $\KR_{M,i}$ jumps for some $i \in I$; or
\item[Folding to increase rank is unsuccessful:] $\KR_{M,i}$ does not jump for each $i \in I$, and $h_M$ is a homeomorphism.
\end{description}
It follows that $\KR_{j,i}$ does not jump for all $J < j < M$ and all $i =1,\ldots, I$.
\end{description}
\noindent
Under Process \#1, the diameter of $S_J \mapsto\cdots\mapsto S_M$ is uniformly bounded, for the following reasons. For all $J<j<M$ and $i =1,\ldots, I$, since $\KR_{j,i}$ does not jump it follows from \pref{ItemDetectJumping} that the restricted map $f_j \from \beta_{j-1,i} \to \beta_{j,i}$ induces a bijection of component sets. Also, the hypothesis of Process \#1 implies that $f_j$ is injective on each component of $\beta_{j-1,i}$. Combining these, each restricted map $f_j \from \beta_{j-1,i} \to \beta_{j,i}$ is a homeomorphism. For each $i=1,\ldots,I$ the composed map $f^J_{M-1} = f_{M-1} \circ \cdots \circ f_J \from \beta_{J,i} \to \beta_{M-1,i}$ is therefore a homeomorphism. For each edge $e \subset S_{M-1}$, it follows that the entire subforest $\beta_{J,i}$ has at most one edge whose image under $f^J_{M-1}$ crosses~$e$. By item~\pref{ItemForestsCover} the forests $\beta_{J,1},\ldots,\beta_{J,I}$ cover $S_J$, and so the total number of edges of $S_J$ whose $f^J_{M-1}$ image crosses $e$ is bounded above by $I \le \MaxEdges(\Gamma;\A)$. Applying Proposition~\ref{PropsEdgeInverseDiamBound}, the diameter in $\FS(\Gamma;\A)$ of the fold sequence $S_J \mapsto\cdots\mapsto S_{M-1}$ is bounded above by $4 \cdot \MaxEdges(\Gamma;\A)$. Adding on one more fold $S_{M-1} \mapsto S_M$, we have proved that 
\begin{enumeratecontinue}
\item\label{ProcessOneBound}
For each ``fold to increase rank'' process, the diameter of the fold sequence $S_J \mapsto \cdots \mapsto S_M$ is bounded by 
$$\Delta_{1.\pref{ProcessOneBound}} = 4 \cdot \MaxEdges(\Gamma;\A) + 2
$$
\end{enumeratecontinue}

\begin{description}
\item[Process \#2: Fold to injection.] Assuming there exists $i=1,\ldots,I$ such that $h_J$ is not injective on each component of $\beta_{J,i}$, choose one such value $i_0$. One may always prioritize folding the subforest $\beta_{J,i_0}$ all the way until it becomes injective, however we will interrupt the process at the moment that some $\KR$ value jumps. To be precise, under process \#2 we prioritize folding $\beta_{J,i_0}$, extending $(*)_J$ to $(*)_M$, where $M$ is the minimum value such that $\KR_{j,i}$ does not jump for $J < j < M$ and all $i=1,\ldots,I$ and such that one of two possible outcomes occurs: 
\begin{description}
\item[Folding to injection is successful:] The map $h_M \from S_M \to T$ is injective on each component of $\beta_{M,i_0}$, and for all $i \in I$, $\KR_{M,i}$ does not jump; or 
\item[Folding to injection is unsuccessful:] There exists $i \in I$ such that $\KR_{M,i}$ jumps (the map $h_M$ may or may not be injective on each component of $\beta_{M,i}$).
\end{description}
\end{description}
Because we have prioritized folding the proper $\Gamma$-invariant subforest $\beta_{J,i_0} \subset S_J$, by applying Proposition~\ref{PropPriorityBound} we obtain:
\begin{enumeratecontinue}
\item\label{ProcessTwoBound}
For each fold to injection process, the diameter of the fold subsequence $S_J \to\cdots\to S_M$ is bounded by
$$\Delta_{1.\pref{ProcessTwoBound}} = 2
$$
\end{enumeratecontinue}

\subparagraph{Remark.} Since each fold map $S_{j-1} \mapsto S_{j}$ of Process \#2 is defined by folding edge pairs of the subforest $\beta_{j-1,i_0} \subset S_{j-1}$, the restricted map $\beta_{j-1,i_0} \to \beta_{j,i_0}$ induces a bijection of components, and hence $\KR_{j,i_0}$ does not jump (see \pref{ItemDetectJumping}); this is true even when $j=M$. It follows that if folding to injection is unsuccessful then the value of $i$ for which $\KR_{M,i}$ jumps satisfies $i \ne i_0$. 

\bigskip

To complete the proof, assuming that $f = h_0 \from S_0 \to T$ is not already a homeomorphism, from Theorem~\ref{TheoremStallings} it follows that iterative application of Processes \#1 and \#2 eventually produces a Stallings fold factorization that ends with a homeomorphism $h_M \from S_M \to T$. The final process will be either an unsuccessful ``Fold to increase rank'' process or a successful ``Fold to injection'' process. 

We now count processes of each type and form the weighted sum of diameter bounds of those processes, to get the desired upper bound for $d(S,T)$.

Each successful ``Fold to increase rank'' process and each unsuccessful ``Fold to injection'' process ends at $R_{M,i}$ jumping (for some $M$ and $i$), and so the total number of such processes is bounded above by the jumping bound $\Delta_{1.\pref{ItemFoldEstimate}}$. Also, there is at most one unsuccessful ``Fold to increase rank'' process because when a single such process occurs we have produced a homeomorphism $h_M \from S_M \to T$ and thus have completed the entire Stallings fold factorization. This proves:
\begin{enumeratecontinue}
\item\label{ItemCountingTwoTypes}
The total number of ``Fold to increase rank'' process plus the number of unsuccessful ``Fold to injection'' processes is bounded above by $1 + \Delta_{1.\pref{ItemFoldEstimate}}$
\end{enumeratecontinue}
Furthermore, each such process has upper diameter bound
$$\max\{\Delta_{1.\pref{ProcessOneBound}}, \Delta_{1.\pref{ProcessTwoBound}}\} = \max\{4 \cdot \MaxEdges(\Gamma;\A) + 2,2\} = 4 \cdot \MaxEdges(\Gamma;\A) = \Delta_{1.\pref{ProcessOneBound}}
$$
Therefore,
\begin{enumeratecontinue}
\item\label{ItemOtherProcessDiamBound}
The sum of the diameters of all ``Fold to increase rank'' processes and all unsuccessful ``Fold to injection'' processes is bounded above by 
$$(1 + \Delta_{1.\pref{ItemFoldEstimate}}) \cdot \Delta_{1.\pref{ProcessOneBound}}
$$
\end{enumeratecontinue}

It remains to bound the number of successful ``Fold to injection'' processes. Consider a \emph{maximally successful sequence of fold to injection processes}, meaning a maximal subsequence of the form
$$S_{J_0} \xrightarrow{f^{J_0}_{J_1}} S_{J_1} \xrightarrow{f^{J_1}_{J_2}} \cdots \xrightarrow{f^{J_{P-1}}_{J_P}} S_{J_P}
$$
where for each $p=1,\ldots,P$ the sequence $S_{J_{p-1}} \mapsto\cdots\mapsto S_{J_p}$ is a successful fold to injection process. By definition $\KR_{j,i'}$ does not jump for all $J_0 < j \le J_P$ and all $i'=1,\ldots,I$. Let $(i_0,\ldots,i_{P-1})$ be the sequence of indices in the set $\{1,\ldots,I\}$ such that the fold sequence $S_{J_p} \mapsto\cdots\mapsto S_{J_{p+1}}$ prioritizes folding $\beta_{J_p,i_p}$. We claim that this sequence is one-to-one: 
\begin{description}
\item[Claim:] For each maximally successful sequence of fold to injection processes as denoted above, if $0 \le p < q \le P-1$ then $i(p) \ne i(q)$. It follows that
$$P \le I \le \MaxEdges(\Gamma;\A)
$$
\end{description}
To prove this claim, note that $K_{j,i_p}$ does not jump for $J_p < j \le J_q$ and so the restricted map $f^{J_p}_{J_q} \from \beta_{J_p,i_p} \to \beta_{J_q,i_p}$ is injective on \emph{the component set} of $\beta_{J_p,i_p}$; but this restricted map is also injective on \emph{each individual component} of $\beta_{J_p,i_p}$, and therefore it is a homeomorphism. It follows that the map $h_{J_q} \from \beta_{J_q,i_p} \to T$ is injective on each component of $\beta_{J_q,i_p}$ proving that $i_p \ne i_q$.

Next we have:
\begin{description}
\item[Claim:] The total number of maximally successful sequences of fold to injection processes is bounded by $1 + \Delta_{1.\pref{ItemFoldEstimate}}$.
\end{description}
This follows because, by maximality, either $h_{J_P} \from S_{J_P} \to T$ is a homeomorphism or the very next process is either a ``fold to increase rank process'' or an unsuccessful ``fold to injection'' process, but from~\pref{ItemCountingTwoTypes} it follows that the number of such processes that do not end in a homeomorphism to $T$ is bounded by the quantity $\Delta_{1.\pref{ItemFoldEstimate}}$.

Putting these two claims together, the total number of successful ``fold to injection'' processes is bounded above by
$$(1 + \Delta_{1.\pref{ItemFoldEstimate}}) \cdot \MaxEdges(\Gamma;\A)  
$$
Combining this with~\pref{ProcessTwoBound} we have:
\begin{enumeratecontinue}
\item\label{ItemSuccFoldToInjBound}
The sum of the diameters of all successful ``Fold to injection'' processes is bounded by
$$ (1 + \Delta_{1.\pref{ItemFoldEstimate}}) \cdot 2 \MaxEdges(\Gamma;\A)
$$
\end{enumeratecontinue}

For our final conclusion, by combining~\pref{ItemOtherProcessDiamBound} and~\pref{ItemSuccFoldToInjBound} we have proved that the distance from $S$ to $T$ is bounded by the sum of the diameters of all processes which is bounded above by the constant
$$(1 + \Delta_{1.\pref{ItemFoldEstimate}}) \cdot (\Delta_{1.\pref{ProcessOneBound}} + 2 \MaxEdges(\Gamma;\A))
$$
Adding $1$ to this number we have therefore finished Step 1 with constant
\begin{align*}
\Delta_1 &=(1 + \Delta_{1.\pref{ItemFoldEstimate}}) \cdot (\Delta_{1.\pref{ProcessOneBound}} + 2 \MaxEdges(\Gamma;\A)) + 1
\end{align*}

\subsection{Step 2: One natural edge over all natural edges}
\label{SectionOneNatOverAllNat}

Using the integer constant $\Delta_1 \ge 1$ from Step 1, and defining the integer constant $\Delta_2 = \Delta_1 + 3$, we prove the statement \emph{One natural edge over all natural edges} as follows: Given free splittings $S,T$ and a foldable map $f \from S \to T$, and assuming that for each natural edge $E \subset S$ the path $f(E)$ does \emph{not} cross every natural edge orbit of $T=T_K$, we prove that $d(S,T) \le \Delta_1 + 2 < \Delta_2$.

Subdividing $S$ at $f^\inv(\VT)$, we may assume $f$ is a simplicial map, taking edges of $S$ to edges of~$T$, so every map in every foldable factorization of $f$ is simplicial, and the restriction of each such map to every subcomplex is simplicial. We use this silently in what follows.

Choose a Stallings fold factorization $f \from S = T_0 \to\cdots\to T_K = T$ using folds of length~$1$. Since $d(S,T) \ge \Delta_1$, there exists a minimum $J \in \{0,1,\ldots,K\}$ such that $d(T_0,T_J) \ge \Delta_1$; it follows that $d(T_0,T_{J-1}) \le \Delta_1 - 1$, and since $d(T_{J-1},T_J) \le 1$ we obtain $d(T_0,T_J) = \Delta_1$. We shall prove $d(T_J,T_K) \le 2$, and hence $d(S,T)=d(T_0,T_K) \le d(T_0,T_J) + d(T_J,T_K) \le \Delta_1 + 2$.

Consider the foldable factorization
$$f \from S = T_0 \xrightarrow{f^0_J} T_J \xrightarrow{f^J_K} T_K=T
$$
Since $d(T_0,T_J) \ge \Delta_1$, by Step~1 there exists a natural edge $E \subset T_0$ such that the path $A_J = f^0_J(E) \subset T_J$ crosses an edge in every edge orbit of $T_J$. It follows that the path $A_K = f^J_K(A_J) = f(E) \subset T_K$ crosses an edge in every edge orbit of $T_K$. By assumption, however, there exists a natural edge $E' \subset A_K$ such that $A_K$ does not cross any natural edge in the orbit of $E'$.  
%
%
%
%
Choose an orientation of $A_K$ with initial and terminal vertices $v_-,v_+$. 

Consider any $\gamma \in \Gamma$ such that the intersection $A_K \intersect (\gamma \cdot E')$ is a nontrivial path, and hence a common subpath of $A_K$ and $\gamma \cdot E'$. Because $\gamma \cdot E'$ is a natural edge of $T_K$, it follows that each endpoint of the subpath $A_K \intersect (\gamma \cdot E')$ is either an endpoint of $A_K$ or of $\gamma \cdot E'$ (or of both). On that basis, one of the following three options holds:
\begin{description}
\item[Full Path:] $A_K \intersect (\gamma \cdot E') = \gamma \cdot E'$; \, or
\item[Half Path:] There exists a choice of sign $\pm$ such that $v_\pm$ is an interior point of $\gamma \cdot E'$ and such that $\eta_\pm \equiv A_K \intersect (\gamma \cdot E')$ is (the closure of) one of the two components of $(\gamma \cdot E') - v_\pm$; \, or
\item[Interior Path:] $A_K = A_K \intersect (\gamma \cdot E')$ is contained in the interior of $\gamma \cdot E'$.
\end{description}
But only ``Half Path'' is possible: ``Full Path'' contradicts that $A_K$ crosses no natural edge in the orbit of $E'$; and ``Interior Path'' contradicts that $A_K$ crosses some edge in the orbit of every edge in $T_K$, in fact $A_K$ crosses no edge in the orbit of an edge $e \subset (\gamma \cdot E') \setminus A_K$. To check this, for any $\delta \in \Gamma$ one of two alternative holds, each leading to the same conclusion that $A_K$ does not cross $\delta \cdot e$: if $\delta \ne \Id$ then the natural edges $\delta \cdot (\gamma \cdot E')$ and $\gamma \cdot E'$ have no edges in common, hence $A_K \subset \gamma \cdot E'$ does not cross $\delta \cdot e$; whereas if $\delta = \Id$ then $A_K$ does not cross $\delta \cdot e = e$. 

Since only ``Half-Path'' occurs, and since $A_K$ has only the two endpoints $v_-$, $v_+$, it follows that there are exactly two values of $\gamma \in \Gamma$ such that $A_K \intersect (\gamma \cdot E')$ is nontrivial, and each of those two satisfies ``Half-Path''. We may therefore denote those two values as $\gamma_-$, $\gamma_+$ respectively, so that $v_\pm$ is an endpoint of $\eta_\pm = A_K \intersect E'_\pm$ where $E'_\pm = \gamma_\pm \cdot E'$. For later use we record:
\begin{description}
\item[Two values:] $A_K \intersect (\gamma \cdot E')$ is nontrivial if and only if $\gamma=\gamma_-$ or $\gamma_+$.
\end{description}
We orient $E'_\pm$ so that the restricted orientation on $\eta_\pm$ agrees with the orientation on $\eta_\pm$ restricted from $A_K$; it follows that $\eta_-$ is both a proper initial subsegment of $A_K$ and a proper terminal subsegment of $E'_-$; and similarly $\eta_+$ is both a proper terminal subsegment of $A_K$ and a proper initial subsegment of $E'_+$. We obtain a concatenation expression $A_K = \eta_- \, \alpha \, \eta_+$ where the subpath $\alpha$ is the union of all natural edges crossed by $A_K$.

Let $\zeta = \gamma_+ \gamma_-^\inv$, and so $\zeta \cdot E'_- = E'_+$. We show that the map $E'_- \mapsto E'_+$ induced by $\zeta$ preserves orientation. If not, consider the initial edge $e_-$ of $E'_-$. It follows that $\Id \cdot e_- = e_- \not\subset A_K$. Also $\zeta \cdot e_-$ is the terminal edge of $E'_+$ and it follows that $\zeta \cdot e_- \not\subset A_K$. But $A_K$ crosses \emph{some} edge in the orbit of $e_-$, so exists $\delta \in \Gamma$ such that $\delta \cdot e_- \subset A_K$, hence $A_K \intersect (\delta \gamma_- \cdot E')$ is nontrivial. Applying ``Two Values'' we have $\delta\gamma_- = \gamma_-$ or $\gamma_+$, and so $\delta = \Id$ or $\zeta$, a contradiction.

Let $e_K$ be the terminal edge of $E'_-$ (incident to the initial vertex of $\alpha$), and so $e_K \subset A_K$. Since $\zeta$ preserves orientation of $E'_-$ and $E'_+$ it follows that $\zeta \cdot e_K \not\subset A_K$. If there existed $\delta \ne \Id$ such that $\delta \cdot e_K \subset A_K$ then as in the previous paragraph it would follow $A_K \intersect (\delta \gamma_- \cdot E')$ is nontrivial, and hence by applying ``Two Values'' that $\delta = \Id$ or $\zeta$, which is again a contradiction. This shows that $A_K$ crosses no other edge in the orbit of $e_K$ except for $e_K$ itself. 

Since $f^J_K$ restricts to a simplicial isomorphism $A_J \mapsto A_K$, there is a unique edge $e_J$ in $A_J$ such that $f^J_K(e_J) = e_K$.
%
%
%
%
%

\begin{description}
\item[Claim:] $e_J$ is the unique edge in all of $T_J$ whose image under $f^J_K$ equals $e_K$. 
\end{description}
Arguing by contradiction, suppose there exists another edge $e'_J \ne e_J$ in $T_J$ such that \hbox{$f^J_K(e'_J)=e_K$.} Knowing that $A_J$ crosses some edge in the orbit of $e'_J$, there exists $\delta \in \Gamma$ such that $\delta \cdot e'_J \subset A_J$. If $\delta \ne \text{Id}$ then 
$$\delta \cdot e_K \, = \, \delta \cdot f^J_K(e'_J) \, = \, f^J_K(\delta \cdot e'_J) \, \subset \, f^J_K(A_J) \, = \, A_K
$$
contradicting that $A_K$ crosses no other edge in the orbit of $e_K$. If $\delta = \text{Id}$ then $e'_J \subset A_J$ and so $f^J_K(e_J)=e_K=f^J_K(e'_J)$, contradicting that $f^J_K$ is injective on~$A_K$. 

From the Claim together with Lemma~\ref{LemmaInjOverEdgeBound} it follows that $d(T_J,T_K) \le 2$, completing Step~2.

\subsection{Step 3: Another natural edge over all natural edges}
\label{SectionTwoOverAllStepThree}
After some preliminaries, the proof of Step 3 is laid out under the heading \emph{Outline of the method}.

Consider any foldable map $f \from S \to T$ between any free splittings $S,T$ of $\Gamma$ rel~$\A$. As in Step~2, after subdividing $S$ and $T$ we silently assume that $f$ is simplicial, as is every map in every foldable factorization of $f$. We may also assume the following lower bound: 
$$d(S,T) \ge 2 \Delta_2
$$
Using that $d(S,T) \ge \Delta_2$, we may apply Step~2 to obtain a natural edge $E^* \subset S$ such that $f(E^*)$ crosses an edge in every natural edge orbit of~$T$. To prove Step~3 we shall find an integer constant $\Delta_3(\Gamma;\A) \ge 2\Delta_2$ such that if $d(S,T) \ge \Delta_3$ then there exists another natural edge of~$S$, in a different orbit that $E^*$, whose image in $T$ also crosses some edge in every natural edge orbit of $T$. Arguing by contradiction, we may assume the following:
\begin{description}
\item[Uniqueness Property I:] The orbit of $E^*$ is the \emph{unique} orbit of natural edges of $S$ having the property that the image under $f$ of each edge in that orbit crosses a representative of every natural edge orbit of $T$.
\end{description}
Our goal is now to apply \emph{Uniqueness Property I} to derive an upper bound on $d(S,T)$.

\smallskip
We next sharpen Uniqueness Property~I. Applying Theorem~\ref{TheoremStallings} we obtain a fold factorization using folds of length~$\le 1$ as follows:
$$f \from S = T_0 \xrightarrow{f_1} T_1 \xrightarrow{f_2} \cdots \xrightarrow{f_K} T_K = T
$$
Since $d(T_0,T_K) \ge 2 \Delta_2$, we may choose $J \in \{0,\ldots,K\}$ so that $d(T_J,T_K) = \Delta_2$, and it follows that $d(T_0,T_J) \ge \Delta_2$. We shall rewrite our notation as
$$\xymatrix{
T_0 \ar[r]^{f^0_J} \ar@{=}[d] & T_J \ar[r]^>>>>>{f^J_K} \ar@{=}[d] & T_K=T \\
S \ar[r]^g & U
}$$
Applying Step~2 to the map $f^J_K$ we obtain a natural edge $E^\sharp \subset T_J=U$ such that $f^J_K(E^\sharp)$ crosses a representative of every natural edge orbit in $T_K$. Applying Step~2 again, this time to the map~$f^0_J$, we obtain a natural edge $E' \subset S=T_0$ such that $g(E')=f^0_J(E')$ crosses a representative of every natural edge orbit of $U=T_J$, including the orbit of $E^\sharp$. However, for any such edge $E'$ it follows that $f^0_K(E')=f^J_K(f^0_J(E'))$ crosses $f^J_K(E^\sharp)$ which crosses a representative of every natural edge orbit of $T_K=T$, and it then follows from Uniqueness Property~I that $E'$ is in the orbit of $E^*$. We have proved: 
\begin{description}
\item[Uniqueness Property II:] In any fold factorization of $f$ as described above, for every natural edge $E \subset S$, its image $g(E)$ crosses a translate of $E^\sharp \subset U$ if and only if $E$ is in the orbit of $E^*$.
\end{description}

In the remainder of the proof we use Uniqueness Property II to obtain an upper bound on $d(S,U)=d(T_0,T_J)$; see Fact~\ref{FactFinalBound} where that bound is summarized. Adding that bound to $\Delta_2 = d(T_J,T_K)$ then gives us the desired upper bound on $d(S,T) = d(T_0,T_K)$. 

\paragraph{Partial fold factorizations of $g \from S \to U$.} Consider any partial fold factorization of $g$ having maximal folds, denoted
$$\xymatrix{
\text{(A)} \qquad 
	S = U_0 \ar[r]_<<<<{g_0} \ar@/^3pc/[rrrrrr]^<<<<<<<<{g=g^0}& 
	\cdots \ar[r]_{g_i} &
	U_i \ar[r]_{g_{i+1}} \ar@/^1.5pc/[rrrr]^<<<<<<<<{g^i}&
	\cdots \ar[r]_{g_M} &
	U_M \ar[rr]^<<<<<<{g^M} &&
	U 
}$$
and with notations $g^i_j = g_j \circ\cdots\circ g_{i+1} \from U_i \to U_j$ and $g^i = g^M \circ g^i_M \from U_i \to U$. Pulling $\Gamma \cdot E^\sharp$ back to each $U_i$ we obtain $U_i^\sharp = (g^i)^*(\Gamma \cdot E^\sharp)$. Define a \emph{piece of $U_i^\sharp$} to be a nondegenerate subarc $\mu \subset U_i$ having the property that for some natural edge $E \subset U_i$ and group element $\gamma \in \Gamma$ we have $\mu = E \intersect (g^i)^*(\gamma \cdot E^\sharp)$; more specifically we say that $\mu$ is a \emph{piece of $U_i^\sharp$ in $E$}. It follows that $g^i \restrict \mu$ is an embedding of $\mu$ in $\gamma \cdot E^\sharp$, and we say that $\mu$ is a \emph{whole} piece if $g^i(\mu)=\gamma \cdot E_\sharp$, otherwise $\mu$ is a \emph{partial} piece. Given a vertex $v$ of $U_i^\sharp$, to say that $v$ is \emph{deep} means that $g^i(v)$ is an interior point of some translate of $\gamma \cdot E^\sharp$, equivalent $g^i(v)$ is \emph{not} a natural vertex of $U$. Note that a piece $\mu$ of $U_i^\sharp$ is whole if and only if neither of its endpoints is deep; and $\mu$ is partial if and only if at least one of its endpoints is deep.

In this new terminology, \emph{Uniqueness Property II} says that every whole piece of $U_0^\sharp$ is a subpath of some translate of $E^*$; there is no constraint (yet) on the partial pieces. 

Let \hbox{$\alpha \subset E^* \subset U_0$} be the smallest subpath of $E^*$ containing every whole piece of $U_0^\sharp$ in $E^*$. Let $\P(U_i)$ be the set of pieces of $U_i^\sharp$. In $U_0$ we define a $\Gamma$-invariant decomposition 
$$\P(U_0) = \P_\alpha(U_0) \sqcup \overline\P_\alpha(U_0)
$$
where $\P_\alpha(U_0)$ is the set of whole pieces, each a subpath of some translate of $\alpha$, and $\overline\P_\alpha(U_0) = \P(U_0) - \P_\alpha(U_0)$ is the set of partial pieces. In Fact~\ref{FactSharp} we shall describe how this decomposition evolves along a partial fold factorization~(A), obtaining decompositions \hbox{$\P(U_i) = \P_\alpha(U_i) \sqcup \overline\P_\alpha(U_i)$.} This will \emph{not} continue to be a decomposition of $U_i^\sharp$ into its whole pieces and partial pieces; in fact as the index~$i$ increases eventually all pieces are whole (see Fact~\ref{FactAlphaSeries}~\pref{ItemEveryKPieceWhole} and Fact~\ref{FactAlphaSubdivision}~\pref{ItemAlphaWhole}). 

\subparagraph{Outline of the method.} Our method of proof for Step~3 is to describe a fold factorization of $g \from S \to U$, broken into phases that are designed to produce certain bounds to be eventually used for the final bound on $d(S,U)$. Roughly speaking the first phase prioritizes folds at deep vertices, as long as such a fold exists. When the first phase stops, at some~$U_K$, every piece of $U_K^\sharp$ is whole (Fact~\ref{FactAlphaSeries}~\pref{ItemEveryKPieceWhole}); furthermore, that wholeness property continues to hold beyond $U_K$ in any fold factorization with maximal fold factors (Fact~\ref{FactAlphaSubdivision}~\pref{ItemAlphaWhole}). The second phase allows arbitrary folding \emph{except} that sewing needle folds disallowed, and continuing as long as possible until some $U_L$ in which one of two outcomes occurs: the descendents $\P_\alpha(U_L)$ of $\P_\alpha(U_0)$ have almost --- but not quite --- gone extinct; or the final foldable map $U_L \mapsto U$ is a sewing needle multifold. The proof is organized into certain statements of fact and descriptions of fold phases.

\smallskip

 For any set $X$ on which $\Gamma$ acts, in particular $X = \P_\alpha(U_i)$ or $\overline\P_\alpha(U_i)$, we use the notation $\abs{X}_\Gamma$ for the cardinality of the set of $\Gamma$-orbits of the action; see for example Fact~\ref{FactPartialPieces}~\pref{ItemXOverAlphaBound} just below. Our first fact is an analysis of pieces that leads to a bound on $\abs{\overline\P(U_0)}_\Gamma$.


\begin{StepThreeFact}[Analysis of partial pieces]
\label{FactPartialPieces}
Consider any partial fold factorization~(A), any natural edge $E \subset U_i$ \hbox{with $0 \le i \le M$,} and any partial piece $\mu$ of $U_i^\sharp$ in $E$. We have $g^i(\mu) \subset \gamma \cdot E^\sharp$ for a unique $\gamma \in \Gamma$. 
\begin{enumerate}
\item\label{ItemXEndpoint}
Some endpoint of $\mu$ is a deep vertex of $U_i^\sharp$. 
\item\label{ItemXInnerVertex}
For any endpoint $x$ of $\mu$ which is a deep vertex of $U_i^\sharp$, the following hold: 
\begin{enumerate}
\item\label{ItemXInterior}
$x$ is an interior point of the subforest $U_i^\sharp \subset U_i$.
\item\label{ItemXNatural}
$x$ is a natural vertex of $U_i$, with trivial stabilizer and valence~$\ge 3$.
\item\label{ItemXTurnFoldable}
Some turn of $U_i$ at $x$ is foldable with respect to $g_i$.
\end{enumerate}
\item\label{ItemXPartialNotAlpha}
Every partial piece of $U_0^\sharp$ has no overlap with $\Gamma \cdot \alpha$.
\end{enumerate}
It follows that
\begin{enumeratecontinue}
\item\label{ItemXOverAlphaBound}
$\displaystyle \abs{\overline\P_\alpha(U_0)}_\Gamma \le 6 \corank(\A) + 4\abs{\A} - 6$
\end{enumeratecontinue}
\end{StepThreeFact}

\begin{proof} Conclusion~\pref{ItemXEndpoint} was already mentioned in our earlier paragraph defining pieces. 

Conclusion~\pref{ItemXInterior} holds for any deep vertex $x$ of $U_i^\sharp$, because every direction of $U_i$ at $x$ is represented by a path that is mapped by $g^i$ to the interior of some translate of $E^\sharp$. We remark that the converse does not hold: a vertex contained in the interior of the subforest $U_i^\sharp$ (even a natural one) need not be an deep vertex of $U_i^\sharp$.

Consider now an endpoint $x$ of a piece $\mu$ such that $x$ is an deep vertex of $U_i^\sharp$. As with any interior point of a natural edge of a free splitting, the point $g^i(x) \in U$ has trivial stabilizer, and so $x$ also has trivial stabilizer. If $x$ did not satisfy conclusion~\pref{ItemXNatural} then it would follow that~$x$ has valence~$2$, and hence $x$ is contained in the interior of some natural edge $E' \subset U_i$. But both directions at $x$ are represented by paths mapped by $g^i$ to the interior of $\gamma \cdot E^\sharp$, contradicting that $\mu = E' \intersect (g_i)^*(\gamma \cdot E^\sharp)$. Conclusion~\pref{ItemXTurnFoldable} follows for $x$ because there are only two directions of $U$ at $g^i(x)$ and therefore, out of the three or more directions of $U_i$ at $x$, at least two of them have the same image direction at $g^i(x)$.

To prove conclusion~\pref{ItemXPartialNotAlpha}, suppose that some partial piece $\mu$ of $U_0^\sharp$ overlaps $\Gamma \cdot \alpha$. Replacing $\mu$ by a translate, we may assume that $\mu$ overlaps $\alpha$.  By definition of $\alpha$, each endpoint of $\alpha$ is incident to a whole piece of $U_0^\sharp$ in $E^*$ that is entirely contained in $\alpha$. It follows the partial piece $\mu$ is entirely contained in the interior of $\alpha$ and so is also contained in the interior of the natural edge $E^*$, which contradicts conclusions~\pref{ItemXEndpoint} and~\pref{ItemXNatural}.

From conclusions~\pref{ItemXEndpoint} and~\pref{ItemXNatural} it follows that each natural edge of $U_0$ can contain at most two partial pieces, at most one incident to each endpoint of that natural edge. The bound~\pref{ItemXOverAlphaBound} then follows because the cardinality $\abs{\overline\P_\alpha(U_0)}_\Gamma$ --- which equals the total number of orbits of partial pieces of $U_0^\sharp$ --- is bounded above by two times the maximum number of natural edge orbits of free splittings of $\Gamma$ rel~$\A$, that number being $3 \corank(\A) + 2\abs{\A} - 3$ as given in \cite[Proposition 3.4~(1)]{\RelFSOneTag}.
\end{proof}

The next fact provides a tool for setting up and studying the evolution of pieces along a fold factorization. 

 
\begin{StepThreeFact}[Evolution of pieces]
\label{FactSharp}
Along any partial fold factorization~(A), for each $1 \le i \le M$ there exists a unique function $g_i^\sharp \from \P(U^\sharp_{i-1}) \to \P(U_i^\sharp)$ with the following properties: 
\begin{enumerate}
\item\label{ItemSharpOverlap}
 For each $\mu \in \P(U^\sharp_{i-1})$, the path $g_i(\mu) \subset U_i$ has nontrivial overlap with the piece $g_i^\sharp(\mu)$;
\item\label{ItemSharpSurjectivity}
 $g_i^\sharp$~is surjective.
\end{enumerate}
Furthermore,
\begin{enumeratecontinue}
\item\label{ItemSharpEquivariant}
$g_i^\sharp$ is $\Gamma$-equivariant; 
\item\label{ItemSharpWhole}
For each $i=1,\ldots,M$, if each piece of $U^\sharp_{i-1}$ is whole then each piece of $U^\sharp_i$ is whole.
\end{enumeratecontinue}
\end{StepThreeFact}

\begin{proof} Consider $\mu \in \P(U^\sharp_{i-1})$ and let $E \subset U_{i-1}$ denote the natural edge containing~$\mu$. The description of $g_i^\sharp(\mu)$ will proceed in cases depending on the behavior of $\mu$. In each case we will see that properties \pref{ItemSharpOverlap} and  \pref{ItemSharpSurjectivity} will force a uniquely assigned value for $g_i^\sharp(\mu)$, which is how we define that value \emph{and} how we prove uniqueness of $g_i^\sharp$. In $U$ we have $g^{i-1}(\mu) = g^i(g_i(\mu)) \subset \gamma \cdot E^\sharp$ for a unique $\gamma \in \Gamma$. It will be evident that each case hypothesis is $\Gamma$-invariant, and that $g_i^\sharp(\mu)$ is $\Gamma$-invariantly defined, which is property~\pref{ItemSharpEquivariant}.

\smallskip

\emph{Case I: $g_i$ is one-to-one over the interior of $\mu$,} and so its image $g_i(\mu)$ is contained in some natural edge $E' \subset U_i$ and $g_i(\mu)$ is contained in some piece $\nu$ of $U_i^\sharp$ in $E'$. Furthermore $\nu$ is the unique piece of $U_i^\sharp$ that overlaps $g_i(\mu)$. We are therefore forced to define $g_i^\sharp(\mu) = \nu$. For proving surjectivity later on, note that $\mu$ is the unique piece of $U_i^\sharp$ such that $g_i(\mu)$ overlaps $\nu$.

\smallskip

\emph{Case II: $g_i$ is not one-to-one over the interior of $\mu$.} It follows that there are oriented natural edges $E,E'$ with common initial vertex~$v$ and with maximal nondegenerate initial segments $\eta \subset E$, $\eta' \subset E'$ satisfying $g^{i-1}(\eta)=g^{i-1}(\eta') \subset U$, such that $g_i(\eta)=g_i(\eta)$, and such that $\mu$ has nontrivial overlap with $\eta$. Let $E_1 \subset U_i$ be the natural edge that contains $g_i(\eta)=g_i(\eta')$. Let $\gamma \cdot E^\sharp$ be the unique translate of $E^\sharp$ such that $g^{i-1}(\mu \intersect \eta) \subset g^{i-1}(\mu) \subset \gamma \cdot E^\sharp$. Let $\beta \subset \eta'$ be the unique subpath such that $g^{i-1}(\mu \intersect \eta) = g^{i-1}(\beta) \subset \gamma \cdot E^\sharp$, and so $g_i(\mu \intersect \eta) = g_i(\beta) \subset E_1$. It follows that $\beta$ is contained in a piece of $U_{i-1}^\sharp$ in $E'$ that we denote $\mu'$, and that $g_i(\mu)$ and $g_i(\mu')$ both overlap the same piece of $U_i^\sharp$ in $E_1$ that we denote $\nu_1$. The pieces $\mu,\mu' \subset U_{i-1}$ and $\nu_1 \subset U_i$ all map to subpaths of $\gamma \cdot E^\sharp$ under respective maps to $U$.

\smallskip

\emph{Case IIa:} Suppose that $w$ is not an interior point of $\mu$. It follows that $\mu \subset \eta$, and $g_i(\mu) \subset g_i(\eta) \subset E_1$, and $g_i(\mu) \subset \nu_1$. As in Case~1, $\nu_1$ is the unique piece of $U_i^\sharp$ that overlaps $g_i(\mu)$, and so we are forced to define $g_i^\sharp(\mu) = \nu_1$. 

\smallskip

\emph{Case IIb:} Suppose that $w$ is an interior point of $\mu$ and hence also of~$E$, so $w$ subdivides $E$ and $\mu$ into nondegenerate subpaths 
$$E = e_1 \, e_2, \qquad \mu = \mu_1 \, \mu_2, \qquad \mu_i = \mu \intersect e_i
$$
The path $\mu_1$ is a terminal segment of $\eta$, and there is a corresponding terminal segment $\mu'_1 = \mu' \intersect \eta'$ of $\eta'$ such that $g_i(\mu_1)=g_i(\mu'_1)$. In $U$, note that $g^{i-1}(\mu_1) = g^{i-1}(\mu'_1)$ is a proper subsegment of $\gamma \cdot E^\sharp$ having an endpoint $g^{i-1}(w)=g^{i-1}(w')$ in the interior of $\gamma \cdot E^\sharp$. It follows that $w'$ is a natural vertex of $E'$ and hence that $E'=\eta'$ and $\mu'=\mu'_1$, because otherwise the terminal segments $\mu_1 \subset \eta$ and $\mu'_1 \subset \eta'$ can be extended beyond $w$ and $w'$ to longer subsegments of $E$ and $E'$ having the same image in $U$ under $g^{i-1}$, contradicting maximality of the fold $g_i$. 

Note that orbits $\Gamma \cdot e_1$, $\Gamma \cdot e_2$ have trivial overlap in $U_i$, for otherwise the initial direction of $e_1$ and the terminal direction of $e_2$ would be in the same orbit and hence $g_i$ would be a sewing needle fold; but in that case $\eta'$ would be a proper initial segment of $E'$, contradicting that $\eta' = E'$. It follows that there is a natural edge $E_2 \subset U_i$ such that $E_1$, $E_2$ have distinct orbits, and such that 
$$g_i(\mu_1)=g_i(\mu'_1) \subset g_i(e_1) \subset E_1 \quad\text{and}\quad g_i(\mu_2) \subset g_i(e_2) \subset E_2
$$
We already know that $\nu_1$ is the unique piece of $U_i^\sharp$ in $E_1$ that contains $g_i(\mu_1)=g_i(\mu'_1)$. There is also a unique piece of $U_i^\sharp$ in $E_2$ that contains $g_i(\mu_2)$; denote it as $\nu_2$. Note that $\mu' = \mu'_1$ is a piece of $U_{i-1}^\sharp$ in $E'$ that falls under Case IIa, and that according to the Case IIa definition we were forced to define $g_i^\sharp(\mu') = \nu_1$. Furthermore, $\mu$ is the \emph{unique} piece of $U_{i-1}^\sharp$ whose image $g_i(\mu)$ overlaps $\nu_2$. The surjectivity requirement of $g_i^\sharp$ therefore forces us to define $g_i^\sharp(\mu) = \nu_2$. 

We note one further outcome of Case IIb, which is relevant to conclusion~\pref{ItemSharpWhole}, namely that $\mu' = \mu'_1 \in \P(U_{i-1}^\sharp)$ is a partial piece of $U_{i-1}^\sharp$: this holds because its endpoint $w'$ has image $g^{i-1}(w')=g^{i-1}(w)$ which is an interior point of $\gamma \cdot E^\sharp$, and so $w'$ is deep.

\smallskip

This completes the construction of $g_i^\sharp$ and the verification of conclusions~\pref{ItemSharpOverlap} and~\pref{ItemSharpEquivariant}. 

To prove the surjectivity conclusion~\pref{ItemSharpSurjectivity}, consider any piece $\nu \in \P(U_i^\sharp)$. Since $U_{i-1}^\sharp = (g^{i-1})^*(\Gamma \cdot E^\sharp)$ maps surjectively to $U_i^\sharp = g_i^*(\Gamma \cdot E^\sharp)$, there is at least one piece $\mu$ of $U_{i-1}^\sharp$ whose image $g_i(\mu)$ overlaps $\nu$. If $\mu$ falls into case I or IIa then $g_i^\sharp(\mu) = \nu$. If $\mu$ falls into Case IIb then (following the notation of that case) we have either $\nu = \nu_1 = g_i^\sharp(\mu')$ or $\nu=\nu_2 = g_i^\sharp(\mu)$. 

Finally, to prove property~\pref{ItemSharpWhole}, suppose that each piece of $U_{i-1}^\sharp$ is whole. For any piece $\mu \in \P(U_{i-1}^\sharp)$, it follows that $\mu$ does \emph{not} fall into case~IIb, because one further outcome of case~IIb is the existence of a partial piece $\mu'$ of $U_{i-1}^\sharp$. Therefore $\mu$ falls into cases I or IIa where $g_{i-1}$ is injective on the interior of $\mu$, and so $g_{i-1}(\mu)=g_{i-1}^\sharp(\mu) \in \P(U_{i}^\sharp)$ is whole. By surjectivity of $g_i^\sharp$ it follows that every piece of $U_i^\sharp$ is whole.
\end{proof}

\paragraph{First Fold Phase:} Choose a partial fold factorization of $g \from S \to U$ of the form (A), such that each fold factor $g_i \from U_{i-1} \to U_i$ folds a turn located at a deep vertex of $U_{i-1}^\sharp$, and such that the partial fold factorization has maximal length with respect to this property. We rewrite this factorization in the form
$$g \from S=U_0 \xrightarrow{g_1} \cdots \xrightarrow{g_K} U_K \xrightarrow{g^K} U
$$
By maximality, when the factorization stops at $U_K$, no turn at a deep vertex of $U_K^\sharp$ is foldable with respect to~$g^K$.


\begin{StepThreeFact}[Properties of the first fold phase] 
\label{FactAlphaSeries}
In the \emph{First Fold Phase} the following hold:
\begin{enumerate}
\item\label{ItemEveryKPieceWhole}
Every piece of $U_K^\sharp$ is whole.
\item\label{ItemFirstAlphaSequence}
There exists a sequences $\alpha^{\vphantom *}_i$, $E^*_i$ defined for $0 \le i \le K$ such that the following hold:
\begin{enumerate}
\item\label{ItemAlphaFirst}
$E^*_0 = E^*$ and $\alpha_0=\alpha$;
\item $E^*_i$ is a natural edge of $U_i$ and $\alpha_i$ is a subpath of $E^*_i$;
\item The set $U_i^\sharp \intersect \alpha_i$ is a union of whole pieces of $U_i^\sharp$ in $E^*_i$ with one such piece incident to each endpoint of $\alpha_i$; 
\item\label{ItemAlphaInjectiveInductive}
If $i \ge 1$ then the map $g_i$ is injective over the interior of $\alpha_i$, and $\alpha_i = g_i(\alpha_{i-1})$. 
\item\label{ItemAlphaPieceBijection}
$g_i$ induces a bijection between pieces of $U_{i-1}^\sharp$ in $\alpha_{i-1}$ and pieces of $U_i^\sharp$ in $\alpha_i$.
\end{enumerate}
\item\label{ItemAlphaInjectsAllTheWay}
$g^0_K \from U_0 \to U_K$ is injective at each point of the interior of $\alpha$, and so $d(U_0,U_K) \le 2$. 
\end{enumerate}
\end{StepThreeFact}

\begin{proof} To prove~\pref{ItemEveryKPieceWhole}, if $U_K^\sharp$ had a partial piece $\mu$ then, by Fact~\ref{FactPartialPieces}, some endpoint of $\mu$ would be a deep vertex of $U_K^\sharp$ at which $g^K$ has a foldable turn, which is impossible when the factorization stops at $U_K$ in the \emph{First Fold Phase}.

To prove~\pref{ItemFirstAlphaSequence}, suppose by induction that such a sequence is defined for $0 \le i \le k < K$ (the base case with $k=0$ is evident). Since each endpoint of $\alpha_k$ is incident to a whole piece contained in $\alpha_k$, it follows that neither endpoint of $\alpha_k$ is a deep vertex of $U_k^\sharp$, and so in the \emph{First Fold Phase} there is no turn that is folded by $g_{k+1}$ that involves an ending direction of $\alpha_k$. Together with the fact that $\alpha_k$ is contained in some natural edge $E^*_k \subset U_k$, it follows that the fold map $g_{k+1}$ is injective on the interior of $\alpha_k$. We define $\alpha_{k+1}=g_{k+1}(\alpha_k)$, and it follows that $\alpha_{k+1}$ is a subpath of some uniquely determined natural edge that we define to be $E^*_{k+1} \subset U_{k+1}$. With these definitions, the remaining portions of~\pref{ItemFirstAlphaSequence} needed to complete the induction quickly follow.

Conclusion~\pref{ItemAlphaInjectsAllTheWay} follows from \pref{ItemAlphaFirst} and~\pref{ItemAlphaInjectiveInductive}.
\end{proof}

Consider now any partial fold factorization of $g \from S \to U$ with maximal folds that extends the First Fold Phase, thus having the following form,
$$(B) \qquad\qquad g \from S=U_0 \xrightarrow{g_1} \cdots \xrightarrow{g_K} U_K \xrightarrow{g_{K+1}} \cdots \xrightarrow{g_{L-1}} U_{L-1} \xrightarrow{g_L} U_L \xrightarrow{g^L} U \qquad\qquad\hphantom{(B)}
$$
and we assume also that none of the folds $g_{K+1},\ldots,g_L$ is a sewing needle fold. 

\smallskip

For $0 \le l \le L$ we have notation $\P(U_l^\sharp)$ for the set of all pieces of $U_l^\sharp$, on which $\Gamma$ acts. We shall describe for each $l$ a $\Gamma$-invariant subdivision 
$$\P(U_l^\sharp) = \overline{\P}_\alpha(U_l^\sharp) \sqcup  \P_\alpha(U_l^\sharp)
$$
starting with the base case $l=0$ which is already defined. To extend this by induction, 
assuming that $0 < l \le L$ and that the decomposition is defined for $l-1$, define
$$\overline \P_\alpha(U_{l}^\sharp) = g_{l}^\sharp(\overline \P_\alpha (U_{l-1}^\sharp)) \qquad\qquad \P_\alpha(U_{l}^\sharp) = \P(U_{l}^\sharp) -  \overline \P_\alpha(U_{l}^\sharp) 
$$


\begin{StepThreeFact}[Beyond the first fold phase]
\label{FactAlphaSubdivision}
In any partial fold factorization of the form (B) the following hold:
\begin{enumerate}
\item\label{ItemAlphaWhole}
For $K \le i \le L$, every piece of $U_i^\sharp$ is whole.
\item\label{ItemAlphaCompRestrict}
For $1 \le i \le L$ we have
$$6 \corank(\A) + 4\abs{\A} - 6 \ge \abs{\overline{\P}_\alpha(U_0^\sharp)}_\Gamma \ge \cdots \ge \abs{\overline{\P}_\alpha(U_K^\sharp)}_\Gamma \ge \cdots \ge  \abs{\overline{\P}_\alpha(U_L^\sharp)}_\Gamma
$$
\item\label{ItemAlphaRestrict}
For each $1 \le i \le K$ the map $g_i^\sharp$ restricts to an equivariant bijection $\P_\alpha(U_{i-1}^\sharp) \mapsto \P_\alpha(U_i^\sharp)$; and for each $K < i \le L$, we have an (equivariant) inclusion $\P_\alpha(U_i^\sharp) \subset g_i^\sharp\bigl(\P_\alpha(U_{i-1}^\sharp)\bigr)$. It follows that
$$\abs{\P_\alpha(U_0^\sharp)}_\Gamma = \cdots = \abs{\P_\alpha(U_K^\sharp)}_\Gamma \ge \cdots \ge \abs{\P_\alpha(U_L^\sharp)}
$$
\item\label{ItemAlphaMore}
If $\abs{\P_\alpha(U_L^\sharp)}_\Gamma \ge 1$ then the sequences $\alpha^{\vphantom *}_i$, $E^*_i$ of Fact~\ref{FactAlphaSeries}~\pref{ItemFirstAlphaSequence}, already defined for $0 \le i \le K$, extend over the whole interval $0 \le i \le L$ with the following additional properties for each $K+1 \le i \le L$:
\begin{enumerate}
\item $E^*_i$ is a natural edge of $U_i$ and $\alpha_i$ is a subpath of $E^*_i$;
\item The set $U_i^\sharp \intersect \alpha_i$ is a union of one representative of each $\Gamma$-orbit of $\P_\alpha(U_i^\sharp)$, with each endpoint of $\alpha_i$ incident to a piece in $\P_\alpha(U_i^\sharp)$;
\item\label{ItemAlphaInjective}
the map $g_i$ is injective over the interior of $\alpha_i$ and $\alpha_i \subset g_i(\alpha_{i-1})$. 
\item $g_i$ induces a bijection from the subset of $\P_\alpha(U_{i-1}^\sharp)$ consisting of translates of pieces contained in $g_i^*(\alpha_{i-1})$ to the set $\P_\alpha(U_i^\sharp)$.
\end{enumerate}
\item\label{ItemAlphaBound}
If $\abs{\P_\alpha(U_L^\sharp)} \ge 1$ then $d(U_0,U_L) \le 2$.
\end{enumerate}
\end{StepThreeFact}

\begin{proof} Item~\pref{ItemAlphaWhole} follows by induction, using Fact~\ref{FactAlphaSeries}~\pref{ItemEveryKPieceWhole} for the base case, and Fact~\ref{FactSharp}~\pref{ItemSharpWhole} for the induction step. In item~\pref{ItemAlphaCompRestrict}, the first inequality is Fact~\ref{FactPartialPieces}~\pref{ItemXOverAlphaBound}; the remaining inequalities are evident from surjectivity of each map $g_l^\sharp \from \overline\P_\alpha(U_{l-1}^\sharp) \to \overline\P_\alpha(U_l^\sharp)$. 

The first sentence of item~\pref{ItemAlphaRestrict} follows from the description of the sequence $\alpha_i$ in Fact~\ref{FactAlphaSeries}~\pref{ItemFirstAlphaSequence}. The rest of~\pref{ItemAlphaRestrict} is evident from the definition of $\P_\alpha(U_l^\sharp)$.

We turn to the proof of item~\pref{ItemAlphaMore} which, by Fact~\ref{FactAlphaSeries}~\pref{ItemFirstAlphaSequence}, is already known for $0 \le i \le K$. Proceeding by induction, we assume that $K < I \le L$ and that \pref{ItemAlphaMore} is already known for $0 \le i \le I-1$; we must prove it for $i=I$. We break into cases depending on the behavior of $E^*_{I-1}$.

\smallskip
\emph{Case 1:} Suppose that no direction of $E^*_{I-1}$ is part of a turn folded by $g_I$. It follows that $g_I$ is injective on the interior $\alpha_{I-1}$. The method of proof of Fact~\ref{FactAlphaSeries}~\pref{ItemFirstAlphaSequence} applies to construct $E^*_{I}$ and $\alpha_{I}$, with stronger conclusions: in (c) we have $\alpha_{I} = g_{I}(\alpha_{I-1})$; and in (d), the domain of the bijection is the whole of $\P_\alpha(U_{I-1}^\sharp)$. 

\smallskip
\emph{Case 2:} Suppose that $g_I$ folds some turn $\{d,d'\}$ where $d$ is a direction of $E^*_{I-1}$. We orient $E^*_{I-1}$ so that $d$ is its initial direction, located at its initial vertex $v \in U_{I-1}$. Let $E' \subset U_{I-1}$ be the oriented natural edge with initial vertex $v$ and initial vertex $d'$. Let $\eta \subset E^*_{I-1}$ and $\eta' \subset E'$ be the maximal initial segments that are identified under $g^{I-1}$, and so $g_I(\eta)=g_I(\eta')$. Having assumed that $g_l$ is not a sewing needle fold, it follows that $E^*_{I-1}$ and $E'$ are in different orbits. Let $w$ be the terminal vertex of $\eta$. We denote the subdivision $E^*_{I-1}$ at $w$ as $E^*_{I-1} = \eta \, \zeta$; it is possible that $\zeta$ degenerate to a point. We break into subcases depending on how $\alpha_{I-1}$ is related to this subdivision of $E^*_{I-1}$.

\smallskip
\emph{Case 2a:} $\alpha_{I-1} \subset \zeta$. In this case $g_I$ is injective on the interior of $\alpha_{I-1}$, and the proof proceeds as in Case 1.

\smallskip
\emph{Case 2b:} $\alpha_{I-1} \subset \eta$. In this case, for each piece $\mu$ of $U_{I-1}^\sharp$ such that $\mu \subset \alpha_{I-1}$, the piece $\mu$ is folded together by $g_I$ with a piece contained in $E'$, the latter of which is in the set $\overline\P_\alpha(U_{I-1}^\sharp)$. It follows that the $g_I(\mu) \in \overline\P_\alpha(U_I^\sharp)$, and that $\P_\alpha(U_I^\sharp) = \emptyset$. But then, from item~\pref{ItemAlphaRestrict}, it follows that $\P_\alpha(U_L^\sharp) = \emptyset$ and so $\abs{\P_\alpha(U_L^\sharp)}_\Gamma = 0$, contradicting the hypothesis of item~\pref{ItemAlphaMore}.

\emph{Case 2c:} $\alpha_{I-1} \not\subset \zeta$ and $\alpha_{I-1} \not\subset \eta$. In this case $w$ is an interior point of $\alpha_{I-1}$, subdividing it into two subpaths $\alpha_{I-1} = \alpha' \, \alpha''$. The point $w$ cannot be in the interior of any piece of $U_{I-1}^\sharp$, for then $U_I^\sharp$ would have a partial piece. It follows that each piece of $U_{I-1}^\sharp$ in $\alpha_{I-1}$ is contained in one of $\alpha'$ or $\alpha''$, and each of those two contains at least one piece of $U_{I-1}^\sharp$. For every piece $\mu$ contained in $\alpha'$, the map $g_L$ folds $\mu$ together with a piece contained in $E'$, and so as in Case 2b it follows that $g_I(\mu) \in \overline\P_\alpha(U_l^\sharp)$. The map $g_I$ is injective over the interior of $\zeta$, and so $g_I(\zeta)$ is contained in a natural edge which we take to be $E^*_I$. Let $\hat\alpha''$ be the smallest subpath of $\zeta$ containing every whole piece in $\hat\alpha''$, and so $\hat\alpha''$ is a nondegenerate path. We then take $\alpha_I = g_I(\hat\alpha'')$. The remaining conclusions of item~\pref{ItemAlphaMore} are now straightforward to prove.

To prove~\pref{ItemAlphaBound}, by combining Fact~\ref{FactAlphaSeries}~\pref{ItemAlphaInjectiveInductive} and Fact~\ref{FactAlphaSubdivision}~\pref{ItemAlphaInjective} it follows that the map $g^0_L \from U_0 \to U_L$ is injective over the interior of $\alpha_L$, and so we may apply Lemma~\ref{LemmaInjOverEdgeBound} to obtain the bound $d(U_0,U_L) \le 2$.
\end{proof}

\emph{Remark.} The upper bound $\bigl| \overline{\P}_\alpha(U_i^\sharp)\bigr| \le 6 \corank(\A) + 4\abs{\A} - 6$ for $i=K$, and hence for $i=L$, can be cut in half by more carefully counting of orbits of end directions of pieces (rather than orbits of pieces themselves) and using the Case~1 conclusion that every piece of $U^\sharp_K$ is whole. While that would improve our ultimate distance bound on $d(S,U)$ needed to complete Step 3, we shall not pursue this issue further.

\paragraph{Second Fold Phase:} Choose a partial fold factorization of $g \from S \to U$ of the form (B) so that $\bigl| \P_\alpha(U_i^\sharp)\bigr| \ge 1$ for all $i$, and of maximal length with respect to this property. From the description of (B), each $g_i$ is a maximal fold factor of $g^{i-1}$, and no $g_i$ is a sewing needle fold if $K < i \le L$. When the factorization stops at $U_L$, one of two possible endings occurs:
\begin{description}
\item[Case 1:] $\bigl|\P_\alpha(U_L^\sharp)\bigr|_\Gamma = 1$; 
\item[Case 2:] Every first fold factor of $g^L \from U_L \to U$ is a sewing needle fold.
\end{description} 
A priori there is one other somewhat plausible case allowed by the Stallings Fold Theorem~\ref{TheoremStallings}, namely that the Second Fold Phase can be continued all the way to a full factorization, without every using sewing needle folds, and with $\bigl|\P_\alpha(U_L^\sharp)\bigr|_\Gamma \ge 2$. But this is impossible, because it contradicts the fact that 
$$\abs{\P_\alpha(U_L^\sharp)}_\Gamma  \le \abs{\P(U_L^\sharp)}_\Gamma =  \abs{\{\gamma \cdot E^\sharp \suchthat \gamma \in \Gamma\}}_\Gamma = 1
$$

%
%
%
%

\medskip

We now put the pieces together to conclude the proof of Step 3:

\begin{StepThreeFact}[The final bound]
\label{FactFinalBound}
There is a constant $C=C(\Gamma;\A)$ such that, assuming the Second Fold Phase, we have $d(S,U) \le C$.
\end{StepThreeFact}

\begin{proof} Using $d(S,U) = d(U_0,U) \le d(U_0,U_L) + d(U_L,U)$ it suffices to give separate bounds on $d(U_0,U_L)$ and $d(U_L,U)$.

In the Second Fold Phase we have $\bigl|\P_\alpha(U_L^\sharp)\bigr|_\Gamma \ge 1$ and so from Fact~\ref{FactAlphaSubdivision}~\pref{ItemAlphaBound} we obtain the bound 
$$d(U_0,U_L) \le 2
$$
The bound on $d(U_L,U)$ is the maximum of two bounds depending on the two cases for ending the Second Fold Phase. In Case~2 we obtain $d(U_L,U) \le 2$ from the Sewing Needle Lemma~\ref{LemmaSewingNeedleFold}. In Case~1, by combining Fact~\ref{FactAlphaSubdivision}~\pref{ItemAlphaCompRestrict} and~\pref{ItemAlphaRestrict} we obtain
\begin{align*}
\abs{\P(U_L^\sharp)}_\Gamma &= \abs{\overline\P_\alpha(U_L^\sharp)}_\Gamma + \abs{\P_\alpha(U_L^\sharp)}_\Gamma \\
& \le (6 \corank(\A) + 4 \abs{\A} - 6)+1 = 6 \corank(\A) + 4 \abs{\A} - 5
\end{align*}
The number $|\P(U_L^\sharp)|_\Gamma$ is equal to the number of pre-images under $g^L \from U_L \to U$ of any point in the interior of $E^\sharp$, because $(g^L)^*(\E^\sharp)$ is the union of exactly one representative of each $\Gamma$-orbit of $\P(U_L^\sharp)$. By applying Proposition~\ref{PropsEdgeInverseDiamBound} we obtain
$$d(U_L,U) \le 4 \cdot (6 \corank(\A) + 4\abs{\A} - 5) = 24 \corank(\A) + 16 \abs{\A} - 20
$$ 
\end{proof}

\bibliographystyle{amsalpha} 


\bibliography{/Users/Lee/Dropbox/Handel_Lyman_Mosher_shared/Lee_bibtex_file/mosher.bib} 

\end{document}